\documentclass[english,a4paper,11pt]{article}
\usepackage[utf8]{inputenc}
\DeclareUnicodeCharacter{2212}{-}
\usepackage[T1]{fontenc}
\usepackage{inputenc}
\usepackage[english]{babel}
\usepackage{csquotes}  
\usepackage[breaklinks=true]{hyperref}
\usepackage{graphicx}
\usepackage{lmodern}
\usepackage{geometry}
\usepackage{amsmath}
\usepackage{amsthm}
\usepackage{indentfirst}
\usepackage{algorithm}
\usepackage[noend]{algpseudocode}
\usepackage{tabto}
\usepackage{color}
\usepackage[nottoc]{tocbibind}
\usepackage{systeme,mathtools}
\definecolor{bf}{rgb}{0,0,0.6} 
\usepackage{amssymb}
\usepackage{amsmath}
\usepackage{stmaryrd}
\usepackage{fullpage}
\usepackage{eso-pic}
\usepackage{bbold}
\usepackage{float}
\usepackage{titlesec}
\usepackage{bigints}
\usepackage{fancyhdr}
\usepackage{mathrsfs}
\usepackage[shortlabels]{enumitem}
\usepackage{scalerel,stackengine}
\usepackage{yhmath}
\usepackage{verbatim}
\usepackage{authblk}
\usepackage{needspace}
\usepackage{graphicx}
\usepackage{subcaption}
\usepackage{titlesec}
\usepackage{cases}
\usepackage{hyperref}
\usepackage{chngcntr}
\SetSymbolFont{stmry}{bold}{U}{stmry}{b}{n}
\providecommand{\keywords}[1]
{
	\small	
	\textbf{\textit{Keywords---}} #1
}
\newcommand\reallywidehat[1]{%
	\savestack{\tmpbox}{\stretchto{%
			\scaleto{%
				\scalerel*[\widthof{\ensuremath{#1}}]{\kern.1pt\mathchar"0362\kern.1pt}%
				{\rule{0ex}{\textheight}}
			}{\textheight}% 
		}{2.4ex}}%
	\stackon[-6.9pt]{#1}{\tmpbox}%
}

\newtheorem{dft}{Definition}[section]
\newtheorem{te}[dft]{Theorem}
\newtheorem{rem}[dft]{Remark}
\newtheorem{prop}[dft]{Proposition}
\newtheorem{nota}[dft]{Notation}
\newtheorem{cor}[dft]{Corollary}
\newtheorem{lemm}[dft]{Lemma}
\newtheorem{ex}[dft]{Example}

\counterwithin{equation}{subsection}

\hypersetup{
	colorlinks=true,
	linkcolor=blue,
	filecolor=magenta,      
	urlcolor=cyan,
	pdfpagemode=FullScreen,
}

\DeclareMathAlphabet{\mathdutchcal}{U}{dutchcal}{m}{n}

\title{Long-time behaviour of a multidimensional age-dependent branching process with a singular jump kernel modelling telomere shortening}
\author{Jules \textsc{Olayé}$^{1,*}$\and Milica \textsc{Tomašević}$^{2}$}

\date{}

\begin{document}
	
\maketitle
\maketitle
\let\thefootnote\relax\footnotetext{$^{1}$ Institut de Mathématiques de Toulouse, CNRS, Université de Toulouse, 31062 Toulouse, France.}
\let\thefootnote\relax\footnotetext{$^{2}$ CMAP, CNRS, INRIA, École polytechnique, Institut Polytechnique de Paris, 91120 Palaiseau, France.}
\let\thefootnote\relax\footnotetext{$^{*}$ Corresponding author: \href{mailto:jules.olaye@math.univ-toulouse.fr}{jules.olaye@math.univ-toulouse.fr}}
\begin{abstract}
	In this article, we investigate the  ergodic behaviour of a multidimensional age-dependent branching process with a singular jump kernel, motivated by studying the phenomenon of telomere shortening in cell populations. Our model tracks individuals evolving within a continuous-time framework indexed by a binary tree, characterised by age and a multidimensional trait. Branching events occur with rates dependent on age, where offspring inherit traits from their parent with random  increase or decrease in some coordinates, while the most of them are left unchanged. Exponential ergodicity is obtained at the cost of an exponential normalisation, despite the fact that we have an unbounded age-dependent birth rate that may depend on the multidimensional trait, and a non-compact transition kernel. These two difficulties are respectively treated by stochastically comparing our model to Bellman-Harris processes, and by using a weak form of a Harnack inequality. We conclude this study by giving examples where the assumptions of our main result are verified. 
\end{abstract} 
	{\hypersetup{hidelinks} \tableofcontents }
	
	\section{Introduction}\label{sect:introduction}
	In this work we consider a population of individuals evolving in continuous time indexed by a binary tree and characterised by their label in the tree, and a trait $(x,a)\in\left(\mathbb{R}_+^{2k}\times\mathbb{R}_+\right)\cup\left\{(\partial,\partial)\right\}$, where $x$ is the trait that motivates our study, $a$ is the age of the individual, and $(\partial,\partial)$ is a cemetery state. Here, $\partial$ is an arbitrary element outside of both $\mathbb{R}_+^{2k}$ and $\mathbb{R}_+$, and we are interested in cases where~$k\in\mathbb{N}^*$ is large. The individuals branch with a rate depending on their age and at branching events the two daughters inherit the mother's trait $x$ randomly perturbed as follows.  We choose a random number of coordinates of $x$ that go through a random elongation procedure according to a given distribution, and $k$ coordinates of $x$ that go to a shortening procedure according to another distribution, while the rest of the coordinates remain unchanged. Between these jump events, the trait remains also unchanged. The hack is that eventually, due to the shortening procedure, some coordinates of the individual traits will hit zero. When this happens, the individual is moved to the cemetery state~$(\partial,\partial)$.
	
	Our main goal is to establish the long-time convergence of the distribution of the trait in the population towards a stationary profile. Obtaining this result is motivated by the biological phenomenon of telomere shortening.
	
	In what follows, we will structure the introduction in three parts. First, we present in Section~\ref{subsect:motivs}  the biological phenomenon mentioned above and how it motivated us to study the stationary profile of our model. Then in Section~\ref{subsect:introresults}, we discuss in more detail  our long-time convergence result, the main difficulties we encountered, and  the strategy to handle them. Thereafter, in Section~\ref{subsect:example_model}, we present a simple model on which we have established our main result. Finally, in Section~\ref{subsect:paper_organisation}, we provide the organisation of the rest of the paper. 
	
	\subsection{Motivations}\label{subsect:motivs}
	%\paragraph{Informal definition of the model.} %Motivated by studying telomere shortening in a cell population, our model aims to  propose a mathematical framework for it, completing the different models of the phenomenon developed in the recent years \cite{arino_mathematical_1995,benetos_stochastic_2024,benetos_branching_2023,lee_stationary_2020,spoljaric_mathematical_2019,wattis_mathematical_2020}.
	\paragraph{Telomere shortening.}  A telomere is a highly repetitive, and {\it a priori} non-coding, region of DNA situated at both ends of a chromosome. Its role is to protect the terminal regions of chromosomal DNA from progressive shortening and ensure the integrity of chromosomes. During DNA replication that precedes the cell division, enzymes essential for it fail to copy the last nucleotids on one end of each chromosome in the cell. Thus, the presence of telomeres at the ends of chromosomes prevents the rapid loss of genetic information crucial to the functioning of the cell. Instead of losing the nucleotides linked to the genetic information, it is the telomeres that are "shortened".
	
	When the telomeres of a cell are not long enough to protect against the loss of coding DNA, the cell becomes senescent, which is characterised by an irreversible cell cycle arrest. It has been deduced experimentally and with the help of simulations that the shortest telomere of all the chromosomes is responsible for senescence onset in cells, see~\cite{abdallah2009,bourgeron_2015,hemann_2001,Martin2021,Xu2013}. To avoid the state of senescence, certain cells such as budding yeast cells have an enzyme, the telomerase, that regenerates telomere sequence by~"lengthening". It acts just before the cell division, and it is yet unclear whether its action concerns both daughter cells or only one of them. Cells with sufficient telomerase activity are somehow considered immortal. In human somatic cells, this lengthening mechanism is not activated, except for cancer cells~\cite{robinson_telomerase_2022}.
	
	\paragraph{Stationary profile.} One experiment done by biologists is to inactivate the enzyme telomerase of budding yeast cells. This allows them to study cells with a similar functioning as the ones of human somatic cells, in the absence of cancer. When this experiment is done, biologists first wait for the telomere lengths distribution to stabilise towards a stationary profile before inactivating the enzyme. The advantage is that it provides a similar starting point between the different experiments, making them easier to compare and to reproduce.
	
	%bourgeron_2015,eugene_effects_2017,Martin2021,rat_individual_2023
	In the biological/biomathematical studies where this experiment is performed or simulated \cite{bourgeron_2015,eugene_effects_2017,Martin2021,rat_individual_2023,Xu2013,xu_two_2015}, only few theoretical results have been obtained to justify this convergence towards a stationary profile. Precisely, a theoretical result has only been established in~\cite{eugene_effects_2017}.
	%but it was on a simple model.}
In this article, the number of removed nucleotides at each division is assumed deterministic. However, it may be assumed stochastic as in~\cite{bourgeron_2015,grassman_stochastic_20,olofsson_branching_2000,Xu2013}, and this seems more realistic for other species than the budding yeast, in which the number of shortened telomeres is more variable~\cite{Ohki2001}. In addition, in~\cite{eugene_effects_2017}, the lengthening value distribution is modelled in a very simple way with a geometric random variable. In fact, no clear distribution has been identified to represent the telomere lengthening value, and the way telomeres are lengthened is very complex compared to shortening~\cite{srinivas_telomeres_2020,wellinger_everything_2012}. Obtaining a justification for the existence of a stationary profile with only this way to model lengthening is thus quite restrictive.

That is why our goal in this work is to provide a theoretical justification for the long-time convergence of the telomere lengths distribution towards a stationary distribution on a general model  for telomere shortening. 
The interest is that this result will be robust as the research in the field advances and the models for telomere shortening evolve. %with respect to the evolution of telomere shortening models, for which the complexity increases as research in the field advances. 
We also believe that our result may be useful in other contexts than telomere shortening, where individuals also jump at a certain rate depending on their age. For example, we have these dynamics in selection-mutation models structured in age, see~\cite{calsina_steady_2013,meleard_2019,roget_2019}. 
We mention here that our work is done in parallel to~\cite{benetos_stochastic_2024}, for which the main difference is the discrete setting and the mathematical difficulties encountered, as discussed in Section~\ref{subsect:introresults}. 

%        Hence, there is also mathematical interest in considering a general~model.\jules{Citer Fritsch/Villemonais ici ?}

%1 chromosome 2 telo et pointer vers la section 5 pour des cas plus realistes. 
%\paragraph{\answer{PARAGRAPHE SUR NOS CHOIX DE MODELISATION ? PRESENTATION MODELE AVANT ?}}	

\subsection{Main result, literature and proof strategy}\label{subsect:introresults}

\paragraph{Main result.} The main result of the paper is the long-time behaviour of the population trait distribution of our model: We prove that the first moment semigroup of our process converges exponentially fast towards a stationary profile at the cost of an exponential normalisation, see Theorem~\ref{te:main_result}. 
This long time behaviour of a jump process in a non-compact and branching setting allows us to extend the results of \cite{velleret_exponential_2023} to branching jump processes, with jump rate depending both on the age and of the position of the particle, and which require Lyapunov functions to obtain the ergodic behaviour.
%Furthermore, our careful and sometimes long computations when obtaining the long time behaviour enable us to directly get a density representation of the stationary profile with respect to the Lebesgue measure, without checking additional statements.

Furthermore, we are able to directly get the density representation of the stationary profile with respect to the Lebesgue measure without checking any additional statements. This provides a new method to establish a density representation when applying the results of~\cite{velleret_exponential_2023}.

Finally, we present models for studying the biological phenomenon of telomere shortening (see Section~\ref{subsect:motivs}) with a continuous state space and non-exponential division times. In particular, we exhibit conditions for the parameters of these models which imply a convergence towards a stationary profile, and we compare them with the biological reality. 

\paragraph{Difficulties.} The main difficulty we face is that jumps occur in a continuous trait space, such that the jump kernel is not absolutely continuous with respect to the Lebesgue measure, due to the fact that certain coordinates of our trait stay fixed during jumps. This implies that the operator associated with the infinitesimal generator of our model is not compact. Hence, it is impossible to use methods based on Krein-Rutman theorem as for instance in \cite{doumic07} or \cite{perthame07}, or even \cite{bertoin_2019}. In addition, for methods based on non-conservative versions of Harris' ergodic theorem as those presented in \hbox{\cite{champagnat_2016,champagnat_2020,champagnat_2021,champagnat_general_2023,bansaye2022}}, and applied for example in \cite{tomasevic2022}, the control of the asymptotic comparison of survival is not trivial. For more information, we refer to Section~\ref{subsubsect:assumption_(A3)F}. 

The second difficulty we face is the age structure of our model, and the fact that the birth rate is not bounded from above in the age variable, nor is it bounded away from $0$. Therefore, obtaining fine exponential estimates for the growth of the total number of individuals or for the rate at which individuals come back to compact sets is arduous. The usual strategies of dominating the jump times with exponential random variables, as done in~\cite{velleret_unique_2022,velleret_exponential_2023,champagnat_general_2023,martinez_existence_2014}, or of bounding the infinitesimal generator and then applying Gronwall’s lemma, see \cite[Prop.~$2$]{bansaye2022} and~\cite[Theo.~$5.1$]{champagnat_general_2023}, fail in our context. Moreover, even in the bounded rate case, these strategies do not give precise enough estimates necessary for our purpose, and we detail why in the introduction of Section~\ref{subsubsect:assumption_(A2)}. The main tool to circumvent this issue is  to use Bellman-Harris processes, see~\cite[Chap.~IV]{athreya_1972}, and more precisely the renewal structure behind these processes. The way these processes are used is presented in the paragraph on our proof strategy.%see~\cite[Chap. IV]{athreya_1972} and
%, which correspond to the most classical age-dependent branching processes

% For more information, we refer to \answer{the proofs given in} Sections~\ref{subsubsect:proof_prop_renewal_set},~\ref{subsubsect:proof_bound_tail_probability} and~\ref{subsubsect:proof(A3)F_control_time}.

%The way these processes are used in . 

\paragraph{Literature.} 

Studying long time behaviour of Markov processes has received considerable attention in the literature since the seminal work of \cite{meyn_markov_1993}. Usually, the authors search for one criteria to verify in a general setting in order to obtain the ergodic behaviour with exponential speed of convergence in total variation norm, see e.g.~\cite{bansaye_ergodic_2020,bansaye2022,champagnat_2020,champagnat_2021,champagnat_general_2023,del_moral_contraction_2003,douc_quantitative_2004,ferre_more_2021,hairer_asymptotic_2011,hairer_yet_2011,kontoyiannis_large_2005,meleard_2012,meyn_markov_1993,velleret_unique_2022}. The problem of the discontinuities with respect to the Lebesgue measure is very recent in the literature, and few articles have tackled it, either on the probability or PDE side. On the probability side, the only paper we have found dealing with this type of jump Markov process is the article of Velleret \cite{velleret_exponential_2023}. This paper is related to other works, by the same author and collaborators, see~\hbox{\cite{mariani2022,velleret_mesures_2020,velleret_unique_2022,velleret_adaptation_2023}}. In these works, it is proposed to use stopping times to couple trajectories, with conditions on it allowing the control of the "rare events" on which  the coupling of trajectories fails. This result can be understood as a weak form of Harnack inequality. As in the example presented in~\hbox{\cite[Section~$5$]{velleret_exponential_2023}}, the rare events in our model are cases where the typical particle representing our branching process has not jumped in all the coordinates. On the PDE side, the recent and fine results obtained in \cite{sanchez2023} also handle the issues we have linked to the discontinuities of our jump kernel with respect to the Lebesgue measure. In particular, we refer to~\cite[Section~$12.2$]{sanchez2023} where this type of model has been studied. %In a forthcoming paper, an alternative proof from the PDE side combining regularisation approach and some ideas of the present work will be proposed. 

Many studies about the ergodic behaviour of age-dependent models have been done in the recent years, see for example \cite{bansaye_ergodic_2020,benetos_branching_2023,cormier_renewal_2024,fonte_long_2022,gabriel_measure_2018,gabriel_steady_2019,madrid2023exponential}. In most cases where the model is structured by a second trait, the birth/jump rate is bounded from above and/or below \cite{benetos_branching_2023,cormier_renewal_2024,fonte_long_2022,gabriel_measure_2018,madrid2023exponential}. This condition, for example, enables one to use the infinitesimal generator and prove with ease that a Foster-Lyapunov criterion is satisfied when the Harris's ergodic theorem for conservative semigroups is used~\cite{fonte_long_2022,madrid2023exponential}. We can identify at least two differences between our model and those of these studies. The first one is that as we use a non-conservative version of the Harris's ergodic theorem, we need to check additional assumptions for verifying the Lyapunov criterion, see the condition between $\gamma_1$ and $\gamma_2$ in~\hbox{\cite[Assump. $(F_2)$]{champagnat_general_2023}}. These additional assumptions correspond to the exponential estimates presented before. The second one is that we work with an unbounded birth rate. It is therefore not possible to compare birth/jump times with exponential random variables, and the computations with the infinitesimal generator are more difficult. This also means that we need to study a distorted version of the process in the age variable. We mention that in \hbox{\cite[Section~$3.3.2$]{bansaye_ergodic_2020}}, a non-conservative version of the Harris's ergodic theorem is used to obtain the stationary profile of a model with an unbounded birth rate. However, Lyapunov techniques were not required for their model as there is no other trait than the age, and as the factor before the exponential that bounds the convergence error is independent of the initial distribution of the trait. We also mention that in~\cite{gabriel_steady_2019}, a stationary profile for a model with an unbounded birth rate is also obtained using operator theory and entropy methods. As entropy methods do not give information about the speed of convergence, there was no need to obtain exponential estimates as precise as we require here. To the best of our knowledge, our method based on stochastic comparisons by Bellman-Harris processes to obtain such exponential estimates has not yet been studied. 

As mentioned in Section~\ref{subsect:motivs}, the study of the long-time behaviour of a model representing the same phenomenon as the one considered here has been done in~\cite{benetos_branching_2023}, in parallel to our work. The setting is quite different there, since the trait belongs to a discrete space and since there is no age~structure. An interesting point in~\cite{benetos_branching_2023} is that they propose the way to efficiently  simulate the stationary profile with a particle system.

%The main difference from a mathematical point of view is that the model is in a discrete state space and without age structure. 

\paragraph{Proof strategy.} As mentioned, the result in \cite{velleret_exponential_2023}  works for absorbing Markov processes with only one particle. Moreover, in our model, the factor before the exponential that bounds the convergence error depends on the initial distribution of the trait, see~\eqref{eq:main_result}. To solve these issues, we first weight the first moment semigroup of our branching process with a Lyapunov function, then normalise it, and finally create an auxiliary particle representing this weighted-normalised semigroup. The creation of the auxiliary particle follows the ideas presented in \cite{champagnat_2020}. Another way to represent our semigroup with an auxiliary particle is the one presented in \cite{bansaye2022}. However, it seems less adapted for a birth and death framework with jumps in a multidimensional space like ours. Once we have our auxiliary process, we verify the assumptions presented in \cite{velleret_exponential_2023} to get the ergodic behaviour of this auxiliary particle and then conclude by coming back to our population model. 

One of the main arguments of our proof is to use results coming from renewal theory to obtain precise exponential bounds when verifying the assumptions of~\cite[Theorem~$2.8$]{velleret_exponential_2023}, stated here in Theorem~\ref{te:assumptions_velleret}. To apply the latter, we first return to the branching representation of our model. Then, we do a coupling of this branching representation with Bellman-Harris processes, see Lemma~\ref{lemm:inequality_for_means}, allowing us to stochastically compare the number of individuals in our model by the number of individuals of Bellman-Harris processes. Precisely, this coupling is done by exploiting the fact that the times between each branching in our model and those in the Bellman-Harris processes can be simulated using the same uniform random variables, see the proof in Section~\ref{subsubsect:proof_inequality_means}. Thereafter, using the exponential growth property of Bellman-Harris processes, which comes from the renewal theory (see~\cite[Chap.~IV.4]{athreya_1972}), we establish exponential bounds for the branching representation of our model. Finally, we transfer these bounds to the typical particle. We believe that this method can be applied in many other settings with an age structure. For more information, we refer to the proofs given in Sections~\ref{subsubsect:proof_prop_renewal_set},~\ref{subsubsect:proof_bound_tail_probability} and~\ref{subsubsect:proof(A3)F_control_time}.
%During our proofs, we always juggle between the branching and the particle representation of our model. To be more precise, we return to the branching representation to do our stochastic comparisons to Bellman-Harris processes. For example, to obtain the accumulation of the mass of the semigroup on a compact set, we first bound from above or from below the (weighted) branching process by Bellman-Harris processes, and then we transmit the bounds we have obtained to the auxiliary process. This procedure is also used to obtain the weak form of Harnack inequality, where we also need exponential estimates. Going back and forth from the branching representation to the particle representation is an interesting way to verify the assumptions presented in \cite{velleret_unique_2022,velleret_exponential_2023}, or in \cite{champagnat_2016,champagnat_2020,champagnat_2021,champagnat_general_2023,bansaye2022}.

\subsection{Example of a model}\label{subsect:example_model}
We present here a particular case of the general model we study, introduced later in Section~\ref{subsect:algorithm_model}, as well as the convergence result obtained on this particular case. The first interest is to give to the reader a more concrete idea of the models discussed in this work, since the description given at the beginning of Section~\ref{sect:introduction} was very informal. The second interest is to present the convergence results we aim to establish. The example presented in this section corresponds to a simplified version of the illustrative model given in Section~\ref{subsect:model_renewal_several_generations} (case where $k = 1$). We refer to this illustrative model and to the one presented in Section~\ref{subsect:model_telomere_shortening} for more complex models. 

%The model we developed in Section~\ref{subsect:algorithm_model} is very general, due to our motivations presented in Section~\ref{subsect:motivs}. To facilitate the understanding of it, we present here
% We present here a particular case of this model, which is not necessarily the most biologically relevant, to give an idea to the reader of the type of model we study here. It corresponds to a simplified version of the model presented in Section~\ref{subsect:model_renewal_several_generations} (case where $k = 1$).} 
The model we present is a branching process, where each individual corresponds to a cell. In this model, each cell has only one chromosome, so each cell has $2$ telomeres. We assume that each cell has a trait that belongs to $\left(x,a\right)\in\left(\mathbb{R}_+^{2}\times\mathbb{R}_+\right)\cup\left\{(\partial,\partial)\right\}$. When the trait of the cell is equal to $\left(\partial,\partial\right)$, it means that the cell is senescent. Otherwise, the cell is non-senescent, and $x$ corresponds to the telomere lengths of the cell, whereas $a$ corresponds to its age (which increases linearly with time). In this model, each cell evolves according to the following dynamics:
\begin{itemize}[leftmargin=0.45cm]
\item[$1.$] At a rate $b(a) = a1_{\{a\geq1\}}$, the cell divides. The rate has this form as we wish to mimic the fact that young cells cannot divide and that the rate seems to increase with age~\cite{tzur_cell_2009}.
\item[$2.$] At each division, telomere lengths are updated. To do this, we first draw a pair of random variables $\left(U,U'\right)\in\left([0,\delta]^2\right)^2$, where $\delta >0$, representing the shortening values in the two daughter cells.  This pair is distributed as follows:
\begin{itemize}[leftmargin=*]
	\item  With probability $\frac{1}{2}$, the variables $U_1$ and $U'_2$ are distributed according to an uniform distribution over $[0,\delta]$, while $U_2 = U'_1 = 0$ a.s..
	\item With probability $\frac{1}{2}$, the variables $U_2$ and $U'_1$ are distributed according to an uniform distribution over $[0,\delta]$, while $U_1 = U'_2 = 0$ a.s..
\end{itemize}
The probability measure $\mu^{(S)}$ describing the law of $(U,U')$ is the following:
%$$
%\begin{aligned}
%	\mu^{(S)}\left(du,du'\right) &= \frac{1}{2}\left[\frac{1}{\delta}1_{[0,\delta]}\left(u_1\right)du_1\delta_0\left(du'_1\right)\right]\left[\frac{1}{\delta}1_{[0,\delta]}\left(u_2\right)du_2\delta_0\left(du'_2\right)\right] \\
%	&+ \frac{1}{2}\left[\frac{1}{\delta}1_{[0,\delta]}\left(u'_1\right)du'_1\delta_0\left(du_1\right)\right]\left[\frac{1}{\delta}1_{[0,\delta]}\left(u'_2\right)du'_2\delta_0\left(du_2\right)\right] .
%\end{aligned}
%$$	
$$
\begin{aligned}
	\mu^{(S)}\left(d\alpha_1,d\alpha_2\right) = \frac{1}{2}&\left[\frac{1}{\delta}1_{[0,\delta]}\left((\alpha_1)_1\right)d(\alpha_1)_1\,\delta_0\left(d(\alpha_1)_2\right)\right]\\
	\times&\left[\frac{1}{\delta}1_{[0,\delta]}\left((\alpha_2)_2\right)d(\alpha_2)_2\,\delta_0\left(d(\alpha_2)_1\right)\right] \\
	+ \frac{1}{2}&\left[\frac{1}{\delta}1_{[0,\delta]}\left((\alpha_1)_2\right)d(\alpha_1)_2\,\delta_0\left(d(\alpha_1)_1\right)\right]\\
	\times&\left[\frac{1}{\delta}1_{[0,\delta]}\left((\alpha_2)_1\right)d(\alpha_2)_1\,\delta_0\left(d(\alpha_2)_2\right)\right] .
\end{aligned}	
$$		
%\end{comment}   
At the end of the shortening, the daughter cell $A$ has for telomere lengths $x-U$, and the daughter cell $B$ has for telomere lengths $x-U'$.

Then, we draw a pair of random variables $\left(V,V'\right)\in\left([0,\Delta]^2\right)^2$, where $\Delta >0$. These random variables represent the lengthening values in the two daughter cells. For all~$i\in\{1,2\}$, we have that $V_i$~is distributed according to an uniform distribution over the set~$\left[0,\Delta\max\left(1,x_i-U_i+1\right)\right]$, and that $V'_i$ is distributed according to an uniform distribution over~$\left[0,\Delta\max\left(1,x_i-U'_i+1\right)\right]$. Hence, the probability measure~$\mu_{(x-U_1,x-U_2)}^{(E)}$ describing the distribution of $(V,V')$ is
$$
\begin{aligned}%\times
	\mu_{\left(x-U,x-U'\right)}^{(E)}\left(d\beta_1,d\beta_2\right) &= \left[\prod_{i = 1}^2\frac{1_{\left[0,\Delta\max\left(1,x_i-U_i+1\right)\right]}\left((\beta_1)_i\right)d(\beta_1)_i}{\Delta\times\max\left(1,x_i-U_i+1\right)}\right]\\
	&\times\left[\prod_{i = 1}^2\frac{1_{\left[0,\Delta\max\left(1,x_i-U'_i+1\right)\right]}\left((\beta_2)_i\right)d(\beta_2)_i}{\Delta\times\max\left(1,x_i-U'_i+1\right)}\right].
\end{aligned}
$$
We obtain after this step that the telomere lengths of the daughter cells $A$ and~$B$ are respectively~$x-U+V$ and~$x-U'+V'$.

\item[$3.$] After telomere lengths have been updated, we first check if $x-U+V\in \mathbb{R}_+^{2}$ or~not. If this is the case, then the trait of the daughter cell $A$ at birth is $\left(x-U+V,0\right)$. Otherwise, the trait of this cell is~$\left(\partial,\partial\right)$. 

Then, we do the same for the daughter cell $B$. If $x-U'+V'\in \mathbb{R}_+^{2}$, then the trait of the daughter cell $B$ at birth is $\left(x-U'+V',0\right)$ . Otherwise, the trait of this cell is~$(\partial,\partial)$.

\end{itemize}
According to our computations done in Section~\ref{subsect:model_renewal_several_generations}, and in particular to the condition given in~\eqref{eq:assumption_Delta_delta_first_model}, we are  able to prove that the distribution of the trait in this model converges towards a stationary profile, as $t\to \infty$, when $\Delta > \left[1-4^{-\frac{1}{6}}\right]^{-1}\delta$. Precisely, for fixed $t\geq0$, $A$ Borel set in $\mathbb{R}_+^2\times\mathbb{R}$ and $(x,a)\in\mathbb{R}_+^2\times\mathbb{R}_+$, let us denote the quantity $M_t(1_A)(x,a)$, that corresponds to the expected number of cells with trait in $A$ when the model starts from one cell with trait $(x,a)$. Then, we are  able to prove that when it holds~$\Delta > \left[1-4^{-\frac{1}{6}}\right]^{-1}\delta$, there exist  $\lambda >0$ and two functions $N : \mathbb{R}_+^2\times\mathbb{R}_+ \rightarrow \mathbb{R}_+$ and~$\phi : \mathbb{R}_+^2\times\mathbb{R}_+ \rightarrow \mathbb{R}_+^*$  verifying 
$$
\int_{ \mathbb{R}_+^2\times\mathbb{R}_+} N(y,s) dy ds = \int_{\mathbb{R}_+^2\times\mathbb{R}_+} N(y,s)\phi(y,s) dy ds = 1,
$$
such that 
$$
\lim_{t\rightarrow+\infty} e^{-\lambda t}M_t(1_A)(x,a) = \phi(x,a) \int_{\mathbb{R}_+^2\times\mathbb{R}_+} 1_A(y,s)N(y,s) dy ds.
$$
Moreover, we are able to prove that the convergence holds exponentially fast. This is the type of result we can obtain, even in more general models, see Theorem~\ref{te:main_result}. It characterises the convergence of the distribution of the trait towards a stationary profile $N$, since in view of the definition of~$\left(M_t\right)_{t\geq0}$, this collection gives information about the distribution of the trait over time.

\noindent \subsection{Organisation of the paper}\label{subsect:paper_organisation}
We now turn to the core of the paper. First, in Section \ref{sect:preliminaries_results}, we present the main notions we use in the article, we give in detail the model under consideration, and  we present the main result. Then, in Section~\ref{sect:auxiliary_process} we present the auxiliary processes we use to apply the result of \cite{velleret_exponential_2023}, and prove the main result of the article  in Section \ref{sect:long_time_behaviour}. Thereafter, in Section~\ref{sect:assumptions_verified}, we give conditions for which the assumptions necessary to use our theorem are satisfied, and present two models where they are verified. Finally, in Section~\ref{sect:discussion}, we discuss the limits of our results and present the perspectives of this work. The appendices~\hyperlink{appendix:everything}{A-D} are devoted to the proof of certain auxiliary statements given during this paper.

\section{Notations, preliminaries and main result}\label{sect:preliminaries_results}
Taking into account our biological motivation, from now on, the individuals in the population process described at the beginning of Section~\ref{sect:introduction} are called cells, as in Section~\ref{subsect:example_model}. Now, each cell has $k$ chromosomes, and each extremity of each chromosome corresponds to one telomere, so the cell has $2k$ telomeres. The traits~\hbox{$(x,a) \in \mathbb{R}_+^{2k}\times\mathbb{R}_+$} of the cells that are not in the cemetery state are the lengths of their telomeres and their age. Branching events are called divisions, and at each division, certain telomere lengths are updated~("shortened" or "lengthened") with jumps in a continuous trait, while the others stay fixed. When the length of a telomere is below zero, we usually say that the cell becomes senescent, meaning that the individual is moved to the cemetery state.

The aim of this section is to present the main result of the paper: the exponential convergence of the first moment semigroup towards a stationary profile. To present this result, it is necessary to introduce a certain number of notions. Hence, the first part of this section is devoted to the presentation of the notations and the model that we use throughout the paper (Sections \ref{subsect:notations} and \ref{subsect:algorithm_model}). We then present the assumptions and the main result in Section \ref{subsect:main_result}.

\subsection{Notations and first definitions}\label{subsect:notations}
We present here the notations and mathematical notions we use throughout the paper. The notations and notions are presented in nine parts.  In each of these parts,  $\left(\mathbb{X},Top(\mathbb{X})\right)$ denotes an arbitrary topological space. The first part we give introduces our probability space, as well as the notations we use for the Borel $\sigma$-algebra of $\mathbb{X}$ and the distribution of a random variable.

\paragraph{$(N_1)\!:$ Probability space and related notations.}

\begin{enumerate}[start=1,label={$(N\textsubscript{1.\arabic*})\!:$},ref={$(N\textsubscript{1.\arabic*})$}, leftmargin=*]
\item We work on the probability space $(\Omega,\mathcal{A},\mathbb{P})$. \label{notation:probability_space}

\item We denote by $\mathcal{B}(\mathbb{X})$ the $\sigma-$algebra of the Borel sets of $(\mathbb{X},Top(\mathbb{X}))$.

\item For any random variable $Z : (\Omega,\mathcal{A},\mathbb{P}) \longrightarrow (\mathbb{X},\mathcal{B}(\mathbb{X}))$, we denote $\mathbb{P}_Z$ the probability measure such that $\mathbb{P}_Z = \mathbb{P} \circ Z^{-1}$. \label{notation:pushforward_measure}
\end{enumerate}
\noindent The second part introduces our trait space and the Ulam-Harris-Neveu notation, which is classical to label individuals in a branching process.
\paragraph{$(N_2)\!:$ Trait space and Ulam-Harris-Neveu notation.}
\begin{enumerate}[start=1,label={$(N\textsubscript{2.\arabic*})\!:$}, ref={$(N\textsubscript{2.\arabic*})$},leftmargin=*]
\item We consider $\mathcal{X} = \mathbb{R}_+^{2k}\times\mathbb{R}_+$. This space represents the trait space for non-senescent cells. 

\item We consider $\partial$ an arbitrary element outside of $\mathbb{R}^d$ for all $d\in \mathbb{N}^*$. \label{nota:cemetery}

\item Let us denote for all $n\in\mathbb{N}^*$ the tuple \hbox{$(\partial)_n := (\partial,\hdots,\partial)$} where there is $n$ times the term $\partial$. Each time it is used, $(\partial)_n$ represents a cemetery for the Cartesian product between $n$ sets $\left(A_i\right)_{1\leq i \leq n}$. We use the set \hbox{$\left(A_1\times\hdots\times A_n\right)\cup\{(\partial)_n\}$} instead of~\hbox{$\left(A_1\times\hdots\times A_n\right)\cup\{\partial\}$}, because this is meaningless to consider a tuple~$(a_1,\hdots,a_n)\in \{\partial\}$.

\item We consider $\mathcal{X}_{\partial} = \mathcal{X}\cup \left\{(\partial)_2\right\}$, such that $(\partial)_2$ is an isolated point of $\mathcal{X}_{\partial}$. This space represents the trait space of individuals, and $(\partial)_2$ represents the trait of senescent individuals.

\item Let $\mathcal{U} = \bigcup_{n\in\mathbb{N}}\{\mathbb{N}^{*}\times\{1,2\}^n\}$. As we use a branching process, this set allows us to denote the individuals in the tree thanks to the classical Ulam-Harris-Neveu notation. In particular, for all $\mathdutchcal{u} = \left(\mathdutchcal{u}_1,\hdots,\mathdutchcal{u}_{m}\right)\in \mathcal{U}$ and $i\in\{1,2\}$, where $m\in\mathbb{N}^*$, we write $\mathdutchcal{u}i = \left(\mathdutchcal{u}_1,\hdots,\mathdutchcal{u}_{m},i\right)$. \label{nota:ulam_harris_neveu}
\end{enumerate}
The third part introduces all the spaces and sets of functions we use during this work.
\paragraph{$(N_3)\!:$ Spaces/sets of functions.}
\begin{enumerate}[start=1,label={$(N\textsubscript{3.\arabic*})\!:$},ref={$(N\textsubscript{3.\arabic*})$},leftmargin=*]

\item We denote by $M\left(\mathbb{X}\right)$ the space of Borel functions $f : \mathbb{X} \longrightarrow \mathbb{R}$, by $M_b\left(\mathbb{X}\right)$ the space of bounded Borel functions $f : \mathbb{X} \longrightarrow \mathbb{R}$, and by  $M_b^{loc}\left(\mathbb{X}\right)$ the space of locally bounded Borel functions \mbox{$f : \mathbb{X} \longrightarrow \mathbb{R}$}.

\item For all $n\in\mathbb{N}^n$ and $S \subset \mathbb{R}^n$, we denote by $L^1(S)$ the set of the functions \hbox{$f : S \longrightarrow \mathbb{R}$} that are integrable with respect to the Lebesgue measure. We also denote by $L_{\text{loc}}^1(S)$ the space of the functions $f : S \longrightarrow \mathbb{R}$ that are locally integrable with respect to the Lebesgue measure.		

\item For all $n\in\mathbb{N}^n$ and $S \subset \mathbb{R}^n$, we denote by $L_{\text{p.d.f.}}^1(S)$ the set of the measurable functions $f : S \longrightarrow \mathbb{R}_+$ such that~$\int_{S} f(x) dx = 1$. \label{notations:L1_pdf}

\item For all $n\in\mathbb{N}^n$ and $S \subset \mathbb{R}^n$, we denote by $\mathcal{C}_c^{\infty}(S)$ the space of $\mathcal{C}^{\infty}$ functions defined on $S$, taking values in $\mathbb{R}$, and with compact support.

\item For all $n\in\mathbb{N}^n$, $S \subset \mathbb{R}^n$, and Borel function $\psi: S \longrightarrow \mathbb{R}$, we denote by $\mathdutchcal{B}(\psi)$ the space of the Borel functions $f: S\longrightarrow \mathbb{R}$ such that
$$
\sup_{x\in S} \left|\frac{f(x)}{\psi(x)}\right|< +\infty.
$$
We endow $\mathdutchcal{B}(\psi)$  with the norm $||f||_{\mathdutchcal{B}(\psi)} := \sup_{x\in S}\left|\frac{f(x)}{\psi(x)}\right|$. \label{notation:start_nota_space_distortion}

\item For all $n\in\mathbb{N}^n$, $S \subset \mathbb{R}^n$  and Borel function $\psi : S \longrightarrow \mathbb{R}$, we denote
$$
L^1(\psi) := \left\{f:S \longrightarrow \mathbb{R}\,\big|\,\int_{S}|f(z)\psi(z)| dz < +\infty\right\}.
$$

\item For all $S \subset\mathcal{X}_{\partial}$ we denote by $\mathcal{C}_b^{m,1}\left(S\right)$ the space of bounded Borel functions \hbox{$f : (x,a)\in S \longrightarrow \mathbb{R}$} verifying the following properties: \label{nota:Cm1}
\begin{itemize}[leftmargin=0.5cm]
	\item $f$ is measurable in the first variable,
	\item The restriction of $f$ on $\mathcal{X}$ is continuously differentiable in the second variable, with bounded derivative in the second variable.
\end{itemize}

\item For all $S \subset\mathcal{X}_{\partial}$ we denote by $\mathcal{C}_b^{1,m,m,1}(\mathbb{R}_+\times\mathcal{U}\times S)$ the space of bounded Borel functions \hbox{$f:(t,\mathdutchcal{u},x,a)\in \mathbb{R}_+\times\mathcal{U}\times S\longrightarrow \mathbb{R}$}, verifying the following properties:
\begin{itemize}[leftmargin=0.5cm]
	\item $f$ is continuously differentiable in the first variable, with bounded derivative in the first variable.
	
	\item $f$ is measurable in the second and third variables.
	
	\item  The restriction of $f$ on $\mathbb{R}_+\times\mathcal{U}\times \mathcal{X}$ is continuously differentiable in the fourth variable, with bounded derivative in the fourth variable.
\end{itemize}
\end{enumerate}
The fourth part presents operations we apply to functions during this work.
\paragraph{$(N_4)\!:$ Operations on functions.}
\begin{enumerate}[start=1,label={$(N\textsubscript{4.\arabic*})\!:$},ref={$(N\textsubscript{4.\arabic*})$},leftmargin=*]
\item We consider for all $f : \mathbb{R}_+  \longrightarrow \mathbb{R}$, $g : \mathbb{R}_+  \longrightarrow \mathbb{R}$ and $n\in\mathbb{N}$, the operation $f*^{(n)}g = f*f*\hdots*f*g$, where there is $n$ times the function $f$. We use the convention that $f*^{(0)}g = g$.\label{notation:convolution} 

\item For any measurable function $f : \mathbb{R}_+ \longrightarrow \mathbb{R}$, we denote by $\mathcal{L}(f)$ its Laplace transform, defined such that~$\mathcal{L}(f)(p) = \int_0^{+\infty}e^{-pt}f(t)dt$ for all $p\in\mathbb{C}$ such that the integral converges. \label{nota:laplace_transform}

\item For any probability density function $F : \mathbb{R}_+ \longrightarrow \mathbb{R}_+$, we denote by $\overline{F}$ its associated complementary cumulative distribution function.
\end{enumerate}
The fifth part provides all the spaces and sets of measures we work with, as well as notations related to them.
\paragraph{$(N_5)\!:$ Spaces/sets of measures, and related notations.}
\begin{enumerate}[start=1,label={$(N\textsubscript{5.\arabic*})\!:$},ref={$(N\textsubscript{5.\arabic*})$},leftmargin=*]
\item We denote by $\mathcal{M}(\mathbb{X})$ the set of positive measures on $(\mathbb{X},\mathcal{B}(\mathbb{X}))$.

\item We denote by $\mathcal{M}_1(\mathbb{X}) \subset \mathcal{M}(\mathbb{X})$ the set of probability measures on $(\mathbb{X},\mathcal{B}(\mathbb{X}))$.

\item We denote by $\mathcal{M}_P\left(\mathbb{X}\right)$ the set of all finite non-negative point measures on $\mathbb{X}$.

\item \label{notation:end_nota_space_distortion} For any Borel function $\psi: \mathbb{X} \longrightarrow \mathbb{R}$ such that $\inf_{x\in \mathbb{X}} \psi(x) >0$, we denote by~$\mathdutchcal{M}_+(\psi)$ the cone of positive measures that integrate $\psi$. We also denote by \hbox{$\mathdutchcal{M}(\psi) = \mathdutchcal{M}_+(\psi) - \mathdutchcal{M}_+(\psi)$} the set containing all the differences of measures in~$\mathdutchcal{M}_+(\psi)$. We endow this set with the norm~$||.||_{\mathdutchcal{M}(\psi)}$. This norm is defined for all~\hbox{$\mu = \mu_+ - \mu_-\in\mathdutchcal{M}(\psi)$} as
$$
||\mu||_{\mathdutchcal{M}(\psi)} := \sup_{f\in\mathdutchcal{B}(\psi),\,||f||_{\mathdutchcal{B}(\psi)}\leq 1} |\mu(f)| = \mu_+(\psi) + \mu_-(\psi).
$$

%\end{enumerate}
%The sixth set of notations corresponds to notations related to measures we use in this work.
%\paragraph{$(N_6):$ Notations related to measures.}
%\begin{enumerate}[start=1,label={$(N\textsubscript{6.\arabic*}):$},ref={$(N\textsubscript{6.\arabic*})$},leftmargin=1.75cm]

\item For every $\mu\in\mathcal{M}\left(\mathbb{X}\right)$ and $f : \mathbb{X}\longrightarrow \mathbb{R}$ integrable with respect to $\mu$, we write
$$
\mu(f) := \int_{x\in\mathbb{X}} f(x) \mu(dx).
$$
We also write $\mu(1) := \int_{x\in\mathbb{X}}1\mu(dx)$.

\item $||.||_{TV,\mathbb{X}}$ is the total variation norm on $\mathcal{M}(\mathbb{X})$. For any $\mu\in\mathcal{M}(\mathbb{X})$, 
$$
||\mu||_{TV,\mathbb{X}} := \sup_{f : \mathbb{X}\longrightarrow [-1,1],\,f\text{ mes.}} \left|\mu(f)\right|.
$$

\item Let $(P_t)_{t\geq 0}$ a family of linear maps from a set $S \subset M\left(\mathbb{X}\right)$ to itself. Then, for all measure \hbox{$\mu \in \mathcal{M}(\mathbb{X})$}, $f \in S$ non-negative and $t\geq0$, we write
$$
\mu P_t(f) :=  \int_{z \in \mathbb{X}} P_t(f)(z) \mu(dz).
$$
The above notation can be extended to the case where $P_t(f)$ is integrable with respect to $\mu$, and $f$ not necessarily non-negative.
\end{enumerate}
The sixth part corresponds to notations we use when we work with sets or subsets.
\paragraph{$(N_6)\!:$ Notations related to set theory.}
\begin{enumerate}[start=1,label={$(N\textsubscript{6.\arabic*})\!:$}, ref={$(N\textsubscript{6.\arabic*})$},leftmargin=*]

\item For any $S \subset \mathbb{X}$, we denote $S^c = \mathbb{X} \backslash S$.

\item For any finite set $S \subset \mathbb{X}$, we denote $\#(S)$ the cardinal of $S$.

\item For all $S \subset \mathbb{X}$, $\text{int}(S)$ is the interior of $S$. 

\item For all $(a,b)\in\mathbb{Z}^2$ such that $a < b$, we denote 
$$
\llbracket a,b\rrbracket := \left\{a,a+1,\hdots,b-1,b\right\}.
$$

\item For any set $X$, we denote $\mathcal{P}(X)$ the power set of $X$ i.e. the set containing all subsets of $X$. 

\item For any finite set $X$, and $l\in\llbracket0,\#(X)\rrbracket$, we denote \hbox{$\mathcal{P}_l(X) \hspace{-0.05mm}= \hspace{-0.05mm} \{S\in\mathcal{P}(X)|\#S = l\}$} the set containing all subsets of $X$ of size $l$.
\end{enumerate}
The seventh part presents notations we use when we work with stochastic processes, as for example in Section~\ref{sect:long_time_behaviour}.
\paragraph{$(N_7)\!:$ Notations related to stochastic processes.}
\begin{enumerate}[start=1,label={$(N\textsubscript{7.\arabic*})\!:$}, ref={$(N\textsubscript{7.\arabic*})$},leftmargin=*]
\item For any stochastic process $(X_t)_{t\geq 0}$ defined on $\mathbb{X}$, we write for all $z\in\mathbb{X}$, $f \in M(\mathbb{X})$ non-negative, $t\geq 0$,
$$
\mathbb{E}_{z}\left[f(X_t)\right] := \mathbb{E}\left[f(X_t)\,|\,X_0 = z\right],
$$
and for any $\mu\in\mathcal{M}_1(\mathbb{X})$
$$
\mathbb{E}_{\mu}\left[f(X_t)\right] := \int_{z\in\mathbb{X}}\mathbb{E}_{z}\left[f(X_t)\right] \mu(dz).
$$
These notations can be extended to the case where $f\in M_b\left(\mathbb{X}\right)$ and not necessarily non-negative.

\item  Let $(X_t)_{t\geq 0}$ a stochastic process defined on $\mathbb{X}$, $z\in\mathbb{X}$ and $t\geq0$. Then, for any random variable~$W$, $A$ a subset of the set where $W$ takes its values, and $f\in M(\mathbb{X})$ non-negative, we denote
$$
\mathbb{E}_{z}\left[f(X_t); W\in A\right] := \mathbb{E}_{z}\left[f(X_t)1_{\{W\in A\}}\right].
$$	
This notation can be extended to the case where $f\in M_b\left(\mathbb{X}\right)$ and not necessarily non-negative.
\end{enumerate}
The eighth part introduces the initial condition of our model, as well as the Poisson point measure used to define it, see~\eqref{eq:SDE_model_intro}.
\paragraph{$(N_8)\!:$ Initial condition and Poisson point measure.}
\begin{enumerate}[start=1,label={$(N\textsubscript{8.\arabic*})\!:$}, ref={$(N\textsubscript{8.\arabic*})$},leftmargin=*]
\item We consider $Y_0$ a $\mathcal{M}_P\left(\mathcal{U}\times\mathcal{X}_{\partial}\right)$-valued random variable. This random variable represents the initial condition of the population. Its distribution will be specified when needed. \label{notation:initial_condition_process}

\item We consider $K\in\mathbb{N}^*$ and $N(ds, d\mathdutchcal{u}, dz, d\theta)$ a Poisson point measure on the set $\mathbb{R}_+\times\mathcal{U}\times\mathbb{R}_+\times[0,1]^K$ with intensity measure $ds\times\left(\sum_{\mathdutchcal{u}\in\mathcal{U}} \delta_{\mathdutchcal{u}}(du)\right)\times dz\times d\theta$. This random measure allows us to describe the jumps that occur in our dynamics. The integer $K$ is the number of uniform random variables needed to update telomere~lengths. \label{notation:poisson_process}

\item We consider $\left(\mathcal{F}_t\right)_{t\geq0}$ the canonical filtration generated by $N(ds, d\mathdutchcal{u}, dz, d\theta)$ and~$Y_0$. 
\end{enumerate}
The last part corresponds to notations that cannot be classified in any of the other parts.
\paragraph{$(N_{9})\!:$ Other notations.}
\begin{enumerate}[start=1,label={$(N\textsubscript{9.\arabic*})\!:$}, ref={$(N\textsubscript{9.\arabic*})$},leftmargin=*]
\item We use the convention that $\prod_{\emptyset} = 1$ and $\inf(\emptyset) = +\infty$ in all the paper.

\item Let $n\in\mathbb{N}^*$. For all $i\in\llbracket1,n\rrbracket$, $e_i$ is the $i-th$ vector in the canonical basis of~$\mathbb{R}^n$.
\end{enumerate}
\color{black}
\subsection{Presentation of the model}\label{subsect:algorithm_model}
We work in continuous time, and we use an individual-based approach where each individual is a cell. When a cell divides, it gives birth to two daughter cells $A$ and $B$ (the choice of which daughter cell is $A$ or $B$ is arbitrary). The traits of each individual are the lengths of its telomeres and its age, or~$(\partial)_2$ when it is a senescent cell. Each cell has $2k$ telomeres, one per chromosome end. We associate each coordinate with the length of a telomere, such that for all $i\in\llbracket1,k\rrbracket$, the coordinates $i$ and $i+k$ of a vector in~$\mathbb{R}_+^{2k}$ represent the lengths of the two telomeres of the same chromosome. Let us now introduce the following objects needed to construct our model.
\begin{itemize}[leftmargin=*]
\item A division rate $b : (x,a) \in \mathbb{R}_+^{2k}\times\mathbb{R}_+  \longrightarrow \mathbb{R}_+$, and constants $\tilde{b} >0$, $d_b\in\mathbb{N}^*$, $b_0>0$,~$a_0> 0$. We assume that the birth rate satisfies
\begin{equation}\label{eq:birth_rate_assumption}
	\begin{aligned}
		&\forall (x,a)\in\mathbb{R}_+^{2k}\times\mathbb{R}_+:\hspace{1.5mm}\,|b(x,a)| \leq \tilde{b}(1+a^{d_b}),\\
		&\forall (x,a)\in\mathbb{R}_+^{2k}\times[a_0,+\infty[:\hspace{1.5mm}b(x,a)\geq b_0.
	\end{aligned}
\end{equation}

\item A probability measure $\mu^{(S)}\in \mathcal{M}_1\left(\mathbb{R}_+^{2k}\times\mathbb{R}_+^{2k}\right)$ representing the distribution of telomere shortening at each cell division. It is assumed independent of telomere lengths. When we draw a pair $(U,U')$ distributed according to $\mu^{(S)}$, $U$ corresponds to the shortening in the daughter cell $A$, and $U'$ the shortening in the daughter cell~$B$. The complete expression of this measure is given later, see~\eqref{eq:measure_shortening}.

\item For all $(s_1,s_2)\in\mathbb{R}^{2k}\times\mathbb{R}^{2k}$, a probability measure $\mu_{(s_1,s_2)}^{(E)}\in\mathcal{M}_1\left(\mathbb{R}_+^{2k}\times\mathbb{R}_+^{2k}\right)$ which represents the distribution of telomere elongation by telomerase, when the two daughter cells $A$ and $B$ have respective telomere lengths $s_1$ and $s_2$ after shortening. When we draw a pair $(V,V')$ distributed according to $\mu_{(s_1,s_2)}^{(E)}$, $V$ corresponds to the lengthening in the daughter cell $A$, and $V'$ the lengthening in the daughter cell $B$. As for $\mu^{(S)}$, its complete expression is given later, see \eqref{eq:measure_elongation}.

\item For every $x\in\mathbb{R}_+^{2k}$, we define the measure $\Pi_x$ such that $\forall C\in\mathcal{B}((\mathbb{R}_+^{2k})^2)$:
\begin{equation}\label{eq:kernel_telomere_updating}
	\begin{aligned}
		\Pi_x(C) := &\int_{(y_1,y_2)\in\left(\mathbb{R}_+^{2k}\right)^2} \bigg[\int_{(z_1,z_2)\in\left(\mathbb{R}_+^{2k}\right)^2} 1_{C}(-y_1 + z_1,-y_2+z_2)\\
		&\times d\mu^{(E)}_{(x - y_1,x-y_2)}(z_1,z_2)\bigg] d\mu^{(S)}(y_1,y_2).
	\end{aligned}
\end{equation}
This is the kernel used to represent the distribution of telomere lengths variation (shortening + lengthening), when a cell with telomere lengths $x\in \mathbb{R}_+^{2k}$ divides.

\item For all $x\in\mathbb{R}_+^{2k}$, we assume that there exists a function \hbox{$R(x,.) : [0,1]^K \longrightarrow \left(\mathbb{R}^{2k}\right)^2$}, where $K$ is defined in Notations~\hyperref[notation:poisson_process]{\ref*{notation:poisson_process}}, such that for all $f: \left(\mathbb{R}^{2k}\right)^2 \longrightarrow \mathbb{R}$ bounded measurable
\begin{equation}\label{eq:uniform_random_variables_assumption}
	\begin{aligned}
		\int_{[0,1]^K} f\left(R(x,\theta)\right) d\theta &=
		\int_{\left(w_1,w_2\right)\in\left(\mathbb{R}^{2k}\right)^2} f(w_1,w_2)  d\Pi_x(w_1,w_2).
	\end{aligned}
\end{equation}
The latter represents a procedure to simulate random variables distributed according to $\Pi_x$ with uniform random variables, see \cite[p. 2655]{marguet2019}.
\end{itemize}
We now give the complete expression for the measures $\mu^{(S)}$ and $\left(\mu_{(s_1,s_2)}^{(E)}\right)_{(s_1,s_2)\in\left(\mathbb{R}^{2k}\right)^2}$. The expression of these measures comes from biological modelling. 

We start with $\mu^{(S)}$. At each cell division, for each chromosome $i\in\llbracket1,k\rrbracket$, we have in view of~\hbox{\cite[Fig.$1$]{eugene_effects_2017}} that:
\begin{itemize}[leftmargin=0.4cm]
\item One of the daughter cells, which can be $A$ or $B$, has the telomere represented by the coordinate $i$ shortened (corresponds to one end of the chromosome $i$), and the telomere represented by the coordinate $i+k$ not shortened (corresponds to the other end of the chromosome $i$).

\item The other daughter cell has the telomere represented by the coordinate $i+k$ shortened, the telomere represented by the coordinate $i$ not shortened.
\end{itemize}
Then, to write an explicit expression for $\mu^{(S)}$, we denote
\begin{equation}\label{eq:set_combination_shortening}
\mathcal{I}_{k} := \bigg\{I\in \mathcal{P}_k\left(\llbracket1,2k\rrbracket\right)  \left.\big|\right. \forall (i,j)\in I^2: i\neq j \implies \left(i \,\text{mod}\, k\right) \neq \left(j\,\text{mod}\,k\right)\bigg\}.
\end{equation}
This corresponds to the set containing all the possible combinations of telomere coordinates that can be shortened at each division in one of the daughter cell (we refer to Example \ref{ex:shortening_set} for an expression of $\mathcal{I}_k$ when $k$ is small). Indeed, the condition~"$\forall (i,j)\in I^2: i\neq j \implies \left(i \,\text{mod}\, k\right) \neq \left(j\,\text{mod}\,k\right)$" corresponds to the fact that for each chromosome $i\in\llbracket1,k\rrbracket$, it is either the coordinate $i$ or the coordinate $i+k$ that is shortened during a cell division, but not both. We also denote $g\in L^1_{\text{p.d.f.}}(\mathbb{R}_+)$, that corresponds to the probability density function of the shortening size. Then, we use the following expression for~$\mu^{(S)}$ 
\begin{equation}\label{eq:measure_shortening}
\begin{aligned}
	d\mu^{(S)}(\alpha_1,\alpha_2) &= \frac{1}{\#\left(\mathcal{I}_k\right)}\sum_{I\in\mathcal{I}_{k}}  \prod_{i\in I}\left[g\left((\alpha_1)_{i}\right)d(\alpha_1)_{i}\delta_0(d(\alpha_2)_{i})\right]\\ 
	&\times\prod_{i'\in\llbracket1,2k\rrbracket\backslash\,I}\left[g\left((\alpha_2)_{i'}\right)d(\alpha_2)_{i'}\delta_0\left(d(\alpha_1)_{i'}\right)\right] =: \frac{1}{\#\left(\mathcal{I}_k\right)}\sum_{I\in\mathcal{I}_{k}} d\mu^{(S;I)}(\alpha_1,\alpha_2).
\end{aligned}
\end{equation}
Qualitatively, this expression means that at each division, we draw one set $I\in \mathcal{I}_k$ uniformly. This set gives us which coordinates are shortened in the daughter cell $A$ (terms $g\left((\alpha_1)_{i}\right)d(\alpha_1)_{i}$), and that are not shortened in the daughter cell $B$ (terms~$\delta_0(d(\alpha_2)_{i})$). Considering now the set $\llbracket1,2k\rrbracket\backslash I$, we also have the coordinates that are shortened in the daughter cell $B$ (terms $g\left((\alpha_2)_{i'}\right)d(\alpha_2)_{i'}$), and not shortened in the daughter cell $A$ (terms $\delta_0\left(d(\alpha_1)_{i'}\right)$).
\begin{ex}\label{ex:shortening_set}
When $k = 2$, we have
$$
\mathcal{I}_k = \{\{1,2\},\{1,4\},\{2,3\},\{3,4\}\}.
$$
If at a cell division we have drawn the set $\{1,4\} \in \mathcal{I}_k$, then the coordinates where there is a shortening in the daughter cell $A$ are $1$ and $4$, and the coordinates where there is a shortening in the daughter cell~$B$ are $2$ and $3$.
\end{ex}
\begin{rem}
To better understand the measure $\left(\mu^{(S;I)}\right)_{I\in\mathcal{I}_k}$ defined in~\eqref{eq:measure_shortening}, let us take  $I\in\mathcal{I}_k$ and $f = f_0\times f_1$, where $\left(f_0,f_1\right)\in \left(M_b\left(\mathbb{R}_+^{2k}\right)\right)^2$. Then, for this choice of test function we have (recall that $I^c = \llbracket1,2k\rrbracket\backslash I$)
\begin{equation}\label{}
	\begin{aligned}
		\mu^{(S;I)}(f) &= \int_{\left(\left(\alpha_1\right)_i,i\in I\right)\in\mathbb{R}_+^k} f_0\left(\sum_{i \in I} \left(\alpha_1\right)_ie_i\right)\prod_{i\in I}\left[g\left((\alpha_1)_{i}\right)d(\alpha_1)_{i}\right]\\
		&\times \int_{\left(\left(\alpha_2\right)_{i'},i'\in I^c\right)\in\mathbb{R}_+^k} f_1\left(\sum_{i' \in I^c} \left(\alpha_2\right)_{i'}e_{i'}\right)\prod_{i'\in I^c}\left[g\left((\alpha_2)_{i'}\right)d(\alpha_2)_{i'}\right] .
	\end{aligned}
\end{equation}
A similar alternative representation for the measure $\mu^{(S)}$ can easily be obtained from the above and~\eqref{eq:measure_shortening}.
\end{rem}

Now, we deal with $\left(\mu_{(s_1,s_2)}^{(E)}\right)_{(s_1,s_2)\in\left(\mathbb{R}^{2k}\right)^2}$. We denote 
$$
\mathcal{J}_{k} := \bigg\{(J,M) \in\mathcal{P}_l\left(\llbracket1,2k\rrbracket\right)\times\mathcal{P}_{m}\left(\llbracket1,2k\rrbracket\right)\,|\,(l,m)\in\llbracket1,2k\rrbracket^2\bigg\}.
$$
This corresponds to the set containing  all the possible combinations of telomere coordinates that can be lengthened at each division. For any pair $(J,M)\in \mathcal{J}_{k}$, $J$~corresponds to the coordinates that are lengthened in the daughter cell $A$, and $M$ corresponds to the coordinates that are lengthened in the daughter cell $B$. Short telomeres are more susceptible to be lengthened compared to large telomeres. Then, the choice of which coordinates are lengthened in the two daughter cells depends on the values of telomere lengths of the daughter cells after shortening. To take this into account, for any~$(s_1,s_2)~\in~(\mathbb{R}^{2k})^2$, we introduce $\left(p_{J,M}(s_1,s_2)\right)_{(J,M)\in \mathcal{J}_{k}}$ a probability mass function on $\mathcal{J}_{k}$ that corresponds to the distribution of the choice of which telomeres are lengthened at each division, when telomere lengths of the daughter cells $A$ and $B$ are respectively $s_1$ and $s_2$ after shortening. As the choice of which daughter cell is the daughter cell $A$ or the daughter cell $B$ is arbitrary, the distribution of elongation must be symmetric. Thus, we assume that for all $(J,M)\in\mathcal{J}_{k}$, $(s_1,s_2)\in(\mathbb{R}^{2k})^2$, we have
\begin{equation}\label{eq:symmetry_proba}
p_{J,M}\left(s_1,s_2\right) = p_{M,J}\left(s_2,s_1\right).
\end{equation}
We finally introduce for all $x\in\mathbb{R}$ the function $h(x,.)\in L_{\text{p.d.f.}}^1(\mathbb{R}_+)$. This function represents the distribution of the lengthening value, when the telomere that is lengthened has a length $x$. We are now able to write the expression we use for~$\mu_{(s_1,s_2)}^{(E)}$. For all $(s_1,s_2)\in\left(\mathbb{R}^{2k}\right)^2$ (recall that $\prod_{\emptyset} = 1$), we have
%\begin{equation}\label{eq:measure_elongation}
\begin{align}
d\mu_{(s_1,s_2)}^{(E)}(\beta_1,\beta_2) &= \hspace{-1.5mm} \sum_{(J,M)\in\mathcal{J}_{k}}\hspace{-0.5mm} p_{J,M}\left(s_1,s_2\right)\left[\prod_{j\in J}h\left((s_1)_j,(\beta_1)_{j}\right)d(\beta_1)_{j}\right]\hspace{-0.1mm}\left[\prod_{j'\in J^c} \delta_0(d(\beta_1)_{j'})\right] \nonumber\\
&\times \left[\prod_{m\in M}h\left((s_2)_m,(\beta_2)_{m}\right)d(\beta_2)_{m}\right]\hspace{-0.1mm}\left[\prod_{m'\in M^c} \delta_0(d(\beta_2)_{m'})\right]\label{eq:measure_elongation} \\ 
&=: \sum_{(J,M)\in\mathcal{J}_{k}}p_{J,M}\left(s_1,s_2\right) d\mu_{(s_1,s_2)}^{\left(E;J,M\right)}(\beta_1,\beta_2). \nonumber
\end{align}
%\end{equation}
Qualitatively, this expression means that at each division, just after shortening, if the telomere lengths of the daughter cells $A$ and $B$ are $s_1$ and $s_2$ respectively, we first draw a pair $(J,M)\in \mathcal{J}_{k}$, according to the distribution $\left(p_{J',M'}(s_1,s_2)\right)_{(J',M')\in \mathcal{J}_{k}}$ to know the coordinates that are lengthened in these daughter cells. Then, for any $j\in J$, the coordinate $j$ of the daughter cell $A$ is lengthened according to a distribution given by the density $h((s_1)_j,.)$, and for any $m\in M$, the coordinate~$m$ of the daughter cell $B$ is lengthened according to a distribution given by the density~$h((s_2)_m,.)$.
\begin{ex}\label{ex:lengthening_set}
We suppose that $k = 1$. If after the shortening we have drawn the pair of sets $(\{1,2\},\{2\}) \in \mathcal{J}_{k}$, then there is a lengthening in the coordinates $1$ and~$2$ of the daughter cell $A$, and a lengthening in the coordinate $2$ of the daughter cell~$B$.
\end{ex}
\begin{rem}
Let us consider a pair $\left(s_1,s_2\right)\in\left(\mathbb{R}^{2k}\right)^2$. To better understand the measures~$\left(\mu_{\left(s_1,s_2\right)}^{(E;J,M)}\right)_{\left(J,M\right)\in\mathcal{J}_k}$ defined in~\eqref{eq:measure_elongation}, we fix $\left(J,M\right)\in\mathcal{J}_k$ and a function~\hbox{$f = f_0\times f_1$}, where $\left(f_0,f_1\right)\in \left(M_b\left(\mathbb{R}_+^{2k}\right)\right)^2$. Then, with this choice of test functions, we have 
%Then, another representation of the measures~$\left(\mu_{\left(s_1,s_2\right)}^{(E;J,M)}\right)_{\left(J,M\right)\in\mathcal{J}_k}$ defined in~\eqref{eq:measure_elongation} is that for all $\left(J,M\right)\in\mathcal{J}_k$ and \hbox{$f = f_0\times f_1$}, where $\left(f_0,f_1\right)\in \left(M_b\left(\mathbb{R}_+^{2k}\right)\right)^2$, we have
$$
\begin{aligned}
	\mu_{(s_1,s_2)}^{(E;J,M)}&(f) = \int_{\left(\left(\beta_1\right)_j,j\in J\right)\in\left(\mathbb{R}_+\right)^{\#(J)}} f_0\left(\sum_{j \in J} \left(\beta_1\right)_je_j\right)\left[\prod_{j\in J}h\left((s_1)_j,(\beta_1)_{j}\right)d(\beta_1)_{j}\right]\\
	&\hspace{-4mm}\times 	\int_{\left(\left(\beta_2\right)_m,m\in M\right)\in\left(\mathbb{R}_+\right)^{\#(M)}} f_1\left(\sum_{m \in M} \left(\beta_2\right)_me_m\right)\left[\prod_{m\in M}h\left((s_2)_m,(\beta_2)_{m}\right)d(\beta_2)_{m}\right].
\end{aligned}
$$
A similar alternative representation for the measure $\mu^{(E)}_{(s_1,s_2)}$ can easily be obtained from the above and~\eqref{eq:measure_elongation}.
\end{rem}
\begin{rem}\label{rem:telomere_shortening_below_0}
One can notice that $\mu_{(s_1,s_2)}^{(E)}$ is defined for ~$(s_1,s_2)\in\left(\mathbb{R}^{2k}\right)^2$, meaning that we authorize negative telomere lengths. One should actually think about this in a different way. When telomere lengths of a cell go under a certain threshold, this cell becomes senescent and thus inactive. Hence, zero should be seen as this threshold, which has a biological interpretation. \\Of course, zero has been taken here for simplicity. For example, in a recent study, this threshold has been estimated at $28$ base pairs for the budding yeast~\cite{rat_individual_2023}.
%This is not very limiting as in our model, the criteria for determining if a cell becomes senescent is that this cell has a telomere with a length below zero. 
\end{rem}
\begin{rem}
A symmetry property holds for the measures $\left(\Pi_x\right)_{x\in\mathbb{R}^{2k}}$ defined in~\eqref{eq:kernel_telomere_updating}. To show it, we denote for all $C\in\mathcal{B}\left(\mathbb{R}_+^{2k}\times\mathbb{R}_+^{2k}\right)$ the set
$$
\underline{C} = \left\{(y_1,y_2)\in\mathbb{R}_+^{2k}\times\mathbb{R}_+^{2k}\,|\, (y_2,y_1)\in C\right\}.
$$
Eq.~\eqref{eq:symmetry_proba} and the fact that $\left\{\llbracket1,2k\rrbracket\backslash I\,|\,I\in \mathcal{I}_k\right\} = \mathcal{I}_k$ imply that for all $(s_1,s_2)\in(\mathbb{R}^{2k})^2$
\begin{equation}\label{eq:symmetry_assumption}
	\mu_{(s_1,s_2)}^{(E)}(C) = \mu_{(s_2,s_1)}^{(E)}(\underline{C}), \text{ and }\mu^{(S)}(C) = \mu^{(S)}(\underline{C}).
\end{equation}
Then, for all measurable function $f : (\mathbb{R}_+^{2k})^2 \longrightarrow \mathbb{R}$ and $(s_1,s_2)\in(\mathbb{R}^{2k})^2$, it holds
$$
\begin{aligned}
	\int_{(\beta_1,\beta_2)\in (\mathbb{R}_+^{2k})^2} f(\beta_1,\beta_2) d\mu_{(s_1,s_2)}^{(E)}(\beta_1,\beta_2) &= \int_{(\beta_1,\beta_2)\in (\mathbb{R}_+^{2k})^2} f(\beta_2,\beta_1) d\mu_{(s_2,s_1)}^{(E)}(\beta_1,\beta_2), \\
	\int_{(\beta_1,\beta_2)\in (\mathbb{R}_+^{2k})^2} f(\alpha_1,\alpha_2) d\mu^{(S)}(\alpha_1,\alpha_2) &= \int_{(\alpha_1,\alpha_2)\in (\mathbb{R}_+^{2k})^2} f(\alpha_2,\alpha_1) d\mu^{(S)}(\alpha_1,\alpha_2),
\end{aligned}
$$
which implies by~\eqref{eq:kernel_telomere_updating} that for all \hbox{$f : \left(\mathbb{R}^{2k}\right)^2  \longrightarrow \mathbb{R}$} measurable and $x\in\mathbb{R}^{2k}$
\begin{equation}\label{eq:symetry_measure}
	\int_{(w_1,w_2)\in(\mathbb{R}^{2k})^2} f(w_1,w_2) d\Pi_x(w_1,w_2) = \int_{(w_1,w_2)\in(\mathbb{R}^{2k})^2} f(w_2,w_1) d\Pi_x(w_1,w_2).
\end{equation}
This symmetry property is the consequence of the fact that the two daughter cells are not distinguished.
\end{rem}

\noindent We now have everything we need to give the algorithm that builds our model. Recall that~$(\partial)_2$ represents the senescent state, and that the trait space is $\mathcal{X}_{\partial} = \left(\mathbb{R}_+^{2k}\times\mathbb{R}_+\right)\cup\{(\partial)_2\}$. Let $(x,a)\in\mathcal{X}_{\partial}$ the trait of a cell, and $\mathdutchcal{u}\in\mathcal{U}$ its label. 
\begin{enumerate}[leftmargin=*]
\item  At a rate $b(x,a)1_{\{(x,a) \neq (\partial,\partial)\}}$, the cell divides. 

\item At each division, we draw a pair of random variables $(U,U')\in\mathbb{R}_+^{2k}\times\mathbb{R}_+^{2k}$ following a distribution given by $\mu^{(S)}$, and a pair of random variables $(V,V')\in\mathbb{R}_+^{2k}\times\mathbb{R}_+^{2k}$ following a distribution given by  $\mu_{(x-U, x-U')}^{(E)}$. The telomeres of the daughter cell $A$ are shortened by $U$, lengthened by $V$. The telomeres of the daughter cell $B$ are shortened by $U'$, lengthened by $V'$. Thus, the telomeres and the labels of the two daughter cells are  updated as follows ($\mathdutchcal{u}1$ and~$\mathdutchcal{u}2$ are defined in Notation~\ref{nota:ulam_harris_neveu})
$$
x \longrightarrow \begin{cases}
	x - U + V, \\
	x - U'+ V',
\end{cases} \hspace{18mm}\mathdutchcal{u} \longrightarrow \begin{cases}
	\mathdutchcal{u}1, \\
	\mathdutchcal{u}2.
\end{cases}
$$
Let  $\theta =\left(\theta_1, \theta_2,\hdots, \theta_K\right)$ a tuple of $K$ independent uniform random variables on~$[0,1]$. A~consequence of~\eqref{eq:uniform_random_variables_assumption},~\eqref{eq:kernel_telomere_updating} and of the way $(U,U')$ and $(V,V')$ are distributed is that $R\left(x,\theta\right) \overset{d}{=} \left(-U+V, - U' + V'\right)$. Then, the updating of the telomeres and the labels can be rewritten as	
$$
x \longrightarrow \begin{cases}
	x + (R(x,\theta))_1, \\
	x + (R(x,\theta))_2,
\end{cases} \hspace{14.5mm}\mathdutchcal{u} \longrightarrow \begin{cases}
	\mathdutchcal{u}1, \\
	\mathdutchcal{u}2.
\end{cases}
$$

\item After the lengthening, if $\min\left(x_i + (R(x,\theta))_i,\,i\in\llbracket1,2k\rrbracket\right)<0$ for one of the daughter cells, then the trait of this daughter cell at birth is $(\partial)_2$. Otherwise, the trait of this daughter cell is $(x - U + V,0)$ or $(x - U' + V',0)$. It is also possible to use a rate to enter in senescence rather than a threshold. However, recent results coming from parameter calibration of telomere shortening models seem to support the fact that a deterministic threshold is a good approximation~\cite{rat_individual_2023}.
\end{enumerate}
Assume that we start from $Y_0$ defined in Notation \hyperref[notation:initial_condition_process]{\ref*{notation:initial_condition_process}}, and denote by $(Y_t)_{t\geq0}$ the  $\mathcal{M}_P\left(\mathcal{U}\times\mathcal{X}_{\partial}\right)$-valued random process representing the dynamics described above. For all $t\geq0$, $V_t \subset \mathcal{U}$ is the set that contains the labels of the individuals at time $t$ with a trait different from $(\partial)_2$. For all $\mathdutchcal{u}\in V_t$, we denote by $x^{\mathdutchcal{u}}$ the telomere lengths of $\mathdutchcal{u}$, by~$T_b(\mathdutchcal{u})$ the birth time of $\mathdutchcal{u}$, and by $a_t^{\mathdutchcal{u}} := t - T_b(\mathdutchcal{u})$ the age of $\mathdutchcal{u}$ at time $t$. Then $(Y_t)_{t\geq0}$ is a solution of the following equation, for all $f\in\mathcal{C}_b^{1,m,m,1}(\mathbb{R}_+\times\mathcal{U}\times\mathcal{X}_{\partial})$, $t\geq 0$,
\begin{equation}\label{eq:SDE_model_intro}
\begin{aligned}
	&\int_{\mathcal{U}\times\mathcal{X}_{\partial}} f(t,\mathdutchcal{u},x,a)Y_t(d\mathdutchcal{u},dx,da) = \int_{\mathcal{U}\times\mathcal{X}_{\partial}} f(0,\mathdutchcal{u},x,a)Y_0(d\mathdutchcal{u},dx,da)  \\
	&+ \int_0^t\int_{\mathcal{U}\times\mathcal{X}_{\partial}} \bigg[\frac{\partial f}{\partial s}\left(s,\mathdutchcal{u},x,a\right) + \frac{\partial f}{\partial a}\left(s,\mathdutchcal{u},x,a\right)1_{\{\mathdutchcal{u}\in V_s\}}\bigg] Y_s(d\mathdutchcal{u},dx,da)ds \\
	&+ \int_{[0,t]\times\mathcal{U}\times\mathbb{R}_+\times[0,1]^K} 1_{\{ z \leq b(x^\mathdutchcal{u},a_{s-}^\mathdutchcal{u}),\mathdutchcal{u}\in\,V_{s-}\}} \Bigg[\sum_{i=1}^{2} \left[f\left(s,\mathdutchcal{u}i,\partial,\partial\right)1_{\{x^\mathdutchcal{u} + (R(x^\mathdutchcal{u},\theta))_{i} \notin\mathbb{R}_+^{2k}\}} \right.   \\ 
	&+\left.f\left(s,\mathdutchcal{u}i,x^\mathdutchcal{u} + (R(x^\mathdutchcal{u} ,\theta))_{i},0\right)1_{\{x^\mathdutchcal{u} + (R(x^\mathdutchcal{u},\theta))_{i} \in\mathbb{R}_+^{2k}\}}\right] - f(s,\mathdutchcal{u},x^\mathdutchcal{u},a_{s-}^\mathdutchcal{u})\bigg]N(ds, d\mathdutchcal{u}, dz, d\theta). 
\end{aligned}
\end{equation}
The following theorem guarantees that the above model is well-posed, and that the number of individuals does not explode in finite time.

\begin{te}[Well-posedness and non-explosion]\label{te:existence_uniqueness}
Let us assume that \eqref{eq:birth_rate_assumption} and \eqref{eq:uniform_random_variables_assumption} hold, and that the variable $Y_0$ defined in Notation \hyperref[notation:initial_condition_process]{\ref*{notation:initial_condition_process}} verifies~$\mathbb{E}\left[Y_0(1)\right] < +\infty$. Then, there \hbox{exists a $\left(\mathcal{F}_t\right)_{t\geq0}$-adapted} càdlàg process $(Y_t)_{t\geq0}$ taking values in $\mathcal{M}_P\left(\mathcal{U}\times\mathcal{X}_{\partial}\right)$, with initial condition $Y_0$, which is solution of~\eqref{eq:SDE_model_intro}. Moreover, the trajectorial uniqueness holds for the solution of~\eqref{eq:SDE_model_intro}.
\end{te}
\noindent This type of result has been obtained as the starting point of many works in the literature, see for example \cite{bansaye2015,bansaye2010,champagnat_2006,fournier2004,marguet2019,tran2006}. The proof has nowadays become standard and we do not detail it here. We just present briefly the main steps of the proof. First, we adjust the proof of~\hbox{\cite[Lemmas $2.1$ and $2.3$]{marguet2019}} to a setting where we need several uniform random variables to update offspring traits. Then, in view of~\eqref{eq:birth_rate_assumption}, we bound from above~$\mathbb{E}\left[\sup_{t\in[0,T]}\left(Y_t(1)\right)\right]$ for all $T >0$ by $2\mathbb{E}\left[Y_0(1)\right]$ times the mean number of individuals of a dynamics of individuals that branches at a rate \hbox{$\tilde{b}(1+a^{d_b})$} and gives birth to $2$ offspring, as done in \hbox{\cite[Corollary $3.10$]{bansaye_ergodic_2020}}. Finally, we use \cite[Corollary $1$,p.$153$]{athreya_1972} to justify that the latter is finite. 

We end this subsection by presenting a deterministic representation of the model. To do this, we first introduce $(n_t(dx,da))_{t\geq0}$ a family of measures such that for all $t\geq0$, $f\in\mathcal{C}_b^{m,1}\left(\mathcal{X}_{\partial}\right)$,
\begin{equation}\label{eq:dft_mean_measure}
n_t(f) := \mathbb{E}\left[\int_{\mathcal{U}\times\mathcal{X}_{\partial}} f(x,a)Y_t(d\mathdutchcal{u},dx,da)\right].
\end{equation}
This measure is called the mean measure of $Y_t$. For all $C\in\mathcal{B}\left(\mathbb{R}_+^{2k}\right)$, $x\in\mathbb{R}_+^{2k}$, we also introduce the following measure, that is used frequently in this paper
\begin{equation}\label{eq:kernel_branching}
\begin{aligned}
	\mathdutchcal{K}(C)(x) &:= \int_{(w_1,w_2)\in \left(\mathbb{R}_+^{2k}\right)^2} 1_C\left(x + w_1\right)1_{\left\{x + w_1 \in\mathbb{R}_+^{2k}\right\}} d\Pi_x(w_1,w_2) \\
	&+ \int_{(w_1,w_2)\in \left(\mathbb{R}_+^{2k}\right)^2} 1_C\left(x+ w_2\right)1_{\left\{x+ w_2 \in\mathbb{R}_+^{2k}\right\}} d\Pi_x(w_1,w_2) \\
	&= 2\int_{
		(w_1,w_2)\in \left(\mathbb{R}_+^{2k}\right)^2} 1_C\left(x + w_1\right)1_{\left\{x + w_1 \in\mathbb{R}_+^{2k}\right\}} d\Pi_x(w_1,w_2).
\end{aligned}
\end{equation}
The last equality is a consequence of \eqref{eq:symetry_measure}. The above represents the transition probability for telomere lengths of the daughter cells of a cell with telomere lengths $x$ that divides. Finally, we introduce the following system
\begin{equation}\label{eq:PDE_model_intro}
\begin{cases}
	\partial_t n(t,dx,da) = - \partial_a\,n(t,dx,da)  -b(x,a)\,n(t,dx,da), & (x,a)\in \mathcal{X},\,t\geq 0, \\
	n(t,dx, a = 0) = \int_{(y,s)\in\mathcal{X}} \mathdutchcal{K}(dx,y) b(y,s)n(t,dy,ds), & x\in \mathbb{R}_+^{2k},\,t\geq 0,\\
	\partial_t n(t,(\partial)_2) = \int_{(y,s)\in\mathcal{X}} (2-\mathdutchcal{K}(1)(y))b(y,s)n(t,dy,ds), & t\geq 0, \\
	n(0,dx,da) = n_0(dx,da), & (x,a)\in \mathcal{X},
\end{cases}
\end{equation}
and present below what we call a measure solution to this system. Our definition of measure solution to~\eqref{eq:PDE_model_intro} is inspired by the one given in~\cite[Def.~$3$]{gabriel_measure_2018}.
\begin{dft}[Measure solution to~\eqref{eq:PDE_model_intro}]\label{dft:measure_solution}
The family $(n_t(dx,da))_{t\geq0}$ is a measure solution to~\eqref{eq:PDE_model_intro} if for all $t\geq0$, $f\in\mathcal{C}_b^{m,1}\left(\mathcal{X}\right)$,
$$
n_t\left(f1_{\mathcal{X}}\right) = n_0\left(f1_{\mathcal{X}}\right)+ \int_0^{t}\left[n_s\left(\frac{\partial f}{\partial a}1_{\mathcal{X}}\right) + n_s\left(\mathdutchcal{K}\left(f(.,0)\right)b1_{\mathcal{X}}\right) - n_s\left(fb1_{\mathcal{X}}\right)\right]ds,
$$
and
$$
n_t\left(1_{\left(\partial,\partial\right)}\right) = n_0\left(1_{\left(\partial,\partial\right)}\right)+ \int_0^{t} n_s\left(\left[2-\mathdutchcal{K}\left(1\right)\right]b1_{\mathcal{X}}\right)ds.
$$
\end{dft}
\noindent Then, the following theorem gives us that~\eqref{eq:PDE_model_intro} is a deterministic representation of~\eqref{eq:SDE_model_intro}.  Again, we do not detail the proof of this result as it has become standard, \hbox{see \cite{bansaye2015,fournier2004,tomasevic2022,tran2006}}.
\begin{te}[Deterministic representation of~\eqref{eq:SDE_model_intro}]\label{te:PDE}
Let us assume that the assumptions of Theorem~\ref{te:existence_uniqueness} hold, and that 
$$
\mathbb{E}\left[\int_{\mathcal{U}\times\mathcal{X}_{\partial}} (1+a^{d_b})Y_0(d\mathdutchcal{u},dx,da)\right]< +\infty.
$$
Then, the family $(n_t(dx,da))_{t\geq0}$ is a measure solution to~\eqref{eq:PDE_model_intro}. Moreover, $n_t\left(f1_{\mathcal{X}}\right)$ is derivable at $t = 0$ for all $f\in \mathcal{C}_b^{m,1}(\mathcal{X})$, with derivative 
$$
\frac{d}{d t} n_t\left(f1_{\mathcal{X}}\right)\Big|_{t = 0} = n_0\left(\frac{\partial f}{\partial a}1_{\mathcal{X}}\right) + n_0\left(b\left[\mathdutchcal{K}(f(.,0))-f\right]1_{\mathcal{X}}\right).
$$
\end{te}

\subsection{Assumptions and main result}\label{subsect:main_result}

Let us remind the reader that from now on, $k$ is fixed. The main result of this article is the convergence of the first moment semigroup of a solution of \eqref{eq:SDE_model_intro} towards a stationary profile. To write the assumptions that we need to verify to obtain this convergence, we introduce the following notations related to the measures $(\mathdutchcal{K}(.)(y))_{y\in\mathbb{R}_+^{2k}}$.
\begin{itemize}[leftmargin=*]
\item For all $x\in\mathbb{R}_+^{2k}$, $n\in\mathbb{N}^*$, for all measurable function $f:\mathbb{R}_+^{2k} \longrightarrow \mathbb{R}$, we denote 
$$
\mathdutchcal{K}^n(f)(x) := (\underbrace{\mathdutchcal{K} \circ \hdots \circ \mathdutchcal{K}}_{n \text{ times}})(f)(x).
$$
We use the convention that $\mathdutchcal{K}^0(f)(x)= 1$.
\item For all $x\in\mathbb{R}_+^{2k}$, $C \subset \mathbb{R}_+^{2k}$ a measurable set, and $f:\mathbb{R}_+^{2k} \longrightarrow \mathbb{R}$ a measurable function, we denote 
\begin{equation}\label{eq:measure_k_restriction}
	\mathdutchcal{K}_{C}(f)(x) := \mathdutchcal{K}(f1_{C})(x).
\end{equation}
\end{itemize}
We now write the assumptions on our model that imply the main result of this paper. They will be discussed at the end of the section, see p.\hyperlink{paragraph:discussion_assumptions}{\pageref*{paragraph:discussion_assumptions}}. We distinguish three sets of assumptions. The first set corresponds to lower/upper bounds on the shortening distribution $g$, the lengthening distributions $\left(h(x,.)\right)_{x\in\mathbb{R}}$, and the lengthening probabilities~$\left(p_{J,M}(s_1,s_2)\right)_{(J,M)\in\mathcal{J}_k,\left(s_1,s_2\right)\in\left(\mathbb{R}^{2k}\right)^2}$, all introduced in Section~\ref{subsect:algorithm_model}. 

\paragraph{$(S_1)\!:$ Lower/upper bounds.}
\hypertarget{paragraph:long_time_behaviour}{}
\hypertarget{paragraph:long_time_behaviour_S1}{}
\begin{enumerate}[start=1,label={$(S\textsubscript{1.\arabic*})\!:$},leftmargin=*]

\item The following four statements hold. 
\hypertarget{paragraph:long_time_behaviour_S1.1}{} 
\begin{enumerate}[$i)$, leftmargin=0.5cm]
	
	\item There exists $\delta>0$ such that 
	$$
	\text{supp}(g) := \{x\in\mathbb{R}_+, g(x) \neq 0\} = [0, \delta].
	$$
	
	\item There exists $g_{\min} > 0$ such that for all $x \in [0,\delta]$
	$$
	g(x) \geq g_{\min}.
	$$

	\item There exists $\Delta > 0$, such that for all $x\in\mathbb{R}$, there exists $\Delta_x \in[0,\Delta]$ such that \hypertarget{paragraph:long_time_behaviour_S1.1.iv}{} 
	$$
	\begin{aligned}
		\text{supp}(h(x,.)) &:= \{y\in\mathbb{R}_+, h(x,y) \neq 0\} = [0, \Delta_{x}].
	\end{aligned}
	$$
	Moreover, the function $x \mapsto \Delta_x$ is bounded from below on compact sets.
	\item There exists $h_{\min} >0$ such that for all $x\in \mathbb{R}$, $y \in [0,\Delta_{x}]$
	$$
	h(x,y) \geq h_{\min}.
	$$
\end{enumerate}
\item \hypertarget{paragraph:long_time_behaviour_S1.2}{} There exist $m_0\in\mathbb{N}\backslash\{0,1\}$ and a sequence of pairs $(\mathbb{I}^n,\mathbb{J}^n,\mathbb{M}^n)_{n \in \llbracket1,m_0\rrbracket}$ taking values in $\mathcal{I}_k\times\mathcal{J}_k$ verifying 
$$
\begin{aligned}
	&\bigcup_{i = 1}^{m_0} \mathbb{I}^i = \llbracket1,2k\rrbracket, \hspace{6mm} \forall i \in\llbracket1,m_0\rrbracket:\, \mathbb{I}^i \subset \mathbb{J}^i,
\end{aligned}
$$ 
such that for all $A > 0$, there exists $p_{\min} >0$ verifying:
$$
\inf_{\substack{i \in \llbracket1,m_0\rrbracket,\\ (s_1,s_2) \in [-\delta,A + m_0\Delta]^2}}\left(p_{\mathbb{J}^i,\mathbb{M}^i}(s_1,s_2)\right) \geq p_{\min}.
$$

\item \hypertarget{paragraph:long_time_behaviour_S1.3}{} There exists $\overline{g} >0$ such that
$$
\forall x\in\mathbb{R}_+: g(x) \leq \overline{g}.
$$
Moreover, the function $h$ is bounded from above on compact sets.
\end{enumerate}
Now, we need assumptions that allow us to justify the existence of a Lyapunov function.
\paragraph{$(S_2)\!:$ Existence of a Lyapunov function.}\ 
\hypertarget{paragraph:long_time_behaviour_S2}{ }
\begin{enumerate}[start=1,label={$(S\textsubscript{2.\arabic*})\!:$},leftmargin=*]
\item \hypertarget{paragraph:long_time_behaviour_S2.1}{} There exist $G\in\mathbb{N}^*$, $\varepsilon_0>0$,\,$B_{\max}> 0$, and a set $K_{\text{ren}} \subset [0,B_{\max}]^{2k}$ non-negligible with respect to the Lebesgue measure, such that for all $x\in K_{\text{ren}}$
$$
\left(\mathdutchcal{K}_{[0,B_{\max}]^{2k}}\right)^{G}\left(1_{\{K_{\text{ren}}\}}\right)(x)  \geq (1 + \varepsilon_0).
$$

\item \hypertarget{paragraph:long_time_behaviour_S2.2}{} There exist $L_{\text{ret}}\in\mathbb{N}^*$, $\mathdutchcal{V} : \mathbb{R}_+^{2k} \longrightarrow (\mathbb{R}_+)^*$ measurable and  $\varepsilon_1 >0$ such that with~$B_{\max}$ defined above:
\begin{enumerate}[$i)$, leftmargin=0.5cm]
	\item For all $x\in \mathbb{R}_+^{2k}$
	$$
	\mathdutchcal{K}_{\left(\left[0,B_{\max} L_{\text{ret}}\right]^{2k}\right)^c}\left(\mathdutchcal{V}\right)(x) \leq (1+ \varepsilon_1)\mathdutchcal{V}(x).
	$$
	
	\item For all $x\in\mathbb{R}_+^{2k}$, $I\in\mathcal{I}_k$
	$$
	\begin{aligned}
		&\hspace{-4.6mm}\sum_{(J,M)\in \mathcal{J}_{k}}\int_{(\alpha_1,\alpha_2)\in (\mathbb{R}_+^{2k})^2} p_{J,M}\left(x - \alpha_1,x - \alpha_2\right)\bigg[\int_{(\beta_1,\beta_2)\in (\mathbb{R}_+^{2k})^2}\mathdutchcal{V}(x-\alpha_1+\beta_1)\\ 
		&\hspace{-4.6mm}\times 1_{\{x-\alpha_1+\beta_1\in\mathbb{R}_+^{2k}\}} d\mu_{(x-\alpha_1,x-\alpha_2)}^{\left(E;J,M\right)}(\beta_1,\beta_2)\bigg] d\mu^{\left(S;I\right)}(\alpha_1,\alpha_2)\leq (1+\varepsilon_1)\mathdutchcal{V}(x).
	\end{aligned} 
	$$
	\item $\mathdutchcal{V}$ is bounded from above on compact sets and there exists $\mathdutchcal{V}_{\min} >0$ such that for all $x\in\left(\mathbb{R}_+\right)^{2k}$ it holds $\mathdutchcal{V}(x)~\geq~\mathdutchcal{V}_{\min}$.
	
	\item There exists $C_{\mathdutchcal{V}} >0$ such that 
	$$
	\sup_{x\in\mathbb{R}_+^{2k}}\left(\sup_{y\in\mathbb{R}_+^{2k}\text{ s.t. }||y-x||_{\infty} \leq \max(\delta,\Delta)} \frac{\mathdutchcal{V}(y)}{\mathdutchcal{V}(x)}\right) \leq C_{\mathdutchcal{V}}. 
	$$
	
\end{enumerate}
\end{enumerate}
Finally, we need assumptions to handle the age structure of our model. To present them, we recall that $\mathcal{L}$ corresponds the the Laplace transform operator, see Notation~\ref{nota:laplace_transform}.
\paragraph{$(S_3)\!:$ From the generational level to the temporal level.}
\hypertarget{paragraph:long_time_behaviour_S3}{ }
\begin{enumerate}[start=1,label={$(S\textsubscript{3.\arabic*})\!:$},leftmargin=*]
\item \hypertarget{paragraph:long_time_behaviour_S3.1}{} There exist $\underline{b} : \mathbb{R}_+ \longrightarrow \mathbb{R}_+$ and $\overline{b} : \mathbb{R}_+ \longrightarrow \mathbb{R}_+$, such that for all $(x,a)\in\mathcal{X}$
$$
\underline{b}(a) \leq b(x,a) \leq \overline{b}(a),
$$
and such that $\int_0^{+\infty} \underline{b}(a) da = +\infty$.

\item \hypertarget{paragraph:long_time_behaviour_S3.2}{} We denote for all $a,\,s\geq0$ 
$\mathdutchcal{F}_a(s) := \underline{b}(a+s)\exp\left(-\int_a^{a+s} \underline{b}\left(u\right) du\right)$ and its associated cumulated distribution function $\overline{\mathdutchcal{F}_a}(s) := \exp\left(-\int_a^{a+s} \underline{b}\left(u\right) du\right)$. Then, there exists~$\alpha >0$ such that
$$
\mathcal{L}(\mathdutchcal{F}_{0})(\alpha) = \frac{1}{(1+\varepsilon_0)^{\frac{1}{G}}}, \hspace{3mm}\text{ and } \hspace{3mm}\underset{t\geq0}{\sup}\,\underset{a\geq0}{\inf}\left(\int_0^t e^{-\alpha s} \mathdutchcal{F}_a(s) ds\right) > 0.
$$

\item \hypertarget{paragraph:long_time_behaviour_S3.3}{} We denote for all $a,s\geq0$ $\mathdutchcal{J}_a(s) := \overline{b}(a+s)\exp\left(-\int_a^{a+s} \overline{b}\left(u\right) du\right)$, and its associated cumulative distribution function $\overline{\mathdutchcal{J}_a}(s) := \exp\left(-\int_a^{a+s} \overline{b}\left(u\right) du\right)$.
Then, there exists~$\beta\in(0,\alpha)$ such that
$$
\mathcal{L}(\mathdutchcal{J}_{0})(\beta) = \frac{1}{1+\varepsilon_1}.
$$

\end{enumerate}

\noindent In Section \ref{sect:assumptions_verified}, we give criteria allowing to check \hyperlink{paragraph:long_time_behaviour_S2.2}{$(S_{2.2})$} and \hyperlink{paragraph:long_time_behaviour_S3}{$(S_{3})$}, and present two models where all the above assumptions are verified. We now give the main result of the paper. 
\paragraph{Main result.}  
We consider the following semigroup, for all $(x,a)\in\mathcal{X}$, $f \in M\left(\mathcal{X}\right)$ non-negative, $t\geq 0$:
\begin{equation}\label{eq:first_moment_semigroup}
M_t(f)(x,a) = \mathbb{E}\left[\sum_{\mathdutchcal{u}\in V_t} f(x^\mathdutchcal{u},a_t^{\mathdutchcal{u}})\,\bigg|\, Y_0 = \delta_{(1,x,a)}\right]\in[0,+\infty].
\end{equation}%  &\answermath{:= \mathbb{E}\left[\int_{\mathcal{U}\times\mathcal{X}} f(x,a)Y_t(d\mathdutchcal{u},dx,da)\,\bigg|\, Y_0 = \delta_{(1,x,a)}\right]}\\
We easily see that $(M_t)_{t\geq0}$ is similar to a solution of \eqref{eq:PDE_model_intro} starting from a Dirac measure, and without the individuals in the cemetery. Under Assumption \hyperlink{paragraph:long_time_behaviour_S2.2}{$(S_{2.2})$}, we also consider the following space ($d_b$ and $\mathdutchcal{V}$ are respectively defined in~\eqref{eq:birth_rate_assumption} and Assumption \hyperlink{paragraph:long_time_behaviour_S2.2}{$(S_{2.2})$})
\begin{equation}\label{eq:defintion_all_distortion_functions}
\Psi = \left\{\psi \in  M_b^{loc}\left(\mathcal{X}\right)\,|\,\exists \,d_{\psi}\in\mathbb{N},\,d_{\psi}\geq d_b \text{ s.t. } \forall (x,a)\in\mathcal{X}:\,\psi(x,a) = (1+a^{d_{\psi}})\mathdutchcal{V}(x) \right\}.
\end{equation}
This space contains all the functions used to distort the space when we obtain our convergence result. We use several functions to distort the space to be able to obtain the convergence of the semigroup defined on a large space of functions. In Lemma~\ref{lemma:well_definition_semigroup}, we prove that we can extend the definition of $(M_t)_{t\geq 0}$ to functions in $\cup_{\psi\in\Psi} \mathdutchcal{B}(\psi)$, where $\left(\mathdutchcal{B}(\psi)\right)_{\psi\in\Psi}$ are introduced in Notation~\ref{notation:start_nota_space_distortion}. We also denote $\mathdutchcal{B}(\Psi) = \cap_{\psi\in\Psi} \mathdutchcal{B}(\psi)$ and $L^1(\Psi) = \cap_{\psi\in\Psi} L^1(\psi)$. We are ready to give the main result. 
\begin{te}[Main result]\label{te:main_result}
Let us assume that \eqref{eq:birth_rate_assumption} and Assumptions \hyperlink{paragraph:long_time_behaviour}{$(S_1)-(S_{3})$} hold. Then, there exists a unique $(N,\phi,\lambda)\in L^1(\Psi)\times\mathdutchcal{B}(\Psi)\times[\alpha,+\infty[$ satisfying $N \geq 0$, $\phi > 0$ and
$$
\int_{(x,a)\in\mathcal{X}} N(x,a) dxda = \int_{(x,a)\in\mathcal{X}} N(x,a)\phi(x,a) dxda = 1,
$$
such that for all $\psi\in \Psi$, there exist $C,\omega >0$ such that for all $t>0$, $\mu\in \mathdutchcal{M}(\psi)$
\begin{equation}\label{eq:main_result}
	\sup_{f\in \mathdutchcal{B}(\psi),||f||_{\mathdutchcal{B}(\psi)} \leq 1}\left|e^{-\lambda t}\mu M_t(f)  - \mu(\phi)\int_{(x,a)\in\mathcal{X}}f(x,a) N(x,a)dx da\right| \leq C||\mu||_{\mathdutchcal{M}(\psi)}e^{-\omega t}.
\end{equation}
Moreover, there exists $N_0\in L^1(\mathbb{R}_+^{2k})$ non-negative such that for all $(x,a)\in\mathcal{X}$ it holds 
\begin{equation}\label{eq:form_stationary_profile}
	N(x,a) = N_0(x)\exp\left(- \int_0^a b(x,s) ds-\lambda a \right),
\end{equation}
and such that denoting $\mathdutchcal{Y}_0(w,s) = b(w,s)\exp\left(-\int_0^s b(w,u) d u\right)$ for all $(w,s)\in\mathcal{X}$, $N_0$ is solution of the following equation
\begin{equation}\label{eq:equation_N_0}
	\forall f\in M_b\left(\mathbb{R}_+^{2k}\right): \hspace{2mm} \int_{x\in\mathbb{R}_+^{2k}} f(x)N_0\left(x\right) dx = \int_{x\in\mathbb{R}_+^{2k}} \mathdutchcal{K}\left(f\right)(x)\mathcal{L}\left(\mathdutchcal{Y}_0(x,.)\right)\left(\lambda\right) N_0(x) dx.	
\end{equation}
\end{te}
To prove this theorem, we first weight the space by $\psi$ and study the weighted-normalised semigroup 
\begin{equation}\label{eq:weighted_renormalised_semigroup}
M^{(\psi)}_t(f)(x,a) := e^{-\lambda_{\psi}t}\frac{M_t (f\psi)(x,a)}{\psi(x,a)},\hspace{6mm} \forall t\geq0,\,\forall f\in M_b(\mathcal{X}),\,\forall (x,a)\in\mathcal{X},
\end{equation}
where $\lambda_{\psi} >0$ is chosen in Lemma~\ref{lemma:inequalities_psi} such that for all \hbox{$(x,a)\in\mathcal{X}$} and $t\geq0$ we have 
$$
e^{-\lambda_{\psi}t}\frac{M_t(\psi)(x,a)}{\psi(x,a)} \leq 1.
$$
This new semigroup satisfies the same equation as the first moment semigroup $\big(P_t^{(\psi)}\big)_{t\geq0}$ of a jump Markov process with an absorbing state. As this equation has a unique solution~(see Lemma~\ref{lemma:unique_solution_equation_semigroup}), we can study the ergodic behaviour of~$\big(P_t^{(\psi)}\big)_{t\geq0}$ to obtain the one of $\big(M^{(\psi)}_t\big)_{t\geq0}$. This kind of method has been presented in~\cite{champagnat_2020}. We then apply~\cite[Theorem~$2.8$]{velleret_exponential_2023} to obtain the convergence of $\big(M_t^{(\psi)}\big)_{t\geq0}$ towards a stationary profile through~$\big(P_t^{(\psi)}\big)_{t\geq0}$. We conclude by the fact that the convergence of $\big(M^{(\psi)}_t\big)_{t\geq 0}$ to a stationary profile on $\mathcal{M}(\mathcal{X})$ implies the convergence of $(M_t)_{t\geq0}$ on $\mathdutchcal{M}(\psi)$. The delicate/most interesting points are the following.
\begin{itemize}[leftmargin=*]
\item We need to handle the age structure of our model. To do so, we stochastically compare our process with Bellman-Harris processes, which allows us to use results from the renewal theory~\hbox{\cite[Chap.~IV.4]{athreya_1972}} to obtain exponential estimates where we need them~(Sections \ref{subsubsect:proof_prop_renewal_set}, \ref{subsubsect:proof_bound_tail_probability} and~\ref{subsubsect:proof(A3)F_control_time}).

\item As the birth rate depends on the multidimensional trait, we do not have independent identically distributed times between jumps, which is one of the main characteristics of Bellman-Harris processes. To retrieve independent identically distributed jumps, we use the fact that the jump rate is bounded (from above or from below) by another jump rate, independent of the multidimensional trait. Then, we control the times these jumps occur using the inverse transform sampling to simulate jumps (Section~\ref{subsubsect:proof_inequality_means}).

\item During the proof that a renewal in several generations implies an exponential growth, we also need to handle the fact that not only telomere lengths can renew, but also the age. This implies that we need to restrict ourselves to truncated versions of the distribution of jump times, such that jump times of this truncated version do not exceed a certain value. We prove that even for a truncated version of the distribution of jump times, we can still make a stochastic comparison of our process with a Bellman-Harris process by considering an alternative birth rate~(Section~\ref{subsubsect:proof_prop_renewal_set}).

\item The irregularities of the jump kernel with respect to the Lebesgue measure make difficult the control of the asymptotic comparisons of survival needed to obtain the stationary profile of the auxiliary particle representing our branching process. To handle this, we use a weak form of a Harnack inequality (Section \ref{subsubsect:assumption_(A3)F}).

\end{itemize}
\noindent We finish this section with a  discussion about the consequences of our assumptions, and why we made them.

%\subsection{\answer{Discussion about our main result and the assumptions}}
%\answer{We finish this section by discussing the limitations of our main results, and explaining our assumptions. We begin by presenting how the statement pres}
%\paragraph{\answer{Are the results closed to optimality ?}} \answer{.}

%\paragraph{\answer{Possible extensions of the model.}} \answer{bb.}

\paragraph{Discussion about the assumptions.}\label{paragraph:discussion_assumptions}\hypertarget{paragraph:discussion_assumptions}{}\

%\answer{Let us now present the consequences of our assumptions, how they are useful in our proofs, and why we made them.}\
%\paragraph{Perspectives.} To continue this study, a new question can be raised: Under which conditions the marginals of the stationary profile of our model over each coordinate are, at least approximately, independent and identically distributed~? In other words, recalling the function $N_0~\in~L^1(\Psi)$ in the expression of the stationary profile, see Theorem~\ref{te:main_result}, under which condition we have the existence of $n_0\in L^1(\mathbb{R}_+)$, such that for all~$x\in\mathbb{R}_+^{2k}$
%$$
%N_0(x) \approx \prod_{i = 1}^{2k} n_0(x_i).
%$$
%This approximation was made in the previous studies of this phenomenon~\cite{Xu2013,bourgeron_2015,Martin2021,rat_individual_2023}, so obtaining theoretical guarantees for it will further improve the rigour behind these studies. However, it should be noted that telomeres on the same chromosome (at coordinates $i$ and $i +k$) are not independent during shortening: both telomeres cannot be shortened simultaneously. Consequently, it seems not trivial to justify that such an approximation is possible, and further numerical/mathematical studies must be done to obtain a good justification. This question is the subject of a future work.
%\vspace{0.5mm}
\noindent \underline{\textbf{Discussion about \hyperlink{paragraph:long_time_behaviour_S1}{$(S_{1})$}}:}\label{paragraph:discussion_S1}\hypertarget{paragraph:discussion_S1}{}\
\begin{itemize}[leftmargin=*]
\item \hyperlink{paragraph:long_time_behaviour_S1}{$(S_{1}):$} All the assumptions of this set are easy to verify, as they only correspond to lower/upper bounds assumptions.
\item \hyperlink{paragraph:long_time_behaviour_S1.1}{$(S_{1.1}):$}  $1)$ This assumption implies that the cumulative distribution functions of $g$ and $(h(x,))_{x\in\mathbb{R}^{2k}}$ are bijective on their support. Thus, we can simulate shortening and lengthening values with the inverse transform sampling method, and then obtain that Eq.~\eqref{eq:uniform_random_variables_assumption} is verified (the expression of the procedure $R(.,.)$ is a bit complicated, so we do not make it explicit here). From the latter, we get that Theorem~\ref{te:existence_uniqueness} holds, so that the model is well-posed and that the number of individuals does not explode in finite~time. 

\item \hyperlink{paragraph:long_time_behaviour_S1.1}{$(S_{1.1}):$} $2)$ The lower bounds for $g$ and $h$ are useful to prove the Doeblin condition. Having such lower bound assumptions is relatively strong, since it implies a sharp transition at the limit of the support of $g$ and of the functions $\left(h(x,.)\right)_{x\in\mathbb{R}}$. %Any other hypothesis would imply more complicated calculations, that we do not do here since we prefer to focus on the intrinsic difficulties of the model.} %We prefer to focus on the intrinsic difficulties of the model.

\item \hyperlink{paragraph:long_time_behaviour_S1.1}{$(S_{1.1}):$} $3)$ The constant $\Delta > 0$ corresponds to the upper bound of the maximum value that can have the support of a function $h(x,.)$, where $x\in\mathbb{R}$. Having this upper bound simplifies the computations when we obtain the weak form of the Harnack inequality.

\item \hyperlink{paragraph:long_time_behaviour_S1.1}{$(S_{1.1}):$} $4)$  The fact that $x \mapsto \Delta_x$ is bounded from below on compact sets simplify the computations when we obtain the Doeblin condition.

\item \hyperlink{paragraph:long_time_behaviour_S1.2}{$(S_{1.2}):$} $1)$ This assumption means that the probability that every coordinate has been shortened and lengthened after~$m_0$ cell divisions is bounded from below on every compact set. It is useful to obtain the Doeblin condition. 

\item \hyperlink{paragraph:long_time_behaviour_S1.2}{$(S_{1.2}):$} $2)$ The assumption that for all $i\in\llbracket1,m_0\rrbracket$ it holds $\mathbb{I}^i \subset \mathbb{J}^i$ means that when a telomere is shortened at a cell division, it is also lengthened. This allows us to simplify a lot of computations. For more complex models where this assumption is not verified, by slightly changing the assumption, the computations made in this paper can be adapted to obtain the existence of a stationary profile. However, it will be much more laborious. 

\item \hyperlink{paragraph:long_time_behaviour_S1.3}{$(S_{1.3}):$} This assumption is useful to obtain the weak form of the Harnack inequality. 
\end{itemize}

\noindent \underline{\textbf{Discussion about \hyperlink{paragraph:long_time_behaviour_S2}{$(S_{2})$}}:}
\begin{itemize}[leftmargin=*]
\item \hyperlink{paragraph:long_time_behaviour_S2.1}{$(S_{2.1}):$} This assumption means that there is a renewal of the number of individuals, with telomere lengths that stay in $[0,B_{\max}]^{2k}$, after $G$ generations. In particular, if we study the dynamics at generations that are multiples of $G$, and starting from an initial condition in $K_{\text{ren}}$, then the population grows exponentially. As we are in a multidimensional setting, there are cases where it is easier in practice to verify that a renewal occurs in several generations rather than in one generation. The simplest examples are cases where the procedure of lengthening acts on too many coordinates. In~Section~\ref{subsect:model_renewal_several_generations}, this type of model is studied.

\item \hyperlink{paragraph:long_time_behaviour_S2.2}{$(S_{2.2}):$} $1)$ The first statement of this assumption means that for the space distortion defined with~$\mathdutchcal{V}$, a cell does not give birth to too many offspring with telomere lengths outside of $\left[0,B_{\max} L_{\text{ret}}\right]^{2k}$.

\item \hyperlink{paragraph:long_time_behaviour_S2.2}{$(S_{2.2}):$} $2)$ The second statement of this assumption means that the space distortion does not change too much the mean of the reproduction law of the branching process, uniformly in~$I\in\mathcal{I}_k$. In particular, this implies that for all $x\in\mathbb{R}_+^{2k}$ we have
\begin{align}
\mathdutchcal{K}\left(\mathdutchcal{V}\right)(x) &= \frac{2}{\#\left(\mathcal{I}_k\right)}\sum_{I\in\mathcal{I}_k} \sum_{(J,M)\in \mathcal{J}_{k}}\int_{(\alpha_1,\alpha_2)\in (\mathbb{R}_+^{2k})^2} p_{J,M}\left(x - \alpha_1,x - \alpha_2\right)\bigg[\int_{(\beta_1,\beta_2)\in (\mathbb{R}_+^{2k})^2} \nonumber\\ 
&\times\mathdutchcal{V}(x-\alpha_1+\beta_1)1_{\{x-\alpha_1+\beta_1\in\mathbb{R}_+^{2k}\}} d\mu_{(x-\alpha_1,x-\alpha_2)}^{(E;J,M)}(\beta_1,\beta_2)\bigg] \nonumber d\mu^{(S;I)}(\alpha_1,\alpha_2) \\
&\leq 2(1+\varepsilon_1)\mathdutchcal{V}(x). \label{eq:upperbound_fullkernel}
\end{align}
It is important to have this condition uniformly in $I\in\mathcal{I}_k$ to be able to apply the weak form of the Harnack inequality. 

\end{itemize}

\noindent \underline{\textbf{Discussion about \hyperlink{paragraph:long_time_behaviour_S3}{$(S_{3})$}}:}
\begin{itemize}[leftmargin=*]
\item \hyperlink{paragraph:long_time_behaviour_S3.2}{$(S_{3.2}):$} $1)$ The first condition, with the Laplace transform, means that the number of individuals grows exponentially, and at a rate larger than $\alpha$. Briefly, this comes from the fact that by \hyperlink{paragraph:long_time_behaviour_S2.1}{$(S_{2.1})$} and \hyperlink{paragraph:long_time_behaviour_S3.1}{$(S_{3.1})$}, the number of individuals of our model can be bounded from below by the expectation of a Bellman-Harris process whose parameters depends on $\underline{b}(a)$, $1+\varepsilon_0$ and $G$ (see Section~\ref{subsubsect:explanation_proof_A2} for more information). The point is that our condition with the Laplace transform implies an exponential growth of this Bellman-Harris process, at a rate $\alpha$. Briefly, this comes from the fact that the Laplace transform of a Bellman-Harris process is explicit, as it satisfies a renewal equation. Hence, performing an inverse Laplace transform, in view of our condition, gives directly this exponential growth. We refer to~\hbox{\cite[Lemma~$7.3$]{iksanov_asymptotic_2024}} for a proof of this result for Crump-Mode-Jagers processes, which correspond to a generalisation of Bellman-Harris processes, see~\cite[Section~$1$]{iksanov_asymptotic_2024}.

\item \hyperlink{paragraph:long_time_behaviour_S3.2}{$(S_{3.2}):$} $2)$ To understand the second condition, notice that if we are interested in a cell that divides at a rate $\underline{b}$, and starting from an initial age $a$, then $\mathdutchcal{F}_a$ is the probability density function of the division time of this cell. Notice also that if this  cell dies at a rate $\alpha$, then the probability that it is still alive at time $s\geq0$ is $e^{-\alpha s}$. From these interpretations, the second condition means that for our cell, there exists a time $t\geq0$ for which, uniformly with respect to its initial age, the probability that it divides before the time $t$ is positive. This is verified when there exists $a_1,\,b_1 >0$ such that $\underline{b}(a) \geq b_1$ for all~$a\geq a_1$, see the proof of~Corollary~\ref{cor:birth_rate_age}.

\item \hyperlink{paragraph:long_time_behaviour_S3.2}{$(S_{3.2}):$} $3)$ If we slightly modify assumption \hyperlink{paragraph:long_time_behaviour_S2}{$(S_{2})$} and \hyperlink{paragraph:long_time_behaviour_S3}{$(S_{3})$}, and write $\varepsilon_0$ instead of $1+\varepsilon_0$, and $\varepsilon_1$ instead of $\varepsilon_1$ each time these terms appear, then it is possible to relax the condition that $\alpha >0$ in~\hyperlink{paragraph:long_time_behaviour_S3.2}{$(S_{3.2})$}. Relaxing this condition implies that Theorem~\ref{te:main_result} can be true in a subcritical regime (i.e. for $\lambda \leq 0$), so allows us to have a stronger result. We do not relax it here because it seems impossible to us to verify the following condition for~$\varepsilon_1 <1$:
\begin{equation}\label{eq:new_S2.2}
\forall x\in\mathbb{R}_+^{2k}:\hspace{2mm} 	\mathdutchcal{K}_{(\left[0,B_{\max} L_{\text{ret}}\right]^{2k})^c}\left(\mathdutchcal{V}\right)(x) \leq \varepsilon_1\mathdutchcal{V}(x),
\end{equation}
which corresponds to the modified version of \hyperlink{paragraph:long_time_behaviour_S2.2}{$\left(S_{2.2}\right)$-$i)$}. We explain below why.

Recall that the condition in~\eqref{eq:new_S2.2} means that there exists a space distortion $\mathdutchcal{V}$ such that individuals create in average $\varepsilon_1$ new offspring with telomere lengths outside of~$\left[0,B_{\max} L_{\text{ret}}\right]^{2k}$. By our biological constraint, for each telomere, there is at least one daughter cell for which this telomere is not shortened (see the explanation in Section~\ref{subsect:algorithm_model} above Eq.~\eqref{eq:set_combination_shortening}). This leads that if a cell with telomere lengths $x\in\mathbb{R}_+^{2k}$ for which $\exists i_0 \in\llbracket1,2k\rrbracket$ with $x_{i_0} > B_{\max}L_{\text{ret}}$ divides, then this cell gives birth to at least one cell with telomere lengths outside of $\left[0,B_{\max} L_{\text{ret}}\right]^{2k}$. This is in contradiction with the explanation of~\eqref{eq:new_S2.2} we have given at the beginning of the paragraph~(as~$\varepsilon_1 < 1$). Thus, we think this is impossible to verify~\eqref{eq:new_S2.2} for $\varepsilon_1 < 1$.

\item \hyperlink{paragraph:long_time_behaviour_S3.3}{$(S_{3.3}):$} $1)$ This assumption means, in view of~\hyperlink{paragraph:long_time_behaviour_S2.2}{$(S_{2.2})$}, that the number of individuals with telomere lengths outside of $\left[0,B_{\max}L_{\text{ret}}\right]^{2k}$ grows at a rate strictly smaller than~$\alpha$. Similarly to \hyperlink{paragraph:long_time_behaviour_S3.1}{$(S_{3.1})$}, this comes from the fact that by \hyperlink{paragraph:long_time_behaviour_S2.2}{$(S_{2.2})$} and~\hyperlink{paragraph:long_time_behaviour_S3.2}{$(S_{3.2})$}, the number of individuals outside of $\left[0,B_{\max}L_{\text{ret}}\right]^{2k}$ can be bounded from above by the expectation of a Bellman-Harris process whose parameters depends on $\overline{b}(a)$ and~$1+\varepsilon_1$ (see Section~\ref{subsubsect:explanation_proof_A2} for more information). Then, as this Bellman-Harris process grows at a rate $\beta$ by the condition on the Laplace transform (see the explanation given in the first point of the discussion about \hyperlink{paragraph:long_time_behaviour_S3.2}{$(S_{3.2})$}), which is strictly smaller than~$\alpha$, the interpretation presented at the beginning of the paragraph comes.

$2)$ The assumption that $\beta < \alpha$ is crucial here, and implies that \hyperlink{paragraph:long_time_behaviour_S3.2}{$(S_{3.2})$} and \hyperlink{paragraph:long_time_behaviour_S3.3}{$(S_{3.3})$} are closely related. Due to the fact that there is the same constant $\varepsilon_0$ in~\hyperlink{paragraph:long_time_behaviour_S2.1}{$(S_{2.1})$} and~\hyperlink{paragraph:long_time_behaviour_S3.2}{$(S_{3.2})$}, and the same constant $\varepsilon_1$ in~\hyperlink{paragraph:long_time_behaviour_S2.2}{$(S_{2.2})$} and~\hyperlink{paragraph:long_time_behaviour_S3.3}{$(S_{3.3})$}, these four assumptions must often be verified together. Criteria allowing to verify \hyperlink{paragraph:long_time_behaviour_S2.2}{$(S_{2.2})$}, \hyperlink{paragraph:long_time_behaviour_S3.2}{$(S_{3.2})$} and~\hyperlink{paragraph:long_time_behaviour_S3.3}{$(S_{3.3})$} from~\hyperlink{paragraph:long_time_behaviour_S2.1}{$(S_{2.1})$} are provided in~Sections~\ref{subsubsect:general_criteria_lyapunov} and~\ref{subsubsect:practical_criteria_lyapunov}.

$3)$ As stated in Corollary~\ref{cor:birth_rate_age}, the condition~$\beta < \alpha$ is equivalent to $1 +\varepsilon_1 < (1+\varepsilon_0)^{\frac{1}{G}}$ when  $\underline{b} \equiv \overline{b} \equiv b$.

\end{itemize}

\section{Space distortion and auxiliary pure jump Markov processes}\label{sect:auxiliary_process}
This section focuses on the introduction of the weighted-normalised semigroup $\big(M_t^{(\psi)}\big)_{t\geq 0}$ and on its associated auxiliary process $\big(Z_t^{(\psi)}\big)_{t\geq0}$, for all $\psi\in\Psi$. First, in Section \ref{subsection:first_moment_semigroup}, we study the well-posedness of $(M_t)_{t\geq0}$ for functions in $\cup_{\psi\in\Psi}\mathdutchcal{B}(\psi)$. Then, we introduce the weighted-normalised semigroup $\big(M_t^{(\psi)}\big)_{t\geq0}$, and prove that the equation satisfied by the latter has a unique solution in Section \ref{subsect:weighted_renormalised_auxiliary}. Finally, we construct the auxiliary process $\big(Z_t^{(\psi)}\big)_{t\geq0}$ thanks to the equation of $\big(M^{(\psi)}_t\big)_{t\geq0}$ for all $\psi\in\Psi$ in Section~\ref{subsect:construction_auxiliary}. Throughout this section, Assumptions \hyperlink{paragraph:long_time_behaviour_S1.1}{$(S_{1.1})$} and \hyperlink{paragraph:long_time_behaviour_S2.2}{$(S_{2.2})$} hold. We also assume for the rest of the paper that~\eqref{eq:birth_rate_assumption} holds. Until Lemma \ref{lemma:well_definition_semigroup} is stated, the semigroup $(M_t)_{t\geq 0}$ introduced in \eqref{eq:first_moment_semigroup} is only defined for non-negative measurable functions. Except Lemmas~\ref{lemm:generalized_duhamel_n=1} and~\ref{lemm:generalized_duhamel}, all the statements given in this section are proved in Appendix~\ref{appendix:proof_auxiliary}. 
\subsection{Well-posedness of weighted-normalised semigroups}\label{subsection:first_moment_semigroup}

To obtain the ergodic behaviour of our semigroup $(M_t)_{t\geq0}$, a naive approach is to construct an absorbing Markov process with first moment semigroup $\big(e^{-\tilde{\lambda}t}M_t\big)_{t\geq0}$ where~$\tilde{\lambda} > 0$, then apply~\hbox{\cite[Theorem~$2.8$]{velleret_exponential_2023}} to this process, and finally conclude by multiplying by $e^{\tilde{\lambda}t}$. However, two problems occur with this approach.
\begin{itemize}[leftmargin=*]
\item The first issue is that when the birth rate $b$ is unbounded, there is no absorbing Markov process with first moment semigroup $\big(e^{-\tilde{\lambda}t}M_t\big)_{t\geq0}$.  Indeed, a necessary condition for the existence of such a process is that for all $t \geq 0$ and~\hbox{$(x,a)\in\mathcal{X}$},  
\begin{equation}\label{eq:bound_number_individuals}
e^{-\tilde{\lambda}t}M_t(1)(x,a) \leq 1. 
\end{equation}
The reason is that this implies that the number of individuals of such a process is in average lower than $1$ at any~time, so can be represented by one individual that may be absorbed by a cemetery state. A consequence of~\eqref{eq:bound_number_individuals} is that for all $(x,a)\in\mathcal{X}$
$$
\frac{d}{dt}\left(M_t(1)(x,a)\right)\Big|_{t=0} = \lim_{t\rightarrow0_+} \frac{1}{t}\left(M_t(1)(x,a)-1\right) \leq \lim_{t\rightarrow0_+} \frac{1}{t}\left(e^{\tilde{\lambda} t}-1\right) = \tilde{\lambda},
$$
so that the following must hold to construct our absorbing Markov process
$$
\sup_{(x,a)\in\mathcal{X}}\left[\frac{d}{dt}\left(M_t(1)(x,a)\right)\Big|_{t=0}\right] < +\infty.
$$
The problem is that in view of Theorem~\ref{te:PDE} and the fact that $\left(M_t\right)_{t\geq0}$ is similar to~$\left(n_t\right)_{t\geq0}$ starting from a Dirac (see~\eqref{eq:dft_mean_measure}-\eqref{eq:first_moment_semigroup}), we have for all $(x,a)\in\mathcal{X}$
$$
\frac{d}{dt}\left(M_t(1)(x,a)\right)\Big|_{t=0} = b(x,a)\left[\mathdutchcal{K}\left(1\right)(x)-1\right],
$$
for which the supremum is not finite when $b$ is not bounded. Hence, this is impossible that~\eqref{eq:bound_number_individuals} holds when $b$ is unbounded, which implies that we cannot find an absorbing Markov process with semigroup $\big(e^{-\tilde{\lambda}t}M_t\big)_{t\geq0}$.
\item The second issue is that the speed of convergence of the semigroup $(M_t)_{t\geq0}$ towards a stationary profile depends on the initial distribution of the trait (the prefactor, not the exponential decay). Indeed, it will take more time for dynamics starting from a cell with very large telomere lengths to converge towards a stationary distribution compared to starting from cells with smaller telomere lengths. Hence, applying directly \cite[Theorem~$2.8$]{velleret_exponential_2023} will not work as this theorem only handle models with a speed of convergence independent of the initial distribution of the trait.
\end{itemize}
In view of these problems, we weight the semigroup $(M_t)_{t\geq0}$ by a function \hbox{$\psi \in \Psi$}, which gives us the semigroup $\big(M_t^{(\psi)}\big)_{t\geq0}$. Studying $\big(M_t^{(\psi)}\big)_{t\geq0}$ rather than $(M_t)_{t\geq0}$ allows us to manage the problems explained above for the following reasons (we recall that~\hbox{$\psi(x,a) = (1+a^{d_{\psi}})\mathdutchcal{V}(x)$} with $d_{\psi}\in\mathbb{N}^*$):
\begin{itemize}[leftmargin=*]
\item In view of the definition of $(M_t^{(\psi)})_{t\geq0}$ given in~\eqref{eq:weighted_renormalised_semigroup}, and Theorem~\ref{te:PDE} combined to the fact that $\psi(,0) = \mathdutchcal{V}$ by~\eqref{eq:defintion_all_distortion_functions} (one need to extend the definition of $(n_t)_{t\geq0}$ to unbounded function as done below), we have for all $(x,a)\in\mathcal{X}$ that
$$
\frac{d}{dt} M_t^{(\psi)}(1)(x,a)\Big|_{t=0} = \frac{\frac{\partial}{\partial a}\psi(x,a)}{\psi(x,a)}+ b(x,a)\left(\frac{\mathdutchcal{K}\left(\mathdutchcal{V}\right)(x)}{\psi(x,a)} - 1\right).
$$
As the latter is bounded, see Lemma~\ref{lemma:inequalities_psi}, the first problem presented above is solved.

\item In view of Assumptions \hyperlink{paragraph:long_time_behaviour_S2}{$(S_2)-(S_{3})$}, weighting the space by $\mathdutchcal{V}$ implies that the time before returning to a cell with telomere lengths in $\left[0,B_{\max} L_{\text{ret}}\right]^{2k}$, even starting from telomere lengths far from this set, is accelerated (for a typical particle conditioned to non-extinction). Therefore, we do not have issues linked to the speed of convergence with this weighting. In Theorem \ref{te:main_result}, the impact of the speed of convergence is reflected in the term $||\mu||_{\mathdutchcal{M}(\psi)}$.
\end{itemize}
In order to use $\big(M_t^{(\psi)}\big)_{t\geq0}$ to obtain the ergodic behaviour of $(M_t)_{t\geq0}$, we first must be sure that $(M_t)_{t\geq0}$ is well-posed for functions in $\cup_{\psi\in\Psi}\mathdutchcal{B}(\psi)$. For that purpose, let us give the following statement that is proved in Appendix~\ref{subsect:proof_inequality_exponential}.
\begin{lemm}[Well-posedness of $(M_t)_{t\geq0}$]\label{lemma:inequality_exponential}
Assume that \hyperlink{paragraph:long_time_behaviour_S1.1}{$(S_{1.1})$} and \hyperlink{paragraph:long_time_behaviour_S2.2}{$(S_{2.2})$} hold. Let us consider $\psi_e(x,a) = e^a\mathdutchcal{V}(x)$ with $\mathdutchcal{V}$ defined in~\hyperlink{paragraph:long_time_behaviour_S2.2}{$(S_{2.2})$}. Then there exists $c >0$ such that for all $T \geq0$, $(x,a)\in\mathcal{X}$,	
$$
M_T (\psi_e)(x,a) \leq \psi_e(x,a)\exp\left(cT\right).
$$
\end{lemm}
\noindent Thanks to the above statement, we are now able to extend the definition of $(M_t)_{t\geq0}$ to functions in $\mathdutchcal{B}(\psi)$. Before extending the definition of $(M_t)_{t\geq0}$, let us introduce some notations. First, we denote the set
$$
\mathcal{Q}_{k} := \mathcal{I}_k\times\mathcal{P}\left(\llbracket1,2k\rrbracket\right).
$$
Then, for all $(I,J) \in \mathcal{Q}_{k}$, we introduce the measure $\pi_x^{I,J}$ defined for all $C\in\mathcal{B}\left(\mathbb{R}^{2k}\right)$ as
\begin{equation}\label{eq:measure_by_event}
\begin{aligned}
\pi_x^{I,J}(C) &:= \frac{1}{\#\left(\mathcal{I}_k\right)}\sum_{\substack{M \in \mathcal{P}(\llbracket1,2k\rrbracket)}}\int_{(\alpha_1,\alpha_2)\in (\mathbb{R}_+^{2k})^2} p_{J,M}\left(x- \alpha_1,x - \alpha_2\right)\\ 
&\times\left[\int_{(\beta_1,\beta_2)\in (\mathbb{R}_+^{2k})^2}1_{C}(x-\alpha_1+\beta_1) d\mu_{(x-\alpha_1,x-\alpha_2)}^{(E;J,M)}(\beta_1,\beta_2)\right]d\mu^{(S;I)}(\alpha_1,\alpha_2), \\
\end{aligned}
\end{equation}
where the measures $\mu^{(S;I)}$ and $\left(\mu_{(s_1,s_2)}^{(E;J,M)}\right)_{(s_1,s_2)\in\mathbb{R}^{2k}}$ are defined in \eqref{eq:measure_shortening} and \eqref{eq:measure_elongation}. This measure represents the kernel for updating telomere lengths of the daughter cell $A$ during the cell division when the indices of the coordinates that are shortened in this daughter cell are those in $I$, and the indices of the coordinates that are lengthened are those in~$J$. All the measures introduced above are helpful to construct a particle~$\big(Z_t^{(\psi)}\big)_{t\geq0}$, for which it is sufficient to only have information of what happens to the daughter cell~$A$, and not to both daughter cells.

\begin{rem}\label{rem:equality_kernel}
Under \hyperlink{paragraph:long_time_behaviour_S1.1}{$(S_{1.1})$}, for all $x\in\mathbb{R}_+^{2k}$, $C\in\mathcal{B}(\mathbb{R}_+^{2k})$
$$
\begin{aligned}
2\sum_{(I,J)\in\mathcal{Q}_{k}} \int_{u\in\mathbb{R}^{2k}}1_{\{x+u\in C\}}d\pi_x^{I,J} &= 2\int_{(w_1,w_2)\in(\mathbb{R}^{2k})^2} 1_{\{x + w_1\in C\}} d\Pi_x(w_1,w_2) = \mathdutchcal{K}(C).
\end{aligned}
$$
\end{rem} 
\noindent We extend the definition of $(M_t)_{t\geq0}$ given in \eqref{eq:first_moment_semigroup} to functions in $\cup_{\psi\in\Psi}\mathdutchcal{B}(\psi)$ with the following lemma that is proved in Appendix~\ref{subsect:proof_well_definition_semigroup}. 

\begin{lemm}[Extension of the definition of $(M_t)_{t\geq0}$]\label{lemma:well_definition_semigroup}
Let us assume that \hyperlink{paragraph:long_time_behaviour_S1.1}{$(S_{1.1})$} and \hyperlink{paragraph:long_time_behaviour_S2.2}{$(S_{2.2})$} hold. Then:
\begin{enumerate}[leftmargin=0.5cm]
\item For any non-negative $f\in\cup_{\psi\in\Psi}\mathdutchcal{B}(\psi)$, $(x,a)\in\mathcal{X}$, $t\geq0$, $M_tf(x,a)$ is finite. In particular, if we consider for all $f\in\cup_{\psi\in\Psi}\mathdutchcal{B}(\psi)$: $f_+$ its positive part and $f_-$ its negative part, then we can extend the definition of $M$ to $\cup_{\psi\in\Psi}\mathdutchcal{B}(\psi)$ as follows
$$
M_t(f)(x,a) := M_t(f_+)(x,a) - M_t(f_-)(x,a).
$$

\item $(M_t)_{t\geq0}$ defined above is a positive semigroup which satisfies the following equation for every $f\in\cup_{\psi\in\Psi}\mathdutchcal{B}(\psi)$ and $(x,a)\in\mathcal{X}$:
\begin{equation}\label{eq:equation_integral_semigroup}
	\begin{split}
		M_t(f)(x,&a) = f(x,a+t)\exp\Bigg(-\int_a^{a+t}b(x,s)ds\Bigg)  + 2\sum_{(I,J)\in\mathcal{Q}_{k}}\int_0^tb(x,a+s)\\ 
		&\times\exp\left(-\int_a^{a+s}b(x,v)dv\right)\left[\int_{u\in\mathbb{R}^{2k}} M_{t-s}f(x+u,0) 1_{\{x+u\in\mathbb{R}_+^{2k}\}}d\pi^{I,J}_x(u)\right]ds.
	\end{split}
\end{equation}
\end{enumerate}
\end{lemm}
\noindent We are now able to study the long time behaviour of $(M_t)_{t\geq0}$ on $\mathdutchcal{M}(\psi)$. We first choose the value of $\lambda_{\psi}$ allowing to construct an auxiliary process and then apply \cite[Theorem~$2.8$]{velleret_exponential_2023}.
\subsection{Choice of normalisations and equations of auxiliary semigroups}\label{subsect:weighted_renormalised_auxiliary}
To be able to represent our weighted-renormalised semigroup $\big(M_t^{(\psi)}\big)_{t\geq0}$ with a Markov process, we must choose $\lambda_{\psi} > 0$ such that for all $(x,a)\in\mathcal{X}$ and $t\geq0$ it holds~$M_t^{(\psi)}(1)(x,a) \leq 1$. The following lemma allows us to do so, and is proved in Appendix~\ref{subsect:proof_inequalities_psi}.
\begin{lemm}[Inequalities related to $\psi$]\label{lemma:inequalities_psi}
Let us assume that \hyperlink{paragraph:long_time_behaviour_S1.1}{$(S_{1.1})$} and \hyperlink{paragraph:long_time_behaviour_S2.2}{$(S_{2.2})$} hold. Then, for all $\psi\in\Psi$, there exist~$\lambda_{\psi}>0$ and~$\overline{\psi} >1$ such that for all $(x,a)\in \mathcal{X}$
\begin{enumerate}[$i)$,leftmargin=0.7cm]
\item $b(x,a)\frac{\mathdutchcal{K}\left(\mathdutchcal{V}\right)(x)}{\psi(x,a)} < \lambda_{\psi} + b(x,a) - \frac{\partial_a\psi(x,a)}{\psi(x,a)}$,
\item $2\frac{b(x,a)}{\psi(x,a)} < \lambda_{\psi} + b(x,a) - \frac{\partial_a\psi(x,a)}{\psi(x,a)}$,
\item $\sup_{s\geq 0}\left|\frac{\psi(x,s+a)}{\psi(x,s)}\right| \leq \overline{\psi}\times\left(1+a^{d_{\psi}}\right)$,
\end{enumerate}
where $\mathdutchcal{K}$ is defined by \eqref{eq:kernel_branching}.
\end{lemm}

\noindent This lemma ensures that the probabilities given later in Eq.~\eqref{eq:probabilities_type_jump_particle} take indeed values in~$[0,1]$, see Remark~\ref{rem:justification_sum_probabilities}. The latter implies that an absorbing Markov process with semigroup~$\left(M_t^{(\psi)}\right)_{t\geq0}$ can be constructed (see Section~\ref{subsect:construction_auxiliary}). This now implies that
$$
M_t^{(\psi)}(1)(x,a) = e^{-\lambda_{\psi}t}\frac{M_t(\psi)(x,a)}{\psi(x,a)} \leq 1.
$$ 
In the rest of the paper, $\lambda_{\psi}$ and $\overline{\psi}$ refer to the bounds in this lemma. In particular, for all $\psi\in\Psi$, the semigroup~$\big(M_t^{(\psi)}\big)_{t\geq0}$
is defined with this $\lambda_{\psi}$. We now give preliminaries to obtain a representation of this semigroup through a process with a particle $\big(Z_t^{(\psi)}\big)_{t\geq0}$. In the rest of this section, we have fixed~$\psi\in\Psi$. By \eqref{eq:equation_integral_semigroup}, $(M^{(\psi)}_t)_{t\geq0}$ satisfies for all $t\geq0$, $f\in M_b(\mathcal{X})$, and~$(x,a)\in\mathcal{X}$ ($d\pi_x^{I,J}$ has been defined in~\eqref{eq:measure_by_event})
\begin{equation}\label{eq:semigroup_renormalised_equation}
\begin{aligned}
M_t^{(\psi)}(f)(x,a) &= f(x,a+t)\frac{\psi(x,a+t)}{\psi(x,a)}\exp\left(-\int_a^{a+t}b(x,s)ds  - \lambda_{\psi}t\right)  \\ 
&+ 2\sum_{(I,J) \in \mathcal{Q}_{k}}\int_0^t\frac{b(x,a+s)}{\psi(x,a)}\exp\left(-\int_a^{a+s}b(x,v)dv -\lambda_{\psi}s\right)\\
&\times\left[\int_{u\in\mathbb{R}^{2k}} M_{t-s}^{(\psi)}(f)(x+u,0) \mathdutchcal{V}(x+u)1_{\{x+u\in\mathbb{R}_+^{2k}\}}d\pi^{I,J}_x(u)\right]ds. 
\end{aligned}
\end{equation}
We can prove that this equation has a unique solution, and thus justify that we can use an auxiliary process $(Z_t^{(\psi)})$ to study $\big(M_t^{(\psi)}\big)_{t\geq0}$. This gives us the following statement that is proved in Appendix~\ref{subsect:proof_unique_solution_equation_semigroup}.
\begin{lemm}[Unique solution to~\eqref{eq:semigroup_renormalised_equation}]\label{lemma:unique_solution_equation_semigroup}
Let us assume that \hyperlink{paragraph:long_time_behaviour_S1.1}{$(S_{1.1})$} and \hyperlink{paragraph:long_time_behaviour_S2.2}{$(S_{2.2})$} hold. We fix \hbox{$T > 0$}, and introduce for all $f\in M_b(\mathcal{X})$ the operator $\Gamma_f : M_b([0,T]\times\mathcal{X}) \rightarrow M_b([0,T]\times\mathcal{X})$, defined for all~$F\in M_b([0,T]\times\mathcal{X})$ and~\hbox{$ (t,x,a)\in[0,T]\times\mathcal{X}$} as:
\begin{equation}\label{eq:operator_semigroup_unique}
\begin{aligned}
	\Gamma_f(F)(t,x,a) &= f(x,a+t)\frac{\psi(x,a+t)}{\psi(x,a)}\exp\left(-\int_a^{a+t}b(x,s)ds  - \lambda_{\psi}t\right)\\
	&+2\sum_{(I,J) \in \mathcal{Q}_{k}}\int_0^t\frac{b(x,a+s)}{\psi(x,a)}\exp\left(-\int_a^{a+s}b(x,v)dv -\lambda_{\psi}s\right)\\
	&\times\left[\int_{u\in\mathbb{R}^{2k}} F(t-s,x+u,0)\mathdutchcal{V}(x+u) 1_{\{x+u\in\mathbb{R}_+^{2k}\}}d\pi^{I,J}_x(u)\right]ds.
\end{aligned}
\end{equation}
Then, there exists a unique $\overline{f}\in M_b\big([0,T]\times\mathcal{X}\big)$ that is solution to $\Gamma_f(\overline{f}) = \overline{f}$.
\end{lemm}
\noindent By Lemma \ref{lemma:unique_solution_equation_semigroup}, if we construct an auxiliary process such that its first moment semigroup satisfies \eqref{eq:semigroup_renormalised_equation}, then the first moment semigroup of this auxiliary process and $\big(M_t^{(\psi)}\big)_{t\geq0}$ have the same values.  Obtaining the ergodic behaviour of this auxiliary process is then equivalent to obtaining the ergodic behaviour of~$\big(M_t^{(\psi)}\big)_{t\geq0}$. Before doing the construction of our auxiliary process, we give a more suitable expression for \eqref{eq:semigroup_renormalised_equation}. We denote for all $(I,J) \in \mathcal{Q}_{k}$ and  $(x,a)\in\mathcal{X}$ the following
\begin{equation}\label{eq:definition_probabilities}
d^{I,J}(x) := \int_{u\in\mathbb{R}^{2k}} \mathdutchcal{V}(x+u)1_{\{x + u\in\mathbb{R}_+^{2k}\}}d\pi^{I,J}_x(u), \hspace{2mm}\text{  and  }\hspace{2mm}q_{\psi}^{I,J}(x,a) = \frac{\frac{2d^{I,J}(x)b(x,a)}{\psi(x,a)}}{\lambda_{\psi} + b(x,a) - \frac{\partial_a \psi(x,a)}{\psi(x,a)}}.
\end{equation}
Then, \eqref{eq:semigroup_renormalised_equation} can be rewritten as follows 
\begin{equation}\label{eq:semigroup_renormalised_equation_algorithm}
\begin{aligned}
&M_t^{(\psi)}(f)(x,a) = f(x,a+t)\exp\left(-\int_a^{a+t}b(x,s)ds - \lambda_{\psi}t + \int_a^{a+t} \frac{\frac{\partial \psi}{\partial a}(x,s)}{\psi(x,s)}ds\right) \\
&+ \int_0^t\bigg[\lambda_{\psi} + b(x,a+s) - \frac{\frac{\partial \psi}{\partial a}(x,a+s)}{\psi(x,a+s)}\bigg]\exp\bigg[-\int_a^{a+s}b(x,v)dv  - \lambda_{\psi}s  + \int_a^{a+s} \frac{\frac{\partial \psi}{\partial a}(x,u)}{\psi(x,u)} du\bigg] \\ 
&\times\sum_{(I,J) \in \mathcal{Q}_{k}}q_{\psi}^{I,J}(x,a+s)\bigg[\frac{1}{d^{I,J}(x)} \int_{u\in\mathbb{R}^{2k}} M_{t-s}^{(\psi)}(f)(x+u,0)\mathdutchcal{V}(x+u)1_{\{x+u\in\mathbb{R}_+^{2k}\}}d\pi^{I,J}_x(u)\bigg]ds.
\end{aligned}
\end{equation}
This means that we can see \eqref{eq:semigroup_renormalised_equation} as the Duhamel's formula of a jump Markov process with an absorbing state such that:
\begin{itemize}[leftmargin=*]
\item The process jumps at a rate 
$$
(x,a)\in\mathcal{X}\mapsto\lambda_{\psi} + b(x,a) - \frac{\frac{\partial}{\partial a}\psi(x,a)}{\psi(x,a)}.
$$

\item At each jump, if we denote $(x,a)\in\mathcal{X}$ is the trait of the particle at the jump, then:
\begin{itemize}[leftmargin=0.4cm]
\item For any $(I,J)\in\mathcal{Q}_{k}$, with probability $q_{\psi}^{I,J}(x,a)$, the trait of the particle is updated such that
$$
(x,a) \longrightarrow (x+U,0),
$$
where the random variable $U$ is distributed according the probability measure
$$
\frac{1}{d^{I,J}(x)}\mathdutchcal{V}(x+u)1_{\{x+u\in\mathbb{R}_+^{2k}\}} d\pi_x^{I,J}(u).
$$
\item With probability $1 - \sum_{(I,J)\in\mathcal{Q}_{k}} q^{I,J}(x,a)$, the jump Markov process jumps to a cemetery.
\end{itemize}
\end{itemize}
We now construct an auxiliary process that follows these dynamics.

\subsection{Algorithmic construction of auxiliary processes associated to weighted-normalised semigroups}\label{subsect:construction_auxiliary}

Let $\psi\in\Psi$. We keep the same $\psi$ in this whole subsection. Even if all the objects introduced in this subsection depend on $\psi$, we will drop the index to mark the dependence in $\psi$ to simplify notations, except for $\big(Z_t^{(\psi)}\big)_{t\geq0}$. Our aim in this subsection is to construct an absorbing Markov process such that its first moment semigroup satisfies \eqref{eq:semigroup_renormalised_equation_algorithm}. To do this, we introduce the following mathematical objects:
\begin{itemize}[leftmargin=*]
\item For all $y\in\mathbb{R}_+^{2k}$, $r,s\geq0$ we define the following functions, that are the complementary cumulative distribution functions and the probability density functions for times between jumps
\begin{align}    
\overline{\mathdutchcal{H}}_{s}(y,r) &= \exp\left(-\int_s^{s+r} b\left(y,u\right) du - \lambda_{\psi}r + \int_s^{s+r}\frac{\frac{\partial \psi}{\partial a}(y,u)}{\psi(y,u)} du\right),\label{eq:tail_events} \\
\mathdutchcal{H}_{s}(y,r) &= \left[\lambda_{\psi} + b(y,s+r) - \frac{\frac{\partial \psi}{\partial a}(y,s+r)}{\psi(y,s+r)}\right]\overline{\mathdutchcal{H}}_{s}(y,r)\label{eq:density_events}.  
\end{align}

\item For all $(y,s)\in\mathcal{X}$, we introduce $\left(T_n(y,s)\right)_{n\in\mathbb{N}}$ an i.i.d. sequence of random variables taking values in~$\mathbb{R}_+$ and distributed according to $\mathdutchcal{H}_{s}(y,.)$. This sequence of random variables is used to describe jump times of the auxiliary process. We assume here that the sequence $\left(T_n(y,s)\right)_{n\in\mathbb{N}}$ is constructed thanks to an inverse transform sampling. The latter means that we first consider a sequence of i.i.d. random variables~$\left(W_n\right)_{n\in\mathbb{N}}$ following an uniform distribution over $[0,1]$, and then define~$T_n(y,s)$ as follows, for all~$n\in\mathbb{N}$,
$$
T_n(y,s) = \inf\left\{r\geq0,\,|\,1 - \overline{\mathdutchcal{H}}_{s}(y,r) \geq W_n\right\}.
$$
We also assume that for all 
$n\in\mathbb{N}$, we use the same uniform variable $W_n$ to generate all the random variables in the collection~$\left(T_n(y,s)\right)_{\left(y,s\right)\in\mathcal{X}}$. The latter allows us to say that our model can be generated by a countable sequence of random variables, which is required later, see Remark~\ref{rem:hilbert_cube}. As such, for all~$n\in\mathbb{N}$, $(y,s)\in\mathcal{X}$ and $(y',s')\in\mathcal{X}$, the random variables~$T_n(y,s)$ and $T_n(y',s')$ are not independent. The latter does not create any issue since these random variables are never drawn together when we construct our particle.

\item For all $(y,s)\in\mathcal{X}_{\partial}$, we introduce $\Big(\big(\tilde{I}_j,\tilde{J}_j\big)(y,s)\Big)_{j\in\mathbb{N}^*}$ an i.i.d. sequence of random variables taking values in the set $\left(\mathcal{Q}_{k}\cup\left\{(\partial)_2\right\}\right)$. Recalling that $q_{\psi}^{I,J}$ is defined in~\eqref{eq:definition_probabilities}, their distribution is the following, for all $j\in\mathbb{N}^*$,
\begin{equation}\label{eq:probabilities_type_jump_particle}
\begin{aligned}
	\forall(y,s)\in\mathcal{X}:\hspace{2.25mm} \mathbb{P}\left[\big(\tilde{I}_j,\tilde{J}_j\big)(y,s) = (I,J)\right] &= \begin{cases}
		q_{\psi}^{I,J}(y,s), & \text{if }(I,J)\in\mathcal{Q}_k,\\
		1 - \underset{(I,J) \in \mathcal{Q}_{k}}{\sum} q_{\psi}^{I,J}(y,s), &\text{if } (I,J) =(\partial)_2,
	\end{cases} \\
	\mathbb{P}\left[\big(\tilde{I}_j,\tilde{J}_j\big)(\partial,\partial) = \left(I,J\right)\right] &= \begin{cases}
		0, & \hspace{30.22mm}\text{if }(I,J)\in\mathcal{Q}_k,\\
		1, &\hspace{30.22mm}\text{if } (I,J) =(\partial)_2.
	\end{cases}
\end{aligned}
\end{equation}
These random variables are constructed with the same method as the random variables~$\left(T_n(y',s')\right)_{(y',s')\in\mathcal{X},n\in\mathbb{N}}$. The only difference is that we now use a sequence of i.i.d. uniform random variables~$\left(W_n'\right)_{n\in\mathbb{N}}$ independent of $\left(W_n\right)_{n\in\mathbb{N}}$ to do the inverse transform sampling. As such, the variables in the collection $\Big(\big(\tilde{I}_j,\tilde{J}_j\big)(y,s)\Big)_{j\in\mathbb{N}^*}$ are independent of the variables in the collection~$\Big(T_n(y',s')\Big)_{(y',s')\in\mathcal{X},n\in\mathbb{N}}$. In addition, for all $j\in\mathbb{N}^*$,~$(y,s)\in\mathcal{X}_{\partial}$ and~$(y',s')\in\mathcal{X}_{\partial}$, the random variables $\big(\tilde{I}_j,\tilde{J}_j\big)(y,s)$ and $\big(\tilde{I}_j,\tilde{J}_j\big)(y',s')$ are not independent. As above, this is not an issue.

We use the variables $\Big(\big(\tilde{I}_j,\tilde{J}_j\big)(y,s)\Big)_{j\in\mathbb{N}^*}$ to describe the coordinates where there is a jump (shortening and/or lengthening), or if the particle jumps in the cemetery. Assume that the trait of the particle at the $j-$th jump is $(y,s)$. If~$\big(\tilde{I}_j,\tilde{J}_j\big)(y,s) \neq\left(\partial\right)_2$, then the coordinates where there is a shortening at the $j-$th jump are indexed by~$\tilde{I}_j(y,s)$, and the coordinates where there is a lengthening are indexed by $\tilde{J}_j(y,s)$. Otherwise, we have $\big(\tilde{I}_j,\tilde{J}_j\big)(y,s) = \left(\partial\right)_2$ and the particle jumps to the cemetery at the $j-$th jump. 
\begin{rem}\label{rem:justification_sum_probabilities}
In view of~\eqref{eq:definition_probabilities},~Remark~\ref{rem:equality_kernel}, and the first statement of Lemma~\ref{lemma:inequalities_psi}, we have for all $(y,s)\in\mathcal{X}$ that
$$
\underset{(I,J) \in \mathcal{Q}_{k}}{\sum} q_{\psi}^{I,J}(y,s) = \frac{\frac{\mathdutchcal{K}\left(\mathdutchcal{V}\right)(x)b(x,a)}{\psi(x,a)}}{\lambda_{\psi} + b(x,a) - \frac{\partial_a \psi(x,a)}{\psi(x,a)}} < 1.
$$
As such, the probabilities defined in~\eqref{eq:probabilities_type_jump_particle} take indeed values in $[0,1]$.
\end{rem}

\item For all $(I,J,y)\in \mathcal{Q}_{k}\times\mathbb{R}_+^{2k}$, we introduce $\big(\tilde{U}_{j,I,J}(y)\big)_{j\in\mathbb{N}^*}$ an i.i.d. sequence of random variables taking values in $\mathbb{R}_+^{2k}$. For all $j\in\mathbb{N}^*$, the random variable $\tilde{U}_{j,I,J}(y)$ is distributed according to the probability measure
\begin{equation}\label{eq:distribution_increment_length}
d\mathbb{P}_{\tilde{U}_{j,I,J}(y)}(u) = \frac{1}{d^{I,J}(y)} \mathdutchcal{V}(y+u)1_{\{y+u\in\mathbb{R}_+^{2k}\}}d\pi^{I,J}_y(u).
\end{equation}
These random variables are used to give the jump value at each jump of the auxiliary process. They are constructed with the same method as~$\left(T_n(y,s)\right)_{(y,s)\in\mathcal{X},n\in\mathbb{N}}$ and~$\Big(\big(\tilde{I}_j,\tilde{J}_j\big)(y,s)\Big)_{(y,s)\in\mathcal{X}_{\partial},j\in\mathbb{N}^*}$. The only difference is that we use a new sequence of i.i.d. uniform random variables~$\left(W_n''\right)_{n\in\mathbb{N}}$ independent of $\left(W_n\right)_{n\in\mathbb{N}}$ and $\left(W_n'\right)_{n\in\mathbb{N}}$ to do the inverse transform sampling.

\item We consider the process~$\left(X_n,A_n,I_n,J_n,\mathdutchcal{T}_n,U_n\right)_{n\in\mathbb{N}}$ taking values on the following set~$ \left(\mathcal{X}\times\mathcal{Q}_k\times\mathbb{R}_+\times\mathbb{R}^{2k}\right)\cup\{(\partial)_6\}$ ($\mathcal{X}$ and $\mathcal{Q}_k$ are considered as the cartesian product between two sets), with initial condition $(X_0,A_0,I_0,J_0,0,0)$, and such that for all~$n\in \mathbb{N}$:
\begin{equation}\label{eq:algorithm_process_jump}
\begin{cases}
	(I_{n+1},J_{n+1}) = \big(\tilde{I}_{n+1},\tilde{J}_{n+1}\big)\left(X_{n},\,T_n\left(X_n, A_n\right)\right), \\
	\\
	U_{n+1}= 
	\begin{cases}
		\tilde{U}_{n+1,I_{n+1},J_{n+1}}\left(X_{n}\right), & \text{if }(I_{n+1},J_{n+1}) \neq (\partial,\partial),\\
		\partial, & \text{otherwise,} 
	\end{cases}\\
	\\
	(X_{n+1},A_{n+1}) = \begin{cases}
		(X_n + U_{n+1},0), & \text{if }  (I_{n+1},J_{n+1}) \neq (\partial,\partial),\\
		(\partial,\partial), & \text{otherwise,}
	\end{cases}\\
	
	\\
	\mathdutchcal{T}_{n+1} = 
	\begin{cases}
		\mathdutchcal{T}_{n} + T_n\left(X_n,  A_n\right), & \text{if }(I_{n+1},J_{n+1}) \neq (\partial,\partial),\\
		\partial, & \text{otherwise.} 
	\end{cases}
\end{cases}
\end{equation}

In the above, $X_0,A_0,I_0,J_0$ are random variables for which the distribution is given when needed (for example when we define semigroups). We also consider the random variables $\mathcal{N}_{\partial}$ and $\tau_{\partial}$ defined as
$$
\begin{aligned}
\mathcal{N}_{\partial} &:= \inf\left\{l\in\mathbb{N},\, (I_l,J_l) = \left(\partial,\partial\right)\right\},\\
\tau_{\partial} &:= \begin{cases}
	\mathdutchcal{T}_{\mathcal{N}_{\partial} - 1} + T_{\mathcal{N}_{\partial} - 1}\left(X_{\mathcal{N}_{\partial}-1},A_{\mathcal{N}_{\partial} - 1}\right), & \text{if }\mathcal{N}_{\partial} \geq 1, \\
	0, & \text{otherwise},
\end{cases}
\end{aligned}
$$
and the process $\left(N_t\right)_{t\geq0}$, defined for all $t\geq0$ as
$$
N_t :=\begin{cases}
\sup\left\{m\in\llbracket 0, \mathcal{N}_{\partial}-1\rrbracket\,|\,\mathdutchcal{T}_m \leq t\right\}, & \text{ if }t < \tau_{\partial}, \\
\mathcal{N}_{\partial}, & \text{ otherwise.}
\end{cases}
$$

The random variables $\mathcal{N}_{\partial}$ and $\tau_{\partial}$ describe the number of jumps and the time before extinction of the dynamics, respectively. For all $t\geq0$, $N_t$ describes the number of jumps that have occurred up to time~$t$. For all $n < \mathcal{N}_{\partial}$, $X_{n}$ describes the values of telomere lengths after $n$ jumps, and $A_n$ is the age of the particle right after the $n$-th jump has occurred. 
For all $n\in \mathbb{N}^*$, $I_{n}$ is the random variable that describes where the coordinates of telomeres of the process~$(X_m)_{m\in \mathbb{N}}$ that are shortened at the $n-$th jump. Similarly, $J_{n}$ describes the coordinates of the telomeres of $(X_m)_{m\in \mathbb{N}}$ that are lengthened at the $n-$th jump. $U_{n}$ describes the value of the jump that the process~$(X_m)_{m\in \mathbb{N}}$ make at the $n-$th jump. 

In order to apply Lemma~\ref{lemma:unique_solution_equation_semigroup} later in the proof, see Section~\ref{subsect:preliminaries_theorem}, we need to have a version of the process $(N_t)_{t\geq 0}$ with an initial condition that can vary, and a cemetery state. Then, we finally introduce a variant to $(N_t)_{\geq0}$ named $(\tilde{N}_t)_{t\geq0}$, that is a process with initial condition $\tilde{N}_0\in\mathbb{N}$, and such that for all $t\geq 0$
\begin{equation}\label{eq:definition_tilde_N}
\tilde{N}_t := \begin{cases}
	\tilde{N}_0 +N_t, & \text{if }t<\tau_{\partial}, \\
	\partial, & \text{otherwise}.
\end{cases}
\end{equation}
This variant is only introduced to be more rigorous, and is not useful in practice.

\begin{rem}\label{rem:age_generation}
One can easily see by~\eqref{eq:algorithm_process_jump} that if $a\in\mathbb{R}_+$ is the initial condition of~$(A_n)_{n\in\mathbb{N}}$, then for all $n \in \mathbb{N}$ it holds 
$$
A_n = \begin{cases}
	a, & \text{ when }n=0, \\
	0, & \text{ when }1\leq n <\mathcal{N}_{\partial},\\
	\partial, & \text{ when }n \geq \mathcal{N}_{\partial} .
\end{cases}
$$
\end{rem}
\begin{rem}\label{rem:initial_condition_discrete_process}
In view of~\eqref{eq:algorithm_process_jump}, one can easily see that the values of the initial conditions $I_0$ and $J_0$ do not influence the values of $\left(X_n,A_n,I_n,J_n,\mathdutchcal{T}_n,U_n\right)_{n\in\mathbb{N}}$. However, as for $(\tilde{N}_t)_{t\geq0}$, we do not fix an initial condition for these two random variables to be able to define rigorously a Markov process later in~Section~\ref{subsect:preliminaries_theorem}.
\end{rem}
\end{itemize}
\noindent Now, we define $\big(Z_t^{(\psi)}\big)_{t\geq0}$ a process taking values in $\mathcal{X}_{\partial}$, such that for all $t\geq0$
\begin{equation}\label{eq:expression_particle}
\begin{aligned}
Z_t^{(\psi)} &= \left(X_{N_t}, A_{N_t} + (t-\mathdutchcal{T}_{N_t})1_{\left\{A_{N_t} \neq \partial\right\}}\right).
\end{aligned}
\end{equation}
One can observe that $(X_0,A_0)$ is the initial condition of $(Z_t^{(\psi)})_{t\geq 0}$. We also consider the following notations, that are very useful for the proof of the main theorem.
\begin{nota}
For all Borel set $A \subset \mathcal{X}$, we denote in the rest of the paper
$$
\tau_{A} := \inf\left\{t>0\,|\,Z_t^{(\psi)}\in A\right\}, \text{ and } T_{A} := \inf\left\{t>0\,|\,Z_t^{(\psi)}\notin A\right\}.
$$
\end{nota}
\noindent We finally introduce the semigroup $(P_t^{(\psi)})_{t\geq0}$ such that for all $f\in M_b(\mathcal{X})$, $t\geq0$, $(x,a)\in\mathcal{X}$
\begin{equation}\label{eq:first_moment_particle}
P_t^{(\psi)}(f)(x,a) = \mathbb{E}\left[f(Z_t^{(\psi)})1_{\{t < \tau_{\partial}\}}\,|\,(X_0,A_0) = (x,a)\right].
\end{equation}
If we condition the expectation with respect to the first jump time, we obtain that this semigroup satisfies~\eqref{eq:semigroup_renormalised_equation_algorithm}. Then, as \eqref{eq:semigroup_renormalised_equation} is equivalent to \eqref{eq:semigroup_renormalised_equation_algorithm}, Lemma \ref{lemma:unique_solution_equation_semigroup} implies that for all $f\in M_b(\mathcal{X})$, $t\geq0$, $(x,a)\in\mathcal{X}$
\begin{equation}\label{eq:equality_semigroup}
P_t^{(\psi)}(f)(x,a) = M_t^{(\psi)} (f)(x,a).
\end{equation}

We now give statements that facilitate the computation of the distribution of the auxiliary process during the proof of Theorem \ref{te:main_result}. First, we introduce the following function that is essential to simplify notations during the proof of the main theorem. For all $(x,a,t)\in\mathcal{X}\times\mathbb{R}_+$
\begin{equation}\label{eq:useful_notation_third}
\begin{aligned}
\mathdutchcal{G}_{a}(x,t) &= 2\frac{b(x,a+t)}{\psi(x,a+t)}\exp\left(-\int_{a}^{a+t} b\left(x,u\right) du - \lambda_{\psi}t + \int_{a}^{a+t} \frac{\frac{\partial }{\partial a}\psi(x,u)}{\psi(x,u)}du\right) \\ 
&= 2\frac{b(x,a+t)}{\psi(x,a)}\exp\left(-\int_{a}^{a+t} b\left(x,u\right) du - \lambda_{\psi}t\right).
\end{aligned}
\end{equation}
Then, we have the following equality, that we easily prove.

\begin{lemm}[Distribution at one jump]\label{lemm:generalized_duhamel_n=1}
Let us assume that \hyperlink{paragraph:long_time_behaviour_S1.1}{$(S_{1.1})$} and \hyperlink{paragraph:long_time_behaviour_S2.2}{$(S_{2.2})$} hold. We consider $t\geq 0$, $B\in\mathcal{B}(\mathbb{R}_+^{2k})$, $(i,j) \in \mathcal{Q}_{k}$ and $C\in\mathcal{B}([0,t])$. Then for any $f \in M_b(\mathcal{X})$ and~$(x,a)\in\mathcal{X}$, we have
$$
\begin{aligned}
&\mathbb{E}_{(x,a)}\left[f\big(Z_t^{(\psi)}\big)\,;\,X_1 \in B,\,N_t = 1,\,(I_1,J_1) = (i,j)\,,\,\mathdutchcal{T}_1 - \mathdutchcal{T}_0 \in C\right] \\
&= \int_{s\in \mathbb{R}_+}\int_{u \in \mathbb{R}^{2k}} f(x + u,t - s) 1_{\{x+u \in B\}} 1_{\{s\in C\}}\mathdutchcal{V}(x+u) \mathdutchcal{G}_a(x,s)\overline{\mathdutchcal{H}}_0(x+u,t-s) ds d\pi_x^{i,\,j}(u).
\end{aligned}
$$
\end{lemm}
\begin{proof}

We denote $\mathcal{E} = \mathbb{E}_{(x,a)}\big[f(Z_t^{(\psi)})\,;\,X_1 \in B,\,N_t = 1,\,(I_1,J_1) = (i,j)\,,\,\mathdutchcal{T}_1 - \mathdutchcal{T}_0  \in C\big]$ in this proof. First, by \eqref{eq:algorithm_process_jump} and~\eqref{eq:expression_particle}, we have on the event $\{N_t = 1,(I_1,J_1) = (i,j)\}$ that it holds \hbox{$Z_t^{(\psi)} =(x + \tilde{U}_{1,I_1,J_1}(x), t - T_0(x,a))$}. Second, we know by the construction of the particle that \hbox{$T_0(x,a) = \mathdutchcal{T}_1 - \mathdutchcal{T}_0$} is the time before the first jump occurs and is distributed according to the density~$\mathdutchcal{H}_a(x,.)$. Finally, on the event $\{N_t = 1\}$, the second jump has not yet occurred at time $t$, implying that $\mathdutchcal{T}_2 - \mathdutchcal{T}_1 = T_1\big(x+\tilde{U}_{1,I_1,J_1}(x),0\big) > t - T_0(x,a)$. From these three points, it comes
\begin{equation}\label{eq:proof_duhamel_generalized_equation}
\begin{aligned}
	\mathcal{E} &= \int_{s\in \mathbb{R}_+} \mathbb{P}_{(x,a)}[(\tilde{I}_1,\tilde{J}_1)(x,a+s) = (i,j)]\left[\int_{u \in \mathbb{R}^{2k}} f(x+u,t - s)1_{\{x+u \in B\}}\right.\\
	&\times\left.1_{\{s\in C \}}\mathbb{P}[T_1(x+u,0) > t-s] d\mathbb{P}_{\tilde{U}_{1,i,j}(x)}(u)\right] \mathdutchcal{H}_a(x,s)ds.    
\end{aligned}
\end{equation}
In addition,  one can easily develop the expression of the following mathematical objects, using~\eqref{eq:density_events}, \eqref{eq:probabilities_type_jump_particle} combined with~\eqref{eq:definition_probabilities}, and~\eqref{eq:distribution_increment_length}, to obtain
$$
\mathdutchcal{H}_a(x,s)\mathbb{P}_{(x,a)}[(\tilde{I}_1,\tilde{J}_1)(x,a+s) = (i,j)]d\mathbb{P}_{\tilde{U}_{1,i,j}(x)}(u) = \mathdutchcal{V}(x + u)\mathdutchcal{G}_a(x,s)1_{\{x+u\in\mathbb{R}_+^{2k}\}}d\pi_x^{i,j}(u).
$$
Plugging the latter in \eqref{eq:proof_duhamel_generalized_equation}, and using the fact that $\mathbb{P}[T_1(x+u,0) > t-s] = \overline{\mathdutchcal{H}}_0(x+u,t-s)$ allows us to conclude that the lemma is true. 
\end{proof}
\noindent This statement can be extended to the event $\{N_t = n\}$, where $n\in\mathbb{N}\backslash\{0,1\}$, by iterating what we have done to obtain Lemma \ref{lemm:generalized_duhamel_n=1}:

\begin{lemm}[Distribution at $n$ jumps]\label{lemm:generalized_duhamel}
Let us assume that \hyperlink{paragraph:long_time_behaviour_S1.1}{$(S_{1.1})$} and \hyperlink{paragraph:long_time_behaviour_S2.2}{$(S_{2.2})$} hold. Let $t\geq 0$ and $n\in\mathbb{N}\backslash\{0,1\}$. We consider $(B_1,\hdots,B_n)\in \left(\mathcal{B}(\mathbb{R}_+^{2k})\right)^n$, $((i_1,j_1),\hdots,(i_n,j_n)) \in\left(\mathcal{Q}_{k}\right)^n$ and $(C_1,\hdots,C_n)\in \left(\mathcal{B}([0,t])\right)^n$. Then, for any $f \in M_b(\mathcal{X})$ and $(x,a)\in\mathcal{X}$, we have
$$
\begin{aligned}
&\mathbb{E}_{(x,a)}\left[f\big(Z_t^{(\psi)}\big)\,;\,N_t = n,\,\forall p \in \llbracket 1, n \rrbracket: X_p \in B_p,\,(I_p,J_p) = (i_p,j_p)\,,\,\mathdutchcal{T}_{p} - \mathdutchcal{T}_{p-1} \in C_p\right] \\
&= \int_{s_1\in [0,t]}\int_{s_2\in [0,t-s_1]} \hdots \int_{s_n \in \left[0,t-\sum_{i = 1}^{n-1} s_i\right]} \int_{u_1 \in\mathbb{R}^{2k}} \hdots \int_{u_n \in\mathbb{R}^{2k}} f\left(x + \sum_{i = 1}^n u_i,t - \sum_{i = 1}^n s_i\right)  \\
&\times\mathdutchcal{V}(x+u_1)\hdots\mathdutchcal{V}\left(x+\sum_{i = 1}^n u_i\right)\mathdutchcal{G}_a(x,s_1)\mathdutchcal{G}_0(x+u_1,s_2)\hdots \mathdutchcal{G}_0\left(x+\sum_{i = 1}^{n-1}u_i ,s_n\right)  \\
&\times \overline{\mathdutchcal{H}}_0\left(x+\sum_{i = 1}^{n}u_i,t-\sum_{i = 1}^n s_i\right)1_{\left\{\forall p \in \llbracket 1,n \rrbracket\,:\,x+\sum_{i = 1}^p u_l \in B_p,\,s_p \in C_p\right\}}\\
&\times\left(ds_1ds_2\hdots ds_n\right) \left(d\pi_x^{i_1,\,j_1}(u_1)\hdots d\pi_{x+\sum_{i = 1}^{n-1} u_i}^{i_n,\,j_n}(u_n)\right).
\end{aligned}
$$
\end{lemm}
\section{Long time behaviour: Proof of the main theorem}\label{sect:long_time_behaviour}
To obtain the ergodic behaviour of our model, we apply \cite[Theorem~$2.8$]{velleret_exponential_2023} to the auxiliary processes constructed in Section~\ref{subsect:construction_auxiliary}. For the reader's convenience, we give it below slightly rewritten to fit our framework and notations. 

\begin{te}[Theorem~$2.8$ in \cite{velleret_exponential_2023}]\label{te:assumptions_velleret}
Let $\Omega$ be of path type (see~\cite[Def. $(23.10)$]{sharpe1988}), $\mathbb{X}$~a~Polish space, $\partial$ an element outside of $\mathbb{X}$, and $(Z_t)_{t\geq 0}$ a strong Markov process for a complete and right-continuous filtration~$\left(\overline{\mathcal{F}}_t\right)_{t\geq 0}$, taking values in $\mathbb{X}\cup\{\partial\}$ and absorbed at $\partial$. Let us assume that there exist $(D_l)_{l\in\mathbb{N}^*}$ a sequence of closed subsets of $\mathbb{X}$, a probability measure $\nu$ on $\mathbb{X}$, a closed measurable set $E\subset \cup_{l \geq 1}D_l$ and a constant $\rho >0$ such that
\begin{itemize}[leftmargin=0.5cm]

\item \hypertarget{te:assumptions_velleret_A0}{$(\underline{A}_0):$}  For any $l\in\mathbb{N}^*$, $D_{l} \subset \text{int}\left(D_{l+1}\right)$.

\item \hypertarget{te:assumptions_velleret_A1}{$(A_1):$}  For any $l\in\mathbb{N}^*$, there exist $L > l$, $c,t >0$ such that for all $z\in D_{l}$
$$
\mathbb{P}_{z}\left[Z_t \in dz',\,t < \min\left(\tau_{\partial}, T_{D_L}\right)\right] \geq c\nu(dz'),
$$

where $\tau_{\partial}$ is the extinction time of $(Z_t)_{t\geq0}$ and $T_{D_L}$ the first time it exits $D_L$. 

\item \hypertarget{te:assumptions_velleret_A2}{$(A_2):$} It holds 
$$
\rho > \sup \left\{\gamma \in \mathbb{R}\,|\, \sup_{L\geq 1} \inf_{t>0} e^{\gamma t}\mathbb{P}_{\nu}\left[t < \min\left(\tau_{\partial}, T_{D_L}\right)\right] = 0\right\} =: \rho_S,
$$
and
$$
\sup_{z\in\mathbb{X}} \left[\mathbb{E}_{z}\left[\exp\left(\rho\min\left(\tau_{\partial}, \tau_E\right)\right)\right]\right] < +\infty.
$$

\item \hypertarget{te:assumptions_velleret_A3}{$(A_3)_F:$} For all $\varepsilon \in(0,1)$, there exist $t_F,c' >0$ such that for any $z\in E$ there exist two stopping times $U_H$ and $V$ such that
\begin{equation}\label{eq:first_statement_(A3F)}
	\mathbb{P}_{z}\left[Z_{U_H} \in dz';\,U_H< \tau_{\partial}\right] \leq c'\mathbb{P}_{\nu}\left[Z_{V} \in dz';\,V< \tau_{\partial}\right],
\end{equation}
and
\begin{equation}\label{eq:second_statement_(A3F)}
	\left\{\min(\tau_{\partial},t_F) < U_H\right\}= \left\{U_H = \infty\right\}, 
\end{equation}
\begin{equation}\label{eq:third_statement_(A3F)}
	\mathbb{P}_{z}\left[U_H = \infty,\,t_F < \tau_{\partial}\right] \leq \varepsilon\exp\left(-\rho t_F\right).
\end{equation}
\end{itemize}
Then, there exist a unique $(\tilde{\gamma},\tilde{\phi},\tilde{\lambda})\in\mathcal{M}_1(\mathbb{X})\times M_b(\mathbb{X})\times\mathbb{R}_+$, and $C,\omega >0$, such that for all $t>0$, and for all $\tilde{\mu}\in \mathcal{M}(\mathbb{X})$ with $||\tilde{\mu}||_{TV,\mathbb{X}} \leq 1$
$$
\left|\left|e^{\tilde{\lambda} t}\tilde{\mu} P_t - \tilde{\mu}(\tilde{\phi})\tilde{\gamma}\right|\right|_{TV,\mathbb{X}} \leq Ce^{-\omega t},
$$
where $\tilde{\mu} P_t(dy) = \mathbb{P}_{\tilde{\mu}}\left(Z_t\in dy,\,t < \tau_{\partial}\right)$.
\end{te}
\begin{rem}
In~\cite[Theorem~$2.1$]{velleret_unique_2022}, the condition $\mathbb{X} = \cup_{l\geq 1} D_l$ is added to Theorem~\ref{te:assumptions_velleret} to ensure that~$\tilde{\gamma}$ is the unique quasi-stationary distribution of the process $\left(Z_t\right)_{t\geq0}$. As we are only interested in the quasi-stationary distribution corresponding to the Yaglom limit of our process (see~\cite[Def.~2]{meleard_2012}), this uniqueness result is not useful here. 
\end{rem}
\noindent Now, we prove Theorem~\ref{te:main_result}. Let us start with preliminaries where we first give the framework we use~(Polish space, probability space, process, filtration), then check that these objects verify the properties required, thereafter prove~\hyperref[te:assumptions_velleret]{$(\underline{A}_0)$}, and finally present the plan of the rest of the proof.
\subsection{Preliminaries}\label{subsect:preliminaries_theorem}
\paragraph{Framework.} We work with $\mathbb{X} = \mathcal{X}$ endowed with the Euclidean topology. It is well-known that this is a Polish space. We also work on the probability space $\left(\Omega,\mathcal{A},\mathbb{P}\right)$ introduced in Notation~\hyperref[notation:probability_space]{\ref*{notation:probability_space}}. Let $\psi\in\Psi$ and $\big(Z_t^{(\psi)}\big)_{t\geq0}$ its associated auxiliary process constructed in Section \ref{subsect:construction_auxiliary}. From now on, to simplify the notations, we drop the index $\psi$ in $\big(Z_t^{(\psi)}\big)_{t\geq0}$. We keep in mind that $\psi$ is fixed up to the second paragraph of Section~\ref{subsect:existence_stationary_profile}. $(Z_t)_{t\geq 0}$ corresponds to the process we are interested in.

Recall the process $\big(\tilde{N}_t\big)_{t\geq0}$ introduced in~\eqref{eq:definition_tilde_N}. Then, we work with the filtration~$\left(\mathcal{G}_t\right)_{t\geq0}$, defined as the augmentation of the canonical filtration (called here augmented filtration) of the process \hbox{$\left(\overline{Z}_t\right)_{t\geq0}:= \big(Z_t,\tilde{N}_t,I_{N_t},J_{N_t}\big)_{t\geq0}$}, see~\cite[Def. $(6.1)$ - $(ii)$]{sharpe1988}. It is essential to use this filtration because it is on it that hitting times are stopping times~\hbox{\cite[Section~$10$]{sharpe1988}}. It is also essential to consider a filtration induced by $\left(\overline{Z}_t\right)_{t\geq0}$ instead of only $\left(Z_{t}\right)_{t\geq0}$ because we use information coming from the other random variables in the proof. We point out that the random variables $\left(\mathdutchcal{T}_{n}\right)_{0\leq n \leq N_t}$, are induced by $(\tilde{N}_t)_{t\geq0}$ because they are hitting times for the process $(N_t)_{t\geq0}$. We also point out that the random variables~$(U_n)_{1\leq n\leq N_t}$ are induced by the process $\left(X_{N_t}\right)_{t\geq0}$, as for all~$n\in\mathbb{N}^*$ it holds~$U_n = \left(X_{n}-X_{n-1}\right)_{n\in\mathbb{N}}$ on the event~$\{N_t = n\}$ (we do not have information for $U_0$, but we do not need it).

\paragraph{Properties of $\Omega$, $\left(\mathcal{G}_t\right)_{t\geq0}$, and $\left(Z_t\right)_{t\geq0}$.} To apply Theorem~\ref{te:assumptions_velleret}, we need to prove that we work on a sample space $\Omega$ of path type. We also need to prove that $\left(\mathcal{G}_t\right)_{t\geq0}$ is right-continuous, and that $\left(Z_t\right)_{t\geq0}$ satisfies the strong Markov property relative to $\left(\mathcal{G}_t\right)_{t\geq0}$. Having a sample space $\Omega$ of path type is necessary to prove that the stopping time $U_H$ presented in Theorem~\ref{te:assumptions_velleret} is regular with respect to the Markov property, see~\hbox{\cite[Prop. $2.2$]{velleret_exponential_2023}}. In~\hbox{\cite[Sect. A.$3$]{velleret_exponential_2023}}, it is suggested to use~\hbox{\cite[Prop.~$8.8$]{kallenberg_foundations_2021}} to justify this. The problem is that this criterion does not apply when we work with the augmented filtration. We thus proceed in another way here to justify that we work on a sample space of path type. The method we propose is based on the proposition presented below, which is proved in Appendix~\ref{sect:proof_prop_path_type}. It also allows us to prove the right-continuity of $\left(\mathcal{G}_t\right)_{t\geq0}$, and the fact that $\left(Z_t\right)_{t\geq0}$ satisfies the strong Markov property relative to $\left(\mathcal{G}_t\right)_{t\geq0}$. What is interesting with this method is that it can be applied to a large class of Markov processes, namely the piecewise deterministic Markov processes (PDMPs), presented for example in~\hbox{\cite[Chapitre~$2$]{davis1993}.}
\begin{prop}\label{prop:path_type}
Let us consider $\mathbb{X}$ a Lusin space, $\partial$ an element isolated from~$\mathbb{X}$, and~\hbox{$\left(\Omega_1,\mathcal{A}_1\right) = \big(\left[0,1\right]^{\mathbb{N}},\mathcal{A}_1\big)$} a measurable space. We introduce $\left(X_t\right)_{t\geq0}$ a PDMP verifying~\cite[$\left(24.8\right)$]{davis1993}, defined on $\left(\Omega_1,\mathcal{A}_1\right)$, taking values on~$\mathbb{X}\cup\{\partial\}$, such that $\partial$ is an absorbing state. We also consider for all $x\in\mathbb{X}\cup\{\partial\}$ the probability measure $\mathbb{P}_x$, which corresponds to the measure describing the law of $\left(X_t\right)_{t\geq0}$ starting from $x$. Then, there exists a Markov process $\big(\Tilde{X}_t\big)_{t\geq0}$ defined on a measurable space~$\big(\Omega_2,\mathcal{A}_2\big)$, taking values on $\mathbb{X}\cup\{\partial\}$, such that $\big(\Tilde{X}_t\big)_{t\geq0}$ is absorbed at $\partial$, with the following properties:
\begin{enumerate}[leftmargin=0.5cm]
\item There exists $\Omega'_1 \subset \Omega_1$ and $\Phi : \Omega'_1 \rightarrow \Omega_2$ such that $\Omega\backslash \Omega'_1$ is negligible, and such that for all $\omega\in\Omega'_1$ and $t\geq0$ we have $X_t\left(\omega\right) = \Tilde{X}_t\left(\Phi(\omega)\right)$.
%\item For all $x\in\left(\mathbb{X}\cup\{\partial\}\right)$, the law of $\big(\Tilde{X}_t\big)_{t\geq0}$ starting from $x$ is given by the measure $\Tilde{\mathbb{P}}_x = \mathbb{P}_x \circ \Phi^{-1}$.
\item The semigroups of $\left(X_t\right)_{t\geq0}$ and $\big(\Tilde{X}_t\big)_{t\geq0}$ are the same.
\item $\big(\tilde{X}_t\big)_{t\geq0}$ satisfies the strong Markov property relative to its augmented filtration. In addition, this filtration is right-continuous.
\item $\Omega_2$ is of path type (i.e. verifies each point of~\cite[Def.~$\left(23.10\right)$]{sharpe1988}).
\end{enumerate}
\end{prop}
\begin{rem}\label{rem:hilbert_cube}
The assumption that the sample space is $\Omega_1 = \left[0,1\right]^{\mathbb{N}}$ means that $\left(X_t\right)_{t\geq0}$ can be simulated thanks to a sequence of uniform random variables on $\left[0,1\right]$.
\end{rem}
\noindent Let us explain how this proposition allows us to verify that $\Omega$ is of path type, the right-continuity of~$\left(\mathcal{G}_t\right)_{t\geq0}$, and the fact that $\left(Z_t\right)_{t\geq0}$ satisfies the strong Markov property. Assume that the assumptions of Proposition~\ref{prop:path_type} are verified for $\left(X_t\right)_{t\geq0} = \big(\overline{Z}_t\big)_{t\geq0}$, where~$\big(\overline{Z}_t\big)_{t\geq0}$ is defined at the beginning of the subsection. Then, by the above proposition there exists a realisation of our semigroup in another sample space, such that the sample space is of path type.  In addition, this other realisation satisfies the strong Markov property relative to its augmented filtration, and this last filtration is right-continuous. Thus, by saying that we work with this other realisation instead of the one on the original sample space, we have that $\Omega$ and~$\left(\mathcal{G}_t\right)_{t\geq0}$ satisfy the properties we need, and that $\big(\overline{Z}_t\big)_{t\geq0}$ satisfies the strong Markov property relative to~$\left(\mathcal{G}_t\right)_{t\geq0}$. The strong Markov property for~$\big(Z_t\big)_{t\geq0}$ then comes from the fact that the strong Markov property of $\big(\overline{Z}_t\big)_{t\geq0}$ implies for all $f\in M_b\left(\mathcal{X}\right)$, $T$ stopping time of $\left(\mathcal{G}_t\right)_{t\geq0}$, and $t\geq0$,
\begin{equation}\label{eq:strong_markov_property}
\mathbb{E}\left[f\left(Z_{T+t}\right);\,T<+\infty\,|\,\mathcal{G}_T\right] = 1_{\left\{T < +\infty\right\}}\mathbb{E}\left[f\left(Z'_{t}\right)\,|\,\overline{Z}'_{0} = \overline{Z}_{T}\right],
\end{equation}
where $\left(\overline{Z}'_{t}\right)_{t\geq0} = \left(Z_t',\Tilde{N}_t',I_t',J_t'\right)_{t\geq0}$ is a process with the same distribution as~$\left(\overline{Z}_t\right)_{\geq0}$, and independent of it. Indeed, as the distribution of $\left(\overline{Z}'_{t}\right)_{t\geq0}$ depends on its initial condition only through $Z_0'$ (see Section~\ref{subsect:construction_auxiliary}), Eq.~\eqref{eq:strong_markov_property} implies that for all $f\in M_b\left(\mathcal{X}\right)$, $T$ stopping time of $\left(\mathcal{G}_t\right)_{t\geq0}$ and~$t\geq0$,
$$
\mathbb{E}\left[f\left(Z_{T+t}\right);\,T<+\infty\,|\,\mathcal{G}_T\right] = 1_{\left\{T < +\infty\right\}}\mathbb{E}\left[f\left(Z'_{t}\right)\,|\,Z'_{0} = Z_{T}\right].
$$

%for $f\in M_b\left(\mathcal{X}\right)$, $T$ stopping time oh , and $t\geq0$, it holds
%$$
%\mathbb{P}\left[f\left(\overline{Z}_{T+t}\right),\,T<+\infty\,|\,\mathcal{G}_T\right] = 1_{\left\{T < +\infty\right\}}\mathbb{P}\left[f\left(\overline{Z}'_{t}\right),\,|\,\overline{Z}'_{0} = \overline{Z}_{T}\right].
%$$
%We thus have

%We then obtain that $\left(Z_t\right)_{t\geq0}$ also satisfies the Strong Markov property relative to~$\left(\mathcal{G}_t\right)_{t\geq0}$, as this is a coordinate of $\left(\overline{Z}_t\right)_{\geq0}$.

It thus only remains to verify that the assumptions of Proposition~\ref{prop:path_type} are verified for~$\left(X_t\right)_{t\geq0} = \big(\overline{Z}_t\big)_{t\geq0}$. First, in view of the construction of the processes $\left(Z_t\right)_{t\geq0}$, $\big(\Tilde{N}_t\big)_{t\geq0}$, and $\left(I_n,J_n\right)_{n\in\mathbb{N}}$ in Section~\ref{subsect:construction_auxiliary}, one can simulate the process $\big(\overline{Z}_t\big)_{t\geq0}$ thanks to a sequence of uniform random variables on~$[0,1]$. Thus, in view of Remark~\ref{rem:hilbert_cube}, we can assume  without loss of generality that $\left(\overline{Z}_t\right)_{t\geq0}$ has been constructed on a sample space $\Tilde{\Omega} = \left[0,1\right]^{\mathbb{N}}$. Moreover, one can easily see that $\left(\overline{Z}_t\right)_{t\geq0}$ corresponds to a piecewise Markovian deterministic process, and that~\hbox{\cite[$\left(24.8\right)$ - $1.$, $2.$, $3.$]{davis1993}} are verified for this process. To verify~\cite[$\left(24.8\right)$ - $4.$]{davis1993}, we need to prove that starting from any initial condition, the expectation of the number of jumps of~$\big(\overline{Z}_t\big)_{t\geq0}$ is finite at every time. To do so, recall that the jump rate of~$\left(\overline{Z}_t\right)_{t\geq0}$ (before extinction) is 	by Section~\ref{subsect:construction_auxiliary} (in particular~\eqref{eq:tail_events}-\eqref{eq:density_events})
$$
(x,a) \in\mathcal{X} \mapsto \lambda_{\psi} + b(x,a) - \frac{\partial_a\psi(x,a)}{\psi(x,a)},
$$
which is bounded on compact sets by~\eqref{eq:birth_rate_assumption}. Then, by this last property, the number of jumps of~$\left(\overline{Z}_t\right)_{t\geq0}$ on intervals of the form $[0,T]$, where $T>0$, can be bounded from above by the number of jumps of a homogeneous Poisson point process. As this Poisson point process has a finite expectation (see~\cite[Theorem $1.1.1$]{tijms_first_2003}), we obtain that $\left(\overline{Z}_t\right)_{t\geq0}$ verifies~\cite[$\left(24.8\right)$ - $4.$]{davis1993}. Hence, all the assumptions of Proposition~\ref{prop:path_type} are satisfied for~$\left(X_t\right)_{t\geq0} = \big(\overline{Z}_t\big)_{t\geq0}$, which implies that $\Omega$, $\big(\mathcal{G}_t\big)_{t\geq0}$ and~$\left(Z_t\right)_{t\geq0}$ satisfy the properties we need.
%, which implies that $\Omega$, $\big(Z_t\big)_{t\geq0}$ and~$\left(\mathcal{G}_t\right)_{t\geq0}$ satisfy the properties we need
%In addition, \cite[$\left(24.8\right)$ - $4$]{davis1993} is verified from the fact that before jumping to the cemetery state $\left(\partial\right)_5$, the jump rate of $\left(\overline{Z}_t\right)_{t\geq0}$ is $\lambda_{\psi} + b(x,a)$ by Section~\ref{subsect:construction_auxiliary},

\paragraph{Assumption~\hyperref[te:assumptions_velleret]{$(\underline{A}_0)$}.} As this is short, we now give the sequence~$(D_l)_{l\geq 1}$ we use and verify~\hyperref[te:assumptions_velleret]{$(\underline{A}_0)$}. To simplify the future computations, we need that the marginal over telomere lengths in $D_1$ is a subset of~$[0,B_{\max}l]^{2k}$, defined in~\hyperlink{paragraph:long_time_behaviour_S2.2}{$(S_{2.2})$}. For the same reason, we also need to be able to compare the age variable to $a_0$, defined above~\eqref{eq:birth_rate_assumption}. That is why we take for all $l\geq 1$ 
\begin{equation}\label{eq:Dl_expression}
D_l = [0,B_{\max}l]^{2k}\times[0,a_0l], \hspace{6mm} \mathcal{D}_l = [0,B_{\max}l]^{2k}. 
\end{equation}
The set $\mathcal{D}_l$ corresponds to the marginal of $D_l$ over telomere lengths, and is frequently used in this paper. It is trivial to see that this is a sequence of closed subsets such that for all $l \geq 1$: $D_l \subset \text{int}(D_{l+1})$. Thus,~\hyperlink{te:assumptions_velleret_A0}{$(\underline{A}_0)$} is verified for this choice of $(D_l)_{l\geq 1}$.

\paragraph{Plan of the proof.} We now verify Assumptions \texorpdfstring{\hyperref[te:assumptions_velleret]{$(A_1)-(A_2)-(A_3)_F$}}{(A0)-(A1)-(A2)-(A3)F}, which require more computations, in Sections~\ref{subsubsect:assumption_A1}, \ref{subsubsect:assumption_(A2)}, and~\ref{subsubsect:assumption_(A3)F} respectively. Then, in Section \ref{subsect:existence_stationary_profile} we prove that the latter implies the existence of a stationary profile for $(M_t)_{t\geq0}$ in~$\mathdutchcal{M}(\psi)$, and that the stationary profile is the same for all the functions in $\Psi$. Finally, in Section~\ref{subsect:density_stationary_profile}, we obtain a representation of the stationary profile by a function.

\subsection{Assumption \texorpdfstring{\protect\hyperlink{te:assumptions_velleret_A1}{$(A_1)$}}{(A1)} : Doeblin condition on the event \texorpdfstring{$\{t < T_{D_L}\}$}{\{t < t\_\{DL\}\}}}\label{subsubsect:assumption_A1}

Assumption \hyperlink{te:assumptions_velleret_A1}{$(A_1)$} corresponds to a Doeblin condition on the event $\{t < T_{D_L}\}$. Our aim is to show that such a condition is verified, for any initial condition supported on one of the sets~$(D_l)_{l\geq1}$. The following proposition implies \hyperlink{te:assumptions_velleret_A1}{$(A_1)$} with $\nu$ defined by 
$$
\nu(dx,da) := 
\frac{1}{C_{\nu}}\left(\prod_{i = 1}^{2k} (x_i)^{m_0}\right)1_{D_1}(x,a)dxda,$$ where~$C_{\nu}=\int_{(u,v)\in D_1}\left(\prod_{i = 1}^{2k} (u_i)^{m_0}\right)dudv$ and $m_0$ has been introduced in~\hyperlink{paragraph:long_time_behaviour_S1.2}{$(S_{1.2})$}.

\begin{prop}[Doeblin condition]\label{prop:to_prove_A1}
Assume that \hyperlink{paragraph:long_time_behaviour_S1.1}{$(S_{1.1})$}, \hyperlink{paragraph:long_time_behaviour_S1.2}{$(S_{1.2})$} and \hyperlink{paragraph:long_time_behaviour_S2.2}{$(S_{2.2})$} hold. Then for all $l\geq 1$ there exist $L \geq l+2$, and $t,c_F>0$, such that for all $(x,a)\in D_l$
$$
\begin{aligned}
\mathbb{P}_{(x,a)}\left[Z_t\in dx'da';\,t < \min\left(\tau_{\partial},\,T_{D_L}\right)\right] \geq c_F\left(\prod_{i = 1}^{2k} (x'_i)^{m_0}\right)1_{D_l}(x',a')dx'da'.\\
\end{aligned}
$$
\end{prop}
\noindent The above statement is stronger than what we need to have \hyperlink{te:assumptions_velleret_A1}{$(A_1)$}, as we have $1_{D_l}$ instead of $1_{D_1}$ in the right-hand side term. This is because we need it later when we prove \hyperlink{te:assumptions_velleret_A3}{$(A_3)_F$}, see Section~\ref{subsubsect:assumption_(A3)F}. Our aim in the current section is to prove Proposition~\ref{prop:to_prove_A1}.

\noindent Let us denote for all $x\in\mathbb{R}_+^{2k},\,r>0$,
$$
\begin{aligned}
&B(x,r) = \left\{y\in\mathbb{R}_+^{2k}\,\big|\, ||y-x||_{\infty} < r\right\}.
\end{aligned}
$$
Let us define the notions of $(r,l,L,t,c)$- and $(r,l,L,t)$-local Doeblin conditions. We remind that $\mathcal{D}_l$, defined in~\eqref{eq:Dl_expression}, corresponds to the marginal of $D_l$ over telomere lengths. 
\begin{dft}[Local Doeblin condition]\label{dft:local_doeblin_condition}
Let $l,\,L\geq 1$, $r >0$. 
\begin{itemize}[leftmargin=0.5cm]
\item For all $t > 0$, $c > 0$, we say that a $(r,l,L,t,c)$-local Doeblin condition holds from $x_I\in \mathcal{D}_l$ to $x_F\in \mathcal{D}_l$ when for all $(x,a)\in \left[B(x_I,r)\cap \mathcal{D}_l\right]\times[0,a_0l]$ we have
$$
\mathbb{P}_{(x,a)}\left[Z_t\in dyds;\,t< \min(\tau_{\partial},T_{D_L})\right] 
\geq c\left(\prod_{i = 1}^{2k} (y_i)^{m_0}\right)1_{\left\{(y,s)\in B(x_F,r)\cap \mathcal{D}_l\times[0,a_0l]\right\}}dyds.
$$

\item For all $t>0$, we say that a $(r,l,L,t)$-local Doeblin condition holds from $x_I\in \mathcal{D}_l$ to $x_F\in \mathcal{D}_l$ when there exists $c > 0$ such that a $(r,l,L,t,c)$-local Doeblin condition holds from $x_I$ to~$x_F$. 
\end{itemize}
\end{dft}
%\noindent As $\mathcal{D}_l$ is a compact set, for all $r >0$ we can find $N\in\mathbb{N}^*$ and $(x_j)_{j\in \llbracket1,N\rrbracket}\in (\mathcal{D}_l)^{N}$ such that \hbox{$\mathcal{D}_l=\cup_{j\in \llbracket1,N\rrbracket} \left[B(x_j,r)\cap \mathcal{D}_l\right]$}. Therefore, if we prove that there exists $t_F>0$ such that for all $(j,j')\in (\llbracket1,N\rrbracket)^2$, a $(r,l,L,t_F)$-local Doeblin conditions holds from~$x_j$ to~$x_{j'}$, then Proposition \ref{prop:to_prove_A1} will be true.
\noindent We fix 
\begin{equation}\label{eq:dft_constant_r}
r = \min\left(\frac{B_{\max}}{2},\frac{|\min_{x\in D_l}(\Delta_x)-\delta|}{10}\right)
\end{equation}
for the radius of the balls we use to cover~$\mathcal{D}_l$. The constants in the definition of $r$ have been introduced in \hyperlink{paragraph:long_time_behaviour_S1.1}{$(S_{1.1})$} and \hyperlink{paragraph:long_time_behaviour_S2.1}{$(S_{2.1})$}. We now provide the steps of the proof, and then give comments on the first step, which is the most crucial step. %\answer{, and then briefly intuit how local Doeblin conditions can be obtained from $x_I\in\mathcal{D}_l$ to $x_F\in\mathcal{D}_l$, where $x_I$ and $x_F$ are two points in a same ball.}

\paragraph{Steps of the proof.} To obtain Proposition \ref{prop:to_prove_A1}, we follow \cite[Section~$4$]{velleret_exponential_2023}. There are three steps in the proof that we present below. The full proof is detailed in Appendix~\ref{appendix:proof_statements_to_obtain_A1} for the reader’s convenience. 
\begin{enumerate}[leftmargin=*]
\item We prove that there exists $t_I >0$ such that if two points are contained in a same ball of radius $r$, then a $(r,l,L,t_I)$-local Doeblin condition holds between them (see Lemma~\ref{lemma:inequality_doeblin_velleret}). %This gives us at the end Lemma~\ref{lemma:inequality_doeblin_velleret}, stated and proved in Appendix~\ref{subsubsect:inequality_doeblin_velleret}.

\item We use Lemma~\ref{lemma:inequality_doeblin_velleret} to prove that there exists $t_a >0$ such that if a $(r,l,L,t)$-local Doeblin condition holds from $x_I \in \mathcal{D}_l$ to $x_F\in\mathcal{D}_l$ for some $t\geq0$, then for all \hbox{$y\in \left(B(x_F,r)\cap\mathcal{D}_l\right)$}, a $(r,l,L,t+t_a)$-local Doeblin condition also holds from $x_I$ to~$y$ (see Lemma~\ref{lemma:vicinity_velleret}). 
%This gives us at the end Lemma~\ref{lemma:vicinity_velleret}, stated and proved in Appendix~\ref{subsubsect:proof_transfer_localDoeblin}.

\item From the latter, we prove that there exists $t_F >0$ such that a $(r, l, L, t_F)$-local Doeblin condition holds for each pair of points of the set~$\mathcal{D}_l$, even if the points are far apart. Then, as $\mathcal{D}_l$ can be covered by a finite union of balls of radius $r$, the proposition is proved. This step corresponds to the proof of Proposition~\ref{prop:to_prove_A1}, and is detailed in Appendix~\ref{subsect:proof_prop_to_prove_A1}.
\end{enumerate}

\paragraph{Comments on the first step.} The first step corresponds, in fact, to the main argument to obtain Proposition~\ref{prop:to_prove_A1}. The two other steps follow quite naturally from it. The biological interpretation of the result obtained in the first step (see Lemma~\ref{lemma:inequality_doeblin_velleret}) is that in a cell lineage, after a certain amount of time, the last cell may have any telomere lengths and age within a neighbourhood of those of the initial cell. This property is observed here mainly as a consequence of three assumptions.
\begin{itemize}[leftmargin=*]
\item The first of these assumptions is~\hyperlink{paragraph:long_time_behaviour_S1.2}{$(S_{1.2})$}. Biologically, it means that for a cell lineage, it is possible to have at least one shortening and lengthening in all the telomeres after $m_0$ divisions. As a result, the cell of the $m_0$-th generation of the lineage can have telomere lengths within a neighbourhood of those of the initial cell. 
\item The second assumption is the second line of~\eqref{eq:birth_rate_assumption}. It implies biologically that any cell with age greater than $a_0$ will divide after a finite amount of time. It is crucial to have this property to be able to control the probability of reaching the~$m_0$-th generation sufficiently fast.
\item The third assumption is the one from the first line of~\eqref{eq:birth_rate_assumption}. It ensures that the division rate of cells is bounded from above on finite time intervals. This condition guarantees that the time between two successive divisions is not excessively short. This property is essential to ensure that the time at which a lineage reaches the $m_0$-th generation can be large. It is also essential to ensure that a lineage can stay at the $m_0$-th generation during an arbitrary long duration.
\end{itemize}
These assumptions imply constraints on the time $t_I$ at which we obtain local Doeblin conditions in the first step. More precisely, we necessarily have $t_I > a_0m_0 + a_0l$. The~first reason is that the second assumption mentioned above only allows us to ensure that a lineage can reach the $m_0$-th generation after the time $a_0m_0$. The second reason is that once this age is reached, one must wait an additional time~$a_0l$ to be able to attain an age of $a_0l$. In our case, the local Doeblin conditions have only been obtained for~$t_I\in \left[(l+2(m_0-1)+1)a_0,(l+2m_0)a_0\right]$ (see Lemma~\ref{lemma:inequality_doeblin_velleret}). %. The reason is that obtaining local Doeblin conditions for all the possible values of $t_I$ is not the main purpose of this work. By considering intervals of the form $[a_0,a_0+\varepsilon]$, for any $\varepsilon \in (0,a_0]$, instead of the intervals $[a_0,2a_0]$ used in Appendix~\ref{subsubsect:inequality_doeblin_velleret}, it seems possible to obtain local Doeblin conditions for any~$t_I\in \left(a_0m_0 + a_0l,(l+2m_0)a_0\right]$. Then, by taking a larger values for $L$ in Appendix~\ref{subsubsect:inequality_doeblin_velleret}, it seems possible to extend this result for any~$t_I> a_0m_0 + a_0l$.

%Nous référons à la preuve donnée dans la Section~\ref{} pour voir comment chacune de ces hypothèses est utilisée.
%Cette hypothèse est importante car elle permet de justifier qu’à partir d’un instant $2a_0m_0$, la probabilité que l’on soit à la génération $m_0$, et donc que la dernière cellule de la lignée est rallongés et raccourcis tous ces télomères, est non-nulle. Ces hypothèses vont donc induire des contraintes sur le temps auquel nous allons obtenir nos premières estimées locales (temps $t_I$ présenté plus haut), qui devra être plus grand que $2a_0m_0 + a_0 l$. En effet, il faut d'abord attendre un temps $2a_0m_0$ pour être capable de borner la probabilité d'être à la $m_0$-ème génération, et ensuite attendre un temps $a_0l$ pour que la cellule de la $m_0$-eme génération puisse avoir un âge de $a_0$. Cela implique également des restrictions sur le nombre de saut que la particule doit avoir fait pour obtenir ces estimées, qui sera nécessairement égal à $m_0$. Pour plus d'information, nous référons à la preuve donnée dans l'Appendix~\ref{subsubsect:inequality_doeblin_velleret}.}

%\answer{Ces hypothèses impliquent des restrictions sur le temps sur lequel les} 

\subsection{Assumption \texorpdfstring{\protect\hyperlink{te:assumptions_velleret_A2}{$(A_2)$}}{(A2)} : Concentration of the mass of the semigroup conditioned to the non-extinction on one of the \texorpdfstring{$D_l$}{Dl}}\label{subsubsect:assumption_(A2)}

As the set $\mathcal{X}$ is not compact, we need to prove that the mass of the semigroup accumulates on a
compact set. The latter can be proved by obtaining exponential estimates that we present in Section~\ref{subsubsect:explanation_proof_A2}. The usual ways to obtain these estimates is to bound from below/above stopping times involved in the dynamics by exponential variables, as done in~\cite{velleret_unique_2022,velleret_exponential_2023,champagnat_general_2023,martinez_existence_2014}, or to bound the infinitesimal generator and then apply Gronwall's lemma, as presented in~\cite[Prop.~$2$]{bansaye2022} and~\cite[Theo.~$5.1$]{champagnat_general_2023}. However, as mentioned in the introduction, these strategies fail for our model. This is due to the fact that the birth rate is neither bounded from above nor bounded away from $0$, and due to the fact that the renewal of individuals in a compact set may occur in several generations (when \hyperlink{paragraph:long_time_behaviour_S2.1}{$(S_{2.1})$} holds with~$G > 1$). Moreover, even when the birth rate is bounded and $G = 1$, the estimates obtained with these methods are not precise enough to verify the condition $\rho_S < \rho$ in~\hyperlink{te:assumptions_velleret_A2}{$(A_2)$}. Indeed, checking this condition comes down to imposing a restriction on $\varepsilon_0$ and~$\varepsilon_1$ ($\varepsilon_1$ and $\varepsilon_0$ are introduced in \hyperlink{paragraph:long_time_behaviour_S2.1}{$(S_{2.1})$}). This restriction is not optimal due to the loss of information that occurs when the birth rate, which depends on the age, is compared with age-independent rates.

We present here a method to handle these issues based on Bellman-Harris processes, which are the most classical age-dependent branching processes \cite[Chap. IV]{athreya_1972}. In our method, the birth rate only needs to be bounded from below/above by other age-dependent rates (see~\hyperlink{paragraph:long_time_behaviour_S3.1}{$(S_{3.1})$}). The restriction on $\varepsilon_0$ and $\varepsilon_1$ is also better optimised, as information on the evolution of the birth rate with age is used~\hbox{(see~\hyperlink{paragraph:long_time_behaviour_S3.2}{$(S_{3.2})$}-\hyperlink{paragraph:long_time_behaviour_S3.3}{$(S_{3.3})$})}. We are therefore able to manage more cases than when the above strategies are used. 

First, in Section~\ref{subsubsect:explanation_proof_A2}, we present a condition to verify Assumption \hyperlink{te:assumptions_velleret_A2}{$(A_2)$}, as well as the notion of Bellman-Harris process. Then, in Section~\ref{subsubsect:useful_statements_A2}, we present auxiliary statements that allow us to check the condition we have introduced, and verify \hyperlink{te:assumptions_velleret_A2}{$(A_2)$}. Finally, we prove all these auxiliary statements in Sections~\ref{subsubsect:proof_inequality_means},~\ref{subsubsect:proof_prop_renewal_set} and \ref{subsubsect:proof_bound_tail_probability}.
\subsubsection{Condition to verify~\texorpdfstring{\protect\hyperlink{te:assumptions_velleret_A2}{$(A_2)$}}{(A2)} and Bellman-Harris processes}\label{subsubsect:explanation_proof_A2}

Qualitatively, Assumption~\hyperlink{te:assumptions_velleret_A2}{$(A_2)$} combined with Assumption~\hyperlink{te:assumptions_velleret_A1}{$(A_1)$} implies that there exists~$l_{\circ} \in \mathbb{N}^*$ such that the mass of the semigroup (conditioned to the non-extinction) is concentrated on $D_{l_{\circ}}$ after a certain time \cite[Theorem $5.1$]{velleret_unique_2022}. A sufficient condition for this to hold should be that there exists $L_1 \in \mathbb{N}^*$ such that when the particle $\left(Z_t\right)_{t\geq0}$ starts from~$\left(D_{L_1}\right)^c$, the rate at which it leaves this set is larger than the rate at which the particle comes back to~$D_{L_1}$ when it starts from it. Rigorously, this occurs for example when there exist $\alpha',\,\beta'\in\mathbb{R}$ such that $\alpha' > \beta'$, $t_1,\,t_2 >0$, a set $A \subset D_1$ non negligible with respect to~$\nu$, and $L_1\in\mathbb{N}^*$ such that
\begin{align}
\label{eq:renewal_explanation}
&\forall (x,a) \in A:\,\hspace{6.8mm} \mathbb{P}_{(x,a)}\left[Z_{t_1}\in A, \min(\tau_{\partial},T_{D_{L_1}}) > t_1\right] \geq e^{\alpha' t_1},\\
\label{eq:tail_stay_Ec_explanation}
&\forall (x,a) \in (D_{L_1})^c,\,\forall t \geq t_2:\,\hspace{1.9mm}\mathbb{P}_{(x,a)}\left[\min(\tau_{D_{L_1}},\tau_{\partial}) > t\right] \leq e^{\beta' t}.
\end{align}
Indeed, as $A \subset D_{1} \subset D_{L_1}$, the rate at which the particle leaves  $(D_{L_1})^c$ is strictly larger than the rate at which the particle renews in the set $D_{L_1}$. 
Let us prove that~\eqref{eq:renewal_explanation} and~\eqref{eq:tail_stay_Ec_explanation} indeed imply Assumption~\hyperlink{te:assumptions_velleret_A2}{$(A_2)$}. 

\begin{prop}[Condition to verify \hyperlink{te:assumptions_velleret_A2}{$(A_2)$}]\label{prop:simplification_A2}
We suppose that there exist $\alpha',\,\beta'\in\mathbb{R}$ such that $\alpha' > \beta'$, $t_1,\,t_2 >0$, a set $A \subset D_1$ non negligible with respect to $\nu$, and $L_1\in\mathbb{N}^*$ such that \eqref{eq:renewal_explanation} and \eqref{eq:tail_stay_Ec_explanation} are satisfied. Then, \hyperlink{te:assumptions_velleret_A2}{$(A_2)$} is verified with $E = D_{L_1}$ and for any~$\rho \in\left(-\alpha',-\beta'\right)$. Moreover, we have $\rho_S \leq -\alpha'$. 
\end{prop}
\begin{proof}
First, in view of \eqref{eq:renewal_explanation} and \cite[Lemma $3.0.2$]{velleret_unique_2022}, we have 
$$
\rho_S \leq -\frac{\ln(e^{\alpha't})}{t}= -\alpha'.
$$
Now, changing the order of integration we have for all $\rho \in\left(-\alpha',-\beta'\right)$, $(x,a) \in (D_{L_1})^c$,
$$
\begin{aligned}
\mathbb{E}_{(x,a)}\left[\exp\left(\rho\min\left(\tau_{\partial}, \tau_E\right)\right)\right] &= \int_0^{+\infty} e^{\rho u}  d\mathbb{P}_{\min\left(\tau_{\partial}, \tau_E\right)}(u) \\
&= \int_0^{+\infty} \left[\int_{0}^u \rho\exp\left(\rho s\right)ds  +1\right] d\mathbb{P}_{\min\left(\tau_{\partial}, \tau_E\right)}(u) \\
&= \rho\int_0^{+\infty} \mathbb{P}_{(x,a)}\left[\min\left(\tau_{\partial}, \tau_E\right) > s\right]\exp\left(\rho s\right)  ds + 1. \\
\end{aligned}
$$
This implies in particular using \eqref{eq:tail_stay_Ec_explanation} that for all $\rho \in\left(-\alpha',-\beta'\right)$, $(x,a) \in (D_{L_1})^c$
\begin{equation}\label{eq:prop_simplification_A2_intermediate}
\begin{aligned}
\mathbb{E}_{(x,a)}\left[\exp\left(\rho\min\left(\tau_{\partial}, \tau_E\right)\right)\right] &\leq \rho\int_0^{t_2} \exp\left(\rho s\right)  ds + \rho\int_{t_2}^{+\infty}e^{(\rho+\beta') s}ds + 1 < +\infty. \\ 
\end{aligned}
\end{equation}
Finally, for all $(x,a)\in D_{L_1}$ we have $\tau_{D_{L_1}} = 0$ a.s., implying that
$$
\mathbb{E}_{(x,a)}\left[\exp\left(\rho\min\left(\tau_{\partial}, \tau_E\right)\right)\right] = 1.
$$
The latter and Eq.~\eqref{eq:prop_simplification_A2_intermediate} imply that Assumption \hyperlink{te:assumptions_velleret_A2}{$(A_2)$} is verified with $E = D_{L_1}$ and for any~\hbox{$\rho \in\left(-\alpha',-\beta'\right)$}.
\end{proof}
\noindent  We underline the fact that \eqref{eq:renewal_explanation} and \eqref{eq:tail_stay_Ec_explanation} correspond more or less to~\cite[Assump. $(G_2)$]{champagnat_2020} in continuous time, and with $\psi_1 = 1$ , $\psi_2 = 1_{K_{\text{ren}}}$. Our aim is to obtain \eqref{eq:renewal_explanation} and \eqref{eq:tail_stay_Ec_explanation}, and then apply Proposition~\ref{prop:simplification_A2} to verify \hyperlink{te:assumptions_velleret_A2}{$(A_2)$}. 

% Our aim is to obtain \eqref{eq:renewal_explanation} and \eqref{eq:tail_stay_Ec_explanation}, and then apply Proposition~\ref{prop:simplification_A2} to verify \hyperlink{te:assumptions_velleret_A2}{$(A_2)$}.
Given that the age of $\left(Z_t\right)_{t\geq0}$ resets to $0$ at each jump, controlling the rate at which the particle exits $\left(D_{L_1}\right)^c$ is essentially similar to controlling the rate at which $\left(X_{N_t}\right)_{t\geq0}$, the marginal of $\left(Z_t\right)_{t\geq0}$ over telomere lengths (see~\eqref{eq:expression_particle}), exits  $\left(\mathcal{D}_{L_1}\right)^c$ (see~\eqref{eq:Dl_expression}). This return to $0$ at each jump also induces a renewal structure for the age, implying that one mainly needs to establish a sufficiently strong renewal for~$\left(X_{N_t}\right)_{t\geq0}$ in a set~$\Tilde{A} \subset \mathcal{D}_1$ to prove~\eqref{eq:renewal_explanation}. One can therefore think that the age structure does not create any particular difficulty in verifying~\eqref{eq:renewal_explanation} and~\eqref{eq:tail_stay_Ec_explanation}, since the properties that must be established are mostly related only to the process $\left(X_{N_t}\right)_{t\geq0}$. In fact, this age structure is causing most of the difficulties here, as we explained in the introduction of Section~\ref{subsubsect:assumption_(A2)}. The root of the difficulties lies on the fact that the process $\left(X_{N_t}\right)_{t\geq0}$ is not Markovian, due to this age structure. Indeed, at any given time, the time until the next jump of $\left(X_{N_t}\right)_{t\geq0}$ depends on the time since its last jump. This process therefore does not admit an infinitesimal generator, which is useful to establish exponential~bounds. To manage this issue, we use \text{Bellman-Harris processes}~\hbox{\cite[Chap. IV]{athreya_1972}}. These correspond to branching processes defined as follows.
\begin{itemize}[leftmargin=*]
\item Lifetimes of individuals are independent and identically distributed.

\item At the end of its life, an individual gives birth to a certain number of new individuals. The number of new individuals is distributed according to a reproduction law~$(p_k)_{k\in\mathbb{N}}$.
\end{itemize}
The property we exploit to get \eqref{eq:renewal_explanation} and \eqref{eq:tail_stay_Ec_explanation} is that the average number of such a process grows exponentially. Indeed, let us consider an arbitrary Bellman-Harris process. We denote by $f$ the probability density function of individual lifetimes, by $\overline{f}$ its associated complementary cumulative distribution function, and by $\gamma = \sum_{k\geq 0}kp_k$ the mean of its reproduction law. We also denote by $m(t)$ the mean number of individuals of this Bellman-Harris process at time $t>0$, when it starts from one individual with age~$0$. If there exists $\alpha' >0$ such that~$\mathcal{L}(f)(\alpha') = \frac{1}{\gamma}$, then by \cite[Eq. $(12)$,\,p.$143$]{athreya_1972} and \hbox{\cite[Theorem $3$.A,p.$152$]{athreya_1972}}, there exists $n_1 >0$ such that
\begin{equation}\label{eq:asymptotic_mean_bellman-harris}
m(t) = \sum_{n\geq 0} \gamma^n\left[f*^{(n)}\overline{f}\right](t) \underset{t\rightarrow+\infty}{\sim} n_1 e^{\alpha ' t},
\end{equation}
where the operation $*^{(n)}$ is defined in Notation~\ref{notation:convolution}. 

Let us explain how we get \eqref{eq:renewal_explanation} and \eqref{eq:tail_stay_Ec_explanation} from \eqref{eq:asymptotic_mean_bellman-harris}. We recall that $(M_t)_{t\geq 0}$, defined in~\eqref{eq:first_moment_semigroup}, is the first moment semigroup of the branching process introduced in Section~\ref{subsect:algorithm_model}. From now on, we call "distorted branching process" the branching process that has for first moment semigroup $\frac{1}{\psi(x,a)}M_t(f\psi)(x,a)$, for all $t\geq 0$,~$f \in M_b(\mathcal{X})$,~$(x,a)\in\mathcal{X}$. By~\eqref{eq:weighted_renormalised_semigroup} and~\eqref{eq:equality_semigroup}, the semigroup of the particle $(Z_t)_{t\geq0}$ is the same as the one of the distorted branching process multiplied by~$e^{-\lambda_{\psi}t}$. Thus, below is the intuition of what we propose to do to obtain \eqref{eq:renewal_explanation} and \eqref{eq:tail_stay_Ec_explanation}.
\begin{itemize}[leftmargin=*]
\item We consider a Bellman-Harris process with individual lifetimes distributed according to the density~\hbox{$\mathdutchcal{F}_0*^{(G-1)}\mathdutchcal{F}_0$} and a reproduction law with mean $1+\varepsilon_0$. We also recall that the Laplace transform of the convolution of two functions is the product of their Laplace transforms. Then, in view of~\eqref{eq:asymptotic_mean_bellman-harris} and~\hyperlink{paragraph:long_time_behaviour_S3.2}{$(S_{3.2})$}, we propose to bound from below the average number of individuals of the distorted branching process that renews in $D_{L_1}$ using the mean of this Bellman-Harris process. Indeed, if we transmit the bound obtained for the distorted branching process to $(Z_t)_{t\geq0}$ by multiplying by~$e^{-\lambda_{\psi}t}$, then we will have~\eqref{eq:renewal_explanation}.

\item We consider a Bellman-Harris process with individual lifetimes distributed according to the density~$\mathdutchcal{J}_{0}$, and a reproduction law with mean~$1+\varepsilon_1$. Then, in view of~\eqref{eq:asymptotic_mean_bellman-harris} and \hyperlink{paragraph:long_time_behaviour_S3.3}{$(S_{3.3})$}, we propose to bound from above the average number of individuals of the distorted branching process that stays in~$(D_{L_1})^c$ using the mean of this Bellman-Harris process.  Again, if we transmit the bound obtained for the distorted branching process to $(Z_t)_{t\geq0}$ by multiplying by~$e^{-\lambda_{\psi}t}$, then we will have~\eqref{eq:tail_stay_Ec_explanation}.

\end{itemize}
We are going to put this intuition into practice. For the first case, we also need to handle the fact that the age of the particle stays in a compact set during renewal. Thus, we do not use exactly the Bellman-Harris presented above to obtain our lower bound. We rather use all the Bellman-Harris processes such that lifetimes are distributed according to $\mathdutchcal{F}_0*^{(G-1)}\mathdutchcal{F}_0$, and such that the reproduction law has a mean in~$\left(1,1+\varepsilon_0\right)$, see Section~\ref{subsubsect:proof_prop_renewal_set}. 

\begin{rem}
The property stated in~\eqref{eq:asymptotic_mean_bellman-harris} is true for Bellman-Harris processes starting from an individual with age $0$. However, in our case, we assume that~$\left(Z_t\right)_{t\geq0}$ can start from any age. This is not a main issue because the age of~$\left(Z_t\right)_{t\geq0}$ immediately returns to~$0$ after one jump. Therefore, before doing our comparisons with Bellman-Harris processes, some steps are required to manage the time before the first jump. We do not present them here as they only correspond to obtaining easy lower and upper bounds~(see~\eqref{eq:conclusion_renewal_intermediate},~\eqref{eq:inequality_lyapunov_intermediate_second} and~\eqref{eq:inequality_lyapunov_intermediate_sixth}).
\end{rem}
\begin{rem}
The process that describes the number of individuals in a Bellman-Harris process, that we call in this remark~$\left(B_t\right)_{t\geq0}$, is not Markovian for the same reason as~$\left(X_{N_t}\right)_{t\geq0}$. However, the property presented in~\eqref{eq:asymptotic_mean_bellman-harris} comes from the fact that the past-dependence of $\left(B_t\right)_{t\geq0}$ dissipates at long time. The reason is that the age distribution of individuals stabilises when $t\rightarrow+\infty$, implying that the time until the next birth of an individual depends only on the number of individuals in the population (and thus on the present). Our way to manage the non-Markovian nature of $\left(X_{N_t}\right)_{t\geq0}$ is thus to show that~$\left(X_{N_t}\right)_{t\geq0}$~can be compared to Markov processes at long time.
\end{rem}
\subsubsection{Auxiliary statements and proof that \texorpdfstring{\protect\hyperlink{te:assumptions_velleret_A2}{$(A_2)$}}{(A2)} is verified}\label{subsubsect:useful_statements_A2}

We now present statements that imply \eqref{eq:renewal_explanation} and \eqref{eq:tail_stay_Ec_explanation}, and use them to verify \hyperlink{te:assumptions_velleret_A2}{$(A_2)$}. We first present the following statement, essential for obtaining \eqref{eq:renewal_explanation} and \eqref{eq:tail_stay_Ec_explanation}. The proof of this statement is given in Section \ref{subsubsect:proof_inequality_means}.
\begin{lemm}[Transfer of means]\label{lemm:inequality_for_means}
Let $m\in\mathbb{N}^*$, $(q(x,dw_1,\hdots,dw_m))_{x\in\mathbb{R}_+^{2k}}$ a family of probability measures on $\left(\mathbb{R}_+^{2k}\right)^m$, $(A_1,\,A_2) \in\left(0,+\infty\right]^2$ such that $A_1 \geq A_2$, and two functions $b_1 : \mathbb{R}_+^{2k}\times[0,A_1) \mapsto \mathbb{R}_+$ and $b_2 : \mathbb{R}_+^{2k}\times[0,A_2) \mapsto \mathbb{R}_+$ such that
$$
\begin{aligned}
&\forall (x,a)\in\mathbb{R}_+^{2k}\times[0,A_2):\hspace{23.5mm} b_1(x,a) \leq b_2(x,a), \\
&\forall i\in\{1,2\},\,\forall (x,a)\in\mathbb{R}_+^{2k}\times[0,A_i):\,\hspace{3mm}\underset{t\longrightarrow A_i}{\lim}\int_a^{t} b_i(x,s) ds = +\infty.
\end{aligned}
$$
We denote for all $i\in\{1,2\}$, $(x,a,t)\in\mathbb{R}_+^{2k}\times[0,A_i)\times\mathbb{R}_+$ 
$$
\begin{aligned}
F_i(x,a,t) &= b_i(x,a)\exp\left(-\int_{a}^{a+t} b_i(x,u)du\right)1_{t\in[0,A_i-a)},\\
\overline{F}_i(x,a,t) &= \exp\left(-\int_{a}^{a+t} b_i(x,u)du\right)1_{t\in[0,A_i-a)},
\end{aligned}
$$
and for all $c> 1$ 
\begin{equation}\label{eq:mean_for_inequality_means}
\begin{aligned}
&E_i(x,a,t,c) = \overline{F}_i(x,a,t) + \sum_{\substack{(n,r)\in \mathbb{N}\times\llbracket0,m-1\rrbracket \\ (n,r)\neq(0,0)}} c^{n}\int_{[0,t]}\hdots \int_{\left[0,t - \sum_{i = 1}^{nm+r-1}s_i\right]} \int_{(\mathbb{R}_+^{2k})^m}\hdots \int_{(\mathbb{R}_+^{2k})^m} \\
&\times \left[F_i(x,a,s_1) F_i(w_1,0,s_2)\hdots F_i(w_{nm+r-1},0,s_{nm+r})\right]\overline{F}_i\left(w_{nm+r},0,t- \sum_{j = 1}^{nm+r}s_{j}\right)\\
&\times\left(ds_1\hdots ds_{nm+r}\right)\left(q(x,dw_1,\hdots,dw_m) \hdots q(w_{nm},dw_{nm+1},\hdots,dw_{(n+1)m})\right).\\ 
\end{aligned}
\end{equation}
Then for all $(x,a,t,c)\in\mathbb{R}_+^{2k}\times\left[0,A_2\right)\times\mathbb{R}_+\times]1,+\infty[$, we have
$$
E_1(x,a,t,c) \leq E_2(x,a,t,c).
$$
\end{lemm}
\begin{rem}\label{rem:inequality_means}
When for $i\in\{1,2\}$, $b_i(x,a) = b_i(a)$ i.e. does not depend on $x$, we can integrate all the measures $q$ in \eqref{eq:mean_for_inequality_means} to obtain, using Notation \ref{notation:convolution},
$$
E_i(x,a,t,c) = \overline{F}_i(a,t) + \sum_{\substack{(n,r)\in \mathbb{N}\times\llbracket0,m-1\rrbracket \\ (n,r)\neq(0,0)}} c^{n}\left(F_i(a,.)\right)*\left(F_i(0,.)*^{(nm+r-1)}\overline{F}_i(0,.)\right)(t).
$$
\end{rem}
\noindent In the above, for all $i\in\{1,2\}$ and $(x,a,t,c)\in\mathbb{R}_+^{2k}\times\left[0,A_2\right)\times\mathbb{R}_+\times]1,+\infty[$, $E_i(x,a,t,c)$ represents the mean number of individuals of a birth and death process, where each individual is structured by a trait in $\mathbb{R}_+^{2k}\times[0,A_i)$. This~birth and death process starts from an individual with trait~$(x,a)\in\mathbb{R}_+^{2k}\times[0,A_i)$. The age of each individual grows with a transport term. Each individual dies at a rate $b_i(x',a')$, where $(x',a')\in\mathbb{R}_+\times[0,A_i)$ is its trait, and gives birth to a random number of offspring at death. At each generation that is a multiple of $m$, the distribution of the number of offspring has for mean~$c$. At the other generations, the number of new offspring is almost surely~$1$. The trait of individuals evolves such that at each generation that is a multiple of $m$, lengths of the offspring of the $m$~next generations are updated thanks to the kernel $q(x',dw_1,\hdots, dw_m)$, where $x'$ is the trait in space of the individual that dies at this generation. %In addition, for all $n\in \mathbb{N}$ and $r\in\llbracket0,m-1\rrbracket$, the mean number of individuals at the $(nm + r)$-th generation is $c^n$. 

Qualitatively, this lemma states that when $b_2 \geq b_1$ and $c >1$, the birth and death process presented above with death rate $b_2$ has in average more individuals than the one with death rate $b_1$. In our proof, we use this property to transfer bounds obtained for a birth and death/branching process with death/branching rate $\overline{b}(a)$ or~$\underline{b}(a)$ to a birth and death/branching process with death/branching rate~$b(x,a)$, using Assumption~\hyperlink{paragraph:long_time_behaviour_S3.1}{$(S_{3.1})$}. In particular, this allows us to handle the fact that the division rate may depend on $x$ in our~model.

In view of the above statement, we obtain \eqref{eq:renewal_explanation} as a direct consequence of the following proposition, whose proof is given in Section \ref{subsubsect:proof_prop_renewal_set}.
\begin{prop}[Renewal]\label{prop:renewal_set}
Assume that \hyperlink{paragraph:long_time_behaviour_S1.1}{$(S_{1.1})$}, \hyperlink{paragraph:long_time_behaviour_S1.2}{$(S_{1.2})$}, \hyperlink{paragraph:long_time_behaviour_S2}{$(S_2)$} and \hyperlink{paragraph:long_time_behaviour_S3}{$(S_3)$} hold. Then for all~$\eta >0$,  there exist $l\geq 1$, $L' \geq l+2m_0$ and $t_1 >0$ such that for all $L\geq \max\left(t_1/a_0,\,L'\right)$, $(x,a)\in [0,B_{\max}]^{2k}\times[0,a_0(l+1)]$, we have
\begin{equation}\label{eq:renewal_prop_inequality}
\mathbb{P}_{(x,a)}\bigg[Z_{t_1} \in [0,B_{\max}]^{2k}\times[0,a_0(l+1)];\, \min(\tau_{\partial}, T_{D_{L}}) > t_1\bigg] \geq e^{(\alpha - \lambda_{\psi} - \eta)t_1}.
\end{equation}
\end{prop}
\noindent As explained in the paragraph about Bellman-Harris processes, the proof of this proposition is based on the fact that in view of \hyperlink{paragraph:long_time_behaviour_S2.1}{$(S_{2.1})$} and \hyperlink{paragraph:long_time_behaviour_S3.1}{$(S_{3.1})$}, we can bound from below the left-hand side term of \eqref{eq:renewal_prop_inequality} by the expectations of Bellman-Harris processes (multiplied by~$e^{-\lambda_{\psi}t}$). We use Lemma~\ref{lemm:inequality_for_means} to obtain these lower bounds.

Similarly, \eqref{eq:tail_stay_Ec_explanation} is obtained as a consequence of the following proposition which is proved in Section~\ref{subsubsect:proof_bound_tail_probability}. 
\begin{prop}[Return in a compact set]\label{prop:bound_tail_probability}
Assume that \hyperlink{paragraph:long_time_behaviour_S1.1}{$(S_{1.1})$} and \hyperlink{paragraph:long_time_behaviour_S2}{$(S_2)-(S_{3})$}  hold. Then for all $L_1\geq L_{\text{ret}}$, $\eta >0$, there exists $t_2 >0$, such that for all $t\geq t_2$, $(x,a)\in \mathcal{X}$
\begin{equation}\label{eq:bound_tail_probability}
\begin{aligned}
\mathbb{P}_{(x,a)}\left[\min\left(\tau_{D_{L_1}},\tau_{\partial}\right) > t\right] \leq e^{(\beta - \lambda_{\psi} +\eta)t}. \\ 
\end{aligned}
\end{equation}
\end{prop}
\noindent We prove this proposition by bounding from above $\mathbb{P}_{(x,a)}\left[\min\left(\tau_E,\tau_{\partial}\right) > t\right]$ by the mean of a Bellman-Harris process (translated by $e^{-\lambda_{\psi}t}$), using Lemma~\ref{lemm:inequality_for_means} to obtain the upper bound. 

In view of Propositions \ref{prop:renewal_set} and \ref{prop:bound_tail_probability}, we now use Proposition \ref{prop:simplification_A2} to obtain that there exists $L_1\in\mathbb{N}^*$ such that \hyperlink{te:assumptions_velleret_A2}{$(A_2)$} is verified with $E = D_{L_1}$ and for any $\rho \in \left(\lambda_{\psi} - \alpha,\,\lambda_{\psi} - \beta \right)$ ($\eta >0 $ of Propositions \ref{prop:renewal_set} and \ref{prop:bound_tail_probability} can be taken as small as possible). In particular, we also have by Proposition~\ref{prop:simplification_A2} that 
\begin{equation}\label{eq:inequality_absorbing_rate}
\rho_S \leq \lambda_{\psi} - \alpha.
\end{equation}
In the rest of the proof, we take arbitrarily $\rho = \lambda_{\psi} - \frac{\alpha + \beta}{2}\in \left(\lambda_{\psi} -\alpha,\lambda_{\psi} - \beta\right)$. We now prove all the statements presented above.
\subsubsection{Proof of Lemma \ref{lemm:inequality_for_means}}\label{subsubsect:proof_inequality_means}
Before starting to prove this lemma, we need to introduce random variables. The random variables introduced here are only useful for this proof. Let $(x,a)\in\mathbb{R}_+^{2k}\times[0,A_2)$. We consider $(\mathcal{W}_l)_{l\in\mathbb{N}}$ a stochastic process taking values in $\mathbb{R}_+^{2k}$, such that:
\begin{itemize}[leftmargin=*]
\item The initial term is $\mathcal{W}_0 = x$.
\item For all $n\in\mathbb{N}$, the distribution of $\left(\mathcal{W}_{nm+1},\hdots,\mathcal{W}_{(n+1)m}\right)$, knowing $\mathcal{W}_{nm}$, is given by the measure $q(\mathcal{W}_{nm},dw_1,\hdots,w_m)$.
\end{itemize}
Let $(V_n)_{n\in\mathbb{N}}$ a sequence of independent uniform random variables on $[0,1]$. We also define for $i\in\{1,2\}$ the sequence of random variables $\big(T_n^{(i)}\big)_{n\in\mathbb{N}}$, such that $T^{(i)}_0 = 0$ a.s. and:
\begin{itemize}[leftmargin=*]    
\item For all $\omega\in\Omega$: $T_1^{(i)}(\omega) = \left(1 - \overline{F}_i\right)^{-1}\left(x,a,V_0(\omega)\right)$, where we have for $(x',a',u)\in \mathbb{R}_+^{2k}\times[0,A_i)\times(0,1)$
$$
(1 - \overline{F}_i)^{-1}(x',a',u) = \inf\left\{y\in[0,A_i-a')\,|\,1 - \overline{F}_i(x',a',y) \geq u\right\}.
$$
The function $(1 - \overline{F}_i)(x,a,.)$ can be seen as the cumulative distribution function linked to the probability density function $F_i(x,a,.)$. Thus, in view of the inverse transform sampling, the distribution of $T_1^{(i)}$ is given by the probability density function $F_i(x,a,.)$.

\item For all $n \in \mathbb{N}^*$, $\omega\in\Omega$: $T_{n+1}^{(i)}(\omega) = T_{n}^{(i)}(\omega) + \left(1 - \overline{F}_i\right)^{-1}\left(\mathcal{W}_n(\omega),0,V_n(\omega)\right)$. Thus, in view of the inverse transform sampling, the distribution of $T_{n+1}^{(i)} - T_n^{(i)}$ knowing~$\mathcal{W}_n$ is given by the probability density function $F_i(\mathcal{W}_n,0,.)$.
\end{itemize}
Let $c > 1$. We finally define for all $t\geq 0$, $i\in\{1,2\}$, $\omega \in\Omega$,
$$
S_t^{(i)}(\omega) = \sum_{n\geq 0} \sum_{j = 0}^{m-1} c^n 1_{\{T_{nm + j}^{(i)}(\omega) \leq t < T_{nm + j+ 1}^{(i)}(\omega)\}}.
$$
One can notice that for all $t\geq 0$ we have $E_{i}(x,a,t,c) = \mathbb{E}\big[S_t^{(i)}\big]$, see~\eqref{eq:mean_for_inequality_means}. Thus, if we prove that almost surely it holds $S_t^{(1)} \leq S_t^{(2)}$, then the lemma is proved. As $c > 1$, the sequence~$(c^n)_{n\in\mathbb{N}}$ increases. Then, we only have to prove that $T_l^{(1)} \geq T_l^{(2)}$ a.s. for all~$l\in\mathbb{N}$ to obtain that $S_t^{(1)} \leq S_t^{(2)}$ a.s..

Let us prove that $T_l^{(1)} \geq T_l^{(2)}$ a.s. for all $l\in\mathbb{N}$ by induction. The base case is trivial. We detail the induction step. We suppose that the assertion is true for $l\in\mathbb{N}$. Let $\omega \in \Omega$. As $b_1(x',a') \leq b_2(x',a')$ for all $(x',a')\in\mathbb{R}_+^{2k}\times[0,A_2)$, we have for all $s\in [0,A_2)$
$$
\begin{aligned}
\left(1 - \overline{F}_1\right)\left(\mathcal{W}_l(\omega),0,s\right) &= 1 - \exp\left(-\int_{0}^{s} b_1\left(\mathcal{W}_l(\omega),u\right)du\right) \\
&\leq 1 - \exp\left(-\int_{0}^{s} b_2\left(\mathcal{W}_l(\omega),u\right)du\right) = \left(1 - \overline{F}_2\right)\left(\mathcal{W}_l(\omega),0,s\right).
\end{aligned}
$$
This implies that
$$
\begin{aligned}
\left(1 - \overline{F}_1\right)^{-1}\left(\mathcal{W}_l(\omega),0,V_l(\omega)\right) &\geq \left(1 - \overline{F}_2\right)^{-1}\left(\mathcal{W}_l(\omega),0,V_l(\omega)\right).
\end{aligned}
$$
The induction assumption and the latter imply that when $l\geq 1$
$$
\begin{aligned}
T_{l+1}^{(1)}(\omega) &= T_{l}^{(1)}(\omega) + \left(1 - \overline{F}_1\right)^{-1}\left(\mathcal{W}_l(\omega),0,V_l(\omega)\right) \\
&\geq T_{l}^{(2)}(\omega) + \left(1 - \overline{F}_2\right)^{-1}\left(\mathcal{W}_l(\omega),0,V_l(\omega)\right) = T_{l+1}^{(2)}(\omega).
\end{aligned}
$$
By the same reasoning, when $l = 0$, we also have $T_{l+1}^{(1)}(\omega) \geq T_{l+1}^{(2)}(\omega)$. Hence, we obtain that $T_l^{(1)} \geq T_l^{(2)}$ almost surely, so that the lemma is proved.
\qed 
\subsubsection{Proof of Proposition \ref{prop:renewal_set}}\label{subsubsect:proof_prop_renewal_set}

The objects we use in this proof are mostly introduced in \hyperlink{paragraph:long_time_behaviour_S2}{$(S_{2})$}, \hyperlink{paragraph:long_time_behaviour_S3.1}{$(S_{3.1})$}, and \hyperlink{paragraph:long_time_behaviour_S3.2}{$(S_{3.2})$}. We first notice that, if for all $\tilde{\alpha} \in \left(0,\alpha\right)$ Eq.~\eqref{eq:renewal_prop_inequality} is true with $\tilde{\alpha}$ instead of $\alpha$ in the equation, then the proposition will be true because $\tilde{\alpha}$ can be taken as close as possible to $\alpha$. Thus, we need to prove that for all $\tilde{\alpha} \in \left(0,\alpha\right)$, Eq.~\eqref{eq:renewal_prop_inequality} is true, replacing $\alpha$ by $\tilde{\alpha}$ in the latter. 

Let $\tilde{\alpha} \in \left(0,\alpha\right)$.  The function $\mathcal{L}\left(\mathdutchcal{F}_0\right)(x) = \int_0^{+\infty} e^{-s x} \mathdutchcal{F}_0(s) ds$ strictly decreases on $[0,\alpha]$, with $\mathcal{L}\left(\mathdutchcal{F}_0\right)(0) = 1$ and $\mathcal{L}\left(\mathdutchcal{F}_0\right)(\alpha) = \frac{1}{(1+\varepsilon_0)^{\frac{1}{G}}}$ in view of~\hyperlink{paragraph:long_time_behaviour_S3.2}{$(S_{3.2})$}. In addition, by~\hyperlink{paragraph:long_time_behaviour_S3.1}{$(S_{3.1})$}, we have \hbox{$\underset{x \rightarrow+\infty}{\lim} \exp\left(-\int_0^{x} \underline{b}(s) ds\right)=0$}. Then, there exists $\tilde{\varepsilon}_0\in\left(0,\varepsilon_0\right)$ and $\tilde{l}\in\mathbb{N}^*$ such that
\begin{equation}\label{eq:equalities_from_intermediate_values_theorem}
\mathcal{L}(\mathdutchcal{F}_0)(\tilde{\alpha})= \frac{1}{\left(1+\tilde{\varepsilon}_0\right)^{\frac{1}{G}}}, \hspace{8mm}1 - \exp\left(-\int_0^{a_0(\tilde{l}+1)} \underline{b}(s) ds\right) \geq \left(\frac{1+\tilde{\varepsilon}_0}{1+\varepsilon_0}\right)^{\frac{1}{G}}.
\end{equation}
The proof that \eqref{eq:renewal_prop_inequality} is true with $\tilde{\alpha}$ instead of $\alpha$ is done in three steps. In Step~\hyperlink{paragraph:step_1_renewal}{$1$}, we prove that there exists \hbox{$C >0$} such that for all $t\geq 0$, $(x,a)\in K_{\text{ren}}\times[0,a_0\tilde{l}]$,  $L\geq t/a_0 + \tilde{l}$, we have
\begin{equation}\label{eq:inequality_duhamel_renewal}
\begin{aligned}
&\mathbb{P}_{(x,a)}\left[Z_t \in [0,B_{\max}]^{2k}\times[0,a_0(\tilde{l}+1)];\, \min(\tau_{\partial}, T_{D_{L}}) > t \right] \\
&\geq  Ce^{-\lambda_{\psi}t}\sum_{\substack{(j,r)\in \mathbb{N}\times\llbracket0,G-1\rrbracket \\ (j,r)\neq(0,0)}}(1+\tilde{\varepsilon}_0)^{j} \left(\mathdutchcal{F}_a*\mathdutchcal{F}_0*^{(jG+r-1)}\overline{\mathdutchcal{F}}_0\right)(t).
\end{aligned}
\end{equation}
Then, we prove in Step~\hyperlink{paragraph:step_2_renewal}{$2$} that for all $\eta' >0$, there exists $t_{\eta'} >0$ such that for all $t\geq t_{\eta'}$
\begin{equation}\label{eq:inequality_exponential_by_residues}
F(t) = \sum_{j'\geq 0}\sum_{r' = 0}^{G-1}\left[\left(1+\tilde{\varepsilon}_0\right)\right]^{j'}\left[\left(\mathdutchcal{F}_0\right)*^{(j'G+r')}\overline{\mathdutchcal{F}_0}\right](t)  \geq e^{(\tilde{\alpha} - \eta')t}.
\end{equation}
Finally, using \eqref{eq:inequality_duhamel_renewal} and \eqref{eq:inequality_exponential_by_residues}, we conclude the proof of the proposition in Step~\hyperlink{paragraph:step_3_renewal}{$3$}.
\paragraph{Step 1:}\hypertarget{paragraph:step_1_renewal}{}
To shorten the number of lines of our equations, we denote in this proof for all~$t\geq0$, $(x,a)\in\mathcal{X}$, and $(l,L)\in(\mathbb{N}^*)^2$
$$
I_B = [0,B_{\max}]^{2k} \text{ and } P_B(t,x,a,l,L) = \mathbb{P}_{(x,a)}\bigg[Z_t \in I_B\times[0,a_0(l+1)]\,;\, \min(\tau_{\partial}, T_{D_{L}}) > t\bigg].
$$
Recall from Section \ref{subsect:construction_auxiliary} that $N_t$ is the number of jumps of the particle up to time $t$, and that on the event $\{N_t = n,t < \tau_{\partial}\}$, we have $Z_t = (X_n,\,a1_{\{n = 0\}} + t-\mathdutchcal{T}_n)$, where $a\geq0$ is the initial age of the particle. Recall also that $\overline{\mathdutchcal{H}}_a(x,s) = \frac{\psi(x,a+s)}{\psi(x,a)}\exp\left(-\lambda_{\psi}s - \int_a^{a+s}b(u) du\right)$ is the complementary cumulative distribution function for the time of the first jump when the initial condition is $(x,a)$. Using these notations, we have for all $t\geq 0$, \hbox{$(x,a)\in K_{\text{ren}}\times[0,a_0\tilde{l}]$}, $L\geq t/a_0 + \tilde{l}$,
\begin{equation}\label{eq:proof_renewal_probabilities_to_sum}
\begin{aligned}
P_B(t,x,a,\tilde{l},L) = \overline{\mathdutchcal{H}}_a(x,t) + \sum_{n\geq 1} \mathbb{P}_{(x,a)}&\left[(X_n,t - \mathdutchcal{T}_n) \in I_B\times[0,a_0(\tilde{l}+1)],\,t < \tau_{\partial},\,N_t = n,\right.\\
&\left.\,\forall i\in\llbracket1,n\rrbracket:\, X_i\in \mathcal{D}_{L}\right].\\
\end{aligned}
\end{equation}
Now, we develop the right-hand side term of \eqref{eq:proof_renewal_probabilities_to_sum} in order to exhibit a lower bound. Recall that for all $x\in\mathbb{R}_+^{2k}$, $a\geq0$, and $s\geq0$, we have introduced the function
$$
\mathdutchcal{G}_a(x,s)  = 2\frac{b(x,a+s)}{\psi(x,a)}\exp\left(-\int_a^{a+s} b\left(x,u\right) du - \lambda_{\psi}s\right),
$$
and that for all $w\in\mathbb{R}_+^{2k}$ and $f\in M_b(\mathcal{X})$ we write $\mathdutchcal{K}_{\mathcal{D}_L}(f)(w) = \mathdutchcal{K}(f1_{\mathcal{D}_L})(w)$. In view of Lemmas~\ref{lemm:generalized_duhamel_n=1}-\ref{lemm:generalized_duhamel} and Remark~\ref{rem:equality_kernel}, we have
\begin{equation}\label{eq:inequality_duhamel_renewal_before_intermediate}
\begin{aligned}
&P_B(t,x,a,\tilde{l},L) = \overline{\mathdutchcal{H}}_a(x,t) + \sum_{n\geq 1}\frac{1}{2^n}\int_{[0,t]\times \mathbb{R}^{2k}}\int_{[0,t-s_1]\times \mathbb{R}^{2k}}\hdots \int_{\left[0,t - \sum_{i = 1}^{n-1}s_i\right] \times \mathbb{R}^{2k}}  \\
&\times \mathdutchcal{V}(w_1)\hdots\mathdutchcal{V}(w_n)\mathdutchcal{G}_{a}(x,s_1)\mathdutchcal{G}_0(w_1,s_2)\hdots\mathdutchcal{G}_0\left(w_{n-1},s_{n}\right)\overline{\mathdutchcal{H}}_0\left(w_n,t-\sum_{i = 1}^{n}s_i\right) \\
&\times 1_{\{w_n\in I_B\}}1_{\left\{t - \sum_{i = 1}^n s_i \leq a_0(\tilde{l}+1)\right\}}\left(ds_1\mathdutchcal{K}_{\mathcal{D}_{L}}(x,dw_1)\right)\hdots \left(ds_{n}\mathdutchcal{K}_{\mathcal{D}_{L}}(w_{n-1},dw_n)\right). \\
\end{aligned}
\end{equation}
To continue our development, we need to introduce new functions, and obtain intermediate equalities and inequalities. We denote for all $x\in\mathbb{R}_+^{2k}$, $a,s\geq0$
\begin{equation}\label{eq:density_pdf_true_birth_rate}
\mathdutchcal{Y}_a(x,s) = b(x,a+s)\exp\left(-\int_a^{a+s} b\left(x,u\right) du\right), \hspace{2mm}\overline{\mathdutchcal{Y}_a}(s) = \exp\left(-\int_a^{a+s} b\left(x,u\right) du\right).
\end{equation}
These correspond respectively to the probability density function and the complementary cumulative distribution function for times between events that occur at a rate $b(x,a)$ (for example, the division time of a cell in the branching process introduced in Section \ref{subsect:algorithm_model}). We also introduce the truncated versions on the interval $[0,a_0(\tilde{l}+1))$ of the functions introduced above
$$
\begin{aligned}
\mathdutchcal{Y}_{a}^{\tilde{l}}(x,s) &= \frac{\mathdutchcal{Y}_a(x,s)}{1 - \overline{\mathdutchcal{Y}}_a(x,a_0 (\tilde{l}+1)-a)}1_{s \in[0,a_0 (\tilde{l}+1)-a)}, \\
\overline{\mathdutchcal{Y}_a}^{\tilde{l}}(x,s) &= \frac{\overline{\mathdutchcal{Y}}_a(x,s)-\overline{\mathdutchcal{Y}}_a(x,a_0 (\tilde{l}+1) - a)}{1 - \overline{\mathdutchcal{Y}}_a(x,a_0 (\tilde{l}+1) -a)}1_{s \in[0,a_0 (\tilde{l}+1)-a)}.
\end{aligned}
$$
We now obtain the intermediate equalities and inequalities needed to develop \eqref{eq:inequality_duhamel_renewal_before_intermediate}. First, in the upper term of the equation below, the terms $\exp(-\lambda s_i)$ in $\mathdutchcal{G}_a(w_0,s_1)$ and~$\mathdutchcal{G}_0(w_i,s_{i+1})$ ($i\in\llbracket1,n-1\rrbracket$) can be multiplied with the term $\exp\left(-\lambda_{\psi}\left(t - \sum_{i = 1}^n s_i\right)\right)$ in $\overline{\mathdutchcal{H}}_0\left(w_n,t - \sum_{i = 1}^n s_i\right)$ to give a term $\exp\left(-\lambda_{\psi}t\right)$. In addition, for any $i\in\llbracket1,n-1\rrbracket$, the term $1/\psi(w_i,0) = 1/\mathdutchcal{V}(w_i)$ in $\mathdutchcal{G}_0(w_i,.)$ can be simplified with the term~$\mathdutchcal{V}(w_i)$. Hence, it comes for all $n\in\mathbb{N}^*$, $t\in\mathbb{R}_+$, $a\geq0$, $w = (w_0,\hdots,w_n)\in(\mathbb{R}_+^{2k})^{n+1}$, $f\in M_b (\mathbb{R}_+^{n+1})$
\begin{equation}\label{eq:convol_exp}
\begin{aligned}
&\int_{[0,t]}\hdots \int_{\left[0,t - \sum_{i = 1}^{n-1}s_i\right]} f(t,s_1,\hdots,s_n) \mathdutchcal{G}_{a}(w_0,s_1)\mathdutchcal{G}_0(w_1,s_2)\hdots\mathdutchcal{G}_0\left(w_{n-1},s_{n}\right) \\ 
&\times\mathdutchcal{V}(w_1)\hdots\mathdutchcal{V}(w_{n})\overline{\mathdutchcal{H}}_0\left(w_n,t-\sum_{i = 1}^{n}s_i\right) ds_1\hdots ds_{n} \\  
&= \frac{2^{n}\mathdutchcal{V}(w_{n})\exp\left(-\lambda_{\psi}t\right)}{\psi(w_0,a)}\int_{[0,t]}\hdots \int_{\left[0,t - \sum_{i = 1}^{n-1}s_i\right]} f(t,s_1,\hdots,s_n)\mathdutchcal{Y}_{a}(w_0,s_1)\mathdutchcal{Y}_0(w_1,s_2)\hdots \\ 
&\times\mathdutchcal{Y}_0\left(w_{n-1},s_{n}\right) \overline{\mathdutchcal{Y}}_0\left(w_{n},t-\sum_{i = 1}^{n}s_i\right)\left[1 + \left(t- \sum_{j = 1}^{n}s_{j}\right)^{d_{\psi}}\right]ds_1\hdots ds_{n}.
\end{aligned}
\end{equation}
Moreover, we easily have by restricting the integration on the interval $[a_0\tilde{l},a_0(\tilde{l}+1)]$, and then applying~\eqref{eq:birth_rate_assumption}, that for all $(x,a,s)\in\mathbb{R}_+^{2k}\times[0,a_0\tilde{l}]\times\mathbb{R}_+$
\begin{equation}\label{eq:ineq_truncated_age}
\begin{aligned}
\mathdutchcal{Y}_a(x,s) &\geq \left(1-\exp\left(-\int_a^{a_0(\tilde{l}+1)} b(x,s) ds\right)\right)\mathdutchcal{Y}_{a}^{\tilde{l}}(x,s) \geq  \big[1-e^{-b_0a_0}\big]\mathdutchcal{Y}_{a}^{\tilde{l}}(x,s).
\end{aligned}
\end{equation}
One can also obtain by similar computations, \hyperlink{paragraph:long_time_behaviour_S3.1}{$(S_{3.1})$} and \eqref{eq:equalities_from_intermediate_values_theorem} that
\begin{equation}\label{eq:ineq_truncated_noage}
\mathdutchcal{Y}_{0}(x,s) \geq  \left(\frac{1+\tilde{\varepsilon}_0}{1+\varepsilon_0}\right)^{\frac{1}{G}} \mathdutchcal{Y}_{0}^{\tilde{l}}(x,s) \hspace{3mm}\text{ and }\hspace{3mm} \overline{\mathdutchcal{Y}}_{0}(x,s) \geq \left(\frac{1+\tilde{\varepsilon}_0}{1+\varepsilon_0}\right)^{\frac{1}{G}}\overline{\mathdutchcal{Y}}_{0}^{\tilde{l}}(x,s).
\end{equation}
Finally, we have as $\psi(x,a) = \mathdutchcal{V}(x)(1+a^{d_{\psi}})$
\begin{equation}\label{eq:time_first_jump}
\overline{\mathdutchcal{H}}_a(x,s) = \frac{\psi(x,a+s)}{\psi(x,a)}\exp\left(-\lambda_{\psi}s - \int_a^{a+s}b(x,u) du\right) \geq e^{-\lambda_{\psi} s} \overline{\mathdutchcal{Y}}_{a}(x,s).
\end{equation}
We are now able to develop \eqref{eq:inequality_duhamel_renewal_before_intermediate}. First, we apply Eq.~\eqref{eq:convol_exp} and~\eqref{eq:time_first_jump}, to transform functions $\mathdutchcal{G}_a$, $\mathdutchcal{G}_0$, $\overline{\mathdutchcal{H}}_a$ and $\overline{\mathdutchcal{H}}_0$ into functions $\mathdutchcal{Y}_a$, $\mathdutchcal{Y}_0$, $\overline{\mathdutchcal{Y}}_a$ and $\overline{\mathdutchcal{Y}}_0$. Then, we use \eqref{eq:ineq_truncated_age} and~\eqref{eq:ineq_truncated_noage} to display the truncated versions of $\mathdutchcal{Y}_0$, $\mathdutchcal{Y}_a$, $\overline{\mathdutchcal{Y}}_0$ and $\overline{\mathdutchcal{Y}}_a$. Thereafter, we use the fact that $1+\left(t-\sum_{i = 1}^n s_i\right)^{d_{\psi}} \geq 1$ and the fact that $\mathdutchcal{V}(w_n)\geq \mathdutchcal{V}_{\min}$ by the third statement of \hyperlink{paragraph:long_time_behaviour_S2.2}{$(S_{2.2})$}. Finally, we group indices in the sum according to their remainder for the Euclidean division by~$G$. It comes for all $t\geq 0$, $(x,a)\in K_{\text{ren}}\times[0,a_0\tilde{l}]$, $L\geq t/a_0 + \tilde{l}$
\begin{equation}\label{eq:inequality_duhamel_renewal_intermediate}
\begin{aligned}
P_B(t,x,a,\tilde{l},L) &\geq  e^{-\lambda_{\psi} t}\overline{\mathdutchcal{Y}}_{a}^{\tilde{l}}(x,t) + \big[1-e^{-b_0a_0}\big]\frac{\mathdutchcal{V}_{\min}e^{-\lambda_{\psi}t}}{\psi(x,a)}\hspace{-0.445mm}\sum_{\substack{(j,r)\in \mathbb{N}\times\llbracket0,G-1\rrbracket,\\
	(j,r)\neq(0,0)}} \left[\left(\frac{1+\tilde{\varepsilon}_0}{1+\varepsilon_0}\right)^{\frac{1}{G}}\right]^{jG + r} \\
	&\times \int_{[0,t]\times \mathbb{R}^{2k}}\hdots \int_{\left[0,t - \sum_{i = 1}^{jG+r-1}s_i\right] \times \mathbb{R}^{2k}}\mathdutchcal{Y}_{a}^{\tilde{l}}(x,s_1)\mathdutchcal{Y}_0^{\tilde{l}}(w_1,s_2)\hdots  \\
	&\times\mathdutchcal{Y}_0^{\tilde{l}}\left(w_{jG+r-1},s_{jG+r}\right)\overline{\mathdutchcal{Y}_0}^{\tilde{l}}\left(w_{jG+r},t- \sum_{j = 1}^{jG+r}s_{j}\right)1_{\{w_{jG+r}\in I_B\}}\\ 
	&\times \left(ds_1\mathdutchcal{K}_{\mathcal{D}_{L}}(x,dw_1)\right)\hdots \left(ds_{jG+r}\mathdutchcal{K}_{\mathcal{D}_{L}}(w_{jG+r-1},dw_{jG+r})\right). 
\end{aligned}
\end{equation}
We denote the family of probability measures $\left(q(y,dw_1,\hdots,dw_{G})\right)_{y\in\mathbb{R}_+^{2k}}$ on $(\mathbb{R}_+^{2k})^{G}$, defined for all~\hbox{$y\in\mathbb{R}_+^{2k}$} in the following way
$$
\begin{aligned}
q(y,dw_1,\hdots, dw_{G}) &= \frac{1_{\{w_{G}\in K_{\text{ren}}\}}}{\left(\mathdutchcal{K}_{I_B}\right)^{G}(1_{\{K_{\text{ren}}\}})(y)} \mathdutchcal{K}_{I_B}(y,dw_{1})\\
&\times\mathdutchcal{K}_{I_B}(w_{1},dw_{2})\hdots\mathdutchcal{K}_{I_B}(w_{G-1},dw_{G}).
\end{aligned}
$$
Our aim now is to bound from below the measures $\left(\mathdutchcal{K}_{\mathcal{D}_{L}}(w_{i-1},dw_i)\right)_{i\in\llbracket1,jG+r+1\rrbracket}$ in~\eqref{eq:inequality_duhamel_renewal_intermediate}, using the measure introduced above. To do so, we first need an auxiliary inequality. Let~$i\in\llbracket1,G-1\rrbracket$ and $y\in I_B$. As at each generation, we have at most $2$ offspring, it comes~$\left(\mathdutchcal{K}_{I_B}\right)^{G-i}(1_{\{K_{\text{ren}}\}})(y) \leq 2^{G-i} \leq 2^{G}$. Then, one can easily obtain from this inequality, \hyperlink{paragraph:long_time_behaviour_S2.1}{$(S_{2.1})$}, and the fact that $\frac{1}{1+\varepsilon_0} \leq 1$, that for all \hbox{$f \in M_b\left(\mathbb{R}_+^{i}\right)$} non-negative, 
\begin{equation}\label{eq:inequality_measures_not_multiple_D}
\begin{aligned}
&\int_{w_1\in\mathbb{R}_+^{2k}}\hdots \int_{w_{G}\in\mathbb{R}_+^{2k}} f(w_1,\hdots,w_i) q(y,dw_1,\hdots,dw_{G}) \\ 
&\leq   2^{G}\int_{w_1\in\mathbb{R}_+^{2k}}\hdots \int_{w_i\in\mathbb{R}_+^{2k}} f(w_1,\hdots,w_i) \mathdutchcal{K}_{I_B}(w_1,dw_2)\hdots \mathdutchcal{K}_{I_B}(w_{i-1},dw_i).
\end{aligned}
\end{equation}
We are now able to bound from below the measures $\left(\mathdutchcal{K}_{\mathcal{D}_{L}}(w_{i-1},dw_i)\right)_{i\in\llbracket1,jG+r+1\rrbracket}$ in Eq.~\eqref{eq:inequality_duhamel_renewal_intermediate}. For all $(j,r)\in\mathbb{N}\times\llbracket0,G-1\rrbracket$ such that $(j,r)\neq (0,0)$:
\begin{itemize}[leftmargin=*]
\item When $r\neq 0$, in view of \eqref{eq:inequality_measures_not_multiple_D} and the fact that $I_B \subset \mathcal{D}_L$, we bound from below the measure $\mathdutchcal{K}_{\mathcal{D}_{L}}(w_{jG},dw_{jG+1})\times\hdots\times\mathdutchcal{K}_{\mathcal{D}_{L}}(w_{jG+r-1},dw_{jG+r})$ in~\eqref{eq:inequality_duhamel_renewal_intermediate} by the measure
$$
\frac{1}{2^{G}}\int_{w_{jG+r+2}\in I_B}\hdots \int_{w_{jG+G-1}\in I_B}\int_{w_{(j+1)G}\in K_{\text{ren}}} q\left(w_{jG},dw_{jG+1},\hdots, dw_{j(G+1)}\right).
$$

\item Then, in view of Assumption \hyperlink{paragraph:long_time_behaviour_S2.1}{$(S_{2.1})$}, we bound from below all the measures of the form
$\mathdutchcal{K}_{\mathcal{D}_{L}}(w_{(i-1)G},dw_{(i-1)G+1})\times\hdots\times\mathdutchcal{K}_{\mathcal{D}_{L}}(w_{(i-1)G + G-1},dw_{iG})$ in~\eqref{eq:inequality_duhamel_renewal_intermediate}, where $i\in\llbracket1,j\rrbracket$, by the measure 
$$
(1+\varepsilon_0) q(w_{(i-1)G},dw_{(i-1)G+1},\hdots, dw_{iG}).
$$

\item In view of the fact that $\left(\frac{1+\tilde{\varepsilon}_0}{1+\varepsilon_0}\right)^{\frac{1}{G}} \leq 1$ (as $\tilde{\varepsilon}_0 \leq \varepsilon_0$), we bound from below the term~$\left[\left(\frac{1+\tilde{\varepsilon}_0}{1+\varepsilon_0}\right)^{\frac{1}{G}}\right]^{jG + r}$ in~\eqref{eq:inequality_duhamel_renewal_intermediate} by the term $\left(\frac{1+\tilde{\varepsilon}_0}{1+\varepsilon_0}\right)^{\frac{G-1}{G}}\left(\frac{1+\tilde{\varepsilon}_0}{1+\varepsilon_0}\right)^{j}$.

\item Finally, in view of the fact that $\tilde{\varepsilon}_0 < \varepsilon_0$ and the fact that
$$
\frac{\mathdutchcal{V}_{\min}}{\psi(x,a)} \leq \frac{\mathdutchcal{V}(x)}{\psi(x,a)} \leq 1,
$$
we bound from below the term $e^{-\lambda_{\psi} t}\overline{\mathdutchcal{Y}}_{a}^{\tilde{l}}(x,s)$ in~\eqref{eq:inequality_duhamel_renewal_intermediate} by the term $$\left(\frac{1+\tilde{\varepsilon}_0}{1+\varepsilon_0}\right)^{\frac{G-1}{G}}\big[1-e^{-b_0a_0}\big]\frac{\mathdutchcal{V}_{\min}e^{-\lambda_{\psi}t}}{\psi(x,a)}\overline{\mathdutchcal{Y}}_{a}^{\tilde{l}}(x,s).
$$
\end{itemize}
At the end, we obtain that for all $t\geq 0$, $(x,a)\in K_{\text{ren}}\times[0,a_0\tilde{l}]$,  $L\geq t/a_0 + \tilde{l}$
$$
\begin{aligned}
P_B(t,x,a,\tilde{l},L) &\geq   \left(\frac{1+\tilde{\varepsilon}_0}{1+\varepsilon_0}\right)^{\frac{G-1}{G}}\big[1-e^{-b_0a_0}\big]\frac{\mathdutchcal{V}_{\min}e^{-\lambda_{\psi}t}}{2^{G}\psi(x,a)}\Bigg[\overline{\mathdutchcal{Y}}_{a}^{\tilde{l}}(x,t)+ \hspace{-0.32mm} \sum_{\substack{(j,r)\in \mathbb{N}\times\llbracket0,G-1\rrbracket,\\
(j,r)\neq(0,0)}}\left(1+\tilde{\varepsilon}_0\right)^{j} \\
&\times \int_{[0,t]}\hdots \int_{\left[0,t - \sum_{i = 1}^{jG+r-1}s_i\right]} \int_{(\mathbb{R}_+^{2k})^{G}}\hdots \int_{(\mathbb{R}_+^{2k})^{G}} \mathdutchcal{Y}^{\tilde{l}}_{a}(x,s_1)\mathdutchcal{Y}_0^{\tilde{l}}(w_1,s_2)\hdots\\
&\times\mathdutchcal{Y}_0^{\tilde{l}}\left(w_{jG+r-1.},s_{jG+r}\right)\overline{\mathdutchcal{Y}_0}^{\tilde{l}}\Bigg(w_{jG+r},t- \sum_{j = 1}^{jG+r}s_{j}\Bigg)\left(ds_1\hdots ds_{jG+r}\right)\\
&\times \left(q(x,dw_1,\hdots,dw_{G}) \hdots q(w_{jG},dw_{jG+1},\hdots,dw_{(j+1)G})\right)\Bigg]. 
\end{aligned}
$$
The functions $\mathdutchcal{Y}^{\tilde{l}}_0$ and $\overline{\mathdutchcal{Y}}^{\tilde{l}}_0$ correspond to the probability density function and the complementary cumulative distribution function of times between events that occur at a rate~\hbox{$b_{\tilde{l}} : \mathbb{R}_+^{2k}\times\big[0,a_0(\tilde{l}+1)\big) \mapsto \mathbb{R}_+$}, defined such that for all $(x,a)\in \mathbb{R}_+^{2k}\times\big[0,a_0(\tilde{l}+1)\big)$
$$
b_{\tilde{l}}(x,a) = -\frac{d}{da}\left(\ln\Big(\overline{\mathdutchcal{Y}}_0^{\tilde{l}}(x,a)\Big)\right) = \frac{b(x,a)\exp\left(-\int_0^a b(x,s) ds\right)}{\exp\left(-\int_0^a b(x,s) ds\right) - \exp\left(-\int_0^{a_0(\tilde{l}+1)} b(x,s) ds\right)}.
$$
We have $b_{\tilde{l}} \geq \underline{b}$ on $\mathbb{R}_+^{2k}\times\big[0,a_0(\tilde{l}+1)\big)$ by \hyperlink{paragraph:long_time_behaviour_S3.1}{$(S_{3.1})$}. Hence, we can apply Lemma \ref{lemm:inequality_for_means} with \hbox{$b_1(x,a)=\underline{b}(a)$}, \hbox{$b_2(x,a)~=~b_{\tilde{l}}(x,a)$}, and the probability measures $\left(q(y,w_1,\hdots,w_{G})\right)_{y\in\mathbb{R}_+^{2k}}$ (see Remark \ref{rem:inequality_means} for the simplification when the birth rate does not depend on telomere lengths), to obtain for all $t\geq 0$, \hbox{$(x,a)\in K_{\text{ren}}\times[0,a_0\tilde{l}]$}, $L\geq t/a_0 + \tilde{l}$ the following
$$
\begin{aligned}
P_B(t,x,a,\tilde{l},L) &\geq  \left(\frac{1+\tilde{\varepsilon}_0}{1+\varepsilon_0}\right)^{\frac{G-1}{G}}\big[1-e^{-b_0a_0}\big]\frac{\mathdutchcal{V}_{\min}e^{-\lambda_{\psi}t}}{2^{G}\psi(x,a)}\bigg(\overline{\mathdutchcal{F}_a}(t) + \sum_{\substack{(j,r)\in \mathbb{N}\times\llbracket0,G-1\rrbracket,\\
(j,r)\neq(0,0)}} \left(1+\varepsilon_0\right)^{j} \\
&\times\left(\left(\mathdutchcal{F}_a\right)*\left(\mathdutchcal{F}_0\right)*^{(jG+r-1)}\left(\overline{\mathdutchcal{F}}_0\right)\right)(t)\bigg). \\
\end{aligned}
$$
Now, in the above equation,  we use the fact that $\overline{\mathdutchcal{F}_a}(t) \geq 0$. Then,  we use the fact that in view of~\hyperlink{paragraph:long_time_behaviour_S2.2}{$(S_{2.2})$}, it holds for all $(x,a)\in \left(K_{\text{ren}}\times[0,a_0\tilde{l}]\right) \subset \left(\mathcal{D}_1\times[0,a_0\tilde{l}]\right)$ 
$$
\frac{1}{\psi(x,a)} = \frac{1}{\mathdutchcal{V}(x)(1+a^{d_{\psi}})} \geq \frac{1}{\underset{y\in\mathcal{D}_{1}}{\sup}\left(\mathdutchcal{V}(y)\right)\left(1+(a_0\tilde{l})^{d_{\psi}}\right)}.
$$
It comes that \eqref{eq:inequality_duhamel_renewal} is true.
\paragraph{Step 2:}\hypertarget{paragraph:step_2_renewal}{} We first prove by induction that for all $n\in\mathbb{N}^*$, $t\geq 0$
\begin{equation}\label{eq:relation_between_severalconv}
\sum_{r' = 0}^{n-1} \left[\mathdutchcal{F}_0*^{(r')}\overline{\mathdutchcal{F}}_0\right](t) = \int_t^{+\infty}  \left(\mathdutchcal{F}_0*^{(n-1)}\mathdutchcal{F}_0\right)(s) ds.
\end{equation}
The base case is trivial with the convention that $(f*^{(0)}g)(t) = g(t)$ (see Notation \ref{notation:convolution}), and the fact that $\int_t^{+\infty} \mathdutchcal{F}_0(s) ds = \overline{\mathdutchcal{F}}_0(t)$. For the induction step, we first observe by developing the convolution that
$$
\begin{aligned}
\int_t^{+\infty}  \left(\mathdutchcal{F}_0*^{(n)}\mathdutchcal{F}_0\right)(s) ds &= \int_t^{+\infty}  \left[\int_0^{t} \left(\mathdutchcal{F}_0*^{(n-1)}\mathdutchcal{F}_0\right)(r)\mathdutchcal{F}_0(s-r) dr\right] ds \\
&+ \int_t^{+\infty}  \left[\int_t^{s} \left(\mathdutchcal{F}_0*^{(n-1)}\mathdutchcal{F}_0\right)(r)\mathdutchcal{F}_0(s-r) dr\right] ds.
\end{aligned}
$$
Then, we switch $\int_t^{+\infty}$ and $\int_0^{t}$ in the first term, and $\int_t^{+\infty}$ and $\int_t^{s}$ in the second term, and use the fact that $\int_{t}^{+\infty} \mathdutchcal{F}_0(s-r) ds=  \overline{\mathdutchcal{F}}_0(t-r)$ and $\int_{r}^{+\infty} \mathdutchcal{F}_0(s-r) ds= 1$, to obtain
$$
\begin{aligned}
\int_t^{+\infty}  \left(\mathdutchcal{F}_0*^{(n)}\mathdutchcal{F}_0\right)(s) ds &= \left(\mathdutchcal{F}_0*^{(n)}\overline{\mathdutchcal{F}}_0\right)(t) + \int_t^{+\infty}  \left(\mathdutchcal{F}_0*^{(n-1)}\mathdutchcal{F}_0\right)(r) dr. 
\end{aligned}
$$
Finally, we apply the induction hypothesis on the above equation. It comes that the induction step is done, and thus that~\eqref{eq:relation_between_severalconv} holds.  

Now, we denote for all $t\geq0$ the function $\mathdutchcal{F}_{0,G}(t) = \left(\mathdutchcal{F}_0*^{(G-1)}\mathdutchcal{F}_0\right)(t)$, that can be seen as a probability density function with complementary cumulative function \hbox{$\overline{\mathdutchcal{F}}_{0,G}(t) = \int_t^{+\infty}\left(\mathdutchcal{F}_0*^{(G-1)}\mathdutchcal{F}_0\right)(s) ds$}. Using \eqref{eq:relation_between_severalconv}, the function $F$ defined in \eqref{eq:inequality_exponential_by_residues} becomes
$$
F(t) = \sum_{j'\geq 0} \left(1+\tilde{\varepsilon}_0\right)^{j'} \left(\mathdutchcal{F}_{0,G}*^{(j')}\overline{\mathdutchcal{F}}_{0,G}\right)(t).
$$
In particular, $F$ is the mean of a Bellman-Harris process with lifetimes distributed according to $\mathdutchcal{F}_{0,G}$ and a reproduction law with mean $1+\varepsilon_0$, see \cite[Eq. $(12)$,\,p.$143$]{athreya_1972}. In addition, as the Laplace transform of the convolution of two functions is the product of their Laplace transform, for all~$x\in\mathbb{R}$ such that the Laplace transform of $\mathdutchcal{F}_0$ is defined, we have $\mathcal{L}\left(\mathdutchcal{F}_{0,G}\right)(x) = \left(\mathcal{L}\left(\mathdutchcal{F}_{0}\right)(x)\right)^{G}$. Then, in view of the first equality in~\eqref{eq:equalities_from_intermediate_values_theorem} and~\eqref{eq:asymptotic_mean_bellman-harris}, there exists $n_1 >0$ such that $F(t)  \sim n_1e^{\tilde{\alpha} t}$ when $t\rightarrow+\infty$. We easily deduce Eq.~\eqref{eq:inequality_exponential_by_residues} from the latter.
\paragraph{Step 3:}\hypertarget{paragraph:step_3_renewal}{} 
Let $\eta'>0$ and $t_{\eta'}$ such that \eqref{eq:inequality_exponential_by_residues} holds for $\eta'$ previously fixed. A consequence of the second condition in~\hyperlink{paragraph:long_time_behaviour_S3.2}{$(S_{3.2})$} and the fact that $\tilde{\alpha} < \alpha$ is that there exists $t_{c},\,c > 0$ such that for all $t\geq t_{c}$, $a\geq 0$, we have
\begin{equation}\label{eq:laplace_transform_age_lower_bound}
\int_{0}^{t} e^{-\tilde{\alpha} s}\mathdutchcal{F}_a(s) ds \geq c.
\end{equation}
We use this inequality and~\eqref{eq:inequality_duhamel_renewal} to bound from below $P_B(t,x,a,\tilde{l},L)$ by an exponential term. First, in view of Tonelli's theorem, we put the convolution by $\mathdutchcal{F}_a$ in~\eqref{eq:inequality_duhamel_renewal} outside the sum. Then, we bound from below what remains in the sum by $F$ defined in~\eqref{eq:inequality_exponential_by_residues}. The bound appears when we change the indices in the sum, and take $r' = r-1$, and $j' = j +1$ when $r' = -1$. Finally, we plug the inequality given in~\eqref{eq:inequality_exponential_by_residues} and~\eqref{eq:laplace_transform_age_lower_bound} in the lower bound we have obtained. This gives us that for all $t \geq t_{\eta'} + t_c$, $(x,a)\in K_{\text{ren}}\times[0,a_0\tilde{l}]$,~$L\geq t/a_0 + \tilde{l}$,
\begin{equation}\label{eq:conclusion_renewal_intermediate}
\begin{aligned}
P_B(t,x,a,l,L) &\geq Ce^{-\lambda_{\psi} t}\int_0^t F(t-s) \mathdutchcal{F}_{a}(s) ds \\
&\geq Ce^{-\lambda_{\psi} t}\int_0^{t-t_{\eta'}} F(t-s) \mathdutchcal{F}_{a}(s) ds \geq c.C.e^{(\tilde{\alpha}-\lambda_{\psi} - \eta')t}.
\end{aligned}
\end{equation}
Using now the Markov property, Proposition \ref{prop:to_prove_A1}, and finally \eqref{eq:conclusion_renewal_intermediate}, we can find $t'_1 >0$, $c_1 >0$, $L'\geq (\tilde{l}+1) + 2m_0$, such that for all $t \geq t_{\eta'} + t_c$, $(x,a)\in \left[0,B_{\max}\right]^{2k}\times[0,a_0(\tilde{l}+1)]$ and $L\geq \max\big(t/a_0 + \tilde{l},\,L'\big)$,
\begin{equation}\label{eq:penultimate_equation_lemma_renewal}
\begin{aligned}
P_B(t+t'_1,x,a,l,L) &\geq c_1\int_{(x',a')\in K_{\text{ren}}\times[0,a_0\tilde{l}]} P_B(t,x',a',l,L) dx'da'\\
&\geq c_1.\text{Leb}(K_{\text{ren}}).(a_0\tilde{l}).c.C.e^{(\tilde{\alpha}-\lambda_{\psi} - \eta')t},
\end{aligned}
\end{equation}
where $\text{Leb}$ is the Lebesgue measure on $\mathbb{R}_+^{2k}$.

Let $\eta > \eta'$. The ratio between the right-hand side term of \eqref{eq:penultimate_equation_lemma_renewal} and $e^{(\tilde{\alpha} - \lambda_{\psi} - \eta)(t+t'_1)}$ tends to infinity when $t\rightarrow+\infty$. Therefore, as $\eta'$ can be taken as small as possible, we easily obtain that \eqref{eq:renewal_prop_inequality} is true replacing $\alpha$ by $\tilde{\alpha}$, for any $\tilde{\alpha}\in(0,\alpha)$. As stated at the beginning of the proof, this yields that the proposition is true.
\qed

\subsubsection{Proof of Proposition \ref{prop:bound_tail_probability}}\label{subsubsect:proof_bound_tail_probability}

Let $L_1 \geq L_{\text{ret}}$. By the definition of $\tau_{D_{L_1}}$, for all $(x,a)\in D_{L_1}$, we have	
$$
\mathbb{P}_{(x,a)}\left[\min\left(\tau_{D_{L_1}},\tau_{\partial}\right) > t\right] = 0.
$$
Thus, we just need to prove \eqref{eq:bound_tail_probability} for all $(x,a)\in (D_{L_1})^c$. The proof that~\eqref{eq:bound_tail_probability} holds for all $(x,a)\in (D_{L_1})^c$ is done in two steps. First, in Step \hyperlink{paragraph:step1_bound_tail_probability}{$1$}, we prove that for all $t\geq0$, $(x,a)\notin D_{L_1}$,
\begin{equation}\label{eq:inequality_Lyapunov}
\begin{aligned}
\mathbb{P}_{(x,a)}\left[\min\left(\tau_{D_{L_1}},\tau_{\partial}\right) > t\right] \leq \overline{\psi}\times\left(1+t^{d_{\psi}}\right)e^{-\lambda_{\psi}t}\bigg[\overline{\mathdutchcal{J}}_a(t) &+ \sum_{n \geq 1} (1+\varepsilon_1)^n\\
&\times\left(\mathdutchcal{J}_a*\mathdutchcal{J}_0*^{(n-1)}\mathdutchcal{J}_0\right)(t)\bigg],
\end{aligned}
\end{equation}
where $\overline{\psi}$ is defined by Lemma~\ref{lemma:inequalities_psi}, $d_{\psi}$ in \eqref{eq:defintion_all_distortion_functions}, and $\mathdutchcal{J}_a$ in \hyperlink{paragraph:long_time_behaviour_S3.3}{$(S_{3.3})$}. The conclusion is then given in Step~\hyperlink{paragraph:step2_bound_tail_probability}{$2$}.

\paragraph{Step 1:}\hypertarget{paragraph:step1_bound_tail_probability}{} Let $(x,a)\in {(D_{L_1})^c}$, $t\geq0$. Recall from Section~\ref{subsect:construction_auxiliary} that $\mathdutchcal{T}_1$ is the waiting time for the first jump of $(Z_t)_{t\geq 0}$, $N_t$ is the total number of jumps $(Z_t)_{t\geq0}$ has made up to time~$t >0$, and $(X_n)_{n\in\mathbb{N}}$ is the evolution of telomere lengths of $(Z_t)_{t\geq0}$ jump by jump. In view of the latter, we have
\begin{equation}\label{eq:inequality_lyapunov_intermediate_first}
\begin{aligned}
\mathbb{P}_{(x,a)}\left[\min\left(\tau_{D_{L_1}},\tau_{\partial}\right) > t\right] = \mathbb{P}_{(x,a)}\left[\mathdutchcal{T}_1 > t\right] + \sum_{n\geq 1}\mathbb{P}_{(x,a)}&\big[\forall i\in\llbracket1,n\rrbracket:\,X_i\in\left(\mathcal{D}_{L_1}\right)^c, \\ 
&\,N_t = n,\,\tau_{\partial} > t\big].
\end{aligned}
\end{equation}
First, we bound from above the first term on the right-hand side. Using the equality 
$$
\int_a^{a+t}\frac{\frac{\partial \psi}{\partial a}(x,u)}{\psi(x,u)} du = \ln\left(\frac{\psi(x,a+t)}{\psi(x,a)}\right)
$$
and the third statement of Lemma \ref{lemma:inequalities_psi}, we obtain
\begin{equation}\label{eq:inequality_lyapunov_intermediate_second}
\begin{aligned}
\mathbb{P}_{(x,a)}\left[\mathdutchcal{T}_1 > t\right] &= \overline{\mathdutchcal{H}}_{a}(x,t) = \exp\left(-\int_a^{a+t} b\left(x,u\right) du - \lambda_{\psi}t + \int_a^{a+t}\frac{\frac{\partial \psi}{\partial a}(x,u)}{\psi(x,u)} du\right) \\ 
&\leq \overline{\psi}\times\left(1+t^{d_{\psi}}\right)\exp\left(-\int_a^{a+t} b\left(x,u\right) du - \lambda_{\psi}t\right).
\end{aligned}
\end{equation}
Now, we bound from above the second term in \eqref{eq:inequality_lyapunov_intermediate_first}. Let $n\in\mathbb{N}^*$. Recall~\eqref{eq:measure_k_restriction} for the definition of~$\mathdutchcal{K}_{(\mathcal{D}_L)^c}$. In view of Lemmas~\ref{lemm:generalized_duhamel_n=1}~and~\ref{lemm:generalized_duhamel}, Remark~\ref{rem:equality_kernel}, and Eq.~\eqref{eq:convol_exp},  we have 
\begin{align}
&\mathbb{P}_{(x,a)}\left[\forall i\in\llbracket1,n\rrbracket:\,X_i \in \left(\mathcal{D}_{L_1}\right)^c,\,N_t = n,\,\tau_{\partial} > t\right] = \frac{e^{-\lambda_{\psi}t}}{\psi(x,a)} \int_{[0,t]\times\mathbb{R}_+^{2k}}\hdots \int_{\left[0,t - \sum_{i = 1}^{n-1}s_i\right] \times \mathbb{R}_+^{2k}} \nonumber\\
&\times \mathdutchcal{Y}_{a_1}(x,s_1)\mathdutchcal{Y}_0(w_1,s_2)\hdots\mathdutchcal{Y}_0\left(w_{n-1},s_{n}\right)\overline{\mathdutchcal{Y}_0}\left(w_n,t- \sum_{j = 1}^{n}s_{j}\right)\left[1 + \left(t- \sum_{j = 1}^{n}s_{j}\right)^{d_{\psi}}\right]\mathdutchcal{V}(w_n)  \label{eq:inequality_lyapunov_intermediate_third} \\
&\times\left(ds_1\mathdutchcal{K}_{(\mathcal{D}_{L_1})^c}(x,dw_1)\right)\hdots\left(ds_{n}\mathdutchcal{K}_{(\mathcal{D}_{L_1})^c}(w_{n-1},dw_n)\right). \nonumber
\end{align}
Let us introduce a family of probability measure $(q(w,dw'))_{w\in\mathbb{R}_+^{2k}}$ such that for all $w\in \mathcal{X}$
$$
q(w,dw') = \frac{\mathdutchcal{V}(w')\mathdutchcal{K}_{(\mathcal{D}_{L_1})^c}(w,dw')}{\mathdutchcal{K}_{(\mathcal{D}_{L_1})^c}(\mathdutchcal{V})(w)}. 
$$
Our aim is to bound from above the measures $\big(\mathdutchcal{K}_{(\mathcal{D}_{L_1})^c}(w_{i-1},dw_{i})\big)_{i\in\llbracket1,n\rrbracket}$ in \eqref{eq:inequality_lyapunov_intermediate_third} (with the convention that $w_0 = x$), using the measures introduced above. To do so, we use the first statement of \hyperlink{paragraph:long_time_behaviour_S2.2}{$(S_{2.2})$} to bound from above the measure $\mathdutchcal{V}(w_n)\mathdutchcal{K}_{(\mathcal{D}_{L_1})^c}(w_{n-1},dw_n)$ by~$(1+\varepsilon_1)\mathdutchcal{V}(w_{n-1})q(w_{n-1},dw_n)$. Then, we iterate this procedure $n-1$ times, such that at the last step of the iteration, we bound from above $(1+\varepsilon_1)^{n-1}\mathdutchcal{V}(w_{1})\mathdutchcal{K}_{(\mathcal{D}_{L_1})^c}(x,dw_{1})$ by $(1+\varepsilon_1)^n\mathdutchcal{V}(x)q(x,dw_{1})$. This gives us an upper bound for~\eqref{eq:inequality_lyapunov_intermediate_third}. We plug  in \eqref{eq:inequality_lyapunov_intermediate_first} the upper bound obtained and~\eqref{eq:inequality_lyapunov_intermediate_second}, and bound from above the term $1 + (t- \sum_{j = 1}^{n}s_{j})^{d_{\psi}}$ in~\eqref{eq:inequality_lyapunov_intermediate_first} by $1+t^{d_{\psi}}$. We obtain after these steps
$$
\begin{aligned}
\mathbb{P}_{(x,a)}[\min&\left(\tau_{D_{L_1}},\tau_{\partial}\right) > t] \leq e^{-\lambda_{\psi}t}(1+t^{d_{\psi}})\bigg(\overline{\psi}.\overline{\mathdutchcal{Y}_a}(t) + \frac{\mathdutchcal{V}(x)}{\psi(x,a)}\sum_{n\geq1} (1+\varepsilon_1)^n  \\
&\times\int_{[0,t]\times\mathbb{R}_+^{2k}}\hdots\int_{\left[0,t - \sum_{i = 1}^{n-1} s_i\right] \times \mathbb{R}_+^{2k}} \mathdutchcal{Y}_{a}(x,s_1)\mathdutchcal{Y}_0(w_1,s_2)\hdots\mathdutchcal{Y}_0\left(w_{n-1},s_{n}\right)\\
&\times\overline{\mathdutchcal{Y}_0}\bigg(w_n,t- \sum_{j = 1}^{n}s_{j}\bigg) \left(ds_1q(x,dw_1)\right)\hdots\left(ds_{n}q(w_{n-1},dw_n)\right)\bigg).
\end{aligned}
$$
As $\overline{\psi} > 1$, we bound from above the term 
$$
\frac{\mathdutchcal{V}(x)}{\psi(x,a)} = \frac{\mathdutchcal{V}(x)}{(1+a^{d_{\psi}})\mathdutchcal{V}(x)}
$$
by $\overline{\psi}$. Then, recalling that \hbox{$\mathdutchcal{J}_a(s)=\overline{b}(a+s)\exp\left(-\int_a^{a+s} \overline{b}\left(u\right) du\right)$}, we obtain \eqref{eq:inequality_Lyapunov} by applying Lemma~\ref{lemm:inequality_for_means} with $b_1 = b$, $b_2 = \overline{b}$ and the probability measures $(q(w,dw'))_{w\in\mathcal{X}}$ (we have $b_1 \leq b_2$ by~\hyperlink{paragraph:long_time_behaviour_S3.1}{$(S_{3.1})$}, and we refer to Remark~\ref{rem:inequality_means} for the simplification of the notations when one of the birth rate does not depend on telomere lengths).

\paragraph{Step 2:}\hypertarget{paragraph:step2_bound_tail_probability}{}

\smallskip

Let us denote for all $t\geq 0$
$$
m(t) = \sum_{n\geq 0}(1+\varepsilon_1)^n\left[\mathdutchcal{J}_0*^{(n)}\overline{\mathdutchcal{J}}_0\right](t).
$$
This function is the mean of a Bellman-Harris process with lifetimes distributed according to $\mathdutchcal{J}_0$, and a reproduction law with mean $1+\varepsilon_1$, see \cite[Eq. $(12)$,\,p.$143$]{athreya_1972}. To have a reproduction law with mean $1+\varepsilon_1$, we take for reproduction law $(p_l)_{l\in\mathbb{N}^*}$ such that $p_1 + rp_r = 1+\varepsilon_1$ and $p_1 + p_r = 1$, where $r\in\mathbb{N}\backslash\{0,1\}$ is taken sufficiently large. By \eqref{eq:asymptotic_mean_bellman-harris} and \hyperlink{paragraph:long_time_behaviour_S3.3}{$(S_{3.3})$}, there exists $n_1 >0$ such that $m(t) \underset{t\rightarrow+\infty}{\sim} n_1e^{\beta t}$. Then, for all $\eta' > 0$, there exists $t'_2>0$ such that for all $t\geq t'_2$
\begin{equation}\label{eq:Bellman_harris_tgeqt_1}
m(t) \leq e^{(\beta + \eta')t}.
\end{equation}
In addition, as $p_0 = 0$, the number of individuals of such a Bellman-Harris increases with time. Then, $m$ is an increasing function. 

Let us fix $\eta' >0$. Using the bound in~\eqref{eq:Bellman_harris_tgeqt_1}, the fact that $m$ is increasing and that $\int_{0}^{\infty} \mathdutchcal{J}_a(s')ds' = 1$, Eq.~\eqref{eq:inequality_Lyapunov} becomes for all $t\geq t'_2$
\begin{equation}\label{eq:inequality_lyapunov_intermediate_sixth}
\begin{aligned}
\mathbb{P}_{(x,a)}\left[\min\left(\tau_{D_{L_1}},\tau_{\partial}\right) > t\right] &\leq \overline{\psi}\times\left(1+t^{d_{\psi}}\right)e^{-\lambda_{\psi}t}\left(\overline{\mathdutchcal{J}}_a(t) + (1+\varepsilon_1)(\mathdutchcal{J}_a*m)(t)\right)\\
&\leq \overline{\psi}\times\left(1+t^{d_{\psi}}\right)e^{-\lambda_{\psi}t}\left(1 + (1+\varepsilon_1)e^{(\beta + \eta')t}\right).
\end{aligned}
\end{equation}
For any $\eta > \eta'$, the ratio between the right-hand-side term of~\eqref{eq:inequality_lyapunov_intermediate_sixth} and $e^{(\beta -\lambda_{\psi}+\eta)t}$ goes to $0$ when~$t$ goes to infinity. The latter and the fact that $\eta'$ can be taken as small as possible imply that the proposition is proved.
\qed

\subsection{Assumption \texorpdfstring{\protect\hyperlink{te:assumptions_velleret_A3}{$(A_3)_F$}}{(A3)} : Asymptotic comparison of survival with a weak form of a Harnack inequality}\label{subsubsect:assumption_(A3)F}

To verify \hyperlink{te:assumptions_velleret_A3}{$(A_3)_F$}, we need to construct a stopping time $U_H$ such that \eqref{eq:first_statement_(A3F)}, \eqref{eq:second_statement_(A3F)} and~\eqref{eq:third_statement_(A3F)} are true. For the stopping time that we present in this section, the fact that~\eqref{eq:second_statement_(A3F)} is true is trivial by its definition. Verifying~\eqref{eq:first_statement_(A3F)}, for its part, involves cumbersome computations. That is why, we only present the main arguments to verify it in Section~\ref{subsubsect:satisfaction_statements_A3F}, and provide the detailed proof for interested readers in Section~\ref{appendix:proof_A3_stepB}. The part that we detail in this section is the proof of \eqref{eq:third_statement_(A3F)}. Again, we need to handle the fact that we have an age-dependent process, with a birth rate that depends both on the age and on telomere lengths of the particle. Hence, a stochastic comparison of the distorted branching process by a Bellman-Harris process is done to verify \eqref{eq:third_statement_(A3F)}. We believe that the proof presented here allows to handle more general cases than the one presented in~\cite{velleret_exponential_2023}, in particular models where the age is involved.

Let us proceed as follows. First, we explain the assumption, and construct the stopping time~$U_H$ in Section~\ref{subsubsect:construction_stopping_time}. Then, we give the auxiliary statements that allow us to get~\eqref{eq:third_statement_(A3F)} in Section~\ref{subsubsect:useful_statements_A3F}, and prove these statements in Sections~\ref{subsubsect:proof(A3)F_stepA}, \ref{subsubsect:proof(A3)F_control_time} and~\ref{subsubsect:proof(A3)F_control_jump}. Finally, we obtain \eqref{eq:first_statement_(A3F)} and~\eqref{eq:third_statement_(A3F)} in Section~\ref{subsubsect:satisfaction_statements_A3F}. The latter allows us to conclude that \hyperlink{te:assumptions_velleret_A3}{$(A_3)_F$} is verified.

\subsubsection{Context and construction of the stopping time \texorpdfstring{$U_H$}{UH}}\label{subsubsect:construction_stopping_time}
\paragraph{Context.} It is well-known since \cite{champagnat_2016} that a Doeblin condition, combined with the fact that the probability of non-extinction starting from $\nu$ is not too small compared to the probability of non-extinction starting from elsewhere, implies the existence of a stationary profile. In the setting of \cite{velleret_unique_2022}, this last criteria is called "asymptotic comparison of survival" and is stated as  
\begin{equation}\label{eq:inequality_control_mass_(A3)}
\limsup_{t\rightarrow+\infty} \sup_{z\in E} \frac{\mathbb{P}_{z}\left[t < \tau_{\partial}\right]}{\mathbb{P}_{\nu}\left[t < \tau_{\partial}\right]} < +\infty. 
\end{equation}
One of the methods to compare the probability of non-extinction starting from $\nu$, and starting from elsewhere, is to use a Harnack inequality. The latter would correspond to the fact that there exist $t_0,\,t_1 >0$ and $C >0$ such that for all $z\in E$
\begin{equation}\label{eq:Harnack_inequality}
\mathbb{P}_z\left[Z_{t_0} \in dz',\,t_0< \tau_{\partial}\right] \leq C.\mathbb{P}_{\nu}\left[Z_{t_1} \in dz',\,t_1< \tau_{\partial}\right].
\end{equation}
By Proposition \ref{prop:to_prove_A1}, we know that the right-hand side term of \eqref{eq:Harnack_inequality} is bounded from below, up to a constant, by the Lebesgue measure restricted on intervals of the form~$[\eta,A]^{2k}$, where $\eta >0$, $A > 0$. Thus, if we prove that the left-hand side term of~\eqref{eq:Harnack_inequality} is bounded from above, up to a constant, by the restriction of the Lebesgue measure on one of these intervals, then we will have \eqref{eq:Harnack_inequality}. However, the particle $(Z_t)_{t\geq 0}$ has a jump kernel that is discontinuous with respect to the Lebesgue measure: there are Dirac measures in the coordinates where the particle does not jump. In addition, for all $t > 0$, $\eta > 0$, $A > 0$, the probability that the particle has telomere lengths outside of $[\eta,A]^{2k}$ at time $t$ is strictly larger than~$0$. Hence, the above tactic does not work. 

Assumption \hyperlink{te:assumptions_velleret_A3}{$(A_3)_F$} allows us to handle these problems by considering events where a Harnack inequality holds, and events where it fails. The latter are called "rare events". When the probability to have a rare event is "sufficiently" small, then we have~\eqref{eq:inequality_control_mass_(A3)}, see~\cite[Theorem~$2.3$]{velleret_exponential_2023}. This can be seen as a "weak form of a Harnack inequality".  

\paragraph{Construction of the stopping time.} The first thing we do is to construct the stopping time $U_H$ involved in this assumption. The above rare events are the following:
\begin{itemize}[leftmargin=*]
\item The particle has not jumped in all the coordinates, implying discontinuities with respect to the Lebesgue measure for the law of $\left(Z_t\right)_{t\geq0}$,
\item The number of jumps is too large, which leads the possibility to have a too large telomere,
\item The particle has a telomere with a length too close to $0$.
\end{itemize}
To write them rigorously, we need to introduce some notions. We recall from Section~\ref{subsect:construction_auxiliary} that for all $l\geq 1$, the random variables $I_l$ and $J_l$ correspond respectively to the sets of "shortened coordinates", and of "lengthened coordinates", at the $l$-th jump. We also recall that $\mathcal{N}_{\partial}$ is the random variable that describes the number of jumps of the particle before extinction. For all $j\in\llbracket1,2k\rrbracket$, we introduce $\mathcal{N}_{j}$ the number of jumps before the first jump in the $j-$th coordinate, defined as
$$
\begin{aligned}
\mathcal{N}_{j} := \inf\left\{l\in\llbracket1,\mathcal{N}_{\partial}\rrbracket,\, j\in I_l\cup J_l\right\}.
\end{aligned}
$$
We also introduce $T_{all}$ the time before the particle $(Z_t)_{t\geq 0}$ has jumped in all coordinates, defined as
$$
T_{all} := \underset{j\in\llbracket1,2k\rrbracket}{\max}\mathdutchcal{T}_{\mathcal{N}_{j}}.
$$
Let $\varepsilon\in(0,1)$. In view of the above list of "rare events", our aim is to find $t_F>0$, $n_J\in\mathbb{N}^*$ and $\eta_0 >0$ such that the statements of~\hyperlink{te:assumptions_velleret_A3}{$(A_3)_F$} are satisfied for $\varepsilon$ and $U_H$ of the form
\begin{equation}\label{eq:stopping_time_UH}
\begin{aligned}
U_H := \begin{cases}
t_F, & \text{ if } t_F < \tau_{\partial}, \text{ }T_{all}  \leq t_F,\text{ } N_{t_F} \leq n_J, \text{ and }X_{N_{t_F}}\in [\eta_0,B_{\max}L_1+n_J\Delta]^{2k},\\
+\infty, & \text{ otherwise.}
\end{cases}
\end{aligned}
\end{equation}
Eq.~\eqref{eq:second_statement_(A3F)} is trivially verified. Thus, we focus on proving that there exists $t_F>0$, $n_J\in\mathbb{N}^*$ and $\eta_0 > 0$ such that~\eqref{eq:first_statement_(A3F)}~and~\eqref{eq:third_statement_(A3F)} are verified. In particular, we now present auxiliary statements that allow us to obtain~\eqref{eq:third_statement_(A3F)}.

\subsubsection{Statements useful to obtain \texorpdfstring{\eqref{eq:third_statement_(A3F)}}{37}}\label{subsubsect:useful_statements_A3F} 

To obtain \eqref{eq:third_statement_(A3F)}, we need to control the probability that $U_H = \infty$ on $\{t_F < \tau_{\partial}\}$. From the definition of $U_H$, the three cases that correspond to this situation are cases where $T_{all}$~is too large, cases where $N_{t_F}$ is too large, and cases where $X_{N_{t_F}} \in \left([\eta_0,+\infty]^{2k}\right)^c$. The third case is directly handled by the statement allowing to obtain~\eqref{eq:first_statement_(A3F)}, see~Section~\ref{subsubsect:satisfaction_statements_A3F}. For the two other cases, we introduce statements allowing to control the probability they occur, and briefly explain how we obtain these statements. Then, in Sections~\ref{subsubsect:proof(A3)F_stepA},~\ref{subsubsect:proof(A3)F_control_time} and~\ref{subsubsect:proof(A3)F_control_jump}, we give all the proofs.

\paragraph{Statements to control the probability that $T_{all}$ is too large.} Inspired by what we have done in Section~\ref{subsubsect:assumption_(A2)}, we propose to use Bellman-Harris processes to control this probability. To be more precise, we need to stochastically compare the distorted branching process presented in Section~\ref{subsubsect:explanation_proof_A2} by the Bellman-Harris process used to prove Proposition~\ref{prop:bound_tail_probability}. This will give us that the probability that~\hbox{$\min(T_{all},\tau_{\partial}) > t$} tends to~$0$ faster than~$e^{(\beta-\lambda_{\psi}+\eta)t}$, for all $\eta > 0$.

To do so, we first obtain an upper bound for the probability that a coordinate is not shortened at one event of division. As the choice of which telomeres are shortened is uniform, we only have to control for all $i\in\llbracket1,2k\rrbracket$ the cardinal of $\{I\in \mathcal{I}_k\,|\,i\notin I\}$ (we recall that $\mathcal{I}_k$ is defined in \eqref{eq:set_combination_shortening}, and is the set that contains all the possible combinations of indices of telomeres that can be shortened). The following lemma, which is proved in Section~\ref{subsubsect:proof(A3)F_stepA}, deals with the latter.
\begin{lemm}[Number of combinations for the shortening]\label{lemma:cardinal_set_part}
Let us consider $i\in\llbracket1,2k\rrbracket$. Then, the functions $f_1 : \{0,1\}^{k} \longrightarrow  \mathcal{I}_{k}$ and \hbox{$f_2: \{0,1\}^{k-1} \longrightarrow  \{I\in\mathcal{I}_k,\,i\notin I\}$} defined such that
$$
\begin{aligned}
\forall x\in\{0,1\}^k:\hspace{6.475mm} f_1(x) &= \{1+kx_1,\hdots,\,k + kx_k\},\\ 
\forall x\in\{0,1\}^{k-1}:\hspace{2.5mm} f_2(x) &=\begin{cases} 
\{j+kx_j\,|\,j\in\llbracket1,k\rrbracket,j\neq i \text{ mod }k\}\cup\{i+k\}, & \text{ if }i\leq k,\\
\{j+kx_j\,|\,j\in\llbracket1,k\rrbracket,j\neq i \text{ mod }k\}\cup\{i-k\}, & \text{ otherwise},\\
\end{cases}
\end{aligned}
$$
are bijective. In particular, it holds
$$
\#\{I\in\mathcal{I}_k,\,i\notin I\} = 2^{k-1}, \hspace{4.25mm}\#(\mathcal{I}_{k}) = 2^k, \hspace{4.25mm}\text{and} \hspace{4.25mm}\frac{\#\{I\in\mathcal{I}_k,\,i\notin I\}}{\#(\mathcal{I}_{k})} = \frac{1}{2}.
$$
\end{lemm}

\noindent From here, we use a Bellman-Harris process to bound from above the probability that~$\min(T_{all},\tau_{\partial}) > t$. This gives us the following lemma, which is proved in Section~\ref{subsubsect:proof(A3)F_control_time}.
\begin{lemm}[Control of the discontinuities]\label{lemma:inequality_time_mixing}
Assume that \hyperlink{paragraph:long_time_behaviour_S1.1}{$(S_{1.1})$}, \hyperlink{paragraph:long_time_behaviour_S2}{$(S_2)$} and \hyperlink{paragraph:long_time_behaviour_S3}{$(S_3)$} hold. Then, for all $\varepsilon > 0$, we can find $t_0 >0$ such that for all $t\geq t_0$
\begin{equation}\label{eq:inequality_time_mixing}
\sup_{(x,a)\in E} \left(\mathbb{P}_{(x,a)}\left[\min(T_{all},\tau_{\partial}) >  t\right]\right) \leq \frac{\varepsilon}{3}\exp\left(-\rho t\right),
\end{equation}
where $\rho = \lambda_{\psi} - \frac{\alpha + \beta}{2}$ is defined in Section~\ref{subsubsect:useful_statements_A2} under~\eqref{eq:inequality_absorbing_rate}.
\end{lemm}
\paragraph{Statement to control the probability that $N_{t_F}$ is too large.} To control the probability that $N_{t_F}$ is too large, we only use the following statement, that we prove in Section~\ref{subsubsect:proof(A3)F_control_jump}. Briefly, we obtain this by using the fact that $\left(\mathbb{P}[\mathcal{N}_{\partial} > n]\right)_{n\geq0}$ decreases exponentially fast.

\begin{lemm}[Control of the number of jumps]\label{lemma:inequality_generation_mixing} Assume that \hyperlink{paragraph:long_time_behaviour_S1.1}{$(S_{1.1})$}, \hyperlink{paragraph:long_time_behaviour_S2}{$(S_2)$} and \hyperlink{paragraph:long_time_behaviour_S3}{$(S_3)$} hold. Then, for all $\varepsilon >0$, there exists an increasing function $n :\mathbb{R}_+ \longrightarrow \mathbb{N}^*$ such that for all~$t\geq0$
\begin{equation}\label{eq:inequality_generation_mixing}
\sup_{(x,a)\in E} \left(\mathbb{P}_{(x,a)}\left[N_{t} > n(t),t < \tau_{\partial}\right]\right) \leq \frac{\varepsilon}{3}\exp\left(-\rho t\right),
\end{equation}
where $\rho = \lambda_{\psi} - \frac{\alpha + \beta}{2}$ is defined in Section~\ref{subsubsect:useful_statements_A2} under~\eqref{eq:inequality_absorbing_rate}.
\end{lemm}

\noindent  Let us prove now all the statements given above, and then obtain \eqref{eq:first_statement_(A3F)} and \eqref{eq:third_statement_(A3F)}.

\subsubsection{Proof of Lemma~\ref{lemma:cardinal_set_part}}\label{subsubsect:proof(A3)F_stepA}

The injectivity of each of these functions is trivial. We  prove the surjectivity for $f_1$ (this is very similar for $f_2$). Let $I\in\mathcal{I}_k$. As $I$ is a subset of $\left\{1,2,\hdots,2k\right\}$, we have for all $i\in I$ that $i-1\in \left\{0,1,\hdots,2k-1\right\}$. Then, for all $i\in I$ there exists $(q_i,r_i)\in\{0,1\}\times\llbracket0,k-1\rrbracket$ such that $i-1 = q_ik +r_i$. 

\noindent By the definition of $\mathcal{I}_{k}$, for all $(i,j)\in I^2$ such that $i\neq j$ we have that $r_i \neq r_j$. In addition, we know that~$\#(I) = k$. Combining these two statements yields that the set $\left\{r_i,i\in I\right\}$ is a set with $k$ different elements. Then, necessarily
$$
\left\{r_i,i\in I\right\} = \left\{0,\hdots,k-1\right\}.
$$
We can now define $\sigma : I \longrightarrow \llbracket1,k\rrbracket$ as $\sigma(i) = r_i +1$ for all $i\in I$. We obtain a new expression for~$I$ using~$\sigma$
$$
I = \left\{q_ik+r_i + 1\,|\,i\in I\right\} = \left\{q_{\sigma^{-1}(j)}k+j\,|\,j\in\llbracket1,k\rrbracket\right\}.
$$
Thus, if we take $x = \left(q_{\sigma^{-1}(1)},q_{\sigma^{-1}(2)},\hdots,q_{\sigma^{-1}(k)}\right)$, then we have $f_1(x) = I$. This implies that $f_1$ is surjective.
\qed

\subsubsection{Proof of Lemma~\ref{lemma:inequality_time_mixing}}\label{subsubsect:proof(A3)F_control_time}

Let $(x,a)\in E$ be the initial condition of the process $(Z_t)_{t\geq0}$. Recall that $\mathcal{N}_{\partial}$ is the random variable that describes the number of jumps before extinction, and for all $j\in\llbracket1,2k\rrbracket$, $\mathcal{N}_j$~is the random variable used to describe the number of jumps before having a jump in the $j-$th coordinate. For every $j\in\llbracket1,2k\rrbracket$, we also define	
$$
\begin{aligned}
\mathbb{N}_{j} := \inf\left\{l\in\llbracket1,\mathcal{N}_{\partial}\rrbracket,\, j\in I_l\right\}
\end{aligned}
$$
the number of jumps before having a shortening in the $j-$th coordinate. We easily see that $\mathdutchcal{T}_{\mathcal{N}_j}$ is a random variable that describes the time of the first jump in the $j-$th coordinate, and $\mathdutchcal{T}_{\mathbb{N}_j}$ is a random variable that describes the time of the first shortening in the $j-$th coordinate. As $N_t = l$ is equivalent to $\mathdutchcal{T}_{l+1}>t\geq \mathdutchcal{T}_l$, and as $T_{all} = \underset{j\in\llbracket1,2k\rrbracket}{\max} \mathdutchcal{T}_{\mathcal{N}_j}$, we have for all $t\geq0$
$$
\begin{aligned}
\mathbb{P}_{(x,a)}\left[\min(T_{all},\tau_{\partial}) > t\right] &= \sum_{l\geq 0}\mathbb{P}_{(x,a)}\left[\min(T_{all},\tau_{\partial}) > t, N_t = l\right]\\
&\leq \sum_{l\geq 0}\sum_{j\in\llbracket1,2k\rrbracket} \mathbb{P}_{(x,a)}\left[\min(\mathdutchcal{T}_{\mathcal{N}_{j}},\tau_{\partial},\mathdutchcal{T}_{l+1}) > t \geq \mathdutchcal{T}_l\right].
\end{aligned}
$$
Recall that in \eqref{eq:inequality_lyapunov_intermediate_second}, we bounded from above the tail of the time of the first event. Then, rewriting the previous equation with the latter implies
\begin{equation}\label{eq:inequality_time_mixing_intermediate_zero}
\begin{aligned}
\mathbb{P}_{(x,a)}[\min(T_{all},\tau_{\partial}) > t] 
&\leq 2k\times\overline{\psi}\times\left(1+t^{d_{\psi}}\right)\exp\left[-\int_a^{a+t} b\left(x,u\right) du - \lambda_{\psi}t\right] \\
&+ \sum_{l\geq 1}\sum_{j\in\llbracket1,2k\rrbracket} L(j,l),
\end{aligned}
\end{equation}
where $L(j,l) := \mathbb{P}_{(x,a)}\left[\min(\mathdutchcal{T}_{\mathcal{N}_{j}},\tau_{\partial},\mathdutchcal{T}_{l+1}) > t \geq \mathdutchcal{T}_l\right]$ for all $j\in\llbracket1,2k\rrbracket$ and $l\in\mathbb{N}^*$. We now bound from above $L(j,l)$, for $j\in\llbracket1,2k\rrbracket$ and $l\in\mathbb{N}^*$. Notice first that for all $i\in\llbracket1,2k\rrbracket$ it holds $\mathdutchcal{T}_{\mathbb{N}_{i}} \geq \mathdutchcal{T}_{\mathcal{N}_{i}}$~a.s.. In addition, on the event $\left\{\min(\mathdutchcal{T}_{\mathbb{N}_{j}},\tau_{\partial},\mathdutchcal{T}_{l+1}) > t \geq \mathdutchcal{T}_l\right\}$, we know that there have been exactly $l$ jumps without visiting~$\partial$ and that the coordinate $j$ has not been shortened yet. This yields the following
\begin{equation}\label{eq:inequality_time_mixing_intermediate_first}
\begin{aligned}
L(j,l) &\leq  \mathbb{P}_{(x,a)}\left[\mathdutchcal{T}_{l+1} > t \geq \mathdutchcal{T}_l,\,\forall i\in\llbracket1,l\rrbracket:\,(I_i,J_i)\neq \left(\partial,\partial\right)\text{ and }j\notin I_i\right].\\
\end{aligned}
\end{equation}
We consider for all $C\in\mathcal{B}\left(\mathbb{R}_+^{2k}\right)$
\begin{equation}\label{eq:kernel_updated_permutation}
\mathdutchcal{K}_{j}(C)(x) = 2\sum_{\substack{(I,J)\in\mathcal{Q}_{k},\\ j\notin I}} \int_{u\in\mathbb{R}^{2k}}1_{\{x+u\in C\}} d\pi^{I,J}_x(u),
\end{equation}
where $\pi_x^{I,J}$ is the measure given in \eqref{eq:measure_by_event}. We use Lemmas~\ref{lemm:generalized_duhamel_n=1}~and~\ref{lemm:generalized_duhamel} in \eqref{eq:inequality_time_mixing_intermediate_first}, then~\eqref{eq:convol_exp} and~\eqref{eq:kernel_updated_permutation}, and finally the fact that $\Big[1 + \big(t- \sum_{j = 1}^{n}y_{j}\big)^{d_{\psi}}\Big] \leq 1+t^{d_{\psi}}$. We obtain (recalling the notation~\eqref{eq:density_pdf_true_birth_rate})
\begin{align}
L(j,l) &\leq \frac{\exp\left(-\lambda_{\psi}t\right)}{\psi(x,a)}\left(1 + t^{d_{\psi}}\right)\int_{[0,t]\times\mathbb{R}_+^{2k}}\hdots \int_{\left[0,t - \sum_{i = 1}^{l-1}s_i\right] \times \mathbb{R}_+^{2k}} \mathdutchcal{Y}_{a}(x,s_1)\mathdutchcal{Y}_0(w_1,s_2)\hdots \nonumber\\ 
&\times\mathdutchcal{Y}_0\left(w_{l-1},s_{l}\right)\overline{\mathdutchcal{Y}_0}\left(w_l,t- \sum_{j = 1}^{l}s_{j}\right)\mathdutchcal{V}(w_l)\left(ds_1\mathdutchcal{K}_{j}(x,dw_1)\right)\hdots \left(ds_{l}\mathdutchcal{K}_{j}(w_{l-1},dw_l)\right). \label{eq:inequality_time_mixing_intermediate_second}
\end{align}
Now, first use the definition of $\pi_x^{I,J}$ given in \eqref{eq:measure_by_event} and the second statement of~\hyperlink{paragraph:long_time_behaviour_S2.2}{$(S_{2.2})$} to bound from above $\mathdutchcal{K}_{j}(\mathdutchcal{V})(x)$. Then, use Lemma \ref{lemma:cardinal_set_part} to write the bound obtained in a more convenient way. It comes this second inequality
\begin{equation}\label{eq:inequality_kernel_power_intermediate}
\mathdutchcal{K}_{j}(\mathdutchcal{V})(x) \leq  \frac{2}{\#\left(\mathcal{I}_k\right)}\sum_{I\in\mathcal{I}_k,\,j\notin I} (1+\varepsilon_1)\mathdutchcal{V}(x) = (1+\varepsilon_1)\mathdutchcal{V}(x).
\end{equation}

We now bound from above $\mathbb{P}_{(x,a)}\left[\min(T_{all},\tau_{\partial}) > t\right]$. Using~\eqref{eq:inequality_time_mixing_intermediate_second} and \eqref{eq:inequality_kernel_power_intermediate}, we proceed exactly as we did to obtain \eqref{eq:inequality_Lyapunov} from the first statement of \hyperlink{paragraph:long_time_behaviour_S2.2}{$(S_{2.2})$}. First, we introduce a family of probability measures $(q(w,dw'))_{w\in\mathbb{R}_+^{2k}}$ defined such that for all~$w\in\mathbb{R}_+^{2k}$
$$
q(w,dw') = \frac{\mathdutchcal{V}(w')\mathdutchcal{K}_{j}(w,dw')}{\mathdutchcal{K}_j(\mathdutchcal{V})(w)}.
$$
Then, we iterate \eqref{eq:inequality_kernel_power_intermediate} to successively bound from above the measures of the form $\mathdutchcal{V}(w_i)\mathdutchcal{K}_j(w_{i-1},dw_i)$ in~\eqref{eq:inequality_time_mixing_intermediate_second}, where $i\in\llbracket1,l\rrbracket$, by the measure $(1+\varepsilon_1)\mathdutchcal{V}(w_{i-1})q(w_{i-1},dw_i)$. Thereafter, we plug the upper bound we have obtained for~\eqref{eq:inequality_time_mixing_intermediate_second} in~\eqref{eq:inequality_time_mixing_intermediate_zero}. Finally, in view of \hyperlink{paragraph:long_time_behaviour_S3.1}{$(S_{3.1})$}, we use Lemma \ref{lemm:inequality_for_means} for $b_1 = b$, $b_2 = \overline{b}$ and the kernel $(q(w,dw'))_{w\in\mathcal{X}}$. We get (recalling that $\overline{\psi} > 1$)
$$
\begin{aligned}
\mathbb{P}_{(x,a)}\left[\min(T_{all},\tau_{\partial}) > t\right] \leq 2k\times\overline{\psi}\times\left(1+t^{d_{\psi}}\right)e^{-\lambda_{\psi}t}\bigg[\overline{\mathdutchcal{J}}_a(t) &+ \sum_{n \geq 1} (1+\varepsilon_1)^n\\
&\times \left(\mathdutchcal{J}_a*\mathdutchcal{J}_0*^{(n-1)}\mathdutchcal{J}_0\right)(t)\bigg].
\end{aligned}
$$
We finally do exactly what we did to obtain~\eqref{eq:inequality_lyapunov_intermediate_sixth} from~\eqref{eq:inequality_Lyapunov} to conclude the proof of the lemma from the above equation.
\qed 

\subsubsection{Proof of Lemma \ref{lemma:inequality_generation_mixing}}\label{subsubsect:proof(A3)F_control_jump}

Let $\varepsilon >0$, $n\geq 1$, and let $(x,a)\in E$ be the initial condition of the process $(Z_t)_{t\geq0}$. Recall that $\mathcal{N}_{\partial}$ is the random variable that describes the number of jumps before extinction. We have for all $t\geq0$ that
\begin{equation}\label{eq:inequality_generation_mixing_intermediate_beforesecond_1}
\mathbb{P}_{(x,a)}\left[N_{t} > n,t < \tau_{\partial}\right] \leq  \mathbb{P}_{(x,a)}\left[\mathcal{N}_{\partial} > n\right] = \mathbb{P}_{(x,a)}\left[\forall i\in\llbracket1,n\rrbracket:\,(I_i,J_i)\neq (\partial,\partial)\right].
\end{equation}
Thus, our aim is to bound from above $\mathbb{P}_{(x,a)}\left[\mathcal{N}_{\partial} > n\right]$. Lemma~\ref{lemm:generalized_duhamel} can be slightly readapted, so that the time is no longer taken into account (we do not have a term $\overline{\mathdutchcal{H}}_0$, as we do not have a condition for~$\mathdutchcal{T}_2 - \mathdutchcal{T}_1$). Readapting this lemma, and using Remark~\ref{rem:equality_kernel}, we develop the right-hand side term of~\eqref{eq:inequality_generation_mixing_intermediate_beforesecond_1} to obtain
\begin{equation}\label{eq:inequality_generation_mixing_intermediate_second}
\begin{aligned}
\mathbb{P}_{(x,a)}\left[\mathcal{N}_{\partial} > n\right] &= \frac{1}{2^{n}}\int_{\mathbb{R}_+\times\mathbb{R}_+^{2k}}\hdots \int_{\mathbb{R}_+\times\mathbb{R}_+^{2k}}\mathdutchcal{V}(w_1)\hdots\mathdutchcal{V}(w_{n})\mathdutchcal{G}_{a}(x,s_1)\mathdutchcal{G}_0(w_1,s_2)\hdots \\
&\times \mathdutchcal{G}_0\left(w_{n-1},s_{n}\right)\left(ds_{1}d\mathdutchcal{K}(x,dw_1)\right)\hdots \left(ds_{n}d\mathdutchcal{K}(w_{n-1},dw_{n})\right).
\end{aligned}
\end{equation}
For all $(a',x',s')\in[0,a_0L_1]\times\mathbb{R}_+^{2k}\times\mathbb{R}_+$, it is easy to notice from Eq.~\eqref{eq:birth_rate_assumption} and the expression $\psi(x',a') = \mathdutchcal{V}(x')(1+(a')^{d_{\psi}})$ that
$$
\begin{aligned}
\mathdutchcal{G}_{a'}(x',s') &= 2\frac{b(x',a'+s')}{\psi(x',a')}\exp\left(-\int_{a'}^{a'+s'} b\left(x',u\right) du - \lambda_{\psi}s'\right) \\
&\leq 2\frac{\tilde{b}(1+(a_0L_1+s')^{d_b})}{\mathdutchcal{V}(x')}\exp\left(- \lambda_{\psi}s'\right).
\end{aligned}
$$
In addition, as $(x,a)\in E = D_{L_1}$, we know that $a\leq a_0L_1$. The latter, combined with~\eqref{eq:inequality_generation_mixing_intermediate_second}, the above equation, and the fact that $ \frac{(\mathdutchcal{K})^{n}(\mathdutchcal{V})(x)}{\mathdutchcal{V}(x)} \leq 2^n(1+\varepsilon_1)^n$ by iterating \eqref{eq:upperbound_fullkernel}, yields
$$
\mathbb{P}_{(x,a)}\left[\mathcal{N}_{\partial} > n\right] \leq \left(2\tilde{b}(1+\varepsilon_1)\int_{\mathbb{R}_+}(1+(a_0L_1+s)^{d_b})\exp\left(- \lambda_{\psi}s\right) ds\right)^{n} =: (c_{\lambda_{\psi}})^n.
$$
One can notice that if $c_{\lambda_{\psi}} < 1$ and
$$
n \geq \frac{\text{ln}\left(\frac{\varepsilon}{3}\right) - \rho t}{\ln\left(c_{\lambda_{\psi}}\right)}
$$
then it holds $(c_{\lambda_{\psi}})^n \leq \frac{\varepsilon}{3}\exp\left(-\rho t\right)$. Thus, if $\lambda_{\psi}$ is large enough to satisfy $c_{\lambda_{\psi}} < 1$,
then the function
$$
n(t) = \begin{cases}
\left\lceil \frac{\ln\left(\frac{\varepsilon}{3}\right) - \rho t}{\ln\left(c_{\lambda_{\psi}}\right)} \right\rceil, & \text{if } t > \frac{\ln\left(\frac{\varepsilon}{3}\right)}{\rho}, \\
1, & \text{if } t \leq \frac{\ln\left(\frac{\varepsilon}{3}\right)}{\rho},
\end{cases}
$$
will verify the statements of Lemma \ref{lemma:inequality_generation_mixing}.

Now, recall that $\lambda_{\psi}$ was chosen to verify the first two statements of Lemma \ref{lemma:inequalities_psi}. If the chosen value is not large enough to verify $c(\lambda_{\psi}) < 1$, then we choose another~$\bar \lambda_{\psi}>\lambda_{\psi}$. 
\qed

\subsubsection{Proof of \texorpdfstring{\eqref{eq:first_statement_(A3F)}}{(4.0.1)} and \texorpdfstring{\eqref{eq:third_statement_(A3F)}}{(4.0.3)}}\label{subsubsect:satisfaction_statements_A3F}

\paragraph{Proof of~\eqref{eq:first_statement_(A3F)}.}

We consider $t_F > 0$,~$n_J = n(t_F)\in\mathbb{N}^*$ and~$\eta_0 >0$, and a stopping time~$U_H$ defined as in \eqref{eq:stopping_time_UH} with these constants. Let us prove that~\eqref{eq:first_statement_(A3F)} holds for this stopping time. We consider $L_2\in\mathbb{N}^*$ such that $L_2\geq L_1$, $L_2B_{\max} > L_1B_{\max} + n_J\Delta$ and $a_0L_2 \geq t_F$. With this choice, one can see that $D_{L_1} \subset D_{L_2}$ and \hbox{$([\eta_0,L_1B_{\max} + n_J\Delta]^{2k}\times[0,t_F])\subset D_{L_2}$}. We also fix $V =t(L_2)$, where $t(L_2)> 0$ is a time such that there exists $C_{L_2} > 0$ satisfying for all $(x,a)\in E = D_{L_1}$
\begin{equation}\label{eq:recall_local_doeblin_condition}
\mathbb{P}_{(x,a)}\left[Z_{t(L_2)}\in dx'da';\,t(L_2)<\tau_{\partial}\right]  \geq C_{L_2}1_{[\eta_0,B_{\max}L_2]^{2k}\times [0,a_0L_2]}(x',a')dx'da'.
\end{equation}
Such a time exists by Proposition~\ref{prop:to_prove_A1}. Notice that on the event~$\{U_H < \tau_{\partial}\}$ it holds $X_{N_{t_F}}\in [\eta_0,B_{\max}L_1+ n_J\Delta]^{2k}$ (see~\eqref{eq:stopping_time_UH}). Recall also that $(X_{N_t})_{t\geq 0}$ is the marginal of~$(Z_t)_{t\geq 0}$ over telomere lengths. In view of these two points and~\eqref{eq:recall_local_doeblin_condition}, if the following statement is true, then we have that \eqref{eq:first_statement_(A3F)} is true with $U_H$ defined above and $V = t(L_2)$. 
\begin{prop}[Upper bound on non-rare events]\label{prop:prop_to_prove_(A3)}
Assume \hyperlink{paragraph:long_time_behaviour_S1.1}{$(S_{1.1})$}, \hyperlink{paragraph:long_time_behaviour_S1.3}{$(S_{1.3})$}, and \hyperlink{paragraph:long_time_behaviour_S2.2}{$(S_{2.2})$} hold. Then for all $t\geq0$, there exists $\overline{C}(t) >0$ such that for every $(x,a)\in E = D_{L_1}$
\begin{equation}\label{eq:eq_to_prove_(A3)}
\mathbb{P}_{(x,a)}[Z_{t} \in dx'da';\,T_{all}  \leq t < \tau_{\partial},\,N_{t} \leq n(t)] \leq \overline{C}(t).1_{[0,B_{\max}L_1+n(t)\Delta]^{2k}\times[0,t]}(x',a')dx'da'.
\end{equation}
Moreover, $\overline{C}(t)$ increases with $t$.
\end{prop}
\noindent We point out here that the increasing character of $\overline{C}(t)$ is only useful to obtain the density representation of the stationary profile in Section~\ref{subsect:density_stationary_profile}, and is not used here.

The proof of Proposition~\ref{prop:prop_to_prove_(A3)} is long, computational and requires a lot of notations to be introduced. Hence, we prefer to leave the detailed proof of the proposition in Appendix~\ref{appendix:proof_A3_stepB}. Below are given briefly the main arguments of this proof.
\begin{itemize}[leftmargin=*]
\item In view of \hyperlink{paragraph:long_time_behaviour_S1.1}{$(S_{1.1})$}, the probability density function $g$ representing the distribution of shortening values is bounded on $\mathbb{R}_+$, and the probability density function $h$ representing the distribution of lengthening values is bounded on~$[-\delta,B_{\max}L_1 + n(t)\Delta]\times[0,\Delta]$. 

\item Proposition~\ref{prop:prop_to_prove_(A3)} concerns initial conditions $(x,a)\in D_{L_1}$. Hence, the initial age is bounded by $a_0L_1$, and before the time $t$, the age of the particle is bounded by~$a_0L_1+t$. This implies by~\eqref{eq:birth_rate_assumption} that the rate at which a jump occurs stays bounded. 

\item All coordinates have been mixed at least once by the Lebesgue measure, and the number of mixing is bounded by $n(t)$. 

\end{itemize}
Since "everything is bounded" and since each coordinate has been mixed at least one time by the Lebesgue measure, the left-hand side term of \eqref{eq:eq_to_prove_(A3)} is bounded by the Lebesgue measure multiplied by a constant~$\overline{C}(t)$. The function $\overline{C}$ increases because $n$ increases by Lemma~\ref{lemma:inequality_generation_mixing}. For telomere lengths, the Lebesgue measure is restricted on the set \hbox{$[0,B_{\max}L_1 + n(t)\Delta]^{2k}$} for the following reason: the initial condition is $x\in[0,B_{\max}L_1]^{2k}$, we have at most $n(t)$ jumps, and the maximum lengthening value is $\Delta$. For the age, we have a restriction of the Lebesgue measure on~$[0,t]$ because on the event~$\{T_{all} \leq t\}$, at least one jump occurs before time $t$. At this jump the age is reset to $0$, so we necessarily have that the age of the particle is at most $t$ at time $t$. 

From these points, Proposition~\ref{prop:prop_to_prove_(A3)} is true, which implies that Eq.~\eqref{eq:first_statement_(A3F)} is verified for $U_H$ defined as above.

\paragraph{Proof of \eqref{eq:third_statement_(A3F)}.} Now, we prove that for all $\varepsilon > 0$, there exists $t_F > 0$, and $\eta_0 > 0$ such that~\eqref{eq:third_statement_(A3F)} is verified for $U_H$ defined with $t_F$, $n_J = n(t_F)$ and $\eta_0$. Before proving~\eqref{eq:third_statement_(A3F)}, we need to control the probability that the particle has a telomere with a length too close to $0$. Let us fix $\varepsilon >0$, and $t_F>0$ and $n_J = n(t_F)\in\mathbb{N}^*$ such that \eqref{eq:inequality_time_mixing} and~\eqref{eq:inequality_generation_mixing} are satisfied. In view of the fact that $\left(X_{N_t}\right)_{t\geq0}$ is the marginal of $\left(Z_t\right)_{t\geq 0}$ over telomere lengths, first apply Proposition~\ref{prop:prop_to_prove_(A3)} to bound from above 
$$
\mathbb{P}_{(x,a)}\big[X_{N_{t_F}}\in\left([\eta,B_{\max}L_1+n_J\Delta]^{2k}\right)^c;\,T_{all}  \leq t_F < \tau_{\partial},\,N_{t} \leq n_J\big].
$$
Then, use the fact that 
$$
[0,B_{\max}L_1+n_J\Delta]^{2k} \backslash\big([\eta,B_{\max}L_1+n_J\Delta]^{2k}\big)\subset \overset{2k}{\underset{i = 1}{\bigcup}} \left\{x\in[0,B_{\max}L_1+n_J\Delta]^{2k}\,|\, x_i \leq \eta\right\}
$$
to bound from above the marginal of the Lebesgue measure over telomere lengths in the bound we have obtained (the measure of the set on the right is bounded from above by~$2k\eta\left(B_{\max}L_1 + n_J\Delta\right)^{2k-1}$). It comes that for all~$\eta > 0$,~$(x,a)\in E$
$$
\begin{aligned}
&\mathbb{P}_{(x,a)}\big[X_{N_{t_F}}\in\left([\eta,B_{\max}L_1+n_J\Delta]^{2k}\right)^c;\,T_{all}  \leq t_F < \tau_{\partial},\,N_{t_F} \leq n_J\big] \\ 
&\leq \overline{C}(t_F)2k\eta\left(B_{\max}L_1 + n_J\Delta\right)^{2k-1}t_F.
\end{aligned}
$$
The latter implies that there exists $\eta_0 > 0$ such that 
\begin{equation}\label{eq:inequality_jump_origin}
\mathbb{P}_{(x,a)}\big[X_{N_{t_F}}\in\left([\eta_0,B_{\max}L_1+n_J\Delta]^{2k}\right)^c;\,T_{all}  \leq t_F < \tau_{\partial},\,N_{t_F} \leq n_J\big] \leq \frac{\varepsilon}{3}\exp\left(-\rho t_F\right).
\end{equation}

Now, let us consider $U_H$ defined as in~\eqref{eq:stopping_time_UH} with $t_F$, $n_J$ and $\eta_0$ of the previous paragraph. Combining~\eqref{eq:inequality_time_mixing},~\eqref{eq:inequality_generation_mixing} and~\eqref{eq:inequality_jump_origin} yields that for all $(x,a)\in E$,
$$
\begin{aligned}
\mathbb{P}_{(x,a)}[U_H = \infty,t_F < \tau_{\partial}] &= \mathbb{P}_{(x,a)}[\min(T_{all},\tau_{\partial}) >  t_F] + \mathbb{P}_{(x,a)}\left[N_{t_F} > n_J,T_{all} \leq t_F,t_F < \tau_{\partial}\right] \\
&+\mathbb{P}_{(x,a)}\big[X_{N_{t_F}}\in\left([\eta_0,B_{\max}L_1+n_J\Delta]^{2k}\right)^c;\\
&\hspace{4mm}T_{all}  \leq t_F < \tau_{\partial},\,N_{t_F} \leq n_J\big] \\
&\leq  \varepsilon \exp\left(-\rho t_F\right)  .%\hspace{-6mm}
\end{aligned}
$$
Then, \eqref{eq:third_statement_(A3F)} is satisfied, which ends the proof of Assumption~\hyperlink{te:assumptions_velleret_A3}{$(A_3)_F$}.

\subsection{Existence and uniqueness of the stationary profile}\label{subsect:existence_stationary_profile}

We recall here Notations \hyperref[notation:start_nota_space_distortion]{\ref*{notation:start_nota_space_distortion}} and \hyperref[notation:end_nota_space_distortion]{\ref*{notation:end_nota_space_distortion}}, and refer to \eqref{eq:first_moment_semigroup}, \eqref{eq:weighted_renormalised_semigroup} and \eqref{eq:first_moment_particle} for the definitions of the semigroups $(M_t)_{t\geq0}$, $(M^{(\psi)}_t)_{t\geq0}$ and~$(P^{(\psi)}_t)_{t\geq0}$ used in this subsection. We also recall that $\psi$ has been fixed in Section~\ref{subsect:preliminaries_theorem}, and that the index $\psi$ was dropped in the process $\big(Z_t^{(\psi)}\big)_{t\geq0}$.

\paragraph{Existence.}  We begin by proving that~$(M_t)_{t\geq0}$ converges in $\mathdutchcal{M}\left(\psi\right)$ towards a stationary profile using what we did in the previous subsections. First, as \hyperlink{te:assumptions_velleret_A1}{$(A_1)$}, \hyperlink{te:assumptions_velleret_A2}{$(A_2)$} and \hyperlink{te:assumptions_velleret_A3}{$(A_3)_F$} are verified for the process~²$\big(Z_t^{(\psi)}\big)_{t\geq0}$, by~\eqref{eq:equality_semigroup} and Theorem \ref{te:assumptions_velleret} there exists a triplet of eigenelements $(\tilde{\gamma},\tilde{\phi},\tilde{\lambda})\in\mathcal{M}_1\left(\mathcal{X}\right)\times M_b\left(\mathcal{X}\right)\times\mathbb{R}_+$ s.t.~$\tilde{\gamma}(\tilde{\phi}) = 1$, and two constants $C,\omega >0$, such that for all $t>0$, $\tilde{\mu}\in \mathcal{M}_1(\mathcal{X})$
\begin{equation}\label{eq:inequality_application_Velleret}
\begin{aligned}
\sup_{\substack{\tilde{f}\in M_b(\mathcal{X}), ||\tilde{f}||_{\infty} \leq 1}} \Big|e^{(\tilde{\lambda} -\lambda_{\psi})t}\tilde{\mu}\frac{M_t (\psi \tilde{f})}{\psi} - \tilde{\mu}\left(\tilde{\phi}\right)\tilde{\gamma}(\tilde{f})\Big|  = \Big|\Big|e^{\tilde{\lambda} t}\tilde{\mu} P_t^{(\psi)} -  \tilde{\mu}\left(\tilde{\phi}\right)\tilde{\gamma}\Big|\Big|_{TV,\mathcal{X}}\leq Ce^{-\omega t}.
\end{aligned}
\end{equation}
From \cite[Prop. 2.10]{velleret_exponential_2023}, we also have $\tilde{\phi} > 0$. 

Second, recalling Notations \hyperref[notation:start_nota_space_distortion]{\ref*{notation:start_nota_space_distortion}} and \hyperref[notation:end_nota_space_distortion]{\ref*{notation:end_nota_space_distortion}}, the following equalities hold:
\begin{align}
\left\{\tilde{f}\psi\text{ s.t. }\tilde{f}\in M_b(\mathcal{X}),\,||\tilde{f}||_{\infty} \leq 1\right\} = \{f \in \mathdutchcal{B}(\psi)\text{ s.t. }||f||_{\mathdutchcal{B}(\psi)} \leq 1\},\label{eq:equality_distorted_functions} \\
\left\{\mu\in\mathdutchcal{M}_+(\psi) \,\big|\,\exists \tilde{\mu}\in\mathcal{M}_1(\mathcal{X}) \text{ s.t. } \mu(.) = \tilde{\mu}\left(\frac{.}{\psi}\right)\right\} = \left\{\mu \in \mathdutchcal{M}_+(\psi)\,\big|\,||\mu||_{\mathdutchcal{M}(\psi)} = 1\right\}. \label{eq:equality_distorted_measures}
\end{align}

Finally, we consider 
\begin{equation}\label{eq:link_eigen_particle_and_branching}
\phi = \tilde{\gamma}\left(\frac{1}{\psi}\right)\tilde{\phi}\psi \in \mathdutchcal{B}(\psi), \hspace{1.5mm}\text{ and }\hspace{1.5mm} \gamma(.) = \left[\tilde{\gamma}\left(\frac{1}{\psi}\right)\right]^{-1}\tilde{\gamma}\left(\frac{.}{\psi}\right) \in \left[\mathdutchcal{M}_+(\psi)\cap\mathcal{M}_1(\mathcal{X})\right].
\end{equation}
As $\tilde{\gamma}(\tilde{\phi}) = 1$, we easily see that  $\gamma(\phi) = 1$. We also introduce $\lambda = \lambda_{\psi} - \tilde{\lambda}$, that belongs to~$[\alpha,+\infty[$ by Proposition \ref{prop:simplification_A2} and the fact that $\rho_S = \tilde{\lambda}$ (see \mbox{\cite[Theorem $2.8$]{velleret_exponential_2023}}).

We can use these three results to conclude on the convergence towards a stationary profile, starting from measures in $\mathdutchcal{M}_+(\psi)$ s.t. $||\mu||_{\mathdutchcal{M}(\psi)} = 1$. First, we replace the terms $\tilde{\gamma}$, $\tilde{\phi}$ and $\tilde{\lambda} -\lambda_{\psi}$ in~\eqref{eq:inequality_application_Velleret} by the terms $\gamma$, $\phi$ and~$\lambda$ respectively, using the definitions of the latter. Then, we use \eqref{eq:equality_distorted_functions} to change functions $\tilde{f}\psi$ in~\eqref{eq:inequality_application_Velleret} into functions $f \in \mathdutchcal{B}(\psi)$. Finally, we use \eqref{eq:equality_distorted_measures} to change measures $\tilde{\mu}\left(\frac{.}{\psi}\right)$ into measures $\mu \in \mathdutchcal{M}(\psi)$. We obtain that for all $\mu \in \mathdutchcal{M}_+(\psi)$ such that $||\mu||_{\mathdutchcal{M}(\psi)} = 1$
\begin{equation}\label{eq:main_result_probability_measure}
\forall \mu \in \mathdutchcal{M}_+(\psi) \text{ s.t. } ||\mu||_{\mathdutchcal{M}(\psi)} = 1: \hspace{2.5mm}\sup_{f \in \mathdutchcal{B}(\psi),\,||f||_{\mathdutchcal{B}(\psi)} \leq 1} \left|e^{-\lambda t}\mu M_t (f)- \mu(\phi)\gamma\left(f\right)\right| \leq Ce^{-\omega t}.
\end{equation}

Now, we extend the result above to measures in $\mathdutchcal{M}(\psi)$. Let us consider $\mu\in\mathdutchcal{M}(\psi)$. By definition of~$\mathdutchcal{M}(\psi)$ (Notation~\hyperref[notation:end_nota_space_distortion]{\ref*{notation:end_nota_space_distortion}}), there exists a couple $(\mu_+,\mu_-)\in\mathdutchcal{M}_+(\psi)\times\mathdutchcal{M}_+(\psi)$ such that $\mu = \mu_+ - \mu_-$. First, we take for initial measures 
$$
\frac{\mu_+}{||\mu_+||_{\mathdutchcal{M}(\psi)}} = \frac{\mu_+}{\mu_+(\psi)} \hspace{3mm}\left(\text{or }\frac{\mu_-}{||\mu_-||_{\mathdutchcal{M}(\psi)}}\right). 
$$
Then, we apply~\eqref{eq:main_result_probability_measure} and multiply by $\mu_+(\psi)$ (or $\mu_-(\psi)$). Finally in view of the equality~\hbox{$\mu = \mu_+ - \mu_-$}, we use the triangular inequality. We obtain at the end the desired~result.

%\subsection{Uniqueness of the stationary profile}\label{subsect:uniqueness_stationary_profile}

\paragraph{Uniqueness.} We recall that the set $\Psi$, that contains all the functions used to distort the space, was defined in~\eqref{eq:defintion_all_distortion_functions}. The function $\psi\in\Psi$ that we have fixed at the beginning of the proof is no more fixed from here.
By the previous paragraph, we have that for all~$\psi\in\Psi$, the first moment semigroup converges towards a stationary profile in  $\mathdutchcal{M}(\psi)$ endowed with~$||.||_{\mathdutchcal{M}(\psi)}$. We now prove that this stationary profile is the same whatever the function~$\psi\in\Psi$. 

Let $(\psi_1,\psi_2)\in\Psi^2$. For all $i\in\{1,2\}$, we denote by $(\gamma_i,\phi_i,\lambda_i)\in\mathdutchcal{M}(\psi_i)\times\mathdutchcal{B}(\psi_i)\times\mathbb{R}_+^*$ the triplet of eigenelements of $(M_t)_{t\geq0}$ obtained by a distortion of the space by~$\psi_i$. These eigenelements also verify the following properties: $\gamma_i\in\mathcal{M}_1(\mathcal{X})$, $\phi_i > 0$ and $\gamma_i(\phi_i) = 1$. By applying~\eqref{eq:main_result_probability_measure} to Dirac measures, and then the fact that $\gamma_1$ and $\gamma_2$ are probability measures, we have that for all $i\in\{1,2\}$, $z\in\mathcal{X}$,
$$
\lim_{t\rightarrow+\infty}  e^{-\lambda_i t}M_t\left(1\right)(z) = \phi_i(z)\gamma_i(1) = \phi_i(z).
$$
Then, from the above equality, we necessarily have that $\left(\phi_1,\lambda_1\right) = \left(\phi_2,\lambda_2\right)$. It thus only remains to prove that $\gamma_1 = \gamma_2$, and the uniqueness will be proved. To do this, notice that by applying~\eqref{eq:main_result_probability_measure} to Dirac measures and indicators, and then using that \hbox{$\left(\phi_1,\lambda_1\right) = \left(\phi_2,\lambda_2\right)$}, we have for all $z\in\mathcal{X}$ and $A\in\mathcal{B}\left(\mathcal{X}\right)$,
$$
\gamma_1\left(A\right) = \lim_{t\rightarrow+\infty} \frac{e^{-\lambda_1 t}M_t\left(1_A\right)(z)}{\phi_1(z)} = \lim_{t\rightarrow+\infty} \frac{e^{-\lambda_2 t}M_t\left(1_A\right)(z)}{\phi_2(z)} = \gamma_2\left(A\right).
$$
Then, from the above, we easily conclude that $\gamma_1 = \gamma_2$, so that $\left(\gamma_1,\phi_1,\lambda_1\right) = \left(\gamma_2,\phi_2,\lambda_2\right)$.

% $\left(\lambda_1,\phi_1,\gamma_1\right) = \left(\lambda_2,\phi_2,\gamma_2\right)$
\subsection{Density representation of the stationary profile}\label{subsect:density_stationary_profile}
Now that we know that there exists a stationary profile which is the same for all $\psi \in \Psi$, we prove that the latter can be represented by a function. In particular, we emphasize the fact that the density representation of the stationary profile can be obtained as a consequence of the statements and the arguments given to prove Assumption~\hyperlink{te:assumptions_velleret_A3}{$(A_3)_F$}. We believe that this proof can be adapted to many other models for which the discontinuities with respect to the Lebesgue of the model can be controlled over time, as for example the models given in \cite{velleret_adaptation_2023,velleret_exponential_2023}. 

First, we prove in Section~\ref{subsubsect:absolute_continuity} that the stationary profile admits a density with respect to the Lebesgue measure thanks to the results obtained in Section \ref{subsubsect:assumption_(A3)F}. Then, we prove in Section \ref{subsubsect:separation_space_age} that the function representing our stationary profile can be seen as the product of two functions, one linked to telomere lengths, and the other linked to the age.

\subsubsection{Absolute continuity of the stationary profile with respect to the Lebesgue measure}\label{subsubsect:absolute_continuity}
Let $\psi\in \Psi$. In view of the first paragraph of Section \ref{subsect:existence_stationary_profile}, we have the equivalence between obtaining the absolute continuity of the stationary profile of $\left(M_t\right)_{t\geq0}$ with respect to~$\psi(x,a) dx da$, and obtaining the absolute continuity of the stationary profile of $\big(P_t^{(\psi)}\big)_{t\geq 0}$ with respect to~$dx da$. We will hence prove that the stationary profile of~$(P_t^{(\psi)})_{t\geq 0}$, denoted $\tilde{\gamma}$, has a density with respect to the Lebesgue measure.

Let $\tilde{\lambda}\in\mathbb{R}_+$ the absorbing rate of $\big(P_t^{(\psi)}\big)_{t\geq 0}$. As $\tilde{\gamma}$ is a quasi-stationary distribution for~$\big(P_t^{(\psi)}\big)_{t\geq 0}$, we easily obtain by~\hbox{\cite[Eq.~$(2.3)$]{velleret_unique_2022}} that for all~$t\geq 0$, $A\in \mathcal{B}(\mathcal{X})$,
\begin{equation}\label{eq:stationary_profile_fixed_point}
\tilde{\gamma}(A) = e^{\tilde{\lambda} t}\int_{(x,a)\in \mathcal{X}}P_t^{(\psi)}(1_A)(x,a)\tilde{\gamma}(dx,da). 
\end{equation}
We use this equality to get the abolute continuity of the stationary profile. Two difficulties arise:
\begin{itemize}[leftmargin=*]
\item As explained in Section \ref{subsubsect:assumption_(A3)F}, $(P_t^{(\psi)})_{t\geq 0}$ has a part absolutely continuous with respect to the Lebesgue measure, and another part discontinuous with respect to the Lebesgue measure. We need to handle the discontinuous part. For that, we use the results obtained in Section \ref{subsubsect:assumption_(A3)F}.

\item The results we have obtained in Section \ref{subsubsect:assumption_(A3)F} are only true for measures starting from $E = D_{L_1}$, where $L_1\in\mathbb{N}^*$ is large and has been fixed in the last paragraph of Section~\ref{subsubsect:useful_statements_A2}. In \eqref{eq:stationary_profile_fixed_point}, the relation between $\tilde{\gamma}$ and $(P_t^{(\psi)})_{t\geq 0}$ involves initial conditions in the set $E^c$. Thus, we need to handle such initial conditions.
\end{itemize}

Let us first control the time we need to return to $E$. As $\rho = \lambda_{\psi} - \frac{\alpha + \beta}{2}$ and $\beta < \alpha$, Proposition \ref{prop:bound_tail_probability} ($L_1$ is chosen sufficiently large to have $L_1 \geq L_0$) implies the following statement.
\begin{lemm}[Control of the time to return to $E$]\label{lemm:proof_representation_inequality_returncompact}
Assume that \hyperlink{paragraph:long_time_behaviour_S1.1}{$(S_{1.1})$}, \hyperlink{paragraph:long_time_behaviour_S2}{$(S_2)$} and~\hyperlink{paragraph:long_time_behaviour_S3}{$(S_3)$} hold. Then there exists $t_1 > 0$ such that for all $t\geq t_1$
$$
\sup_{(x,a)\in \mathcal{X}}\left(\mathbb{P}_{(x,a)}\left[\min\left(\tau_E,\tau_{\partial}\right) > t\right]\right) \leq \frac{\varepsilon}{3}\exp\left(-\rho t\right).
$$
\end{lemm}

\noindent Now that we have Lemma \ref{lemm:proof_representation_inequality_returncompact}, using the constant $t_0$ introduced in Lemma~\ref{lemma:inequality_time_mixing}, we control $P_t^{(\psi)}$ starting from the restriction of $\tilde{\gamma}$ on $E^c$ for $t \geq t_0 + t_1$. This implies the following statement, that is proved at the end of the subsection.
\begin{lemm}[Control of the discontinuities starting from $E^c$]\label{lemm:representation_upperbound_outsidelocalset}
Assume that \hyperlink{paragraph:long_time_behaviour_S1.1}{$(S_{1.1})$}, \hyperlink{paragraph:long_time_behaviour_S2}{$(S_2)$} and~\hyperlink{paragraph:long_time_behaviour_S3}{$(S_3)$} hold. Let $t_0,t_1 >0$ the two constants introduced in Lemmas \ref{lemma:inequality_time_mixing} and \ref{lemm:proof_representation_inequality_returncompact}. Then for all $t\geq t_0 + t_1$, $A\in\mathcal{B}(\mathcal{X})$, the following holds
\begin{equation}\label{eq:representation_upperbound_outsidelocalset}
\begin{aligned}
\int_{(x,a)\in E^c}P_t^{(\psi)}(1_A)(x,a)\tilde{\gamma}(dx,da) &\leq  \overline{C}(t)\text{Leb}(A) + \frac{2\varepsilon}{3}.e_{\mathcal{T}}.e^{-\rho t} + \frac{\varepsilon}{3} e^{-\rho(t-t_0)},
\end{aligned}
\end{equation}
where $e_{\mathcal{T}} := \sup_{(x,a)\in\mathcal{X}} \left[\mathbb{E}_{(x,a)}\left[\exp\left(\rho\min\left(\tau_{\partial}, \tau_E\right)\right)\right]\right]$ (finite by \hyperlink{te:assumptions_velleret_A2}{$(A_2)$}), and $\overline{C}(t)$ is the constant introduced in Proposition~\ref{prop:prop_to_prove_(A3)}.
\end{lemm}

For the restriction of $\tilde{\gamma}$ on $E$, a similar upper bound can be obtained with the same arguments. The only slight difference is that we do not need to handle the fact that we need to return to $E$. Here is the statement that corresponds to this upper bound. As the proof of this statement consists of applying exactly the same arguments as those used to obtain \eqref{eq:representation_upperbound_outsidelocalset} without handling the return to $E$, we do not detail it.
\begin{lemm}[Control of the discontinuities starting from $E$]\label{lemm:representation_upperbound_insidelocalset}
Assume that \hyperlink{paragraph:long_time_behaviour_S1.1}{$(S_{1.1})$}, \hyperlink{paragraph:long_time_behaviour_S2}{$(S_2)$} and~\hyperlink{paragraph:long_time_behaviour_S3}{$(S_3)$} hold. Let $t_0>0$ the constant introduced in Lemma~\ref{lemma:inequality_time_mixing}. Then for all $t\geq t_0$, $A\in\mathcal{B}(\mathcal{X})$, the following holds
\begin{equation}\label{eq:representation_upperbound_insidelocalset}
\begin{aligned}
\int_{(x,a)\in E} P_t^{(\psi)}(1_A)(x,a)\tilde{\gamma}(dx,da) &\leq \overline{C}(t)\text{Leb}(A) + \frac{2\varepsilon}{3} e^{-\rho t},  
\end{aligned}
\end{equation}
where $\overline{C}(t)$ is the constant introduced in Proposition~\ref{prop:prop_to_prove_(A3)}.
\end{lemm}
Now, we obtain that $\tilde{\gamma}$ can be represented by a function. Plugging \eqref{eq:representation_upperbound_outsidelocalset} and \eqref{eq:representation_upperbound_insidelocalset} in \eqref{eq:stationary_profile_fixed_point}, we get for all $t\geq t_0 + t_1$
$$
\begin{aligned}
\tilde{\gamma}(A) \leq 2e^{\tilde{\lambda} t}\overline{C}(t)\text{Leb}(A) + \frac{2\varepsilon}{3}.e_{\mathcal{T}}.e^{-\rho t + \tilde{\lambda} t} + \frac{2\varepsilon}{3} e^{-\rho(t-t_0) + \tilde{\lambda} t} + \frac{\varepsilon}{3} e^{-\rho t + \tilde{\lambda} t}.
\end{aligned}
$$
As $\rho = \lambda_{\psi} - \frac{\alpha + \beta}{2} > \tilde{\lambda}$ (by \eqref{eq:inequality_absorbing_rate} and the fact that $\tilde{\lambda} = \rho_S$, see \cite[Theorem $2.8$]{velleret_exponential_2023}), there exists $t_2 \geq t_0 + t_1$ such that $e^{-\rho t_2 + \tilde{\lambda} t_2} \leq \min\left(1,\left(e_{\mathcal{T}}\right)^{-1},e^{-\rho t_0}\right)$. Then, take $t = t_2$ in the above equation to obtain
$$
\begin{aligned}
\tilde{\gamma}(A) \leq 2e^{\tilde{\lambda} t_2}\overline{C}(t_2)\text{Leb}(A)+ \frac{5\varepsilon}{3}.
\end{aligned}
$$
From the above and~\cite[Prop.~$15.5$.b]{nielsen_introduction_1997},~$\tilde{\gamma}$ is absolutely continuous with respect to the Lebesgue measure. 
\begin{rem}
We mention that it is also possible to obtain the density representation of $\tilde{\gamma}$ by the usual way starting from an absolutely continuous initial condition. Namely, one should prove that for all $\mu\in\mathcal{M}_1\left(\mathcal{X}\right)$ absolute continuous with respect to the Lebesgue measure, $\mu P_t^{(\psi)}$~is absolute continuous with respect to the Lebesgue measure.  Then, the conclusion follows from~\eqref{eq:inequality_application_Velleret} and the fact that the limit in total variation norm preserves the regularity with respect to the Lebesgue measure. We do not proceed like this here because the first step implies cumbersome computations similar to the ones in Appendix~\ref{appendix:proof_A3_stepB}, and because we think that our new method is interesting in itself.
\end{rem}

We now denote $\gamma$ the stationary profile of $(M_t)_{t\geq 0}$. We prove that it can be represented by a function in $L^1(\Psi)$. As $\tilde{\gamma}$ is absolutely continuous with respect to the Lebesgue measure, the right-hand side of~\eqref{eq:link_eigen_particle_and_branching} implies that there exists \hbox{$N \in L^1_{p.d.f.}(\mathcal{X})\cap L^1(\psi)$} (see Notation~\ref{notations:L1_pdf}) such that for all~$A\in\mathcal{B}(\mathcal{X})$
\begin{equation}\label{eq:stationary_measure_representable_function}
\gamma(A) = \int_{A} N(x,a) dx da. 
\end{equation}
Thus, as $\psi\in\Psi$ was chosen arbitrarily, and as the stationary profile is the same whatever the function $\psi \in \Psi$, we have $N \in \bigcap_{\psi\in\Psi} L^1(\psi) = L^1(\Psi)$. It remains to prove Lemma \ref{lemm:representation_upperbound_outsidelocalset}.
\begin{proof}[Proof of Lemma \ref{lemm:representation_upperbound_outsidelocalset}]
Let us consider $t\geq t_0 + t_1$, and $A\in\mathcal{B}(\mathcal{X})$. We have
$$
\begin{aligned}
\int_{(x,a)\in E^c}P_t^{(\psi)}(1_A)(x,a)\gamma(dx,da) &= \int_{(x,a)\in E^c}\mathbb{P}_{(x,a)}\Big[Z_{t}^{(\psi)}\in A;\tau_E \leq t-t_0,t < \tau_{\partial}\Big]\gamma(dx,da) \\
&+ \int_{(x,a)\in E^c}\mathbb{P}_{(x,a)}\Big[Z_{t}^{(\psi)}\in A;\tau_E > t-t_0,t < \tau_{\partial}\Big]\gamma(dx,da).
\end{aligned}
$$
One can easily see that if $\tau_E > t-t_0$ and $\tau_{\partial} > t$, then $\min(\tau_E,\tau_{\partial}) > t-t_0\geq t_1$. The latter combined with Lemma \ref{lemm:proof_representation_inequality_returncompact} yields 
\begin{equation}\label{eq:representation_upperbound_outsidelocalset_first}
\begin{aligned}
\int_{(x,a)\in E^c}P_t^{(\psi)}(1_A)(x,a)\gamma(dx,da) &\leq \int_{(x,a)\in E^c}\mathbb{P}_{(x,a)}\left[Z_{t}^{(\psi)}\in A;\tau_E \leq t-t_0,t < \tau_{\partial}\right]\gamma(dx,da) \\
&+ \frac{\varepsilon}{3} \exp\left(-\rho (t - t_0)\right).
\end{aligned}
\end{equation}
We now obtain an upper bound for the first term on the right-hand side term of \eqref{eq:representation_upperbound_outsidelocalset_first}. With the strong Markov property (and an abuse of notations), we have for all $(x,a)\in E^c$

\begin{multline}\label{eq:representation_upperbound_outsidelocalset_second}
\mathbb{P}_{(x,a)}\left[Z_{t}^{(\psi)}\in A\,;\,\tau_E \leq t-t_0,t < \tau_{\partial}\right] \\ 
= \mathbb{E}_{(x,a)}\left[\mathbb{P}_{Z_{\tau_E}^{(\psi)}}\left[\tilde{Z}_{t-\tau_E}^{(\psi)} \in A, t - \tau_E < \tilde{\tau}_{\partial}\right];\,\tau_E \leq t-t_0,\tau_E < \tau_{\partial}\right], 
\end{multline}  
where $\big(\tilde{Z}_t^{(\psi)}\big)_{t\geq0}$ is an independent copy of $\big(Z_t^{(\psi)}\big)_{t\geq0}$ that has the same distribution, and $\tilde{\tau}_{\partial}$ is the equivalent of $\tau_{\partial}$ for $\big(\tilde{Z}_t^{(\psi)}\big)_{t\geq0}$. One can decompose the probability in the right-hand side term as follows
$$
\begin{aligned}
\mathbb{P}_{Z_{\tau_E}^{(\psi)}}\Big[\tilde{Z}_{t-\tau_E}^{(\psi)} \in A, t - \tau_E < \tilde{\tau}_{\partial}\Big] &= \mathbb{P}_{Z_{\tau_E}}\left[\tilde{Z}_{t-\tau_E}^{(\psi)} \in A, \tilde{T}_{all} \leq t - \tau_E < \tilde{\tau}_{\partial},\tilde{N}_{t - \tau_E} \leq n(t - \tau_E)\right]\\ 
&+ \mathbb{P}_{Z_{\tau_E}^{(\psi)}}\left[\tilde{Z}_{t-\tau_E}^{(\psi)} \in A, \tilde{T}_{all} \leq t - \tau_E < \tilde{\tau}_{\partial},\tilde{N}_{t - \tau_E} > n(t - \tau_E)\right] \\ 
&+ \mathbb{P}_{Z_{\tau_E}^{(\psi)}}\left[\tilde{Z}_{t-\tau_E}^{(\psi)} \in A, \tilde{T}_{all} > t - \tau_E,  t - \tau_E < \tilde{\tau}_{\partial}\right],\\
\end{aligned}
$$
where $\tilde{T}_{all}$ and $\tilde{N}_t$ are the equivalents of $T_{all}$ and $N_t$ for $\big(\tilde{Z}_t^{(\psi)}\big)_{t\geq0}$. Then, in view of Proposition \ref{prop:prop_to_prove_(A3)} for the first term, Lemma~\ref{lemma:inequality_generation_mixing} for the second term, and Lemma \ref{lemma:inequality_time_mixing} for the last term, it holds on the event~$\{\tau_E \leq t-t_0,\tau_E < \tau_{\partial}\}$
$$
\mathbb{P}_{Z_{\tau_E}^{(\psi)}}\left[\tilde{Z}_{t-\tau_E}^{(\psi)} \in A, t - \tau_E < \tilde{\tau}_{\partial}\right] \leq \overline{C}(t-\tau_E)\text{Leb}(A) + \frac{2\varepsilon}{3}\exp\left(-\rho(t-\tau_E)\right).
$$
Plugging this in \eqref{eq:representation_upperbound_outsidelocalset_second}, then using the fact that $\overline{C}$ increases (see Prop.~\ref{prop:prop_to_prove_(A3)}), and finally noticing that $\tau_E = \min(\tau_E,\tau_{\partial})$ on the event $\left\{\tau_E < \tau_{\partial}\right\}$ yields ($e_{\mathcal{T}}$ is defined in the statement of the lemma)
$$
\begin{aligned}
\mathbb{P}_{(x,a)}\left[Z_{t}^{(\psi)}\in A\,;\,\tau_E \leq t-t_0,t < \tau_{\partial}\right] \leq \overline{C}(t)\text{Leb}(A) + \frac{2\varepsilon}{3}.e^{-\rho t}.e_{\mathcal{T}}.
\end{aligned}
$$
Then, plugging the latter in \eqref{eq:representation_upperbound_outsidelocalset_first} implies that the lemma is proved.
\end{proof}

\subsubsection{Separation of the space and age}\label{subsubsect:separation_space_age}
Here, we prove that there exists $N_0\in L^1(\mathbb{R}_+^{2k})$ non-negative, such that for all $(x,a)\in\mathcal{X}$
\begin{equation}\label{eq:separation_age_space}
N(x,a) = N_0(x)\exp\left(-\lambda a - \int_0^a b(x,s) ds\right),
\end{equation}
and such that $N_0$ is solution to~\eqref{eq:equation_N_0}. To do this, we begin by proving~\eqref{eq:separation_age_space}. Let \hbox{$\varphi\in \mathcal{C}_b^{m,1}\left(\mathcal{X}\right)$} (see~Notation~\ref{nota:Cm1}) and $s\geq0$. First, we use~\eqref{eq:main_result_probability_measure} with $\mu = \gamma$ and $f = M_s(\varphi)$ (one can normalise $f$ to have $||f||_{\mathdutchcal{B}(\psi)} \leq 1$) to obtain that 
$$
\underset{t\rightarrow + \infty}{\lim} e^{-\lambda t} \gamma M_{t+s}(\varphi) = \gamma(\phi)\gamma(M_s(\varphi)) = \gamma(M_s(\varphi)).
$$
Then, we apply again~\eqref{eq:main_result_probability_measure} with  $\mu = \gamma$ and $f = \varphi$ to obtain that 
$$
\underset{t\rightarrow + \infty}{\lim} e^{-\lambda t} \gamma M_{t}(\varphi) =  \gamma(\varphi). 
$$
It comes from these two equalities and \eqref{eq:stationary_measure_representable_function} that
$$
e^{\lambda s}\int_{(x,a) \in \mathcal{X}}  \varphi(x,a) N(x,a) dx da = \int_{(x,a) \in \mathcal{X}}  M_s(\varphi)(x,a) N(x,a) dx da.
$$
Now, we take the derivative of the two terms of this equality in $s=0$, using Theorem~\ref{te:PDE} for the second term. This yields 
\begin{equation}\label{eq:fixed_point_equation}
\begin{aligned}
\lambda\int_{(x,a) \in \mathcal{X}}  \varphi(x,a) N(x,a) dx d&a = \int_{(x,a) \in \mathcal{X}} \frac{\partial \varphi}{\partial a}(x,a)N(x,a) dx da \\
&+ \int_{(x,a)\in\mathcal{X}} \left[\mathdutchcal{K}\left(\varphi(.,0)\right)(x)-\varphi(x,a)\right] b(x,a)N(x,a) dx da.
\end{aligned}
\end{equation}
Then, as we have $\varphi(.,0) \equiv 0$ when $\varphi\in\mathcal{C}_c^{\infty}(\text{int}(\mathcal{X}))$, we have for all $\varphi\in\mathcal{C}_c^{\infty}(\text{int}(\mathcal{X}))$ that 
$$
\lambda\int_{(x,a) \in \text{int}(\mathcal{X})}  \varphi(x,a) N(x,a) dx da = \int_{(x,a) \in \text{int}(\mathcal{X})} \left[\frac{\partial \varphi}{\partial a}(x,a) -\varphi(x,a)b(x,a)\right]N(x,a) dx da,
$$
so that $N$ is the solution of the following equation (in the weak sense)
\hbox{$\frac{\partial}{\partial a}N  + (\lambda + b)N = 0$}. From this equation, we have that the weak derivative in the variable $a$ of the function~$\tilde{N}(x,a) := \exp\left(\lambda a + \int_0^a b(x,s) ds\right)N(x,a)$ is zero. Then, there exists $N_0\in L^1_{\text{loc}}(\mathbb{R}_+^{2k})$ such that $N(x,a) = N_0(x)\exp\left(-\lambda a - \int_0^a b(x,s) ds\right)$  for all $(x,a)\in \mathcal{X}$. As $N$ is non-negative, we have that $N_0$ is non-negative. By \eqref{eq:birth_rate_assumption}, we also have
$$
\int_{x\in\mathbb{R}_+^{2k}} N_0(x)dx \times \int_{a\in\mathbb{R}_+}\exp\left(-\lambda a - \tilde{b}\int_0^a (1+s^{d_b}) ds\right)da\leq \int_{(x,a) \in \mathcal{X}} N(x,a) dx da < +\infty.
$$
Then, it holds $N_0\in L^1(\mathbb{R}_+^{2k})$, which implies that~\eqref{eq:separation_age_space} is proved. % holds, and we have obtained that the stationary profile is the product of two functions, one linked to the variable $x$, the other linked to the variable $a$.

It remains to prove that $N_0$ is a solution of~\eqref{eq:equation_N_0}. Let $f\in M_b\left(\mathbb{R}_+^{2k}\right)$. By applying~\eqref{eq:fixed_point_equation} for $\varphi(x,a) = f(x)$, in view of the fact that $\frac{\partial \varphi}{\partial a} = 0$, and then putting the last term in the left-hand side, we have that
$$
\int_{(x,a) \in \mathcal{X}}  \left(\lambda +b(x,a)\right)f(x) N(x,a) dx da = \int_{(x,a)\in\mathcal{X}} \mathdutchcal{K}\left(f\right)(x) b(x,a)N(x,a) dx da.
$$
Thus, by plugging~\eqref{eq:separation_age_space} in the above, and then integrating in $a$, we obtain that
$$
\int_{x\in\mathbb{R}_+^{2k}}f(x)N_0(x) dx = \int_{(x,a)\in\mathbb{R}_+^{2k}\times\mathbb{R}_+} \hspace{-0.0033mm}\mathdutchcal{K}(f)(x) b(x,a) \exp\left(-\int_0^a b(x,s)ds - \lambda a\right) N_0(x)dxda.
$$
The above implies that~$N_0$ is a solution of~\eqref{eq:equation_N_0}, in view of the definition of $\mathdutchcal{Y}_0$ given in Theorem~\ref{te:main_result}, and of the definition of the Laplace transform given in Notation~\ref{nota:laplace_transform}.

\section{Criteria for Lyapunov functions and two illustrative models}\label{sect:assumptions_verified}
Now that we have proved Theorem \ref{te:main_result}, we give conditions and models for which Assumptions~\hyperlink{paragraph:long_time_behaviour}{$(S_1)-(S_{3})$} are verified. First, we present in Section~\ref{subsect:S2.2} conditions implying Assumptions~\hyperlink{paragraph:long_time_behaviour_S2.2}{$(S_{2.2})$} and~\hyperlink{paragraph:long_time_behaviour_S3}{$(S_{3})$}. Then, we present in Sections~\ref{subsect:model_renewal_several_generations} and~\ref{subsect:model_telomere_shortening} two models where all the assumptions are verified. The first one is a model for which all coordinates are lengthened at each division. The second one is a model for which for each telomere, its probability to be lengthened is independent of the other telomeres. For each of these models, we devote a large part of their study to check~\hyperlink{paragraph:long_time_behaviour_S2.1}{$(S_{2.1})$}, which is the most difficult assumption to verify.

\subsection{The Lyapunov function}\label{subsect:S2.2}
We first present in Section~\ref{subsubsect:general_criteria_lyapunov} the key statement of this subsection that allows us to check~\hyperlink{paragraph:long_time_behaviour_S2.2}{$(S_{2.2})$}. We also discuss about the relation between $\varepsilon_1$ and $\varepsilon_0$ that must occur to check~\hyperlink{paragraph:long_time_behaviour_S3}{$(S_{3})$} from this statement. Then, in Section~\ref{subsubsect:practical_criteria_lyapunov}, we give two criteria useful in practice to prove that~\hyperlink{paragraph:long_time_behaviour_S2.2}{$(S_{2.2})$} and~\hyperlink{paragraph:long_time_behaviour_S3}{$(S_{3})$} hold. Finally, in Section~\ref{subsubsect:proof_prop_S2.1_verified}, we prove the key statement.

\subsubsection{Statement implying~\texorpdfstring{\protect\hyperlink{paragraph:long_time_behaviour_S2.2}{$(S_{2.2})$}}{S2.2} and its link with~\texorpdfstring{\protect\hyperlink{paragraph:long_time_behaviour_S3}{$(S_{3})$}}{S3}}\label{subsubsect:general_criteria_lyapunov}
The function $\mathdutchcal{V}$ that we need to find to verify \hyperlink{paragraph:long_time_behaviour_S2.2}{$(S_{2.2})$} corresponds to a Lyapunov function. Qualitatively, this is a function that measures the impact of the initial condition of the semigroup on the speed of convergence towards the stationary profile. We thus need to think about how the speed of convergence varies when telomere lengths vary to obtain our Lyapunov function. In our case, the speed of convergence increases exponentially when telomere lengths increase. Hence, the Lyapunov function must have an exponential form. Until the next proposition is stated, we assume that $\mathdutchcal{V}$ has an exponential form.

By the first statement of \hyperlink{paragraph:long_time_behaviour_S2.2}{$(S_{2.2})$}, our aim is to bound from above $\frac{\mathdutchcal{K}(\mathdutchcal{V})(x)}{\mathdutchcal{V}(x)}$ when $\max_{i\in\llbracket1,2k\rrbracket}(x_i)$ is "large". To do so, under~\hyperlink{paragraph:long_time_behaviour_S1.1}{$(S_{1.1})$}  and~\hyperlink{paragraph:long_time_behaviour_S2.1}{$(S_{2.1})$}, we first introduce for all~$x\in\mathbb{R}^{2k}$, $L\in\mathbb{N}^*$ the following set
$$
\mathcal{E}_{L}(x) := \{i\in\llbracket1,2k\rrbracket\,|\,x_i>B_{\max}L-2\Delta - \delta\}.
$$
This set is used to represent the coordinates of $x$ which are considered to be "large". We also consider the following function, for all $\lambda \geq 0$, $L \in\mathbb{N}^*$ 
\begin{equation}\label{eq:definition_condition_supremum}
\Lambda(\lambda, L) := \sup_{(r,s)\in(\mathbb{R}^{2k})^2}\left[\sum_{(J,M)\in\mathcal{J}_k}p_{J,M}(r,s)\left(\prod_{i\in J\cap \mathcal{E}_{L}(r)}\mathcal{L}(h(r_i,.))(-\lambda)\right) \right],
\end{equation}
where for all $y >0$, $\mathcal{L}(h(y,.))$ denotes the Laplace transform of $h(y,.)$. As $\mathdutchcal{V}$ has an exponential form, the Laplace transform term gives information about the integral of the marginals of $\mathdutchcal{V}$ with respect to the measures $\left(h(x_i,y)dy\right)_{i\in \mathcal{E}_L(x)}$. Then, as $h$ is the probability density function for lengthening values, the function $\Lambda$ allows us to bound from above how the lengthening impacts the value of $\frac{\mathdutchcal{K}(\mathdutchcal{V})(x)}{\mathdutchcal{V}(x)}$.

To bound from above how the shortening impacts the value of $\frac{\mathdutchcal{K}(\mathdutchcal{V})(x)}{\mathdutchcal{V}(x)}$, we use the function $1 + \mathcal{L}(g)$, where $\mathcal{L}(g)$ is the Laplace transform of the shortening density $g$. To be more precise, it is possible to bound from above $\frac{\mathdutchcal{K}(\mathdutchcal{V})(x)}{\mathdutchcal{V}(x)}$ when $\max_{i\in\llbracket1,2k\rrbracket}(x_i)$ is large by the product between~$\Lambda$ and~$1+\mathcal{L}(g)$. This gives us the following key statement, that allows us to verify~\hyperlink{paragraph:long_time_behaviour_S2.2}{$(S_{2.2})$}, and that is proved in Section~\ref{subsubsect:proof_prop_S2.1_verified}.

\begin{prop}[Existence of a Lyapunov function]\label{prop:S2.2_verified}
Let us assume that \hyperlink{paragraph:long_time_behaviour_S1.1}{$(S_{1.1})$}, \hyperlink{paragraph:long_time_behaviour_S1.3}{$(S_{1.3})$} and~\hyperlink{paragraph:long_time_behaviour_S2.1}{$(S_{2.1})$} hold with $B_{\max} > \Delta + \delta$. Then for all $\lambda_0 > 0$, $L\in\mathbb{N}^*$, Assumption~\hyperlink{paragraph:long_time_behaviour_S2.2}{$(S_{2.2})$} is verified with $L_{\text{ret}} = L$ and the following function $\mathdutchcal{V}$ and constant~$\varepsilon_1$:
\begin{equation}\label{eq:form_function_lyapunov_and_epsilon}
\begin{aligned}
\forall x\in\mathbb{R}_+^{2k}: \hspace{2mm}\mathdutchcal{V}(x) &= \exp\left[\lambda_0.\sum_{i = 1}^{2k}\max\left(x_i-B_{\max}L+\Delta+\delta,0\right)\right],\\
\varepsilon_1 &= \left(1+\mathcal{L}(g)(\lambda_0)\right)\Lambda(\lambda_0, L) - 1.
\end{aligned}
\end{equation}
\end{prop}
\noindent However, the main difficulty is not to verify \hyperlink{paragraph:long_time_behaviour_S2.2}{$(S_{2.2})$}, but to verify both \hyperlink{paragraph:long_time_behaviour_S2.2}{$(S_{2.2})$} and \hyperlink{paragraph:long_time_behaviour_S3}{$(S_{3})$}. Indeed, the values of the constants $\alpha$ and $\beta$ in \hyperlink{paragraph:long_time_behaviour_S3}{$(S_{3})$} depend on the values of $\varepsilon_0$ and $\varepsilon_1$ and we must have $\beta < \alpha$. Thus, we need to have a suitable relation between $\varepsilon_1$ and $\varepsilon_0$ to verify these assumptions. When the birth rate does not depend on $x$, this relation is~$1+\varepsilon_1 < (1+\varepsilon_0)^{\frac{1}{G}}$. We illustrate this by proving the following statement.

\begin{cor}\label{cor:birth_rate_age}
Let us assume that the assumptions of Proposition~\ref{prop:S2.2_verified} hold. Assume also that the birth rate $b$ is independent of $x$, and that there exists $\lambda_0 > 0$, $L\in\mathbb{N}^*$ such that
\begin{equation}\label{eq:condition_Lyapunov_indep_age}
\left(1+\mathcal{L}(g)(\lambda_0)\right)\Lambda(\lambda_0, L) < (1+\varepsilon_0)^{\frac{1}{G}}.
\end{equation}
Then, Assumptions~\hyperlink{paragraph:long_time_behaviour_S2.2}{$(S_{2.2})$} and~\hyperlink{paragraph:long_time_behaviour_S3}{$(S_{3})$} are verified.
\end{cor}
\begin{proof}
One can first easily apply Proposition~\ref{prop:S2.2_verified} to verify~\hyperlink{paragraph:long_time_behaviour_S2.2}{$(S_{2.2})$} with the constant \hbox{$\varepsilon_1 = \left(1+\mathcal{L}(g)(\lambda_0)\right)\Lambda(\lambda_0, L) - 1$}, and then verify \hyperlink{paragraph:long_time_behaviour_S3.1}{$(S_{3.1})$} with $\underline{b} \equiv b$ and $\overline{b} \equiv b$. We thus focus on verifying Assumptions~\hyperlink{paragraph:long_time_behaviour_S3.2}{$(S_{3.2})$} and~\hyperlink{paragraph:long_time_behaviour_S3.3}{$(S_{3.3})$}. We define for all $x \geq 0$ the function $f(x) = \int_0^{+\infty}b(s)\exp\left(-\int_0^s b(u) du\right)e^{-xs} ds$. We observe by~\eqref{eq:birth_rate_assumption} that for all $x > 0$
$$
f(x) \leq \tilde{b}\int_0^{+\infty}(1+s^{d_b}) e^{-xs} ds.
$$
Thus, it holds $\lim_{x \rightarrow +\infty} f(x) = 0$. We also have by classical computations that $f(0) = 1$ and that $f$ is a continuous decreasing function. Then, by the intermediate values theorem and the fact that $1+\varepsilon_1 < \left(1+\varepsilon_0\right)^{\frac{1}{G}}$, there exist $ 0< \beta < \alpha$ such that $f(\alpha) = \frac{1}{(1+\varepsilon_0)^{\frac{1}{G}}}$ and~$f(\beta) = \frac{1}{1+\varepsilon_1}$.
These correspond to the equalities we need to verify~\hyperlink{paragraph:long_time_behaviour_S3.2}{$(S_{3.2})$} and \hyperlink{paragraph:long_time_behaviour_S3.3}{$(S_{3.3})$}. Thus, it only remains to prove the second condition of~\hyperlink{paragraph:long_time_behaviour_S3.2}{$(S_{3.2})$} and the corollary will be proved. 

We notice that for all $a,\,s \geq 0$, we have that 
$\mathdutchcal{F}_a(s) = b(a+s)\exp\left(-\int_a^{a+s} b(u) du\right)$ and \hbox{$\overline{\mathdutchcal{F}}_a(s) = \exp\left(-\int_a^{a+s} b(u) du\right)$}. Let us bound from below $\int_0^{t}  \mathdutchcal{F}_a(s)e^{-\alpha s} ds$ for all pair $(a,t)\in\mathbb{R}_+\times[a_0,+\infty)$. First, after an integration by parts, use the fact that $\overline{\mathdutchcal{F}}_a(s) \leq 1$ for $a,\,s\geq0$, and the fact that $\overline{\mathdutchcal{F}}_a(s) \leq \exp(-b_0(s-a_0))$ when $s\geq a_0$ by \eqref{eq:birth_rate_assumption}. Then, integrate in $s$. It comes for all $t\geq a_0$, $a\geq0$
$$
\int_0^{t}  \mathdutchcal{F}_a(s)e^{-\alpha s} ds = \left[-\overline{\mathdutchcal{F}}_a(s)e^{-\alpha s}\right]_0^{t} - \alpha\int_0^t e^{-\alpha s}\overline{\mathdutchcal{F}}_a(s) ds  \geq  e^{-\alpha a_0}- e^{-\alpha t} - \frac{\alpha}{\alpha + b_0}e^{-\alpha a_0}.
$$
We easily deduce that the second condition of \hyperlink{paragraph:long_time_behaviour_S3.2}{$(S_{3.2})$} is verified from the above equation by taking~$t$ sufficiently large.
\end{proof}
\noindent When the birth rate depends on $x$, a stronger inequality needs to be verified to check~\hyperlink{paragraph:long_time_behaviour_S3}{$(S_{3})$}. This inequality depends on the gap between $\underline{b}$ and $\overline{b}$. For example, when the birth rate is bounded from above and from below by a positive constant, the inequality is the following. 
\begin{cor}\label{cor:birth_rate_bounded_above_below}
Let us assume that the assumptions of Proposition~\ref{prop:S2.2_verified} hold. Assume also that there exists $b_1,\,b_2 > 0$ such that for all $(x,a) \in \mathcal{X}$ it holds $b_1 \leq b(x,a) \leq b_2$, and $\lambda_0 > 0$, $L\in\mathbb{N}^*$ such that
\begin{equation}\label{eq:condition_Lyapunov_rates}
\left(1+\mathcal{L}(g)(\lambda_0)\right)\Lambda(\lambda_0, L) < \frac{b_1}{b_2}\left((1+\varepsilon_0)^{\frac{1}{G}} - 1\right) + 1.
\end{equation}
Then, Assumptions~\hyperlink{paragraph:long_time_behaviour_S2.2}{$(S_{2.2})$} and~\hyperlink{paragraph:long_time_behaviour_S3}{$(S_{3})$} are verified.
\end{cor}
\begin{rem}
This result shows that the more the birth rate varies with respect to telomere lengths, the more it is difficult to ensure the existence of a stationary profile.
\end{rem}
\begin{proof}
One can first easily apply Proposition~\ref{prop:S2.2_verified} to verify~\hyperlink{paragraph:long_time_behaviour_S2.2}{$(S_{2.2})$} with the constant \hbox{$\varepsilon_1 = \left(1+\mathcal{L}(g)(\lambda_0)\right)\Lambda(\lambda_0, L) - 1$}, and then verify \hyperlink{paragraph:long_time_behaviour_S3.1}{$(S_{3.1})$} with $\underline{b} \equiv b_1$ and $\overline{b} \equiv b_2$. We thus focus on verifying~\hyperlink{paragraph:long_time_behaviour_S3.2}{$(S_{3.2})$}  and~\hyperlink{paragraph:long_time_behaviour_S3.3}{$(S_{3.3})$}. First, we have by easy computations that for all~$y \in\mathbb{R}_+$
$$
\int_0^{+\infty} e^{-ys}\mathdutchcal{F}_0(s) ds = b_1\int_0^{+\infty} e^{-(y+b_1)s} ds = \frac{b_1}{y + b_1}.
$$
From the above, the equality on the left in~\hyperlink{paragraph:long_time_behaviour_S3.2}{$(S_{3.2})$} is verified for $\alpha = b_1\big((1+\varepsilon_0)^{\frac{1}{G}} -1\big)$. The inequality on the right is then trivially verified because with similar computations as the above, we have for all $t\geq
0$, $a\geq 0$: 
$$
\int_0^{t} e^{-\alpha s}\mathdutchcal{F}_a(s) ds = \frac{b_1}{\alpha + b_1}\left(1 - e^{-\left(\alpha+b_1\right)t}\right).
$$

To verify~\hyperlink{paragraph:long_time_behaviour_S3.3}{$(S_{3.3})$}, we first observe that for all $y \geq 0$: 
$$
\int_0^{+\infty} e^{-ys}\mathdutchcal{J}_0(s) ds = \frac{b_2}{y + b_2}.
$$
Then, for $\beta = b_2\varepsilon_1$, the equality stated in~\hyperlink{paragraph:long_time_behaviour_S3.3}{$(S_{3.3})$} holds. The fact that $\beta < \alpha$ easily comes from~\eqref{eq:condition_Lyapunov_rates}.
\end{proof}
In practice, verifying formally inequalities such as those presented in~\eqref{eq:condition_Lyapunov_indep_age} or~\eqref{eq:condition_Lyapunov_rates} is quite difficult (but it can be easily done numerically). In addition, this requires conditions quite restrictive on the birth rate, see assumptions of Corollaries~\ref{cor:birth_rate_age} and~\ref{cor:birth_rate_bounded_above_below}. To solve these issues, we now present two conditions on the lengthening densities $(h(x,.))_{x\in\mathbb{R}}$ or the probability mass functions~$\left(p_{J,M}(r,s)\right)_{(J,M)\in\mathcal{J}_k,(r,s)\in\mathbb{R}^2}$, that allow to check~\hyperlink{paragraph:long_time_behaviour_S2.2}{$(S_{2.2})$} and~\hyperlink{paragraph:long_time_behaviour_S3}{$(S_{3})$} in practice.

\subsubsection{Practical criteria to verify~\texorpdfstring{\protect\hyperlink{paragraph:long_time_behaviour_S2.2}{$(S_{2.2})$}}{S2.2} and~\texorpdfstring{\protect\hyperlink{paragraph:long_time_behaviour_S3}{$(S_{3})$}}{S3}}\label{subsubsect:practical_criteria_lyapunov}
In view of Proposition~\ref{prop:S2.2_verified}, we need to control the value of $(1+\mathcal{L}(g))\Lambda$. Indeed, if the latter can be taken as small as we want, then by the intermediate values theorem it will always be possible to find~$\varepsilon_1 >0$ and $\alpha > \beta >0$ such that~\hyperlink{paragraph:long_time_behaviour_S2.2}{$(S_{2.2})$} and~\hyperlink{paragraph:long_time_behaviour_S3}{$(S_{3})$} are verified.  In fact, $\mathcal{L}(g)$ can be taken as small as we want. Thus, if for all $\lambda >0$ we can control the value of $\Lambda(\lambda,.)$, then we will be able to verify the assumptions. We illustrate this by proving the following statement.

\begin{prop}[Criterion to check~\hyperlink{paragraph:long_time_behaviour_S2.2}{$(S_{2.2})$} and~\hyperlink{paragraph:long_time_behaviour_S3}{$(S_{3})$}]\label{prop:limsup_Lambda}
Let us assume that \hyperlink{paragraph:long_time_behaviour_S1.1}{$(S_{1.1})$}, \hyperlink{paragraph:long_time_behaviour_S1.3}{$(S_{1.3})$}, \hyperlink{paragraph:long_time_behaviour_S2.1}{$(S_{2.1})$} and~\hyperlink{paragraph:long_time_behaviour_S3.1}{$(S_{3.1})$} hold with $B_{\max} > \Delta + \delta$, and that 
\begin{eqnarray}
&\forall \lambda > 0,\,\varepsilon >0,\,\exists L\in\mathbb{N}^* \textnormal{ s.t. } \Lambda(\lambda,L) \leq 1 + \varepsilon,  \label{eq:limsup_Lambda}\\
&\exists b_1 > 0,\,d_1\in\mathbb{N}\textnormal{ s.t. }\forall (x,a)\in\mathbb{R}_+^{2k}\times[a_1,+\infty):\, \overline{b}(x,a) \leq b_1(1+a^{d_1}), \label{eq:upper_bound_rate_superior} \\
&\exists b_2 > 0,\,a_2\geq 0\textnormal{ s.t. }\forall (x,a)\in\mathbb{R}_+^{2k}\times[a_2,+\infty):\, \underline{b}(x,a) \geq b_2. \label{eq:lower_bound_rate_inferior}
\end{eqnarray}
Then,~\hyperlink{paragraph:long_time_behaviour_S2.2}{$(S_{2.2})$} and~\hyperlink{paragraph:long_time_behaviour_S3}{$(S_{3})$} are verified.
\end{prop}
\begin{proof}
First, proceeding as in the proof of Corollary~\ref{cor:birth_rate_age}, one can easily obtain from the intermediate values theorem,~\eqref{eq:upper_bound_rate_superior} and~\eqref{eq:lower_bound_rate_inferior}, that there exists $\alpha > 0$ such that \hyperlink{paragraph:long_time_behaviour_S3.2}{$(S_{3.2})$} is verified. Then, we focus on proving \hyperlink{paragraph:long_time_behaviour_S2.2}{$(S_{2.2})$}~and~\hyperlink{paragraph:long_time_behaviour_S3.3}{$(S_{3.3})$}. 

Let $\beta \in(0,\alpha)$ and $\varepsilon_1 = \left[\mathcal{L}(\mathdutchcal{J}_0)(\beta)\right]^{-1}-1$. First, by \hyperlink{paragraph:long_time_behaviour_S1.1}{$(S_{1.1})$}, it holds for all $x \in \mathbb{R}_+^*$: $\mathcal{L}(g)(x) \leq \frac{\overline{g}}{x}$. The latter, the fact that $\mathcal{L}(g)(0) = 1$ and $\varepsilon_1 <1$, and the continuity of~$\mathcal{L}(g)$, imply that there exists $\lambda_{0}>0$ such that
\begin{equation}\label{eq:equation_laplace_for_H8}
\mathcal{L}(g)(\lambda_{0})  < (1+\varepsilon_1)^{\frac{1}{2}} - 1. 
\end{equation}
Second, in view of~\eqref{eq:limsup_Lambda}, there exists $L > 0$ such that
$$
\Lambda(\lambda_0,L) < (1+\varepsilon_1)^{\frac{1}{2}}.
$$
Combining~\eqref{eq:equation_laplace_for_H8} and the above equation yields that~~$\left(1+\mathcal{L}(g)(\lambda_0)\right)\Lambda(\lambda_0, L) < 1+\varepsilon_1$. In view of Proposition~\ref{prop:S2.2_verified} and the fact that $\varepsilon_1 = \left[\mathcal{L}(\mathdutchcal{J}_0)(\beta)\right]^{-1}-1$, we then obtain that Assumptions~\hyperlink{paragraph:long_time_behaviour_S2.2}{$(S_{2.2})$} and~\hyperlink{paragraph:long_time_behaviour_S3.3}{$(S_{3.3})$} are verified. 
\end{proof}
\noindent The most difficult assumption of Proposition~\ref{prop:limsup_Lambda} to check is~\eqref{eq:limsup_Lambda}. Let us fix $\lambda > 0$. The quantities that can be large in the expression of $\Lambda(\lambda,.)$ are the Laplace transforms of the functions~$(h(x,.))_{x\in\mathbb{R}}$. Here are two ways to control them.
\begin{itemize}[leftmargin=*]
\item By finding conditions implying that these Laplace transforms are very small for large coordinates.
\item By finding conditions implying that for all $(r,s) \in \left(\mathbb{R}^{2k}\right)^2$ and $(J,M)\in\mathcal{J}_k$, the lengthening probabilities~$p_{J,M}(r,s)$ are small when $r_j$ is large for at least one\hbox{~$j\in J$}. Indeed, according to \eqref{eq:definition_condition_supremum}, the lengthening probabilities will compensate for the Laplace transforms of the functions $(h(x,.))_{x\in\mathbb{R}}$ under this condition. %if they are small when one coordinate is large 
\end{itemize}
Let us now give two statements, each corresponding one of the above points, implying~\eqref{eq:limsup_Lambda}. We begin with the following statement, providing a condition allowing to directly control the values of the Laplace transforms. Its qualitative interpretation is that the longer a telomere is, the smaller its lengthening value becomes.
\begin{prop}[Control of the Laplace transforms]\label{prop:value_lengthening_vanishes}
Let us assume that~\hyperlink{paragraph:long_time_behaviour_S1.1}{$(S_{1.1})$} and~\hyperlink{paragraph:long_time_behaviour_S2.1}{$(S_{2.1})$} hold with $B_{\max} > \Delta + \delta$, and that
\begin{equation}\label{eq:value_lengthening_zero}
\lim_{x\rightarrow +\infty} \Delta_x = 0.
\end{equation}
Then, Eq.~\eqref{eq:limsup_Lambda} holds.
\end{prop}
\begin{proof}
Let $\lambda > 0$. First, as the functions $(h(x,.))_{x\in\mathbb{R}}$ are probability density functions, one has for all $x\in \mathbb{R}$ that 
\begin{equation}\label{eq:upperbound_laplace_functionh}
1\leq \mathcal{L}(h(x,.))(-\lambda) =  \int_0^{\Delta_x} \exp\left(\lambda u\right) h(x,u) du\leq \exp\left(\lambda\Delta_x\right).
\end{equation}
Then, by \eqref{eq:value_lengthening_zero} we have $\lim_{x\rightarrow +\infty} \mathcal{L}(h(x,.))(-\lambda) = 1$, implying that for all $\varepsilon >0$, there exists $C_{\varepsilon} >0$ such that for all~$x > C_{\varepsilon}$
\begin{equation}\label{eq:laplace_transform_h_confined}
1\leq \mathcal{L}(h(x,.))(-\lambda) \leq (1+ \varepsilon)^{\frac{1}{2k}}.
\end{equation}
We now use the latter to prove the proposition. Let $\varepsilon >0$, and $L\in \mathbb{N}^*$ such that
$$
L > \frac{C_{\varepsilon}+\Delta+\delta}{B_{\max}}. 
$$
First, we use~\eqref{eq:laplace_transform_h_confined} to bound the term $\prod_{i\in J\cap \mathcal{E}_{L}(r)}\mathcal{L}(h(r_i,.))(-\lambda)$ by $(1+ \varepsilon)$. Then, we use the fact that for all $(r,s)\in(\mathbb{R}^{2k})^2$, the sum of the probabilities~$(p_{J,M}(r,s))_{(J,M)\in\mathcal{J}_k}$ is $1$. It comes that $\Lambda(\lambda,L) \leq 1 + \varepsilon$, which ends the proof of the proposition.
\end{proof}
\noindent Biologically, it seems that even when telomeres are long, they can be lengthened by a large number of nucleotides, see~\cite[Figures~$4$.D and $5$.C]{teixeira2004}. Hence, the condition given in Proposition~\ref{prop:value_lengthening_vanishes} is not really biologically relevant. However, we think that this condition is interesting mathematically because it helps to understand why a Lyapunov function may exist on this type of model.

Now, we provide a condition allowing to control the values of the lengthening probabilities $p_{J,M}(r,s)$, where $(r,s) \in \left(\mathbb{R}^{2k}\right)^2$ and $(J,M)\in\mathcal{J}_k$, when $r_j$ is large for at least one~$j\in J$. Its qualitative interpretation is that the longer a telomere is, the smaller its probability to be lengthened becomes. This condition, contrary to the one in Proposition~\ref{prop:value_lengthening_vanishes}, is highly relevant from a biological viewpoint, see~\cite[Figures~$4$.E and~$5$.D]{teixeira2004}.
\begin{prop}[Control of the lengthening probabilities]\label{prop:proba_lengthening_vanishes}
Let us assume that \hyperlink{paragraph:long_time_behaviour_S1.1}{$(S_{1.1})$}, \hyperlink{paragraph:long_time_behaviour_S1.3}{$(S_{1.3})$} and~\hyperlink{paragraph:long_time_behaviour_S2.1}{$(S_{2.1})$} hold with $B_{\max} > \Delta + \delta$, and that for all $(J,M)\in\mathcal{J}_{k}$
\begin{equation}\label{eq:probability_lengthening_to_zero}
\lim_{\underset{j \in J}{\max}(x_j) \rightarrow+\infty}\left(\sup_{y\in\mathbb{R}^{2k}}\left(p_{J,M}(x,y)\right)\right) = 0.
\end{equation}
Then, Eq.~\eqref{eq:limsup_Lambda} holds.
\end{prop}
\begin{rem}
In view of \eqref{eq:symmetry_proba}, Proposition~\eqref{prop:proba_lengthening_vanishes} is still true if we replace in Eq.~\eqref{eq:probability_lengthening_to_zero} the terms $\sup_{y\in\mathbb{R}^{2k}}$ and \hbox{$\max_{j \in J}(x_j) \rightarrow+\infty$} with $\sup_{x\in\mathbb{R}^{2k}}$ and $\max_{m\in M} (y_m) \rightarrow+\infty$ respectively.
\end{rem}
\begin{proof}
Let $\lambda > 0$.  By~\eqref{eq:probability_lengthening_to_zero} and the fact that $\#(\mathcal{J}_k) < +\infty$, for all $\varepsilon > 0$ there exists~$C_{\varepsilon} >0$ such that for all $(x,y)\in (\mathbb{R}_+^{2k})^2$ and $(J,M)\in\mathcal{J}_{k}$ verifying $\max_{j\in J}(x_{j}) \geq C_{\varepsilon}$, we have
\begin{equation}\label{eq:assumption_propH8}
p_{J,M}(x,y) < \frac{1}{\sup_{r\in \mathbb{R}}\left(\mathcal{L}(h(r,.))(-\lambda)\right)^{2k}}\frac{\varepsilon}{\#(\mathcal{J}_k)}  =: p_{0}.
\end{equation}
The fact that $\sup_{x\in \mathbb{R}}\left(\mathcal{L}(h(x,.))(-\lambda)\right) < +\infty$ is a consequence of~\eqref{eq:upperbound_laplace_functionh} and the fact that~$\Delta_x \leq \Delta$ for all $x\in \mathbb{R}$, see~\hyperlink{paragraph:long_time_behaviour_S1.1}{$(S_{1.1})$}.

Let us now use~\eqref{eq:assumption_propH8} to get~\eqref{eq:limsup_Lambda}. We fix $\varepsilon > 0$ and $L\in \mathbb{N}^*$ such that 
$$
L > \frac{C_{\varepsilon}+\Delta+\delta}{B_{\max}}.
$$
First, in~\eqref{eq:definition_condition_supremum}, we use \eqref{eq:assumption_propH8} to bound from above $p_{J,M}(r,s)$ when~\hbox{$J \cap \mathcal{E}_{L}(r) \neq \emptyset$} by~$p_0$. Then, for all $(r,s)\in\mathbb{R}^{2k}\times\mathbb{R}^{2k}$, we bound from above the sum of the probabilities $\left(p_{J,M}(r,s)\right)_{(J,M)\in\mathcal{J}_k\text{ s.t. }J\cap\mathcal{E}_L(r) = \emptyset}$ by~$1$. Finally, using the definition of $p_0$ (see~\eqref{eq:assumption_propH8}), we bound from above the other sum by~$\varepsilon$, and sum this bound with the bound obtained for the first sum. It comes that $\Lambda(\lambda,L) \leq 1 + \varepsilon$, which ends the proof of the proposition.
\end{proof}
\noindent We present in Sections~\ref{subsect:model_renewal_several_generations} and \ref{subsect:model_telomere_shortening} two models where the criteria given in this section are used. We now conclude this subsection by proving Proposition~\ref{prop:S2.2_verified}. 

\subsubsection{Proof of Proposition~\ref{prop:S2.2_verified}}\label{subsubsect:proof_prop_S2.1_verified}

Let $\lambda_0 > 0$ and $L \in\mathbb{N}^*$. We first define for all $y\in\mathbb{R}$
$$
v(y) := \exp\left[\lambda_{0}.\max\left(y-B_{\max}L+\Delta+\delta,0\right)\right].
$$
We also define for all $x\in\mathbb{R}_+^{2k}$, $(I,J,M)\in\mathcal{I}_k\times\mathcal{J}_k$
$$
\begin{aligned}
A(x,I,J,M) &:= \int_{(\alpha_1,\alpha_2)\in ([0,\delta]^{2k})^2} p_{J,M}\left(x - \alpha_1,x - \alpha_2\right)\left[\int_{(\beta_1,\beta_2)\in\left([0,\Delta]^{2k}\right)^2} \mathdutchcal{V}(x-\alpha_1+\beta_1)\right.\\
&\times\left.1_{\{x-\alpha_1+\beta_1 \in \mathbb{R}_+^{2k}\}}d\mu_{(x-\alpha_1,x-\alpha_2)}^{(E;J,M)}(\beta_1,\beta_2)\right] d\mu^{(S;I)}(\alpha_1,\alpha_2).
\end{aligned}
$$
This represents the kernel for telomere lengths updating of the branching process when the sets that index which telomeres are shortened and lengthened in the daughter cell~$A$ are respectively $I$ and $J$, and in the daughter cell $B$ are respectively $I^c$ and $M$. In addition, we notice that for all $x\in\mathbb{R}_+^{2k}$ we have 
\begin{equation}\label{eq:notice_proof_S2.1_general}
\mathdutchcal{V}(x) = \prod_{i\in\llbracket1,2k\rrbracket}v(x_i), \hspace{3.5mm} \text{and} \hspace{3.5mm}\mathdutchcal{K}(\mathdutchcal{V})(x) = \frac{2}{\#(\mathcal{I}_k)}\sum_{(I,J,M)\in\mathcal{I}_k\times\mathcal{J}_{k}} A(x,I,J,M). 
\end{equation}
In this proof, we extend the definition of $\mathdutchcal{V}$ to the set $\left(\mathbb{R}_+^{2k}\right)^c$ by writing for all $y\in \left(\mathbb{R}_+^{2k}\right)^c$ the following: $\mathdutchcal{V}(y) = \underset{i\in\llbracket1,2k\rrbracket}{\prod}v(y_i)$ .

The proof of Proposition~\ref{prop:S2.2_verified} is done in three steps. In Step~\hyperlink{paragraph:proof_lyapunov_step1}{$1$}, we prove the following auxiliary inequality. For all $x\in \mathbb{R}_+^{2k}$ and $(\alpha_1,\alpha_2)\in\left([0,\delta]^{2k}\right)^2$ 
\begin{equation}\label{eq:step1_proof_lyapunov}
\sum_{(J,M)\in\mathcal{J}_k}\hspace{-0.14mm}p_{J,M}(x,y)\int_{(\beta_1,\beta_2)\in([0,\Delta]^{2k})^2} \mathdutchcal{V}(x-\alpha_1+\beta_1) d\mu_{(x-\alpha_1,x-\alpha_2)}^{(E;J,M)}(\beta_1,\beta_2)\leq \mathdutchcal{V}(x-\alpha_1)\Lambda(\lambda_0,L).
\end{equation}
Then, we prove in Step~\hyperlink{paragraph:proof_lyapunov_step2}{$2$} a second auxiliary inequality. For all $x\in \mathbb{R}_+^{2k}$, $I \in \mathcal{I}_k$
\begin{equation}\label{eq:lyapunov_function_step2}
\begin{aligned}
&\int_{(\alpha_1,\alpha_2)\in\left([0,\delta]^{2k}\right)^2} \mathdutchcal{V}(x-\alpha_1) d\mu^{(S;I)}(\alpha_1,\alpha_2) \\
&\leq  \left(1_{\{\forall i\in I:\, x_{i} \leq B_{\max}L-\Delta\}} + \mathcal{L}(g)(\lambda_0)1_{\{\exists i_0\in I \text{ s.t. } x_{i_0} > B_{\max}L-\Delta\}}\right)\mathdutchcal{V}(x).
\end{aligned}
\end{equation}
Finally, in Step \hyperlink{paragraph:proof_lyapunov_step3}{$3$} we verify \hyperlink{paragraph:long_time_behaviour_S2.2}{$(S_{2.2})$} using the previous inequalities.

\paragraph{Step $1$:}\hypertarget{paragraph:proof_lyapunov_step1}{} Let $x\in \mathbb{R}_+^{2k}$, $(\alpha_1,\alpha_2)\in\left([0,\delta]^{2k}\right)^2$, $(\beta_1,\beta_2)\in[0,\Delta]^{2k}$, and $(J,M)\in\mathcal{J}_{k}$. For all~\hbox{$i\in\left(\mathcal{E}_{L}(x)\right)^c$}, we have by definition of $\mathcal{E}_{L}(x)$ that $x_i - (\alpha_1)_i + (\beta_1)_i  \leq B_{\max}L -\Delta- \delta$. This implies 
\begin{equation}\label{eq:step1_lyapunov_intermediate_first}
v(x_i - (\alpha_1)_i + (\beta_1)_i) = 1 = v(x_i - (\alpha_1)_i).
\end{equation}
In addition, by \eqref{eq:measure_elongation}, the measure $\mu_{(x-\alpha_1,x-\alpha_2)}^{(E;J,M)}$ is a Dirac measure in $0$ for the coordinates that are not in~$J$, and is the measure $h(x_i-(\alpha_1)_i,u)du$ for the coordinates $i\in J$. By the equality on the left in~\eqref{eq:notice_proof_S2.1_general}, Eq.~\eqref{eq:step1_lyapunov_intermediate_first} and the latter, we thus have
\begin{equation}\label{eq:step1_lyapunov_intermediate_second}
\begin{aligned}
&\int_{(\beta_1,\beta_2)\in\left([0,\Delta]^{2k}\right)^2} \mathdutchcal{V}(x-\alpha_1+\beta_1) d\mu_{(x-\alpha_1,x-\alpha_2)}^{(E;J,M)}(\beta_1,\beta_2) \\
&= \left[\prod_{\substack{i \in \llbracket1,2k\rrbracket\\i\notin\mathcal{E}_{L}(x) \text{ or } i\notin J}}v(x_i - (\alpha_1)_i)\right]\left[\prod_{\substack{i \in \llbracket1,2k\rrbracket\\i\in J\cap \mathcal{E}(x)}}\left(\int_{u \in [0,\Delta]} v(x_i - (\alpha_1)_i + u) h(x_i - (\alpha_1)_i,u) du\right)\right].
\end{aligned}
\end{equation}
We now obtain Eq.~\eqref{eq:step1_proof_lyapunov} from the above equality. First, notice that for all $u\in[0,\Delta]$ we have \hbox{$v(x_i - (\alpha_1)_i + u) \leq v(x_i - (\alpha_1)_i)\exp\left(\lambda_0 u\right)$,} and apply this inequality to the right-hand side term of~\eqref{eq:step1_lyapunov_intermediate_second}. Then, integrate in $u$ to obtain for all $i\in J\cap \mathcal{E}(x)$ the Laplace transform of $h(x_i-(\alpha_1)_i,.)$ in $-\lambda_0$ and group all the $(v(x_i - (\alpha_1)_i))_{i\in \llbracket1,2k\rrbracket}$ into one product to get~$\mathdutchcal{V}(x-\alpha_1)$. Finally, multiply both sides of the inequality by $p_{J,M}(x,y)$, sum with respect to the indices $(J,M)\in\mathcal{J}_k$, and take the supremum for the right-hand side term. This gives~\eqref{eq:step1_proof_lyapunov}.

\paragraph{Step $2$:}\hypertarget{paragraph:proof_lyapunov_step2}{} Let $x\in\mathbb{R}_+^{2k}$ and $I \in \mathcal{I}_k$. Assume first that for all $i\in I$: $x_i \leq B_{\max}L -\Delta$. Notice that the function~$v$ is non-decreasing. Then, by the form of $\mathdutchcal{V}$ given in~\eqref{eq:notice_proof_S2.1_general}, we have for all $\alpha_1 \in \mathbb{R}_+^{2k}$: $\mathdutchcal{V}(x-\alpha_1) \leq \mathdutchcal{V}(x)$. Plugging this in the left-hand side term of \eqref{eq:lyapunov_function_step2} and then using the fact that $\mu^{(S;I)}$ is a probability measure yields that~\eqref{eq:lyapunov_function_step2} is true when  for~all~$i\in I$: $x_i \leq B_{\max}L - \Delta$.

Now, assume that there exists $i_0 \in I$ such that $x_{i_0}> B_{\max}L-\Delta$. Recall that $v$ is non-decreasing. Then, by the form of $\mathdutchcal{V}$ given in~\eqref{eq:notice_proof_S2.1_general} and the fact that $x_{i_0}> B_{\max}L-\Delta$, we have for all $\alpha_1 \in [0,\delta]^{2k}$
\begin{equation}\label{eq:step2_lyapunov_intermediate}
\mathdutchcal{V}(x-\alpha_1) \leq v(x_{i_0} - (\alpha_1)_{i_0})\prod_{\substack{i \in\llbracket1,2k\rrbracket\\ i\neq i_0}} v(x_i) = \exp\left(-\lambda_0 (\alpha_1)_{i_0}\right)\mathdutchcal{V}(x).
\end{equation}
In addition, for all $I \in \mathcal{I}_k$, the measure $\mu^{(S;I)}$ is the measure $g(u) du$ for the coordinates in $I$, and is a probability measure for the other coordinates (see~\eqref{eq:measure_shortening}). Then, we obtain that~\eqref{eq:lyapunov_function_step2} is true by integrating both sides of~\eqref{eq:step2_lyapunov_intermediate} with respect to $\mu^{(S;I)}$.

\paragraph{Step $3$:}\hypertarget{paragraph:proof_lyapunov_step3}{} We now verify \hyperlink{paragraph:long_time_behaviour_S2.2}{$(S_{2.2})$}. The third statement is easy to verify. We begin with the first statement. Let $x\in[0,B_{\max}L-\Delta]^{2k}$. As the maximum lengthening value is $\Delta$, we have that the offspring of a cell that has telomere length $x$ necessarily have telomere length in $[0,B_{\max}L]^{2k}$ after one division. Then it holds~$\mathdutchcal{K}_{\left([0,B_{\max}L]^{2k}\right)^c}(\mathdutchcal{V})(x) = 0$ and the first statement is verified for \hbox{$x\in [0,B_{\max}L-\Delta]^{2k}$}. Now, let $x \in\left([0,B_{\max}L-\Delta]^{2k}\right)^c$ and $i_0 \in \llbracket1,2k\rrbracket$ satisfying $x_{i_0} > B_{\max}L-\Delta$. As it holds \hbox{$\mathdutchcal{K}_{\left([0,B_{\max}L]^{2k}\right)^c}(\mathdutchcal{V})(x) \leq \mathdutchcal{K}(\mathdutchcal{V})(x)$}, we prove the inequality for $\mathdutchcal{K}(\mathdutchcal{V})(x)$ instead of $\mathdutchcal{K}_{\left([0,B_{\max}L]^{2k}\right)^c}(\mathdutchcal{V})(x)$. First, develop~$\mathdutchcal{K}(\mathdutchcal{V})(x)$ as in~\eqref{eq:notice_proof_S2.1_general} and bound from above  for all $(I,J,M)\in\mathcal{I}_k\times\mathcal{J}_k$ the term $A(x,I,J,M)$ by applying~\eqref{eq:step1_proof_lyapunov}. Then, split the new sum in two sums, grouping the sets $(I,J,M)\in\mathcal{I}_k\times\mathcal{J}_k$ such that $i_0 \in I$ in one sum, and the other sets in a second sum. Finally, apply~\eqref{eq:lyapunov_function_step2} to bound from above the terms that are summed in the first sum by $\mathcal{L}(g)(\lambda_0)\mathdutchcal{V}(x)$, and the terms that are summed in the second sum by~$\mathdutchcal{V}(x)$. We mention that by the fact that~$\mathcal{L}(g)(\lambda_0) \leq 1$, we can bound from above all the terms such that $i_0\notin I$  by~$\mathdutchcal{V}(x)$, even when there exists $i_1 \in I$ such that $x_{i_1} > B_{\max}L -\Delta$. It comes
$$
\mathdutchcal{K}(\mathdutchcal{V})(x) \leq \frac{2}{\#(\mathcal{I}_k)}\sum_{\substack{I\in\mathcal{I}_k\\ i_0\in I}} \Lambda(\lambda_0,L)\mathcal{L}(g)(\lambda_0)\mathdutchcal{V}(x) + \frac{2}{\#(\mathcal{I}_k)}\sum_{\substack{I\in\mathcal{I}_k\\ i_0\notin I}}\Lambda(\lambda_0,L)\mathdutchcal{V}(x).
$$
Now, we use Lemma~\ref{lemma:cardinal_set_part} in the above to get that 
$$
\frac{2}{\#(\mathcal{I}_k)}\sum_{\substack{I \in \mathcal{I}_k,i_0 \in I}} = \frac{2}{\#(\mathcal{I}_k)}\sum_{\substack{I \in \mathcal{I}_k,i_0 \notin I}} = 1.
$$
We obtain that the first statement of~\hyperlink{paragraph:long_time_behaviour_S2.2}{$(S_{2.2})$} is verified.

To verify the second statement of~\hyperlink{paragraph:long_time_behaviour_S2.2}{$(S_{2.2})$}, first notice that as $v$ is non-decreasing and as~$\Lambda(\lambda_0,L) \leq 1+\varepsilon_1$, it holds $\mathdutchcal{V}(x-\alpha_1)\Lambda(\lambda_0,L)  \leq \mathdutchcal{V}(x)(1+\varepsilon_1)$. Then, plug this in the right-hand side term of~\eqref{eq:step1_proof_lyapunov}, and integrate both sides of the new inequality with respect to the measure $\mu^{(S;I)}$. As this is a probability measure, the right-hand side term stays equal to~$\mathdutchcal{V}(x)(1+\varepsilon_1)$. It comes that the second statement of~\hyperlink{paragraph:long_time_behaviour_S2.2}{$(S_{2.2})$} is verified.

We finally prove the fourth statement of~\hyperlink{paragraph:long_time_behaviour_S2.2}{$(S_{2.2})$}. We consider $(x,y)\in(\mathbb{R}_+^{2k})^2$ verifying the inequality~$||y-x||_{\infty} \leq \max(\delta,\Delta)$. As $v$ is a non-decreasing function, we have from the last inequality
$$
\begin{aligned}
\mathdutchcal{V}(y) &\leq \exp\left[\lambda_0.\sum_{i = 1}^{2k} \max\left(x_i + \max(\delta,\Delta) - B_{\max}L + \Delta + \delta, 0 \right)\right]\\
&\leq  \mathdutchcal{V}(x)\exp\left(2k\lambda_{0}\max(\delta,\Delta)\right).
\end{aligned} 
$$
Then, the fourth statement of~\hyperlink{paragraph:long_time_behaviour_S2.2}{$(S_{2.2})$} comes from the above inequality by taking the ratio between the left and right-hand side terms.
\qed

\subsection{A model where all telomeres are lengthened}\label{subsect:model_renewal_several_generations}

We now present two illustrative models for telomere shortening for which all the assumptions of Theorem~\ref{te:main_result} are verified. The first model is a model in which at each cell division, every telomere is lengthened. This is a concrete example of model for which this is easier to verify \hyperlink{paragraph:long_time_behaviour_S2.1}{$(S_{2.1})$} with $G > 1$. Let us first consider constants $\Delta >\delta >0$ verifying
\begin{equation}\label{eq:assumption_Delta_delta_first_model}
\left(\frac{\Delta - \delta}{\Delta}\right)^{2k^2+4k} > \frac{1}{4}.
\end{equation}
This condition means that $\delta$ must be sufficiently small in comparison to $\Delta$. We also consider a constant~$a_0 >0$. Then, our model is the branching process introduced in Section~\ref{subsect:algorithm_model} defined in the following~way:
\begin{itemize}[leftmargin=*]
\item For all $(x,a)\in\mathbb{R}_+^{2k}\times\mathbb{R}_+$: $b(x,a) = a1_{\left\{a\geq a_0\right\}}$,
\item For all $y\in\mathbb{R}_+$: $g(y) = \frac{1}{\delta}1_{[0,\delta]}(y)$,
\item For all $(x,y)\in\mathbb{R}\times\mathbb{R}_+$:
$$
h(x,y) = \frac{1}{\Delta}1_{[0,\Delta]}(y)1_{\{x < 0\}} + \frac{x+1}{\Delta}1_{\left[0,\frac{\Delta}{x+1}\right]}(y)1_{\{x \geq 0\}},
$$
\item For all $(s_1,s_2) \in (\mathbb{R}^{2k})^2$, $(J,M)\in\mathcal{J}_k$,
$$
p_{J,M}(s_1,s_2) = \begin{cases}
1 & \text{ if } (J,M) = \left(\{1,\hdots,2k\},\{1,\hdots,2k\}\right), \\
0, & \text{ otherwise.} \\
\end{cases} 
$$
As $p_{J,M}(s_1,s_2)$ does not depend on $s_1$ and $s_2$, we drop $(s_1,s_2)$ in the notation. We also denote $(J_0,M_0) = \left(\{1,\hdots,2k\},\{1,\hdots,2k\}\right)$.
\end{itemize}
The birth rate has been chosen such that cell with an age below a certain threshold cannot divide, and such that the division rate increases with the cell age. This is consistent with the biological reality, see~\cite{tzur_cell_2009}. We easily see that Assumption \hyperlink{paragraph:long_time_behaviour_S1}{$(S_1)$} is verified for this model. We also easily verify \hyperlink{paragraph:long_time_behaviour_S3.1}{$(S_{3.1})$} with~$\underline{b} \equiv b \equiv \overline{b}$. Therefore, by Propositions~\ref{prop:limsup_Lambda} and~\ref{prop:value_lengthening_vanishes}, if \hyperlink{paragraph:long_time_behaviour_S2.1}{$(S_{2.1})$}~is~satisfied, then all the other assumptions are verified. We thus focus on verifying~\hbox{\hyperlink{paragraph:long_time_behaviour_S2.1}{$(S_{2.1})$}}. We first explain why we do not verify \hyperlink{paragraph:long_time_behaviour_S2.1}{$(S_{2.1})$} with $G = 1$ and then we verify it with $G = k + 2$.

\paragraph{Problem for $G = 1$.}  A natural candidate for the set $K_{\text{ren}}$ is a set of the form~$[0,L]^{2k}$, where~$L >0$. The reason is that these sets have a very simple form, and thus computations are easier. However, we cannot verify \hyperlink{paragraph:long_time_behaviour_S2.1}{$(S_{2.1})$} with $K_{\text{ren}}$ of this form and $G = 1$ as all the telomeres of a cell are lengthened at each division.  Let us formalise it. We fix $L > 0$, and then consider $x = \sum_{i = 1}^{2k} L\,e_i$, where~$(e_i)_{i\in\llbracket1,2k\rrbracket}$ are the vectors of the canonical basis. We study a cell with telomere lengths $x$ that divides. By Eq.~\eqref{eq:measure_shortening}, each of its daughter cells has at least $k$ telomeres that stay unchanged after shortening. In addition, as all telomeres are lengthened in this model (because the probability to draw $(J_0,M_0)$ is $1$), the $k$ telomeres that stayed unchanged during the shortening are lengthened. We then necessarily have at least $k$ telomeres with a length strictly larger than~$L$ after lengthening. In particular, daughter cells of the cell that divides necessarily have telomere lengths outside of~$[0,L]^{2k}$. As $\mathdutchcal{K}(.)(x)$ is the transition probability for telomere lengths of the daughter cells, the latter yields that 
$$
\mathdutchcal{K}(1_{[0,L]^{2k}})(x) = 0.
$$
Thus, the inequality presented in \hyperlink{paragraph:long_time_behaviour_S2.1}{$(S_{2.1})$} fails for $x = \sum_{i = 1}^{2k}Le_i$ and it is impossible to have a renewal in one generation for a set of the form $[0, L]^{2k}$, where $L > 0$. We thus have two options to verify~\hyperlink{paragraph:long_time_behaviour_S2.1}{$(S_{2.1})$}: either we should consider a set for the renewal with a more complex form, or we should verify~\hyperlink{paragraph:long_time_behaviour_S2.1}{$(S_{2.1})$} with~$G > 1$. The problem with the first option is that due to the multidimensional trait space and the fact that the expression of the jump kernel is complicated, exhibiting the good set is not trivial. We then choose the second option. Let us begin with some preliminaries.

\paragraph{Preliminaries.} In view of~\eqref{eq:assumption_Delta_delta_first_model}, we first introduce $\gamma \in(0,1)$ such that 
\begin{equation}\label{eq:condition_proportion_delta}
(1 - \gamma)^{2k}\left(\frac{\Delta - \delta}{\Delta}\right)^{2k^2+4k} > \frac{1}{4}.
\end{equation}
We use $\gamma$ in order to write the quantity $\gamma\delta$, which is a minimum shortening value for the cells we are interested in later in the proof. The above condition implies that $\gamma$ must be sufficiently small so that we can control the number of non-senescent offspring of a cell for which telomeres are shortened by at least~$\gamma\delta$. 

We now obtain auxiliary results useful to choose the maximum value $L >0$ of the set $K_{\text{ren}} = [0,L]^{2k}$. Let us consider for all $x \geq 0$: 
$$
f(x) = x - \gamma \delta + \frac{\Delta}{x+1} + \frac{\Delta}{\max\left(x-\delta + 1,1\right)}.
$$
We first easily have that it holds $\underset{x \rightarrow +\infty}{\lim} f(x) - (x- \frac{\gamma}{2}\delta) = -\frac{\gamma}{2}\delta < 0$, so that~$f$~is smaller than the function $\text{Id} -\frac{\gamma}{2}\delta$ after a certain value~$L_0 > 0$. In addition, we have that $f$ is differentiable on $(\delta,+\infty)$ and that~$\underset{x \rightarrow +\infty}{\lim} f'(x)  = 1$, which implies that there exists $L_1 > 0$ such that $f$ strictly increases on $[L_1,+\infty)$. As~$\underset{x \rightarrow +\infty}{\lim} f(x) = +\infty$, we finally have that there exists $L_2 > 0$ such that for all $x \geq L_2$: $f(x) \geq \max_{y\in[0,L_1]} f(y)$. These three results respectively imply that
\begin{alignat}{2}
&\forall x \geq \max(L_0,L_1,L_2): \hspace{7.5mm} &&f(x) \leq  x - \frac{\gamma}{2}\delta, \nonumber \\
&\forall x \in\left[L_1,\max(L_0,L_1,L_2)\right]: \hspace{7.5mm} &&f(x) \leq  f(\max(L_0,L_1,L_2)), \label{eq:condition_maximum_set_renew_first_intermediate}\\
&\forall x \in\left[0,L_1\right]: \hspace{7.5mm} &&f(x) \leq \max_{y\in[0,L_1]} f(y) \leq  f(\max(L_0,L_1,L_2)). \nonumber
\end{alignat}
By combining the first inequality in~\eqref{eq:condition_maximum_set_renew_first_intermediate} with the second one, and then the first with the third one, we obtain that for all $x\in \left[0,\max(L_0,L_1,L_2)\right]$: $f(x) \leq \max(L_0,L_1,L_2) - \frac{\gamma}{2}\delta$. From the latter and the first inequality in~\eqref{eq:condition_maximum_set_renew_first_intermediate}, it comes the following, which is the auxiliary inequality we need:
\begin{equation}\label{eq:condition_maximum_set_renew_first}
\forall L\geq \max(L_0,L_1,L_2),\,\forall x \in [0,L]: \hspace{6mm} f(x) \leq  L - \frac{\gamma}{2}\delta.
\end{equation}

We now introduce the set and constants for which we verify \hyperlink{paragraph:long_time_behaviour_S2.1}{$(S_{2.1})$}. We first fix an arbitrary \hbox{$L > \max\left(\Delta,L_0,L_1,L_2\right)$} verifying
\begin{equation}\label{eq:condition_maximum_set_renew_second}
\frac{k\Delta}{L- \frac{\gamma}{2}\delta+1} \leq \frac{\gamma}{2}\delta,
\end{equation}
We also consider the constants 
$$
\varepsilon_0 = 4(1 - \gamma)^{2k}\left(\frac{\Delta - \delta}{\Delta}\right)^{2k^2+4k} - 1, \hspace{3.5mm} G = k+2, \hspace{3.5mm} B_{\max} = L + G\Delta,
$$
and the set~\hbox{$K_{\text{ren}} = [0,L]^{2k}$}. To verify~\hyperlink{paragraph:long_time_behaviour_S2.1}{$(S_{2.1})$}, we proceed in two steps. We first prove in Step~\hyperlink{paragraph:step1_renewal_model1}{$1$} that for all $x \in \left[0,L-\frac{\gamma}{2}\delta\right]^{2k}$ we have
\begin{equation}\label{eq:step1_renewal_model1}
\left(\mathdutchcal{K}\right)^{k}(1_{K_{\text{ren}}})(x) \geq 2^{k}\left(\frac{\Delta - \delta}{\Delta}\right)^{2k^2}.
\end{equation}
Then, we conclude in Step~\hyperlink{paragraph:step2_renewal_model1}{$2$}. We start with the following remark:
\begin{rem}
As we start from $x\in K_{\text{ren}}= [0,L]^{2k}$, as the maximum lengthening value is~$\Delta$, and as $B_{\max}~=~L + G\Delta$, the offspring of $x$ from the first to the $G-$th generations stays in~$[0,B_{\max}]^{2k}$. Therefore, if the inequality presented in \hyperlink{paragraph:long_time_behaviour_S2.1}{$(S_{2.1})$} is obtained for $\mathdutchcal{K}$, then it will be also obtained for $\mathdutchcal{K}_{[0,B_{\max}]^{2k}}$ (see~\eqref{eq:measure_k_restriction} for the definition of $\mathdutchcal{K}_{[0,B_{\max}]^{2k}}$). Hence, in the rest of the subsection, we omit the term~$[0,B_{\max}]^{2k}$ in $\mathdutchcal{K}_{[0,B_{\max}]^{2k}}$ and rather study $\mathdutchcal{K}$.
\end{rem}
\paragraph{Step 1:}\hypertarget{paragraph:step1_renewal_model1}{} Let  us fix $x\in\left[0,L - \frac{\gamma}{2}\delta\right]^{2k}$. We also consider $y\in\left[0,L-\Delta\left(L-\frac{\gamma}{2}\delta+1\right)^{-1}\right]^{2k}$ and the set 
$$
I_y := \left\{w\in[0,L]^{2k}\,\big|\,w_i \leq \max\left(y_i + \frac{\Delta}{L-\frac{\gamma}{2}\delta+1}, \Delta\right)\right\}.
$$
We first bound from below $\mathdutchcal{K}(1_{I_y})(y)$. To do so, we use the following inequalities, that easily come from the expression of $h$ and the fact that the maximum of $j(x) = x + \frac{\Delta}{x+1}$ on $\left[0,L-\frac{\gamma}{2}\delta\right]$ is in $x = 0$ or~$x = L-\frac{\gamma}{2}\delta$:
\begin{align}
&\forall s\in \left[0,L-\frac{\gamma}{2}\delta\right]: \hspace{4mm}\int_{0}^{\frac{\Delta}{s+ 1}} 1_{\left\{s + u \in \left[0,\max\left(s + \Delta\left(L-\frac{\gamma}{2}\delta+1\right)^{-1},\Delta\right)\right]\right\}} h\left(s,u\right) du = 1 \geq \frac{\Delta - \delta}{\Delta}, \nonumber\\
&\forall s\in [-\delta,0): \hspace{4mm} \frac{1}{\Delta}\int_{0}^{\Delta} 1_{\left\{s + u \in \left[0,\max\left(s+\Delta\left(L-\frac{\gamma}{2}\delta+1\right)^{-1},\Delta\right)\right]\right\}} h\left(s,u\right) du \geq \frac{1}{\Delta}\int_{\delta}^{\Delta}du= \frac{\Delta - \delta}{\Delta}. \label{eq:first_intermediate_points_in_[0,L-1]}
\end{align}
First, develop $\mathdutchcal{K}(1_{I_y})(y)$ and only keep the pair of sets $(J_0,M_0)$ in the sum of the elements over~$\mathcal{J}_k$~(as the other probabilities are $0$). Then, noticing that for all \hbox{$(s_1,s_2)\in\left(\mathbb{R}^{2k}\right)^2$} the marginal of the measure~$\mu_{(s_1,s_2)}^{(E;J_0,M_0)}$ over each coordinate is the Lebesgue measure weighted by the function~$h$~(see~\eqref{eq:measure_elongation}), use Eq.~\eqref{eq:first_intermediate_points_in_[0,L-1]} to bound from below the integral with respect to $d\mu_{\left(x-\alpha_1,x-\alpha_2\right)}^{(E;J_0,M_0)}$ by~$\left(\frac{\Delta - \delta}{\Delta}\right)^{2k}$. Finally, integrate with respect to the probability measure $\mu^{(S;I)}$, and simplify $\frac{1}{\#(\mathcal{I}_k)}$ by~$\sum_{I\in\mathcal{I}_k}$. It comes
\begin{equation}\label{eq:second_intermediate_points_in_[0,L-1]}
\begin{aligned}
\mathdutchcal{K}(1_{I_y})(y) &= \frac{2}{\#\left(\mathcal{I}_k\right)} \sum_{I\in\mathcal{I}_{k}}\int_{(\alpha_1,\alpha_2)\in (\mathbb{R}_+^{2})^{2k}}\left[\int_{(\beta_1,\beta_2)\in (\mathbb{R}_+^{2})^{2k}}1_{\left\{x-\alpha_1+\beta_1\in[0,L]^{2k}\right\}} \right. \\
& \times d\mu_{\left(x-\alpha_1,x-\alpha_2\right)}^{(E;J_0,M_0)}(\beta_1,\beta_2)\bigg] d\mu^{(S;I)}(\alpha_1,\alpha_2) \geq 2\left(\frac{\Delta - \delta}{\Delta}\right)^{2k}.
\end{aligned}
\end{equation}
Now, let us consider $(w^{(1)},\hdots, w^{(k)}) \in (\mathbb{R}_+^{2k})^{k}$ such that $w^{(1)} \in I_x$ and such that for all \hbox{$i \in \llbracket1,k-1\rrbracket$:} $w^{(i+1)} \in I_{w^{(i)}}$. First, iterate the condition in the definition of the sets~$(I_y)_{y\in\mathbb{R}_+^{2k}}$. Then, use Eq.~\eqref{eq:condition_maximum_set_renew_second}, the fact that $x\in[0,L-\frac{\gamma}{2}\delta]^{2k}$, and the fact that~$L > \Delta$. We obtain that for all~\hbox{$i\in \llbracket1,2k\rrbracket$}
$$
\left(w^{(k)}\right)_i \leq \max\left(\left(w^{(k - 1)}\right)_i + \frac{\Delta}{L-\frac{\gamma}{2}\delta+1}, \Delta\right) \leq \hdots \leq \max\left(x + \frac{k\Delta}{L-\frac{\gamma}{2}\delta+1}, \Delta\right) \leq L.
$$
In particular, we have $w^{(k)} \in K_{\text{ren}}$. Using the latter, we get the following inequality:
$$
\left(\mathdutchcal{K}\right)^{k}(1_{K_{\text{ren}}})(x) \geq \int_{w^{(1)}\in I_x} \hdots \int_{w^{(k)}\in I_{w^{(k)}}} \mathdutchcal{K}(x,dw^{(1)}) \hdots \mathdutchcal{K}(w^{(k-1)},dw^{(k)}),
$$
for which we can easily obtain~\eqref{eq:step1_renewal_model1} by iterating~\eqref{eq:second_intermediate_points_in_[0,L-1]}.
\paragraph{Step $2$:}\hypertarget{paragraph:step2_renewal_model1}{} Now, we fix $x \in K_{\text{ren}}$. Our aim is to bound from below $(\mathdutchcal{K})^2\big(1_{[0, L - \frac{\gamma}{2}\delta]^{2k}}\big)(x)$.  Let us fix $(\alpha_1,\alpha'_1)\in \left([0,\delta]^{2k}\right)^{2}$ and $(\beta_1,\beta'_1)\in \left([0,\Delta]^{2k}\right)^{2}$. Assume that for all $i \in \llbracket1,2k\rrbracket$ either~\hbox{$(\alpha_1)_i \geq \gamma\delta$} and $(\alpha'_1)_i = 0$, or $(\alpha_1)_i = 0$ and $(\alpha'_1)_i \geq \gamma\delta$. Assume also that
\begin{itemize}[leftmargin=*]
\item  If $x_i - (\alpha_1)_i  < 0$, then  
$$
\delta \leq (\beta_1)_i \leq \Delta.
$$

\item  If $x_i - (\alpha_1)_i  \geq 0$, then
$$
0\leq (\beta_1)_i \leq \frac{\Delta}{x_i + 1}.
$$

\item If $x_i - (\alpha_1)_i + (\beta)_i -(\alpha'_1)_i < 0$, then 
$$
\delta \leq (\beta'_1)_i \leq \Delta.
$$

\item If $x_i - (\alpha_1)_i + (\beta)_i -(\alpha'_1)_i \geq 0$, then 
$$
0\leq (\beta'_1)_i \leq \frac{\Delta}{x_i - (\alpha_1)_i + (\beta)_i -(\alpha'_1)_i + 1}.
$$
This implies by the conditions on $\alpha_1$, $\alpha'_1$ and $\beta_1$ that 
$$
(\beta'_1)_i \leq \frac{\Delta}{x_i - \delta + 1}
$$
when $x_i \geq \delta$, so that
$$
(\beta'_1)_i \leq \frac{\Delta}{\max\left(x_i-\delta+1,1\right)}
$$ 
whatever the value of $x_i$ (as~$(\beta'_1)_i \leq \Delta$).
\end{itemize}
The above statements correspond to the conditions we have on $\alpha_1$, $\alpha'_1$, $\beta_1$ and $\beta'_1$  in the integrals used to develop $(\mathdutchcal{K})^2\big(1_{[0, L - \frac{\gamma}{2}\delta]^{2k}}\big)(x)$, see~\eqref{eq:second_intermediate_points_in_[0,L-1]} for the development of $\mathdutchcal{K}$. In view of the above, and Eq.~\eqref{eq:condition_maximum_set_renew_first} combined with the fact that $L \geq \max(L_0,L_1,L_2)$,
we have for all~$i\in \llbracket1,2k\rrbracket$ that 
$$
0 \leq x_i - (\alpha_1)_i + (\beta)_i -(\alpha'_1)_i + (\beta'_1)_i  \leq x_i - \gamma\delta + \frac{\Delta}{x_i + 1} + \frac{\Delta}{\max\left(x_i - \delta + 1,1\right)} \leq L - \frac{\gamma}{2}\delta.
$$
Then, it holds $x - \alpha_1 + \beta - \alpha'_1 + \beta'_1 \in [0, L - \frac{\gamma}{2}\delta]^{2k}$ when the above conditions are verified.  Now, we use the latter to bound from below~$(\mathdutchcal{K})^2\big(1_{[0, L - \frac{\gamma}{2}\delta]^{2k}}\big)(x)$. First, we develop the measure~$\mathdutchcal{K}$ two times, using the expression in the first line of~\eqref{eq:second_intermediate_points_in_[0,L-1]}, and only keep the pair of sets $(J_0,M_0)$ in the sum of the elements over~$\mathcal{J}_k$. This gives us a sum of elements over $(I,I')\in(\mathcal{I}_k)^2$. Then, we only keep the indices~$(I,I^c)$ where $I \in \mathcal{I}_k$ in this sum. Finally, we add indicators to restrict the domain of integration, such that the conditions implying that $x - \alpha_1 + \beta - \alpha'_1 + \beta'_1 \in [0, L - \frac{\gamma}{2}\delta]^{2k}$ hold. We get
\begin{equation}\label{eq:first_model_step2_intermediate}
\begin{aligned}
&(\mathdutchcal{K})^2\big(1_{[0, L - \frac{\gamma\delta}{2}]^{2k}}\big)(x) \geq \bigg(\frac{2}{\#(\mathcal{I}_k)}\bigg)^2 \sum_{I\in\mathcal{I}_{k}} \int_{(\alpha_1,\alpha_2)\in ([0,\delta]^{2k})^2}\int_{(\beta_1,\beta_2)\in ([0,\Delta]^{2k})^2}\bigg[\int_{(\alpha'_1,\alpha'_2)\in ([0,\delta]^{2k})^2} \\
&\times\left. \int_{(\beta'_1,\beta'_2)\in ([0,\Delta]^{2k})^2} 1_{\left\{\forall i \in I: (\alpha_1)_i \geq \gamma\delta, \forall j \in I^c: (\alpha'_1)_j \geq \gamma\delta\right\}}1_{\{ \forall i \in \llbracket1,2k\rrbracket \text{ s.t. } x_i - (\alpha_1)_i + (\beta)_i -(\alpha'_1)_i < 0:\, (\beta'_1)_i \geq \delta\}}\right.\\
&\times d\mu^{(S;I^c)}(\alpha'_1,\alpha'_2) d\mu_{\left(x-\alpha_1 + \beta_1 - \alpha'_1,x-\alpha_1 + \beta_1 - \alpha'_2\right)}^{(E;J_0,M_0)}(\beta'_1,\beta'_2) \Big]d\mu^{(S;I)}(\alpha_1,\alpha_2)  d\mu_{\left(x-\alpha_1,x-\alpha_2\right)}^{(E;J_0,M_0)}(\beta_1,\beta_2). 
\end{aligned}
\end{equation}
Now, as the measure $\mu_{(s_1,s_2)}^{(E;J_0,M_0)}$ over each coordinate is the Lebesgue measure weighted by the function~$h$~(see~\eqref{eq:measure_elongation}), use Eq.~\eqref{eq:first_intermediate_points_in_[0,L-1]} to bound from below the integrals with respect to the measures~$\left(\mu_{(s_1,s_2)}^{(E;J_0,M_0)}\right)_{(s_1,s_2)\in\mathbb{R}_+^{2k}}$ in~\eqref{eq:first_model_step2_intermediate} by~$\left(\frac{\Delta - \delta}{\Delta}\right)^{2k}$. Then, integrate with respect to the measures~$\mu^{(S;I)}$ and $\mu^{(S;I^c)}$. Each integral is equal to $\left(1 - \gamma\right)^k$. It comes in view of Lemma~\ref{lemma:cardinal_set_part}
$$
\begin{aligned}
(\mathdutchcal{K})^2\left(1_{[0, L - \frac{\gamma}{2}\delta]^{2k}}\right)(x) &\geq \left(\frac{2}{\#\left(\mathcal{I}_k\right)}\right)^2 \sum_{I\in\mathcal{I}_{k}}\left(\left(1-\gamma\right)^k\left(\frac{\Delta - \delta}{\Delta}\right)^{2k}\right)^2 \\
&= \frac{4}{2^k}\left(1-\gamma\right)^{2k}\left(\frac{\Delta - \delta}{\Delta}\right)^{4k}.
\end{aligned}
$$
The latter implies in view of~\eqref{eq:step1_renewal_model1} that 
$$
\begin{aligned}
(\mathdutchcal{K})^{G}(1_{K_{\text{ren}}})(x) &\geq (\mathdutchcal{K})^2\left(1_{[0, L - \frac{\gamma}{2}\delta]^{2k}}(.)(\mathdutchcal{K})^{k}(1_{K_{\text{ren}}})(.)\right)(x) \\
&\geq 4\left(1 - \gamma\right)^{2k}\left(\frac{\Delta - \delta}{\Delta}\right)^{2k^2 + 4k} = 1+\varepsilon_0.
\end{aligned}
$$
Thus, \hyperlink{paragraph:long_time_behaviour_S2.1}{$(S_{2.1})$} is verified and the model admits a stationary profile.

\subsection{A model with mutually independent lengthening probabilities}\label{subsect:model_telomere_shortening}

We now present a second illustrative model. Contrary to the previous model, the maximum lengthening values $\left(\Delta_x\right)_{x\in\mathbb{R}}$ are all equal in this model, while the lengthening probabilities $\left(p_{J,M}\right)_{(J,M)\in\mathcal{J}_k}$ depend on telomere lengths of the two daughter cells after shortening. Let us consider two constants $\Delta >\delta >0$ satisfying 
\begin{equation}\label{eq:assumption_Delta_delta_second_model}
\left(\frac{\Delta - \delta}{\Delta}\right)^{2k} > \frac{1}{2}.
\end{equation}
We also introduce another constant $a_0$, and a continuous function $q : \mathbb{R} \mapsto (0,1]$ such that~$\underset{x \rightarrow+\infty}{\lim} q(x)~=~0$ and
\begin{equation}\label{eq:indequality_q}
\underset{x\in[-\delta,0)}{\inf} q(x) \geq \left(\frac{\left(\frac{\Delta - \delta}{\Delta}\right)^{2k} + \frac{1}{2}}{2\left(\frac{\Delta - \delta}{\Delta}\right)^{2k}}\right)^{\frac{1}{8k}}.
\end{equation}
Our model in this section is the branching process constructed in Section~\ref{subsect:algorithm_model}, with:
\begin{itemize}[leftmargin=*]
\item For all $(x,a) \in \mathcal{X}$: $b(x,a) = a1_{\left\{a\geq a_0\right\}}$, 
\item For all $y\in\mathbb{R}_+$: $g(y) = \frac{1}{\delta}1_{[0,\delta]}(y)$,
\item For all $(x,y)\in\mathbb{R}\times\mathbb{R}_+$: $h(x,y) = \frac{1}{\Delta}1_{[0,\Delta]}(y)$,
\item For all $(x,y)\in\left(\mathbb{R}^{2k}\right)^2$, $(J,M)\in\mathcal{J}_k$,
$$
p_{J,M}(x,y) = \left(\prod_{j \in J}q(x_j)\right)\left(\prod_{m \in M}q(y_m)\right)\left(\prod_{j \notin J}\left(1-q(x_j)\right)\right)\left(\prod_{m \notin M}\left(1-q(y_m)\right)\right).
$$
\end{itemize}
Qualitatively, this choice for the function $p_{J,M}$ means that at each division, for each telomere of each daughter cell, if its length after shortening is $y\in\mathbb{R}$ (see Remark~\ref{rem:telomere_shortening_below_0}), then the probability that it is lengthened is $q(y)$. Conversely, the probability that it is not lengthened is $1 - q(y)$. In other words, the lengthening of each telomere is independent from the other telomeres. This way to model lengthening is used for example in \cite{benetos_stochastic_2024}.

The condition stated in \eqref{eq:assumption_Delta_delta_second_model} implies a significant difference between the maximum lengthening and shortening values. It has been experimentally deduced that this is the case, see~\cite{teixeira2004}, which is encouraging. For example, in the case of budding yeast cells, the maximum shortening value can be set as $\delta = 10$ and we have $k = 16$ chromosomes. Thus, according to \eqref{eq:assumption_Delta_delta_second_model}, we require $\Delta > 10/\left[1 - \left(\frac{1}{2}\right)^{\frac{1}{32}}\right] \simeq 466.68$ to be sure to observe a convergence towards a stationary profile. This a little too much, but not totally absurd as telomeres can be lengthened by more than $100$ nucleotides at a cell division, see~\cite{teixeira2004}. To compare, in the model presented in Section~\ref{subsect:model_renewal_several_generations}, for the same $k$ and $\delta$ we must have~$\Delta > 10/\left[1 - \left(\frac{1}{4}\right)^{\frac{1}{576}}\right] \simeq 4159.96$, which is much less realistic.

The condition stated in \eqref{eq:indequality_q} means that the probability that small telomeres are lengthened must be sufficiently large. When we have a large number of chromosomes, the condition implies that the probability must be very close to $1$. Again, the condition is not optimal, and can be refined by being more meticulous in our computations, or by choosing other distributions than uniform distributions for~$g$ and~$h$. 

Assumption \hyperlink{paragraph:long_time_behaviour_S1}{$(S_1)$} and \hyperlink{paragraph:long_time_behaviour_S3.1}{$(S_{3.1})$} are trivially verified for this model (we take $\underline{b} \equiv b \equiv \overline{b}$ for \hyperlink{paragraph:long_time_behaviour_S3.1}{$(S_{3.1})$}). Assumption \hyperlink{paragraph:long_time_behaviour_S2.1}{$(S_{2.1})$} is also verified, and the proof is given in the paragraph below. Using finally Propositions~\ref{prop:limsup_Lambda} and~\ref{prop:proba_lengthening_vanishes}, we conclude that all the assumptions of Theorem~\ref{te:main_result} are verified, so that there exists a stationary profile for this model.

\paragraph{Proof that \protect\hyperlink{paragraph:long_time_behaviour_S2.1}{$(S_{2.1})$} is verified.}  We begin with a preliminary result. As $\underset{y \rightarrow+\infty}{\lim} q(y) = 0$, and as by~\eqref{eq:assumption_Delta_delta_second_model} we have
$$
1 > \frac{\left(\frac{\Delta - \delta}{\Delta}\right)^{2k} + \frac{1}{2}}{2\left(\frac{\Delta - \delta}{\Delta}\right)^{2k}}, 
$$
there exists $B_{\max} > 2\Delta $ such that
\begin{equation}\label{eq:inequality_large_model}
\forall y >  B_{\max} - \Delta\,:\,1 - q(y) \geq \left(\frac{\left(\frac{\Delta - \delta}{\Delta}\right)^{2k} + \frac{1}{2}}{2\left(\frac{\Delta - \delta}{\Delta}\right)^{2k}}\right)^{\frac{1}{8k}}.
\end{equation}
Our aim is to prove that \hyperlink{paragraph:long_time_behaviour_S2.1}{$(S_{2.1})$} is verified with  $K_{\text{ren}} = [0,B_{\max}]^{2k}$. To do so,  we first obtain an auxiliary inequality. For all $(x,y) \in \left([-\delta, B_{\max}]^{2k}\right)^2$, we consider $\mathcal{A}(x,y) \subset \mathcal{J}_k$ such that
\begin{equation}\label{eq:condition_lengthening_model}
\begin{aligned}
&\mathcal{A}(x,y) = \Big\{(J,M)\in \mathcal{J}_k\hspace{0.54mm}\big|\hspace{0.54mm} \left\{i \in \llbracket1,2k\rrbracket\hspace{0.47mm}|\hspace{0.47mm} x_i \in[-\delta,0)\right\} \subset J\text{ and }\left\{i \in \llbracket1,2k\rrbracket\hspace{0.47mm}|\hspace{0.47mm} y_i \in[-\delta,0)\right\} \subset M,\\
& \hspace{40mm}\text{ and }\left\{i \in \llbracket1,2k\rrbracket\,|\, x_i \in(B_{\max} - \Delta, B_{\max}]\right\} \cap J = \emptyset, \\
&\hspace{40mm} \text{ and }\left\{i \in \llbracket1,2k\rrbracket\,|\, y_i \in(B_{\max} - \Delta, B_{\max}]\right\} \cap M = \emptyset \Big\}.
\end{aligned}
\end{equation}
When $k = 1$, if we denote $\left(x^*,y^*\right) = \left(\left(-\frac{\delta}{2},\frac{B_{\max}}{2}\right),\left(B_{\max} - \frac{\Delta}{2}, B_{\max} - \frac{\Delta}{2}\right)\right)$, then we have 
$$
\mathcal{A}(x^*,y^*) =  \left\{(\left\{1,2\right\},\emptyset),\left(\{1\},\emptyset\right)\right\}.
$$
In particular, from the definition of $p_{J,M}$, the following holds:
$$
\begin{aligned}
\sum\nolimits_{(J,M) \in \mathcal{A}(x^*,y^*)} p_{J,M}(x^*,y^*) &= q(x^*_1)q(x^*_2)(1-q(y^*_1))(1-q(y^*_2)) \\
&+ q(x^*_1)(1-q(x^*_2))(1-q(y^*_1))(1-q(y^*_2)) \\
&= q(x^*_1)(1-q(y^*_1))(1-q(y^*_2)).
\end{aligned}
$$
In other terms, $\sum_{(J,M) \in \mathcal{A}(x^*,y^*)} p_{J,M}(x^*,y^*)$ can be seen as a product of functions $q$ for the coordinates in $[-\delta,0)$, and functions $1-q$ for the coordinates in $(B_{\max} - \Delta,B_{\max}]$. In fact, this equality can be generalised for all $k \in \mathbb{N}^*$ and $(x,y) \in \left([-\delta, B_{\max}]^{2k}\right)^2$. The reason is that in view of the conditions given in \eqref{eq:condition_lengthening_model}, for all $(J,M)\in\mathcal{A}(x,y)$, the following holds:
\begin{itemize}[leftmargin=*]
\item For all $i \in \llbracket1, 2k\rrbracket$ such that $x_i \in [0,B_{\max} - \Delta]$, if $i\in J$, then $(J\backslash\{i\},M)\in\mathcal{A}(x,y)$. Conversely, if $i\notin J$, then $(J\cup\{i\},M)\in\mathcal{A}(x,y)$.
\item For all $i \in \llbracket1, 2k\rrbracket$ such that $y_i \in [0,B_{\max} - \Delta]$, if $i\in M$, then $(J,M\backslash\{i\})\in\mathcal{A}(x,y)$. Conversely, if $i\notin M$, then $(J,M\cup\{i\})\in\mathcal{A}(x,y)$.
\end{itemize}
Hence, if we sum all the $(J,M)\in\mathcal{A}(x,y)$, then we will always have a term $q$ and $1-q$ that will sum to give~$1$ for the coordinates in $[0,B_{\max} - \Delta]$. The latter, combined with the inequalities given in \eqref{eq:indequality_q} and~\eqref{eq:inequality_large_model} yields for all $(x,y) \in \left([-\delta, B_{\max}]^{2k}\right)^2$ that
\begin{equation}\label{eq:inequality_model_probabilities}
\begin{aligned}
\sum_{(J,M)\in\mathcal{A}(x,y)} p_{J,M}(x,y) &= \left(\prod_{\substack{i \in \llbracket1,2k \rrbracket\\ x_i \in [-\delta,0)}} q(x_i) \right) \left(\prod_{\substack{i\in\llbracket1,2k\rrbracket\\x_i\in(B_{\max} - \Delta, B_{\max}]}}\left(1-q(x_j)\right)\right) \\
&\times\left(\prod_{\substack{i \in \llbracket1,2k \rrbracket\\ y_i \in [-\delta,0)}} q(y_i) \right)\left(\prod_{\substack{i\in\llbracket1,2k\rrbracket\\y_i\in(B_{\max} - \Delta, B_{\max}]}}\left(1-q(y_j)\right)\right) \geq \frac{\left(\frac{\Delta - \delta}{\Delta}\right)^{2k} + \frac{1}{2}}{2\left(\frac{\Delta - \delta}{\Delta}\right)^{2k}}. 
\end{aligned}
\end{equation}
Now, we are able to verify \hyperlink{paragraph:long_time_behaviour_S2.1}{$(S_{2.1})$}. We denote $(x,y)\in\left([-\delta, B_{\max}]^{2k}\right)^2$ the telomere lengths, after shortening but before lengthening, of two daughter cells of a cell that divides. If the two following hold:
\begin{itemize}[leftmargin=*]
\item The coordinates where there is a lengthening are indexed by a pair \hbox{$(J,M)\in\mathcal{A}(x,y)$}, 
\item The coordinates $i\in \llbracket1,2k\rrbracket$ such that $x_i < 0$ or $y_i < 0$, are lengthened by a value greater than~$\delta$,
\end{itemize}
we easily have that the telomere lengths of the daughter cells stay in $[0,B_{\max}]^{2k} = K_{\text{ren}}$ after lengthening. In view of the two points above, in \eqref{eq:renewal_equation_below}, we do the following steps:
\begin{enumerate}[leftmargin=*]
\item We bound from below the sum of the pairs $(J,M)$ over $\mathcal{J}_k$ by a sum of $(J,M)$ over $\mathcal{A}(x-\alpha_1,y-\alpha_1)$, and apply~\eqref{eq:inequality_model_probabilities}.

\item  We restrict the integrals for the measures $h((\beta_1)_i)d(\beta_1)_i$ or $h((\beta_2)_i)d(\beta_2)_i$ (see~\eqref{eq:measure_elongation}) on the interval~$[\delta,\Delta]$ for the coordinates where there is a shortening. The function~$1_{\{x-\alpha_1+\beta_1\in K_{\text{ren}}\}}$ is equal to $1$ on this domain of integration by the two points presented above. 

\item We integrate in~$\beta$.  As there are~$2k$ coordinates in the two daughter cells that are shortened at each division, see~\eqref{eq:set_combination_shortening} and~ \eqref{eq:measure_shortening}, and as~$h$ is the uniform distribution, the integral is bounded from below by $\left(\frac{\Delta - \delta}{\Delta}\right)^{2k}$.

\item  We integrate in~$\alpha$, using the fact that the measures $\left(\mu^{(S;I)}\right)_{I\in\mathcal{I}_k}$ are probability measures, and simplify~$\frac{1}{\#(\mathcal{I}_k)}$ by~$\sum_{I \in \mathcal{I}_k}$. 
\end{enumerate}
It comes
\begin{equation}\label{eq:renewal_equation_below}
\begin{aligned}
\mathdutchcal{K}\left(1_{K_{\text{ren}}}\right)(x) &= \frac{2}{\#\left(\mathcal{I}_k\right)}\sum_{I\in\mathcal{I}_k} \sum_{(J,M)\in \mathcal{J}_{k}}\int_{(\alpha_1,\alpha_2)\in (\mathbb{R}_+^{2k})^2} p_{J,M}\left(x - \alpha_1,x - \alpha_2\right)\\ 
&\left[\int_{(\beta_1,\beta_2)\in (\mathbb{R}_+^{2k})^2}1_{\{x-\alpha_1+\beta_1\in K_{\text{ren}}\}} d\mu_{(x-\alpha_1,x-\alpha_2)}^{(E;J,M)}(\beta_1,\beta_2)\right] d\mu^{(S;I)}(\alpha_1,\alpha_2) \\
&\geq 2 \left(\frac{\Delta - \delta}{\Delta}\right)^{2k}\frac{\left(\frac{\Delta - \delta}{\Delta}\right)^{2k} + \frac{1}{2}}{2\left(\frac{\Delta - \delta}{\Delta}\right)^{2k}} = \left(\frac{\Delta - \delta}{\Delta}\right)^{2k} + \frac{1}{2}.
\end{aligned}
\end{equation}
Then, as $\left(\frac{\Delta - \delta}{\Delta}\right)^{2k} + \frac{1}{2} > 1$ by~\eqref{eq:assumption_Delta_delta_second_model}, we obtain that \hyperlink{paragraph:long_time_behaviour_S2.1}{$(S_{2.1})$} is verified  with $G = 1$ and $\varepsilon_0 = \left(\frac{\Delta - \delta}{\Delta}\right)^{2k} - \frac{1}{2}$.

\section{Discussion}\label{sect:discussion}
% by $\alpha$
This work enabled us to establish the convergence of the semigroup of our model towards a stationary profile, at an exponential rate. Our proof allowed us to obtain a precise lower bound on the principal eigenvalue of the semigroup $\lambda$, and provided qualitative information about the form of the stationary profile, \hbox{see~\eqref{eq:form_stationary_profile}-\eqref{eq:equation_N_0}}. We have also been able to establish very general conditions under which our assumptions, and in particular~\hyperlink{paragraph:long_time_behaviour_S2.2}{$(S_{2.2})$} and~\hyperlink{paragraph:long_time_behaviour_S3}{$(S_{3})$} are verified, see~Section~\ref{subsubsect:general_criteria_lyapunov}. Our results are thus very satisfying. There are however still different ways to improve them.%, that we now present

%Obtaining better insight into the influence of age on the convergence rate towards a stationary profile is thus one of the first directions for improving our~results.
The first concerns the influence of the age on the convergence rate towards a stationary profile (see the term before the exponential in~\eqref{eq:main_result}). Indeed, if we work with~$b(x,a) = a$ for all $a \geq 0$, then we have from~\eqref{eq:defintion_all_distortion_functions} and the fact that~\eqref{eq:birth_rate_assumption} is verified with $d_b = 1$ that
$$
\Psi = \left\{\psi \in  M_b^{loc}\left(\mathcal{X}\right)\,|\,\exists \,d_{\psi}\in\mathbb{N},\,d_{\psi}\geq 1 \text{ s.t. } \forall (x,a)\in\mathcal{X}:\,\psi(x,a) = (1+a^{d_{\psi}})\mathdutchcal{V}(x) \right\}.
$$	
This implies that for all $t \geq 0$, $(x,a) \in \mathcal{X}$, $A \in \mathcal{B}\left(\mathcal{X}\right)$ and $d_\psi \geq 1$, we have by computing~\eqref{eq:main_result} for~$\mu = \delta_{(x,a)}$, $\psi(x,a) = \left(1+a^{d_{\psi}}\right)\mathdutchcal{V}(x)$ and $f = \frac{1_A}{\left|\left|1_A\right|\right|_{\mathdutchcal{B}(\psi)}}$ that
\begin{equation}\label{eq:convergence_rate_discussion}
\left|e^{-\lambda t}M_t\left(1_A\right)(x,a)  - \phi(x,a)\int_{(y,s)\in A} N(y,s)dy ds\right| \leq C\left|\left|1_A\right|\right|_{\mathdutchcal{B}(\psi)}\left(1+a^{d_{\psi}}\right)\mathdutchcal{V}(x)e^{-\omega t}.
\end{equation}
The bound in the right-hand side of~\eqref{eq:convergence_rate_discussion} increases polynomially with the age of the initial condition. This bound does not seem optimal at all. Indeed, if we start from a very large age, then the initial cell will divide almost immediately into two daughter cells of age $0$, since the birth rate explodes when~$a\rightarrow+\infty$. The convergence rate presented in~\eqref{eq:convergence_rate_discussion} should therefore be very close to that of an initial condition starting from two cells of age $0$, which is not the case here. %Obtaining better insight into the influence of age on the convergence rate towards a stationary profile is thus one of the first directions for improving our~results. %As for the influence of the initial telomere lengths, it seems to be qualitative, given the very sharp conditions we managed to derive in order to construct our Lyapunov function.}

The second way to improve our work would be to obtain a law of large numbers result from Theorem~\ref{te:main_result}, similar to the one established in~\cite[Theorem~$1.4$]{tomasevic2022}, stating that for all $\psi\in\Psi$ and~$f\in\mathdutchcal{B}(\psi)$, 
\begin{equation}\label{eq:law_large_numbers}
\lim_{t\rightarrow+\infty}\frac{Z_t(f)}{Z_t(1)} = \int_{(y,s)\in\mathcal{X}}f(y,s) N(y,s)dy ds,
\end{equation}
where the limit holds in probability or almost surely. 
The motivation behind this is that Theorem~\ref{te:main_result} only gives information on the limit of~$e^{-\lambda t} \mathbb{E}\left[Z_t(f)\right]$ when $t\rightarrow+\infty$. However, in practice, $\lambda$ and $\mathbb{E}\left[Z_t(f)\right]$ are never observed, contrary to $Z_t(f)$ and~$Z_t(1)$. The result presented in~\eqref{eq:law_large_numbers} seems, therefore, more suitable for practical questions. We believe it can be derived by proceeding as in the proof of \hbox{\cite[Theorem 1.4]{tomasevic2022}}. However, we note that our framework is different since~\cite{tomasevic2022} studies a growth-fragmentation equation and not a jump process. This difference of setting should not play a major role in the proof, but this remains to be~verified.

Another way to improve this work would be to add complexity when modelling senescence and death.
%to our model, in order to be more relevant from a biological point of view.
Namely, from a biological point of view, dead cells and senescent cells are completely different. Including cell death in our model seems to be one of the first steps to take, especially to have a growth rate~$\lambda$ closer to the reality. Moreover, it has been identified in a recent study that there are two types of senescent cells~\cite{xu_two_2015}. The difference between these two types is that for one of them, cells enter in senescence after a succession of abnormally long cell cycles and normal cell cycles. For the other type, cells enter in senescence without doing any abnormally long cell cycle. Making a distinction between these two mechanisms of senescence would allow
us to improve the biological relevance of our work.

Finally, a new question motivated by biological studies can be raised: Under which conditions the marginals of the stationary profile of our model over each coordinate are, at least approximately, independent and identically distributed~? In other words, recalling the function $N_0~\in~L^1(\Psi)$ in the expression for the stationary profile, see Theorem~\ref{te:main_result}, under which conditions we have the existence of $n_0\in L^1(\mathbb{R}_+)$, such that for all~$x\in\mathbb{R}_+^{2k}$
$$
N_0(x) \approx \prod_{i = 1}^{2k} n_0(x_i).
$$
The reason is that this approximation is made in certain biological/biomathematical studies of this phenomenon~\cite{bourgeron_2015,Martin2021,rat_individual_2023,Xu2013}. Thus, obtaining theoretical guarantees for it will further improve the rigour behind these studies. It should be noted that telomeres on the same chromosome (at coordinates $i$ and $i +k$) are not independent during shortening: both telomeres cannot be shortened simultaneously. Consequently, it seems not trivial to justify that such an approximation is possible, and further numerical/mathematical studies must be done to obtain a good justification. This question is the subject of a future work.
\appendix
\section{Proof of the statements given in Section \ref{sect:auxiliary_process}}\label{appendix:proof_auxiliary}
\hypertarget{appendix:everything}{}
\subsection{Proof of Lemma \ref{lemma:inequality_exponential}}\label{subsect:proof_inequality_exponential}
Let $T>0$, and $(x,a)\in\mathcal{X}$. We suppose that $Y_0 = \delta_{(1,x,a)}$. For all $p\in\mathbb{N}^*$, we consider $h_p\in \mathcal{C}_b^{1,m,m,1}\left(\mathbb{R}_+\times\mathcal{U}\times\mathcal{X}\right)$ such that for all $(s,\mathdutchcal{u},x,a)\in \mathbb{R}_+\times\mathcal{U}\times\mathcal{X}$
$$
\begin{aligned}
h_p(s,\mathdutchcal{u},x,a) &= \left[\mathdutchcal{V}(x)1_{\{x\in[0,p]^{2k}\}} + \mathdutchcal{V}_{\min}1_{\{x\notin[0,p]^{2k}\}}\right]\bigg[e^a1_{\{a\in[0,p]\}}\\ 
&+ \left(e^p + e^p\frac{2\sin\left(\frac{\pi}{2}(a-p)\right)}{\pi}\right)1_{\{a\in(p,p+1)\}} +\left(e^p + \frac{2e^p}{\pi}\right)1_{a\geq p+1}\bigg].
\end{aligned}
$$
It is easy to check that $(h_p)_{p\geq0}$ is an increasing sequence in $\mathcal{C}_b^{1,m,m,1}\left(\mathbb{R}_+\times\mathcal{U}\times\mathcal{X}\right)$ that converges pointwise to $e^a\mathdutchcal{V}$. First, we apply \eqref{eq:SDE_model_intro}, \eqref{eq:uniform_random_variables_assumption}, \eqref{eq:kernel_branching}, and \eqref{eq:upperbound_fullkernel} to obtain
$$
\begin{aligned}
\mathbb{E}\left[ \sum_{\mathdutchcal{u}\in\,V_{T}} h_p(s,\mathdutchcal{u},x^\mathdutchcal{u},a_t^\mathdutchcal{u}) \right] &\leq \psi_e(x,a) + \mathbb{E}\bigg[\int_{s\in[0,T]}\sum_{\mathdutchcal{u}\in\,V_{s}}
\frac{\partial h_p}{\partial a}(s,\mathdutchcal{u},x^\mathdutchcal{u},a^\mathdutchcal{u}_{s})ds\bigg] \\
&+ 2(1+\varepsilon_1)\mathbb{E}\bigg[\int_{s\in[0,T]}\sum_{\mathdutchcal{u}\in\,V_{s-}}
b\left(x^\mathdutchcal{u},a_{s-}^\mathdutchcal{u}\right)\mathdutchcal{V}(x^\mathdutchcal{u})ds\bigg]. \\ 
\end{aligned}
$$
Then, we use the inequality 
$$
\left|\frac{\partial h_p}{\partial a}\right| \leq \psi_e,
$$
Eq.~\eqref{eq:birth_rate_assumption}, and Tonelli's theorem to have
$$
\begin{aligned}
\mathbb{E}\left[ \sum_{\mathdutchcal{u}\in\,V_{T}} h_p(s,\mathdutchcal{u},x^\mathdutchcal{u},a_t^\mathdutchcal{u}) \right] &\leq  \psi_e(x,a) + \int_{s\in[0,T]}\mathbb{E}\left[\sum_{\mathdutchcal{u}\in\,V_{s}}
\psi_e(x^\mathdutchcal{u},a_s^\mathdutchcal{u})\right]ds \\ 
&+2(1+\varepsilon_1)\tilde{b}\int_{s\in[0,T]}\mathbb{E}\left[\sum_{\mathdutchcal{u}\in\,V_{s-}}
\mathdutchcal{V}(x^\mathdutchcal{u})\left(1+(a_{s-}^\mathdutchcal{u})^{d_b}\right)\right] ds. \\
\end{aligned}
$$
Finally, we use the fact that there exists $C_{d_b} >0$ such that $1+a^{d_b} \leq C_{d_b}e^a$, and we let $p$ go to infinity by using the monotone convergence theorem to get
$$
M_T(\psi_e)(x,a) \leq \psi_e(x,a) + (1+2\tilde{b}(1+\varepsilon_1)C_{d_b})\int_{s\in[0,T]} M_s(\psi_e)(x,a) ds.
$$
Thus, by applying Gronwall's lemma, the lemma is proved.
\qed
\subsection{Proof of Lemma \ref{lemma:well_definition_semigroup}}\label{subsect:proof_well_definition_semigroup}

\begin{enumerate}[leftmargin=0.75cm]
\item This statement is immediate thanks to Lemma \ref{lemma:inequality_exponential} and the fact that for all function~\hbox{$f\in\cup_{\psi\in\Psi}\mathdutchcal{B}(\psi)$}, there exists $C_f >0$ such that $\left|f(x,a)\right| \leq C_f\mathdutchcal{V}(x)e^a$ for all \hbox{$(x,a)\in\mathcal{X}$}.

\item The semigroup property is a direct consequence of the branching Markov property. The positivity of this semigroup is trivial. To obtain the equation of the semigroup, first condition with respect to the first event of division in Eq.~\eqref{eq:SDE_model_intro} and use the strong Markov property. Then, apply~\eqref{eq:uniform_random_variables_assumption} and~\eqref{eq:kernel_branching}. It comes for all $t\geq 0$, $(x,a)\in\mathcal{X}$ and $f\in\cup_{\psi\in\Psi}\mathdutchcal{B}(\psi)$
$$
\begin{aligned}
M_t f(x,a) &= f(x,a+t)\exp\left(-\int_a^{a+t}b(x,u)du\right) + \int_0^tb(x,a+s)\\
&\times\exp\left(-\int_a^{a+s}b(x,u)du\right)\left[\int_{y\in\mathbb{R}_+^{2k}} M_{t-s}(f)(y,0) 1_{\{x+w_1\in\mathbb{R}_+^{2k}\}}\mathdutchcal{K}(dy,x)\right]ds.
\end{aligned}
$$
In view of Remark \ref{rem:equality_kernel}, it follows directly that the lemma is proved from the above equation. \qed
\end{enumerate}

\subsection{Proof of Lemma \ref{lemma:inequalities_psi}}\label{subsect:proof_inequalities_psi}
Let $\psi\in\Psi$. We prove each inequality one by one.
\begin{enumerate}[$i)$, leftmargin=0.75cm]
\item First, as for all $(x,a)\in \mathcal{X}$: $\psi(x,a) = \mathdutchcal{V}(x)\left(1+a^{d_{\psi}}\right)$, we easily have that
$$
b(x,a)\frac{\mathdutchcal{K}\left(\mathdutchcal{V}\right)(x)}{\psi(x,a)} - \left[b(x,a) - \frac{\frac{\partial}{\partial a}\psi(x,a)}{\psi(x,a)}\right] = b(x,a)\bigg[\frac{\mathdutchcal{K}\left(\mathdutchcal{V}\right)(x)}{\mathdutchcal{V}(x)(1+a^{d_{\psi}})} - 1\bigg] + \frac{d_{\psi}a^{d_{\psi} - 1}}{1+a^{d_{\psi}}}.
$$
It is well-known that there exists a constant $C_{\psi} >0$ such that for all $a\in\mathbb{R}_+$,
\begin{equation}\label{eq:proof_lemma_inequalities_psi_intermediate_first}
\frac{d_{\psi}a^{d_{\psi} - 1}}{1+a^{d_{\psi}}} \leq C_{\psi}.
\end{equation}
Then, using \eqref{eq:birth_rate_assumption}, \eqref{eq:upperbound_fullkernel}  and~\eqref{eq:proof_lemma_inequalities_psi_intermediate_first}, we obtain
$$
b(x,a)\frac{\mathdutchcal{K}\left(\mathdutchcal{V}\right)(x)}{\psi(x,a)} - \left[b(x,a) - \frac{\frac{\partial}{\partial a}\psi(x,a)}{\psi(x,a)}\right] \leq  2\tilde{b}(1+\varepsilon_1)\frac{1+a^{d_b}}{1+a^{d_{\psi}}} + C_{\psi}. 
$$
As $d_{\psi}\geq d_b$ (see \eqref{eq:defintion_all_distortion_functions}), there exists $C'_{\psi}$ such that for all $a\geq 0$
\begin{equation}\label{eq:proof_lemma_inequalities_psi_intermediate_second}
\frac{1+a^{d_b}}{1+a^{d_{\psi}}} \leq C'_{\psi}.
\end{equation} 
Thus, the first inequality of the lemma is proved for $\lambda_{\psi} > 2\tilde{b}(1+\varepsilon_1)C'_{\psi} + C_{\psi}$.

\item The proof of the second statement is very similar to the proof of the first one. Let us fix~$(x,a)\in\mathcal{X}$. First, we apply Eq.~\eqref{eq:birth_rate_assumption} and Eq.~\eqref{eq:proof_lemma_inequalities_psi_intermediate_first}. Then, we use Eq.~\eqref{eq:proof_lemma_inequalities_psi_intermediate_second}. It comes
$$
2\frac{b(x,a)}{\psi(x,a)} - \left[b(x,a) - \frac{\frac{\partial}{\partial a}\psi(x,a)}{\psi(x,a)}\right] \leq  \frac{2\tilde{b}(1+a^{d_b})}{\mathdutchcal{V}(x)(1+a^{d_{\psi}})} + C_{\psi} \leq \frac{2\tilde{b}C'_{\psi}}{\mathdutchcal{V}(x)} + C_{\psi}.
$$
Using the third statement of \hyperlink{paragraph:long_time_behaviour_S2.2}{$(S_{2.2})$}, we now conclude that the first and second statements of this lemma are true for 
$$
\lambda_{\psi} > \max\left[2\tilde{b}C'_{\psi}(1+\varepsilon_1) + C_{\psi}, \frac{2\tilde{b}C'_{\psi}}{\mathdutchcal{V}_{\min}} + C_{\psi}\right].
$$

\item For all $(x,a)\in\mathcal{X}$, $s\geq0$,
$$
\frac{\psi(x,s+a)}{\psi(x,s)} = \frac{1+(a+s)^{d_{\psi}}}{1+ s^{d_{\psi}}} =  1 + \frac{(a+s)^{d_{\psi}} - s^{d_{\psi}}}{1+ s^{d_{\psi}}} \leq  1 + \sum_{j = 1}^{d_{\psi}} \binom{d_{\psi}}{j}\frac{s^{d_{\psi}-j}}{1+ s^{d_{\psi}}}a^j.
$$
As $s\mapsto s^{d_{\psi}-j}\left(1+s^{d_{\psi}}\right)^{-1}$ is bounded when $j\in\llbracket1,d_{\psi}\rrbracket$, we easily obtain from this inequality that the third statement of the lemma is true, which concludes the proof of the lemma.
\end{enumerate}
\qed 
\subsection{Proof of Lemma \ref{lemma:unique_solution_equation_semigroup}}\label{subsect:proof_unique_solution_equation_semigroup}
The proof is based on a fixed point argument. Let $\tilde{T} >0$. We endow the set $M_b\big([0,\tilde{T}]\times\mathcal{X}\big) $ with the supremum norm. By Remark~\ref{rem:equality_kernel} and \eqref{eq:upperbound_fullkernel}, we have that for all~\hbox{$(t,x,a)\in[0,\tilde{T}]\times\mathcal{X}$}, and $(g_1,g_2)\in\Big(M_b\big([0,\tilde{T}]\times\mathcal{X}\big)\Big)^2$,
\begin{equation}\label{eq:proof_semigroup_uniqueness}
\begin{aligned}
\left|\Gamma_f(g_1)(t,x,a) - \Gamma_f(g_2)(t,x,a)\right| &\leq 2(1+\varepsilon_1)||g_1-g_2||_{\infty}\mathdutchcal{V}(x)\int_0^t\frac{b(x,a+s)}{\psi(x,a)}\\
&\times\exp\left(-\int_a^{a+s}b(x,u)du -\lambda_{\psi}s\right) ds.  
\end{aligned}
\end{equation}
In addition, by \eqref{eq:birth_rate_assumption}, the fact that $d_{\psi} \geq d_b$, and then the third statement of Lemma~\ref{lemma:inequalities_psi}, there exists~$C_{\psi} > 0$ such that for all $(x,a,s)\in\mathcal{X}\times\mathbb{R}_+$ it holds 
$$
\frac{b(x,a+s)}{\psi(x,a)} \leq \frac{\tilde{b}\left(1+(a+s)\right)^{d_{\psi}}}{\mathdutchcal{V}(x)(1+a^{d_{\psi}})} \leq \frac{\tilde{b}C_{\psi}\times\overline{\psi}\times(1+s^{d_{\psi}})}{\mathdutchcal{V}(x)}.
$$
Plugging this inequality in~\eqref{eq:proof_semigroup_uniqueness} yields
$$
\left|\left|\Gamma_f(g_1)- \Gamma_f(g_2)\right|\right| _{\infty} \leq 2\tilde{b}C_{\psi}\times\overline{\psi}\times(1+\varepsilon_1)||g_1-g_2||_{\infty}\int_0^{\tilde{T}}(1+s^{d_{\psi}}) ds.
$$
Choosing $\tilde{T}$ such that $\int_0^{\tilde{T}}(1+s^{d_{\psi}}) ds < \left[2\tilde{b}C_{\psi}\times\overline{\psi}\times(1+\varepsilon_1)\right]^{-1}$, we have by the Banach fixed point theorem that there exists a unique solution $\overline{f}$ to the equation $\Gamma_f(\overline{f}) = \overline{f}$. Then, as usual, iterating this procedure on $[\tilde{T},2\tilde{T}]$, $[2\tilde{T},3\tilde{T}]$ etc... yields that for all $T > 0$, we have a unique solution to $\Gamma(\overline{f}) = \overline{f}$ in~$M_b([0,T]\times\mathcal{X})$. The lemma is thus proved. \qed 

\section{Proof of Proposition~\ref{prop:path_type}}\label{sect:proof_prop_path_type}
By~\cite[Theorem $\left(27.8\right)$]{davis1993}, we have that $\left(X_t\right)_{t\geq0}$ corresponds to a Borel right process relative to its augmented filtration, which is a large class of right-continuous Markov processes defined in~\hbox{\cite[$\left(27.7\right)$]{davis1993}}. More generally, $\left(X_t\right)_{t\geq0}$ corresponds to a right process relative to its augmented filtration, a class of right-continuous Markov processes defined in~\hbox{\cite[Def.~$\left(8.1\right)$]{sharpe1988}} that contains Borel right processes (see~\hbox{\cite[$\left(20.6\right)$]{sharpe1988}}). In addition, by \hbox{\cite[Theorem $(18.1)$]{sharpe1988}}, we have that there exists $\Omega'_1 \subset \Omega_1$ such that $\Omega\backslash\Omega_1'$ is negligible, and such that for all $w\in \Omega'_1$, $X_t(\omega)$ is right-continuous in the Ray topology, a topology defined in~\cite[p.~$91$]{sharpe1988}. From these two points, if we introduce
$$
\begin{aligned}
\Omega_2 = 	\left\{w:\mathbb{R}_+ \rightarrow \mathbb{X}\cup\{\partial\}\right. \,|\, &w \text{ is right-continuous in the original topology of $\mathbb{X}\cup\{\partial\}$} \\ 
&\text{and the Ray topology, and verifies}\\
&\left.  \forall(s,t)\in\mathbb{R}_+^2:\,w(s) = \partial \text{ and }s\leq t \implies w(t) = \partial \right\},
\end{aligned}
$$
then we can define a map $\Phi : \Omega'_1 \rightarrow \Omega_2$ and a collection of functions $\big(\Tilde{X}_t\big)_{t \geq 0}$ from $\Omega_2$ to~$\mathbb{X}\cup\{\partial\}$, such that for every $\omega \in \Omega'_1$ and every $t \geq 0$,
$$
\Tilde{X}_t\left(\Phi(\omega)\right) = X_t\left(\omega\right).
$$
If we denote for all $x\in\mathbb{X}\cup\{\partial\}$ the measure $\Tilde{\mathbb{P}}_x = \mathbb{P}_x \circ \Phi^{-1}$ on $\left(\Omega_2,\mathcal{A}_2\right)$, then we have following the proof of~\hbox{\cite[\hbox{Prop.~$19.6$ - $(ii)$}]{sharpe1988}} that under $\big(\Tilde{\mathbb{P}}_x\big)_{x\in\mathbb{X}\cup\{\partial\}}$, $\big(\Tilde{X}_t\big)_{t\geq0}$ is a Markov process whose semigroup is the same as the one of $\left(X_t\right)_{t\geq0}$. In addition, in view of the fact that $\big(\Tilde{X}_t\big)_{t\geq0}$ is a right process relative to its augmented filtration by~\cite[Theorem~$(19.7)$]{sharpe1988}, we have that~$\big(\Tilde{X}_t\big)_{t\geq0}$ satisfies the strong Markov property with respect to its augmented filtration by~\cite[Theorem $(7.4)$]{sharpe1988}. One can get that this augmented filtration is right-continuous by applying \hbox{\cite[Corollary~$(7.7)$]{sharpe1988}}. We thus only have to prove that $\Omega_2$ is of path type. We do this in the following paragraph.

First, the points \cite[Def. $(23.10)$ - $(iii)$,$(iv)$,$(v)$]{sharpe1988} are almost direct to verify. Indeed, we can take as a shift operator (see~\cite[$\left(2.4\right)$]{sharpe1988}) the operator $\theta_t\big(\left(w_s\right)_{s\geq0}\big) = \left(w_{t+s}\right)_{s\geq0}$ for all $t \geq 0$ and $\left(w_s\right)_{s\geq0} \in \Omega_2$, and then obtain the existence of the three other operators from \hbox{\cite[p.~$63$, p.~$110$, p.~$118$]{sharpe1988}}. In addition, the points \cite[Def. $(23.10)$ - $(ii)$,$(vi)$,$(vii)$]{sharpe1988} are trivial to verify from the definition of our sample space~$\Omega_2$. It thus only remains to prove that~\hbox{\cite[Def. $(23.10)$ - $(i)$]{sharpe1988}} is verified for $\Omega_2$, and we will have that it corresponds a sample space of path type. This consists in proving that for all $\omega\in \Omega$, the trajectory~$\left(\Tilde{X}_t(\omega)\right)_{t\geq0}$ is right-continuous in both the original and the Ray topology of~$\mathbb{X}\cup\{\partial\}$, and has left limits in a Ray compactification of $\mathbb{X}\cup\{\partial\}$, see~\hbox{\cite[p.~$91$]{sharpe1988}} for the definition of Ray compactification. To prove these two properties, recall that in view of \cite[Theorem~$(19.7)$]{sharpe1988}, $\big(\Tilde{X}_t\big)_{t\geq0}$~is a right process relative to its augmented filtration. Then, in view of~\hbox{\cite[Theorem $(18.1)$]{sharpe1988}} and the fact that a right process is right-continuous, we have that the two properties of~\cite[Def. $(23.10)$ - $(i)$]{sharpe1988} presented above are verified. We thus conclude that $\Omega_2$ is of path type, so that our proposition is proved. \qed

\section{Proof of Proposition~\ref{prop:to_prove_A1}}\label{appendix:proof_statements_to_obtain_A1}

This section is devoted the proof of the Doeblin condition presented in Section~\ref{subsubsect:assumption_A1}. Let us start by introducing for all $l\geq 1$, $x \in \mathcal{D}_l$, $L\geq l$, $c>0$ the following set that allows us to know where the $(r,l,L,t,c)$-local Doeblin conditions~(see Definition~\ref{dft:local_doeblin_condition}) hold:
$$
R_{l,L}^{x}(c) := \big\{(t,y)\in \mathbb{R}_+\times\mathbb{R}_+^{2k}\,|\,\text{a }(r,l,L,t,c)\text{-local Doeblin condition holds from }x \text{ to }y\big\}.
$$
The plan of the proof of Proposition~\ref{prop:to_prove_A1} is the following. First, in Sections~\ref{subsubsect:inequality_doeblin_velleret} and \ref{subsubsect:proof_transfer_localDoeblin}, we prove the two following statements, that correspond to the first two steps of the proof of \hyperlink{te:assumptions_velleret_A1}{$(A_1)$}, see Section~\ref{subsubsect:assumption_A1}. 
\begin{lemm}[Local Doeblin conditions in balls of radius $r$]\label{lemma:inequality_doeblin_velleret}
Assume that \hyperlink{paragraph:long_time_behaviour_S1.1}{$(S_{1.1})$}, \hyperlink{paragraph:long_time_behaviour_S1.2}{$(S_{1.2})$}, and \hyperlink{paragraph:long_time_behaviour_S2.2}{$(S_{2.2})$} hold. Let us fix $l\geq 1$. Then, there exists $L \geq l+2m_0$, such that for all $t\in\left[(l+2(m_0-1)+1)a_0,(l+2m_0)a_0\right]$, there exists $C(t)>0$ such that for all $x_I\in \mathcal{D}_l$ and $x'\in B(x_I,r)\cap \mathcal{D}_l$, a $\left(r,l,L,t,C(t)\right)$-Doeblin condition holds from $x_I$ to $x'$. 
\end{lemm}
\begin{lemm}[Transfer of local Doeblin conditions]\label{lemma:vicinity_velleret}
Assume that \hyperlink{paragraph:long_time_behaviour_S1.1}{$(S_{1.1})$}, \hyperlink{paragraph:long_time_behaviour_S1.2}{$(S_{1.2})$} and \hyperlink{paragraph:long_time_behaviour_S2.2}{$(S_{2.2})$} hold. Let us fix $l\geq 1$. Then, there exist $L \geq l+2m_0$ and $t_a,\,c_a >0$, such that for all  $(x_I,x_F)\in (\mathcal{D}_l)^2$, $t>0$, if~$(t,x_F)\in R_{l,L}^{x_I}(c)$, then 
$$
\left\{t+t_a\right\}\times \left(B(x_F,r)\cap \mathcal{D}_l\right) \subset R_{l,L}^{x_I}(c.c_a).
$$
\end{lemm}
\noindent Then, we prove Proposition~\ref{prop:to_prove_A1} in Section~\ref{subsect:proof_prop_to_prove_A1}.

\subsection{Proof of Lemma \ref{lemma:inequality_doeblin_velleret}}\label{subsubsect:inequality_doeblin_velleret}

In this proof, we mostly use the objects introduced
in Assumption~\hyperlink{paragraph:long_time_behaviour_S1}{$(S_{1})$}. We use the convention that for all sequence $\left(\xi_n\right)_{n\in\mathbb{N}}$ and $(n_1,n_2)\in\mathbb{N}^2$ such that $n_2 < n_1$, we have $\sum_{n_1}^{n_2} \xi_n= 0$. We also use the convention that $[0,\delta]^0 = \{\emptyset\}$ and that for all $n\in\mathbb{N}^*$, $C \in \mathcal{B}(\mathbb{R}^n)$ and $f : [0,\delta]^0\times\mathbb{R}^n \rightarrow \mathbb{R}$ measurable:
$$
\int_{\tilde{\alpha}\in[0,\delta]^0}\int_{\tilde{\beta}\in C} f(\tilde{\alpha},\tilde{\beta}) d\tilde{\alpha}d\tilde{\beta} = \int_{\tilde{\beta}\in C} f(\emptyset,\tilde{\beta}) d\tilde{\beta}.
$$
To simplify notations, we finally denote $\kappa = \min_{x\in D_l}(\Delta_x)$, and notice that with this notation and~\eqref{eq:dft_constant_r}, we have~\hbox{$r = \min\left(\frac{B_{\max}}{2},\frac{\min(\delta,\kappa)}{10}\right)$}. 

Let $l\geq 1$ and $L\geq l+2m_0$ such that $B_{\max}L \geq B_{\max}l + m_0\Delta$. We begin with some preliminaries. We first consider $p_{\min} > 0$ that verifies the statement of \hyperlink{paragraph:long_time_behaviour_S1.2}{$(S_{1.2})$} with $A = B_{\max} L$, and the following constants, for all $t\geq0$:
\begin{equation}\label{eq:constants_doeblin}
\begin{aligned}
C_1(t) &= \left[\frac{2b_0\mathdutchcal{V}_{\min}}{\left(\sup_{y\in\mathcal{D}_L} \mathdutchcal{V}(y)\right)(1+(a_0L)^{d_{\psi}})}\right]^{m_0}\exp\left(-(m_0+1)(\lambda_{\psi} + \tilde{b}(1+(a_0L)^{d_b}))t\right), \\
C_2 &= \left[\frac{p_{\min}(g_{\min})^k\left(h_{\min}\right)^{2k}}{\#\left(\mathcal{I}_k\right)}\right]^{m_0}. 
\end{aligned}  
\end{equation}
These constants are useful to obtain lower bounds. In view of~\hyperlink{paragraph:long_time_behaviour_S1.2}{$(S_{1.2})$}, we then define for all~$j\in \llbracket1,2k\rrbracket$
$$
\gamma_j := \# \left\{n \in \llbracket1,m_0\rrbracket\,|\,j\in \mathbb{I}^n\right\}, \hspace{4mm} \omega_j := \# \left\{n \in \llbracket1,m_0\rrbracket\,|\,j\in \mathbb{J}^n\right\}.
$$
They correspond respectively to the number of times the telomere associated to the coordinate $j$ is shortened and the number of times it is lengthened on the event 
\begin{equation}\label{eq:event_for_doeblin_condition}
\mathdutchcal{E} := \left\{\left(I_l,J_l\right)_{l\in\llbracket1,m_0\rrbracket} = \left(\mathbb{I}^l,\mathbb{J}^l\right)_{l\in\llbracket1,m_0\rrbracket}\right\}.
\end{equation}
By~\hyperlink{paragraph:long_time_behaviour_S1.2}{$(S_{1.2})$}, it holds $\omega_j \geq \gamma_j \geq 1$. We finally define for all~$j\in\llbracket1,2k\rrbracket$,~\hbox{$(y,\theta)\in\mathbb{R}_+\times\mathbb{R}$}, \hbox{$(\tilde{\alpha},\tilde{\beta})\in [0,\delta]^{\gamma_j - 1}\times[0,\kappa]^{\omega_j}$}, the following:
\begin{equation}\label{eq:function_doeblin}
\begin{aligned}
g_{j,1}(\theta,\tilde{\alpha},\tilde{\beta}) &= 1_{\left\{\theta - \sum_{i = 1}^{\gamma_j-1} (\tilde{\beta}_i-\tilde{\alpha}_i) - \sum_{i = \gamma_j}^{\omega_j} \tilde{\beta}_i\in\left[-\delta,0\right]\right\}},\\
g_{j,2}(y,\theta,\tilde{\alpha},\tilde{\beta}) &= \left(1_{\left\{\forall l \in \llbracket1,\gamma_j-1\rrbracket:\, y+\sum_{i = 1}^{l} (\tilde{\beta}_i-\tilde{\alpha}_i)\geq 0\right\}}1_{\{\gamma_j \geq 2\}} + 1_{\{\gamma_j = 1\}}\right)1_{\left\{y+\theta-\sum_{i = \gamma_j +1}^{\omega_j} \tilde{\beta}_i \geq 0\right\}}, \\
F_j(y,\theta) &=
\int_{\overline{\alpha}\in[0,\delta]^{\gamma_j - 1}}\int_{\overline{\beta}\in[0,\kappa]^{\omega_j}} g_{j,1}\left(\theta,\overline{\alpha},\overline{\beta}\right)g_{j,2}\left(y,\theta,\overline{\alpha},\overline{\beta}\right) d\overline{\alpha} d\overline{\beta}.
\end{aligned}
\end{equation}
In the above, the functions $\left(g_{j,1}\right)_{1\leq j \leq 2k}$ and $\left(g_{j,2}\right)_{1\leq j \leq 2k}$ are only useful to define the functions~$\left(F_j\right)_{1\leq j \leq 2k}$. The interest of the functions in this last collection is presented in the two paragraphs below. We mention that the indices of the variables $\left(\tilde{\alpha}_i\right)_{i\in\llbracket1,\gamma_j-1\rrbracket}$ and~$\big(\tilde{\beta}_i\big)_{i\in\llbracket1,\omega_j\rrbracket}$ in the definition of the functions $\left(g_{j,1}\right)_{1\leq j \leq 2k}$ and $\left(g_{j,2}\right)_{1\leq j \leq 2k}$ do not correspond to the order in which the shortenings and elongations occur in the coordinate~$j$ on the event $\mathdutchcal{E}$ defined in~\eqref{eq:event_for_doeblin_condition}. The indices have been chosen in order to write efficiently the constraint presented in the second paragraph below.

For all $j\in\llbracket1,2k\rrbracket$, the function $F_j(y,.)$ defined in~\eqref{eq:function_doeblin} is a lower bound for the density of telomere length variation in the \hbox{$j$-th} coordinate of the particle on the event \hbox{$\left\{N_t = m_0\right\}\cap \mathdutchcal{E}$} 
when the initial length of the $j$-th telomere is $y$. Let us explain why. First, we notice by the change of variable $\tilde{\alpha}' = -\tilde{\alpha}$ that the integral
$$
\begin{aligned}
\int_{\tilde{\alpha}\in[0,\delta]^{\gamma_j - 1}}\int_{\tilde{\beta}\in[0,\kappa]^{\omega_j}} g_{j,1}(\theta,&\tilde{\alpha},\tilde{\beta})d\tilde{\alpha} d\tilde{\beta} = \int_{\tilde{\alpha}'\in\mathbb{R}^{\gamma_j - 1}}\int_{\tilde{\beta}\in\mathbb{R}^{\omega_j}} 1_{\left\{s  - \sum_{i = 1}^{\gamma_j-1} \tilde{\alpha}'_i - \sum_{i = 1}^{\omega_j} \tilde{\beta}_i \in\left[-\delta,0\right]\right\}} \\
&\times 1_{\left\{\tilde{\alpha}'_{1}\in[-\delta,0]\right\}}\hdots1_{\left\{\tilde{\alpha}'_{\gamma_j-1}\in[-\delta,0]\right\}} 1_{\left\{\tilde{\beta}_{1}\in[0,\kappa]\right\}}\hdots1_{\left\{\tilde{\beta}_{\omega_j}\in[0,\kappa]\right\}}d\tilde{\alpha}' d\tilde{\beta}
\end{aligned}
$$
corresponds to the convolution between $\omega_j$ times the function~$1_{[0,\kappa]}(\theta)$ and $\gamma_j$ times the function~$1_{[0,\delta]}(-\theta)$. Second, we have by~\hyperlink{paragraph:long_time_behaviour_S1.2}{$(S_{1.2})$} that the probability density functions for the shortening and lengthening values, namely $g$ and $(h(x,.))_{x\in\mathbb{R}}$, are bounded from below on $[0,\delta]$ and $[0,\kappa]$ respectively. Finally, we recall that the probability density function of the sum of random variables is given by the convolution of their densities. By these three results, we obtain that~$\int_{\tilde{\alpha}\in[0,\delta]^{\gamma_j - 1}}\int_{\tilde{\beta}\in[0,\kappa]^{\omega_j}} g_{j,1}(\theta,\tilde{\alpha},\tilde{\beta})d\tilde{\alpha} d\tilde{\beta}$ is, up to a constant, a lower bound for the probability density function of the sum of~$\gamma_j$ random variables of shortening and~$\omega_j$ random variables of lengthening. 

There is another constraint that we need to take into account to have that $F_j(y,.)$ is the lower bound for the density of telomere length variation in the $j$-th coordinate on the event $\left\{N_t = m_0\right\}\cap \mathdutchcal{E}$. During the first $m_0$ jumps, the length of this telomere must remain non-negative~(otherwise there is extinction of the particle). The conditions for the variables $y$, $\tilde{\alpha}$ and $\tilde{\beta}$ in the indicators of the function $g_{j,2}$ allow to handle this. Indeed, they mean that the length of the telomere must remain non-negative when one of the $\gamma_j$ events where there is a shortening occurs (there is also a lengthening at these events as for all~$i\in\llbracket1,m_0\rrbracket$ it holds~$\mathbb{I}^i \subset \mathbb{J}^i$). We do not need any condition for the other events because the length of the telomere cannot become negative when they occur (as there is no shortening). The latter, the above paragraph, and the fact that $F_j$ is the integral of~$g_{j,1}g_{j,2}$ yield that $F_j(y,.)$ corresponds to the lower bound stated above. 

Now, let $t\in[(l+2(m_0-1)+1)a_0,(l+2m_0)a_0]$, $x_I\in\mathcal{D}_l$, $(x,\,x')\in \left(B(x_I,r)\cap \mathcal{D}_l\right)^2$, \hbox{$a\in[0,a_0l]$}, and $f$ a bounded measurable function on $\mathbb{R}_+$. We know that $L\geq l+2m_0$,  $t \leq (l+2m_0)a_0$, $B_{\max}L \geq B_{\max}l + m_0\Delta$, and that~$\Delta$ is the maximum lengthening value. Thus, if $m_0$ jumps have occurred during the interval of time $[0,t]$, if the time before the first jumps is smaller than $2a_0$, and if the particle has not jumped to the cemetery, then the particle $(Z_t)_{t\geq 0}$ has stayed in $D_L$ during all the interval of time~$[0,t]$. In particular, the following holds
\begin{equation}\label{eq:sequence_events_later_proof}
\begin{aligned}
\mathbb{E}_{(x,a)}\left[f(Z_t);\,t < \min(\tau_{\partial},T_{D_L})\right] &\geq \mathbb{E}_{(x,a)}\Big[f(X_{m_0}, t-\mathdutchcal{T}_{m_0});\,N_t = m_0,\\
&\forall i \in \llbracket1,m_0\rrbracket:\,(I_i,J_i) = \left(\mathbb{I}^i,\mathbb{J}^i\right),\\
&\forall i \in \llbracket1, m_0-1 \rrbracket: \mathdutchcal{T}_i - \mathdutchcal{T}_{i-1} \in[a_0,2a_0],\\
&\,\mathdutchcal{T}_{m_0} - \mathdutchcal{T}_{m_0-1} \in[a_0, t- \mathdutchcal{T}_1]\Big]. 
\end{aligned}
\end{equation}
Now, we use Lemma~\ref{lemm:generalized_duhamel} to develop the right-hand side term of the above equation and obtain
\begin{equation}\label{eq:simpilfication_psuforward_measure}
\begin{aligned}
&\mathbb{E}_{(x,a)}\left[f(Z_t);\,t < \min(\tau_{\partial},T_{D_L})\right] \geq \int_{[a_0,2a_0]\times \mathbb{R}^{2k}} \hdots \int_{[a_0,2a_0]\times \mathbb{R}^{2k}}\int_{\left[a_0,t - \sum_{i = 1}^{m_0-1} s_i\right]\times\mathbb{R}^{2k}_+} \\
& \times f\left(x + \sum_{i = 1}^{m_0} u_i,t - \sum_{i = 1}^{m_0} s_i\right)\mathdutchcal{V}(x+u_1) \hdots \mathdutchcal{V}\left(x+\sum_{i = 1}^{m_0}u_i\right)\mathdutchcal{G}_{a}(x,s_1)\mathdutchcal{G}_0(x+ u_1,s_2)\hdots  \\
&\times \mathdutchcal{G}_0\left(x+ \sum_{i = 1}^{m_0-1}u_i,s_m\right)\overline{\mathdutchcal{H}_0}\left(x+\sum_{j = 1}^{m_0}u_j,t- \sum_{j = 1}^{m_0}s_{j}\right)1_{\left\{\forall j \in\llbracket1,m_0\rrbracket:\,x+\sum_{i = 1}^j u_j\in\mathcal{D}_l\right\}}\\
&\times \left(ds_1d\pi_{x}^{\mathbb{I}^1,\mathbb{J}^1}(u_1)\right)\hdots \left(ds_{m_0}d\pi_{x+\sum_{i = 1}^{m_0-1}u_i}^{\mathbb{I}^{m_0},\mathbb{J}^{m_0}}\left(u_{m_0}\right)\right). \\
\end{aligned}
\end{equation}
By \eqref{eq:birth_rate_assumption}, Assumption \hyperlink{paragraph:long_time_behaviour_S2.2}{$(S_{2.2})$} and the definition of $\psi \in \Psi$ given in \eqref{eq:defintion_all_distortion_functions}, we know that for all $(w,s)\in D_L$
$$
\begin{aligned}
b(w,s) \leq \tilde{b}(1+(a_0L)^{d_b}), \text{      }\psi(w,s) \leq \sup_{y \in \mathcal{D}_L}(\mathdutchcal{V}(y))\left(1+(a_0L)^{d_{\psi}}\right), \text{   and   }\mathdutchcal{V}(w) \geq \mathdutchcal{V}_{\min},
\end{aligned}
$$
and that $b(w,s) \geq b_0$ when $s\geq a_0$. Thus, in view of the expression of $\mathdutchcal{G}_a$ and $\mathdutchcal{G}_0$ given in \eqref{eq:useful_notation_third}, the expression of $\overline{\mathdutchcal{H}}_0$ given in \eqref{eq:tail_events}, and the definition of the constant $C_1(t)$ introduced in \eqref{eq:constants_doeblin}, one can bound from below \eqref{eq:simpilfication_psuforward_measure} to get
\begin{equation}\label{eq:intermediate_doeblin}
\begin{aligned}
\mathbb{E}_{(x,a)}\left[f(Z_t);\,t < \min(\tau_{\partial},T_{D_L})\right] &\geq C_1(t)\int_{[a_0,2a_0]\times \mathbb{R}^{2k}} \hdots \int_{[a_0,2a_0]\times \mathbb{R}^{2k}}\int_{\left[a_0,t - \sum_{i = 1}^{m_0-1} s_i\right]\times\mathbb{R}^{2k}_+} \\
&\times f\left(x + \sum_{i = 1}^{m_0} u_i,t - \sum_{i = 1}^{m_0} s_i\right)1_{\left\{\forall j \in\llbracket1,m_0\rrbracket:\,x+\sum_{i = 1}^j u_j\in\mathcal{D}_l\right\}}\\
&\times\left(ds_1d\pi_{x}^{\mathbb{I}^1,\mathbb{J}^1}(u_1)\right)\hdots \left(ds_{m_0}d\pi_{x+\sum_{i = 1}^{m_0-1}u_i}^{\mathbb{I}^{m_0},\mathbb{J}^{m_0}}\left(u_{m_0}\right)\right). \\ 
\end{aligned}
\end{equation}
We recall that the definition of $(\pi_y^{I,J})_{y\in\mathbb{R}_+^{2k},(I,J)\in\mathcal{Q}_k}$ is given in \eqref{eq:measure_by_event}, and that the measures used to define it are given in~\hbox{\eqref{eq:measure_shortening} and~\eqref{eq:measure_elongation}}. We continue to bound from below the left-hand side term of~\eqref{eq:intermediate_doeblin}. First, we use \hyperlink{paragraph:long_time_behaviour_S1.1}{$(S_{1.1})$} (without loss of generality, we can assume that $h_{\min} < 1$) and the second statement of~\hyperlink{paragraph:long_time_behaviour_S1.2}{$(S_{1.2})$}, and do the change of variables $v = \sum_{i = 1}^{m_0} u_i$ to bound from below the measures $\left(\pi_{x+\sum_{i = 1}^{p-1}u_i}^{\mathbb{I}^p,\mathbb{J}^p}(u_p)\right)_{p \in \llbracket1,m_0\rrbracket}$ in the right-hand side term of~\eqref{eq:intermediate_doeblin} by the measure~\hbox{$C_2.\prod_{i = 1}^{2k} F_i\left(x_i,v_i\right)dv_i$}~(see~\eqref{eq:constants_doeblin} and~\eqref{eq:function_doeblin}). Then, we do the change of variables $w = v + x - x'$. Finally, we restrict the domain of integration to~$[-r,r]^{2k}$. It comes
\begin{equation}\label{eq:doeblin_k=N_intermediate_second}
\begin{aligned}
&\mathbb{E}_{(x,a)}\left[f(Z_t);\,t < \min(\tau_{\partial},T_{D_L})\right]
\geq C_1(t) C_2\int_{w\in[-r,r]^{2k}}\int_{[a_0,2a_0]} \hdots \int_{[a_0,2a_0]}\int_{\left[a_0,t - \sum_{i = 1}^{m_0-1} s_i\right]}  \\ 
&\times f\bigg(x' + w, t - \sum_{i = 1}^{m_0} s_i\bigg)1_{\{x'+w\in\mathbb{R}_+^{2k}\}}F_1(x_1,(w + x'-x)_1)\hdots F_{2k}(x_{2k},(w + x'-x)_{2k})dwds.
\end{aligned}
\end{equation}
To continue our computations, we now obtain lower bounds for the functions $(F_j)_{j\in\llbracket1,2k\rrbracket}$. Let us consider \hbox{$j\in\llbracket1,2k\rrbracket$,~$y\in[0,B_{\max}l]$},~\hbox{$\theta\in[\max(-3r,-y),3r]$}, \hbox{$\tilde{\alpha}\in[0,\delta]^{\gamma_j-1}$} and finally \hbox{$\tilde{\beta}\in[0,\kappa]^{\omega_j}$}. Assume that~$\sum_{i = 1}^{\gamma_j-1}(\tilde{\beta}_i-\tilde{\alpha}_i) \in  \left[0,\frac{\min(\delta,\kappa)}{20}\right]$, that \hbox{$\tilde{\beta}_{\gamma_j}\in \left[\frac{3\min(\delta,\kappa)}{10},\frac{6\min(\delta,\kappa)}{10}\right]$} and finally that~\hbox{$\sum_{\gamma_j+1}^{\omega_j}\tilde{\beta}_i \leq \frac{\min(\delta,\kappa)}{20}$}. Then, by using that $|\theta| \leq 3r \leq \frac{3\min(\delta,\kappa)}{10}$, one can easily obtain~that 
$$
-\min\left(\delta,\kappa\right) \leq \theta - \sum_{i = 1}^{\gamma_j-1} (\tilde{\beta}_i-\tilde{\alpha}_i) - \sum_{i = \gamma_j}^{\omega_j} \tilde{\beta}_i \leq 0,
$$
so that $g_{j,1}(y,\theta, \tilde{\alpha}, \tilde{\beta}) = 1$ by the first line of~\eqref{eq:function_doeblin}. Combining this with the fact that $y\geq 0$ (to bound from below the first indicator in $g_{j,2}$), we obtain
$$
\begin{aligned}
F_j(y,\theta) &\geq \int_{\tilde{\alpha}\in[0,\delta]^{\gamma_j - 1}}\int_{\tilde{\beta}\in[0,\kappa]^{\omega_j}}\left(1_{\left\{\forall l \in \llbracket1,\gamma_j-1\rrbracket:\, \sum_{i = 1}^{l} (\tilde{\beta}_i-\tilde{\alpha}_i)\in \left[0,\frac{\min(\delta,\kappa)}{20}\right]\right\}}1_{\{\gamma_j \geq 2\}} + 1_{\{\gamma_j = 1\}}\right) \\
&\times 1_{\left\{\tilde{\beta}_{\gamma_j}\in \left[\frac{3\min(\delta,\kappa)}{10},\frac{6\min(\delta,\kappa)}{10}\right]\right\}}1_{\left\{\sum_{i = \gamma_j +1}^{\omega_j} \tilde{\beta}_i \leq \min(y+s,\frac{\min(\delta,\kappa)}{20})\right\}}d\tilde{\alpha} d\tilde{\beta}.
\end{aligned}
$$
Now, first observe that 
$$
\int_{\tilde{\alpha}\in[0,\delta]^{\gamma_j - 1}}\int_{\tilde{\beta}\in[0,\kappa]^{\gamma_j-1}}\Big(1_{\left\{\forall l \in \llbracket1,\gamma_j-1\rrbracket:\, \sum_{i = 1}^{l} (\tilde{\beta}_i-\tilde{\alpha}_i)\in \left[0,\frac{\min(\delta,\kappa)}{20}\right]\right\}}1_{\{\gamma_j \geq 2\}}+ 1_{\{\gamma_j = 1\}}\Big)  d\tilde{\alpha}d\tilde{\beta} > 0 
$$ 
because it is the integral of a non-negative function, positive on a non negligible set. Thereafter, notice that
$$
\int_{\tilde{\beta}_{\gamma_j}\in\mathbb{R}}1_{\left\{\tilde{\beta}_{\gamma_j}\in \left[\frac{3\min(\delta,\kappa)}{10},\frac{6\min(\delta,\kappa)}{10}\right]\right\}} d\tilde{\beta}_{\gamma_j} >0
$$
for the same reason. Finally, notice that if \hbox{$\omega_j - \gamma_j \geq 1$} and for all $i\in \llbracket \gamma_j+1,\omega_j\rrbracket$ we have
$$
\tilde{\beta}_i \leq \frac{1}{\omega_j - \gamma_j}\min\left(y+\theta,\frac{\min(\delta,\kappa)}{20}\right),
$$
then it holds that
$$
\sum_{i = \gamma_j +1}^{\omega_j} \tilde{\beta}_i \leq \min\left(y+\theta,\frac{\min(\delta,\kappa)}{20}\right). 
$$
We get from these two results that there exists $c_{j,1} > 0$, independent of $y$ and $\theta$, such that 
\begin{equation}\label{eq:lower_bound_Fj_intermediate}
\begin{aligned}
F_j(y,\theta)&\geq c_{j,1} \left(\left(\int_0^{\frac{1}{\omega_j - \gamma_j}\min\left(y+\theta,\frac{\min(\delta,\kappa)}{20}\right)}du\right)^{\omega_j - \gamma_j} 1_{\{\omega_j - \gamma_j \geq 1\}} + 1_{\{\omega_j - \gamma_j = 0\}}\right)\\
&= c_{j,1}\left[\frac{1}{\omega_j - \gamma_j}\min\left(y+\theta,\frac{\min(\delta,\kappa)}{20}\right)\right]^{\omega_j - \gamma_j}1_{\{\omega_j - \gamma_j \geq 1\}} + 1_{\{\omega_j - \gamma_j = 0\}}.
\end{aligned}
\end{equation}
In addition, as $y + \theta \in[0,B_{\max}l+3r]$, we have that 
$$
\min\left(y+\theta,\frac{\min(\delta,\kappa)}{20}\right) \geq (y+\theta)\min\left(1, \frac{\min(\delta,\kappa)}{20(B_{\max}l+3r)}\right).
$$
We also have by classical computations that as $\omega_j - \gamma_j \leq m_0$, there exists $c_{j,2} > 0$ such that for all~\hbox{$x \in[0,B_{\max}l+3r]$ it holds $x^{\omega_j - \gamma_j} \geq c_{j,2}x^{m_0}$}. Combining these two results with~\eqref{eq:lower_bound_Fj_intermediate}, we finally obtain that there exists~$c_{j,3} > 0$~such that for all $y\in [0,B_{\max}l]$ and~$\theta\in[\max(-3r,y),3r]$ we have~\hbox{$F_j(y,\theta) \geq c_{j,3}(y+\theta)^{m_0}$}.

We now conclude. Let us denote $\underline{C}(t) = C_1(t)C_2\left(\prod_{j = 1}^{2k}c_{j,3}\right)$ and $C(t) = (a_0)^{m-1}\underline{C}(t)$. First, we bound from below the functions $(F_j)_{j \in \llbracket1,2k\rrbracket}$ in~\eqref{eq:doeblin_k=N_intermediate_second} by the bounds obtained above. Then, we do the changes of variables $w' = x' + w$ and $s' = t - \sum_{i = 1}^{m_0} s_i$. Finally, we use the fact that $t - 2(m_0-1) a_0 - a_0 \geq l a_0$. The following comes to end the proof:
$$
\begin{aligned}
\mathbb{E}_{(x,a)}\big[f(Z_t);t < \min(\tau_{\partial},T_{D_L})] &\geq \underline{C}(t)\int_{w'\in B(x',r)} \int_{[a_0,2a_0]} \hdots \int_{[a_0,2a_0]}\int_{\left[0, t - \sum_{i = 1}^{m_0-1} s_i - a_0\right]}\\
&\times f(w', s')\left(\prod_{j = 1}^{2k}(w'_{j})^{m_0}\right)dw'ds_1\hdots ds_{m_0-1}ds'  \\ 
&\geq C(t) \int_{w'\in B(x',r)}\int_{0}^{la_0} f(w', s')\left(\prod_{j = 1}^{2k} (w'_{j})^{m_0}\right)dw'ds'.
\end{aligned}
$$
\qed

\subsection{Proof of Lemma \ref{lemma:vicinity_velleret}}\label{subsubsect:proof_transfer_localDoeblin}

Let $l\geq 1$, $t_a = \left(l+2(m_0-1)+1\right)a_0$, and $C(t_a)$ its associated constant obtained by Lemma \ref{lemma:inequality_doeblin_velleret} (also associated to a certain $L \geq l$). We consider the constant
\begin{equation}\label{eq:doeblin_constant_diminution}
c_a = a_0.l.C(t_a)\inf_{y\in \mathcal{D}_l}\left[\int_{u\in B(y,r)\cap \mathcal{D}_l} \left(\prod_{i = 1}^{2k} (u_i)^{m_0}\right)du\right].
\end{equation}
Let $x_I\in\mathcal{D}_l$, $c > 0$, $(t,x_F)\in\mathbb{R}_+^*\times\mathcal{D}_l$ such that $(t,x_F)\in R_{l,L}^{x_I}(c)$. By the Markov property and the definition of $R_{l,L}^{x_I}(c)$, we have for all $(x,a) \in \left(B(x_I,r)\cap \mathcal{D}_l\right)\times[0,a_0l]$
$$
\begin{aligned}
&\mathbb{E}_{(x,a)}\left[f(Z_{t+t_a});\,t + t_a < \min(\tau_{\partial},T_{D_L})\right] \geq c\int_{u\in B(x_F,r)\cap \mathcal{D}_l}\int_0^{a_0l}\\
&\times \left[\int_{z\in\mathcal{X}} f(z) \mathbb{P}_{(u,s)}\left[Z_{t_a} \in dz;t_a \leq \min(\tau_{\partial},T_{D_L})\right]\right]\left(\prod_{i = 1}^{2k} (u_i)^{m_0}\right) du ds. \\
\end{aligned}
$$
Now, first apply Lemma \ref{lemma:inequality_doeblin_velleret} with $x_I = u$ and $x' = x_F$. Then, integrate in $\mathrm{d}u$ and $\mathrm{d}s$. Finally, take the infimum. It comes
$$
\begin{aligned}
&\mathbb{E}_{(x,a)}\left[f(Z_{t+t_a});\,t + t_a < \min(\tau_{\partial},T_{D_L})\right] \\ &\geq c.c_a\int_{w\in B(x_F,r)\cap\mathcal{D}_l}\int_0^{a_0l} f(w,S) \left(\prod_{i = 1}^{2k} (w_i)^{m_0}\right) dwdS.\\
\end{aligned}
$$
It remains to prove that $c_a >0$ to end the proof. Let $y\in \mathcal{D}_l$. For all $i\in\llbracket1,2k\rrbracket$, we consider $u_i\in[0,r]$. When $y_i \leq \frac{B_{\max}}{2}l$, as $r \leq \frac{B_{\max}}{2}$, we have
$$
y_i \leq y_i + u_i \leq \frac{B_{\max}}{2}l + r \leq lB_{\max}.
$$
Conversely, when $y_i > \frac{B_{\max}}{2}l$, as $r \leq \frac{B_{\max}}{2}$, we have
$$
y_i - u_i \geq y_i - r \geq \frac{B_{\max}}{2}l - r \geq 0.
$$
Hence, as $\mathcal{D}_l = [0,B_{\max}l]^{2k}$, one can easily obtain
$$
\begin{aligned}
\int_{u\in B(y,r)\cap \mathcal{D}_l} \left(\prod_{i = 1}^{2k} (u_i)^{m_0}\right)du &\geq \hspace{-0.75mm}  \prod_{\substack{i\in\llbracket 1,2k\rrbracket,\\y_i\in\left[0,\frac{B_{\max}}{2}l\right]}}\left[\int_{y_i}^{y_i+r}(u_i)^{m_0}du_i\right]\prod_{\substack{j\in\llbracket 1,2k\rrbracket,\\y_j\in\left(\frac{B_{\max}}{2}l,l\right]}}\left[\int_{y_j-r}^{y_j}(u_j)^{m_0}du_j\right]\\
&\geq \left(\int_0^r v^{m_0} dv\right)^{2k} = \left(\frac{r^{m_0+1}}{m_0+1}\right)^{2k} >0.
\end{aligned}
$$
Then, in view of \eqref{eq:doeblin_constant_diminution}, we have that 
$$
c_a \geq a_0l C(t_a)\left(\frac{r^{m_0+1}}{m_0+1}\right)^{2k} >0,
$$
which concludes the proof of this lemma.
\qed

\subsection{Proof of Proposition \ref{prop:to_prove_A1}}\label{subsect:proof_prop_to_prove_A1}

Let $l\in\mathbb{N}^*$ and $L\geq l$ sufficiently large to apply Lemma~\ref{lemma:vicinity_velleret}. We begin with two preliminary results. First, by compactness, there exist $J\in\mathbb{N}^*$ and $(x^j)_{1\leq j\leq J}\in (\mathcal{D}_l)^J$ such that
$$
\bigcup_{j\in\llbracket1,J\rrbracket}\left[B(x^j,r)\cap \mathcal{D}_l\right] = \mathcal{D}_l.
$$
Second, by Lemma \ref{lemma:inequality_doeblin_velleret}, there exists $(t_I,c_I)\in\left(\mathbb{R}_+^*\right)^2$ such that for all $j\in\llbracket1,J\rrbracket$ it holds~ $(t_I,x^{j})\in R^{x^{j}}_{l,L}(c_I)$. 
\smallskip

\noindent Now, we fix $(j,j')\in\llbracket 1, J\rrbracket^2$ such that $j\neq j'$. As $(t_I,x^{j})\in R^{x^{j}}_{l,L}(c_I)$, we need to successively apply Lemma~\ref{lemma:vicinity_velleret} to obtain that there exist $t_F,\,c_F >0$ such that $(t_F,x^{j'})\in R^{x^{j}}_{l,L}(c_F)$. Let~$N\in \mathbb{N}^*$ satisfying 
$$
N > \underset{(j,j')\in\llbracket1,J\rrbracket^2}{\max}\frac{||x^{j'} - x^j||_{\infty}}{r},
$$
and $(u_\mathbb{k})_{\mathbb{k}\in\llbracket0,N\rrbracket}$ a sequence defined as
$$
\begin{cases}
u_0 = x^{j}, &  \\
u_{\mathbb{k}+1} = u_\mathbb{k} + \frac{1}{N}(x^{j'} - x^{j}), & \forall \mathbb{k}\in\llbracket0,N-1\rrbracket.
\end{cases}
$$
One can easily see that as $(u_{\mathbb{k}})_{\mathbb{k}\in\llbracket0,N\rrbracket}$ is an arithmetic sequence: $u_{N} = x_{j'}$. 

For all $\mathbb{k}\in\llbracket0,N-1\rrbracket$, we successively apply Lemma~\ref{lemma:vicinity_velleret} with $x_I = u_{\mathbb{k}}$ and $x_F = u_{\mathbb{k}+1}$. At each iteration, as $u_{\mathbb{k}}\in R^{x^{j}}_{l,L}\left(c_I(c_a)^{\mathbb{k}}\right)$ and $u_{\mathbb{k}+1}\in (B(u_\mathbb{k},r)\cap \mathcal{D}_l)$, we have 
$$
(t_0 + (\mathbb{k}+1)t_a,u_{\mathbb{k}+1})\in R^{x^{j}}_{l,L}\left(c_I(c_a)^{\mathbb{k}+1}\right). 
$$
As $u_{N} = x^{j'}$, we finally obtain that $\big(t_0 + Nt_a,x^{j'}\big)\in R^{x^{j}}_{l,L}\left(c_I(c_a)^{N}\right)$. Hence, there exist two positive constants $t_F = t_I + Nt_a$ and $c_F = c_I(c_a)^{N}$ such that for all $(j,j')\in\llbracket1,J\rrbracket^2$ it holds $\big(t_F,x^{j'}\big)\in R_{l,L}^{x^{j}}(c_F)$. In particular, for all $(j,j')\in\llbracket 1,J\rrbracket^2$, \hbox{$(x,a)\in \left(B(x^{j},r)\cap \mathcal{D}_l\right)\times[0,a_0l]$}, we have
$$
\mathbb{P}_{(x,a)}\left[Z_{t_F}\in dx'da';\,\min\left(\tau_{\partial},\,T_{D_L}\right)\right] \geq c_F\left(\prod_{i = 1}^{2k} (x'_i)^{m_0}\right)1_{\{B(x^{j'},r)\cap\mathcal{D}_l\times[0,a_0l]\}}(x',a')dx'da'.
$$
As $\underset{j'\in\llbracket1,J\rrbracket}{\bigcup}\big[B(x^{j'},r)\cap \mathcal{D}_l\big] = \mathcal{D}_l$, summing in $j'$ in the above equation yields that Proposition~\ref{prop:to_prove_A1} is proved.
\qed 

\section{Proof of Proposition~\ref{prop:prop_to_prove_(A3)}}\label{appendix:proof_A3_stepB}

We consider in this section $\left(\mathcal{H}_{n}\right)_{n\geq 0}$ the filtration that corresponds to the augmented filtration generated by the process $\left(X_n,A_n,\mathdutchcal{T}_n,I_n,J_n\right)_{n\in\mathbb{N}}$. It contains all the information we have at the $n$-th generation. Each time we condition with respect to this filtration in this section, we work on an event of the form $\{\mathdutchcal{T}_l \leq t\}$, where $l\in\mathbb{N}$ and $t >0$, that belongs to $\mathcal{G}_t$. Then, this conditioning does not pose any issues with the fact that we work with the filtration $(\mathcal{G}_t)_{t\geq0}$. 

The proof of Proposition~\ref{prop:prop_to_prove_(A3)} requires auxiliary statements to be introduced. In Section~\ref{subsect:proofA3_give_auxliary_and_proof}, we present these auxiliary statements and prove the proposition. Then, we prove all the auxiliary statements in Section~\ref{subsect:proof_auxiliaries_complicatedpart}.

\subsection{Auxiliary statements and proof of Proposition~\ref{prop:prop_to_prove_(A3)}}\label{subsect:proofA3_give_auxliary_and_proof}

\subsubsection{Auxiliary statements}
\noindent \textit{1. Decomposition of the measure}\ 

To obtain Proposition~\ref{prop:prop_to_prove_(A3)}, we first use the following statement, that allows us to decompose the measure on the left-hand side term of~\eqref{eq:eq_to_prove_(A3)}. We prove this statement in Section~\ref{subsubsect:proof_auxiliary_complicatedpart_first}.
\begin{lemm}[Decomposition of the measure]\label{lemma:preliminaries_prop_to_prove_(A3)}
Assume that \hyperlink{paragraph:long_time_behaviour_S1.1}{$(S_{1.1})$} and \hyperlink{paragraph:long_time_behaviour_S1.3}{$(S_{1.3})$} hold. Then the following statement holds for all non-negative function $f\in M_b(\mathcal{X})$, $t >0$, $(x,a)\in E$
\begin{equation}\label{eq:to_prove_A3_intermediate_second}
\begin{aligned}
\mathbb{E}_{(x,a)}\left[f\left(Z_{t}\right);\,T_{all} \leq t < \tau_{\partial}, N_{t} \leq n(t)\right] &\leq \sum_{l = 1}^{n(t)}\sum_{(i_1,i_2,\hdots,i_{2k})\in\llbracket1,l\rrbracket^{2k}} \mathbb{E}_{(x,a)}\Big[f\left(X_l,t-\mathdutchcal{T}_l\right);\\
&\,X_l\in [0,B_{\max}L_1+l\Delta]^{2k},\,\mathdutchcal{T}_l \leq t,\\
&\,\forall p\in\llbracket1,l\rrbracket: (I_{p},J_p)\neq \left(\partial,\partial\right),\\ 
&\forall j\in\llbracket1,2k\rrbracket : j \in I_{i_j}\cup J_{i_j}\Big].
\end{aligned}
\end{equation} 
\end{lemm}
\noindent We now give statements allowing to bound from above the right-hand side term of \eqref{eq:to_prove_A3_intermediate_second}. 

\medskip

\noindent \textit{2. Bound from above jump by jump}\

Let $t\geq0$. Bounding from above the right-hand side term of \eqref{eq:to_prove_A3_intermediate_second} can be done by successively conditioning with respect to $\mathcal{H}_{n-1}$, for all $n\in\llbracket1,n(t)\rrbracket$. This reduces the problem to obtaining an upper bound jump by jump. Let $l\in\llbracket1,n(t)\rrbracket$ and \hbox{$i=(i_1,\hdots, i_{2k})\in\llbracket1,l\rrbracket^{2k}$} that we keep fixed until the end of the subsection. We introduce:
\begin{itemize}[leftmargin=*]
\item The following set for all $m\in\llbracket1,l\rrbracket$
$$
\mathcal{C}_{m,i} := \{j\in\llbracket1,2k\rrbracket,\,i_j = m\}. 
$$
This set contains the coordinates for which we are sure that on the event 
$$
\left\{\forall j\in\llbracket1,2k\rrbracket : j \in I_{i_j}\cup J_{i_j}\right\},
$$
there is a shortening or a lengthening at the $m$-th jump.

\item For all $(I,J,j)\in\mathcal{Q}_{k}\times \llbracket1,2k\rrbracket$, the measure $\nu_{I,J,j}$, defined for all $A\in\mathcal{B}\left(\mathbb{R}\right)$ as 
$$
\begin{aligned}
\nu_{I,J,j}(A) := \begin{cases}
\delta_0(A), &\text{ if }j\notin I\cup J,  \\ 
\int_{u\in[-\delta,\Delta]} 1_A(u)du, & \text{ otherwise.} 
\end{cases}
\end{aligned}
$$
This measure partially composes the measure we use to bound from above the right-hand side term of~\eqref{eq:to_prove_A3_intermediate_second}.

\item For all $m\in\llbracket1,l\rrbracket$, the measures $\mu_{m,i}^1$, $\mu_{m}^2$ and $\mu_{m,i}$, defined such that for all $A~\in~\mathcal{B}(\mathbb{R}_+^{2k})$  
\begin{equation}\label{eq:dft_measure_bound_A3F}
\begin{aligned}
\mu_{m,i}^1(A) &:= \sum_{\substack{(I,J)\in\mathcal{Q}_{k},\\ \mathcal{C}_{m,i} \subset \left(I\cup J\right)}}\left[\int_A \left(\prod_{j \in \mathcal{C}_{m,i}} d\omega_j\right)\left(\prod_{j \notin \mathcal{C}_{m,i}} d\nu_{I,J,j}(w_{j})\right)\right], \\ 
\mu_m^2(A) &:= \sum_{(I,J)\in\mathcal{Q}_{k}} \int_A \prod_{j \in \llbracket1,2k\rrbracket} d\nu_{I,J,j}(w_{j}),
\end{aligned}
\end{equation}
and
$$
\mu_{m,i}(A) := \begin{cases}
\mu_{m,i}^1(A), &\text{ if }\mathcal{C}_{m,i} \neq \emptyset, \\
\mu_m^2(A), &\text{ otherwise.}
\end{cases}
$$
These measures are used to bound from above the right-hand side term of~\eqref{eq:to_prove_A3_intermediate_second}.
\end{itemize}
The following two lemmas are useful to show how we obtain an upper bound at one jump. The proof of Lemma~\ref{lemma:inequality_control_mass} is given in Section~\ref{subsubsect:proof_auxiliary_complicatedpart_second}. The proof of Lemma \ref{lemma:inequality_control_mass_empty} is not given as it follows the same lines of proof.
\begin{lemm}[Upper bound when $\mathcal{C}_{m,i} \neq\emptyset$]\label{lemma:inequality_control_mass}
Assume that \hyperlink{paragraph:long_time_behaviour_S1.1}{$(S_{1.1})$} and \hyperlink{paragraph:long_time_behaviour_S1.3}{$(S_{1.3})$} hold. Then for all~$l\in\mathbb{N}^*$, there exists $\overline{C_1}(l) >0$ such that for all $t\geq0$, $(m,i)\in\llbracket1, l\rrbracket\times\llbracket1,l\rrbracket^{2k}$ verifying~$\mathcal{C}_{m,i} \neq \emptyset$, $v\in\left[-\delta l, \Delta l\right]^{2k}$ and $f\in M_b(\mathcal{X})$ non-negative
$$
\begin{aligned}
&\mathbb{E}_{(x,a)}\Big[f\left(X_m  + v,t-\mathdutchcal{T}_m\right);\,X_m + v\in [0,B_{\max}L_1+l\Delta]^{2k}, \mathdutchcal{T}_m \leq t,\\
& \left(I_{m},J_m\right) \neq \left(\partial,\partial\right),\,\mathcal{C}_{m,i}\subset (I_m\cup J_m)\,|\,\mathcal{H}_{m-1}\Big]\\
&\leq \overline{C_1}(l)\int_{u\in [-\delta,\Delta]^{2k}}\int_{s\in[0,t-\mathdutchcal{T}_{m-1}]} f\left(X_{m-1} + u + v,t-s- \mathdutchcal{T}_{m-1}\right)\\
&\times 1_{\left\{X_{m-1} + u + v\in[0,B_{\max}L_1+l\Delta]^{2k}\right\}} d\mu_{m,i}^1(u) ds.
\end{aligned}
$$
\end{lemm}
\begin{lemm}[Upper bound when $\mathcal{C}_{m,i} =\emptyset$]\label{lemma:inequality_control_mass_empty}
Assume that \hyperlink{paragraph:long_time_behaviour_S1.1}{$(S_{1.1})$} and \hyperlink{paragraph:long_time_behaviour_S1.3}{$(S_{1.3})$} hold. Then, for all~$l\in\mathbb{N}^*$, there exists $\overline{C_2}(l) >0$ such that for all $t\geq 0$, $(m,i)\in\llbracket1, l\rrbracket\times\llbracket1,l\rrbracket^{2k}$ verifying~$\mathcal{C}_{m,i} =\emptyset$, $v\in\left[-\delta l, \Delta l\right]^{2k}$, and $f\in M_b(\mathcal{X})$ non-negative
$$
\begin{aligned}
&\mathbb{E}\Big[f\left(X_m + v,t -\mathdutchcal{T}_m \right);X_m + v\in [0,B_{\max}L_1+l\Delta]^{2k},\mathdutchcal{T}_m \leq t,\,\left(I_{m},J_m\right) \neq \left(\partial,\partial\right)\,|\,\mathcal{H}_{m-1}\Big]\\
&\leq \overline{C_2}(l) \int_{u\in [-\delta,\Delta]^{2k}} \int_{s\in[0,t-\mathdutchcal{T}_{m-1}]}f\left(X_{m-1} +u+v,t-s- \mathdutchcal{T}_{m-1}\right)\\
&\times 1_{\left\{X_{m-1} + u + v\in[0,B_{\max}L_1+l\Delta]^{2k}\right\}} d\mu_m^2(u) ds.
\end{aligned}
$$
\end{lemm}
\noindent We now have all the auxiliary statements required to prove Proposition~\ref{prop:prop_to_prove_(A3)}.
\subsubsection{Proof of Proposition~\ref{prop:prop_to_prove_(A3)}}
Let $t\geq 0$, $l\in\llbracket1,n(t)\rrbracket$, $i = (i_1,\hdots,i_{2k})\in\llbracket1,l\rrbracket^{2k}$. To obtain an upper bound for the right-hand side term of \eqref{eq:to_prove_A3_intermediate_second}, we successively condition with respect to $\mathcal{H}_{n-1}$ ($n\in \llbracket1,l\rrbracket$). At each conditioning, we apply Lemma \ref{lemma:inequality_control_mass} or Lemma \ref{lemma:inequality_control_mass_empty}. When \hbox{$\mathcal{C}_{n,i}\neq \emptyset$}, we apply Lemma~\ref{lemma:inequality_control_mass}, and when $\mathcal{C}_{n,i} = \emptyset$, we apply Lemma \ref{lemma:inequality_control_mass_empty}. For example, if we suppose that~$\mathcal{C}_{l,i} \neq \emptyset$, then applying Lemma~\ref{lemma:inequality_control_mass} and using Tonelli's theorem yields
\begin{align}
& \mathbb{E}_{(x,a)}\Big[f\left(X_{l},t-\mathdutchcal{T}_l\right);\,X_l\in[0,B_{\max}L_1+l\Delta]^{2k},\,\mathdutchcal{T}_l \leq t,\,\forall p\in\llbracket1,l\rrbracket: \left(I_{p},J_{p}\right) \neq \left(\partial,\partial\right), \nonumber\\
&\forall j\in\llbracket1,2k\rrbracket:\,j \in I_{i_j}\cup J_{i_j}\Big] \leq \overline{C_1}(l)\int_{(u,s)\in [-\delta,\Delta]^{2k}\times\mathbb{R}_+}  \nonumber\\
&\times \mathbb{E}_{(x,a)}\Big[f\left(X_{l-1} + u,t - s -\mathdutchcal{T}_{l-1}\right)1_{\left\{X_{l-1} + u\in[0,B_{\max}L_1+l\Delta]^{2k}\right\}}\,;\,s+\mathdutchcal{T}_{l-1}\leq t, \nonumber\\
&\forall p\in\llbracket1,l-1\rrbracket:\,\left(I_{p},J_{p}\right) \neq \left(\partial,\partial\right),\,\forall j\in\llbracket1,2k\rrbracket\text{ s.t. }\,i_j\leq l-1:\,j \in I_{i_j}\cup J_{i_j}\Big] d\mu_{l,i}^1(u) ds. \label{eq:proof_prop_A3_intermediate}
\end{align}
We continue to iterate this for $X_{l-1}$, $X_{l-2}$ $\cdots$ We obtain at the end an upper bound for the left-hand side term of \eqref{eq:proof_prop_A3_intermediate}. We plug this upper bound in~\eqref{eq:to_prove_A3_intermediate_second}. We get
$$
\begin{aligned}
&\mathbb{E}_{(x,a)}\Big[f\left(Z_{t}\right);\,T_{all}  < t < \tau_{\partial}, N_{t} \leq n(t)\Big] \leq \sum_{l = 1}^{n(t)}\sum_{i\in\llbracket1,l\rrbracket^{2k}}\left(\max_{m_1,m_2 \in\llbracket1,l\rrbracket} \overline{C_1}(l)^{m_1}\overline{C_2}(l)^{m_2}\right)  \\
&\times\int_{u_1\in [-\delta,\Delta]^{2k}}\hdots\int_{u_l\in [-\delta,\Delta]^{2k}}\int_{s_1\in [0,t]}\hdots\int_{s_l\in [t - \sum_{r = 1}^{l-1} s_i]}f\left(x +\sum_{r = 1}^{l} u_r,t - \sum_{r = 1}^{l} s_r\right) \\
&\times 1_{\left\{x +\sum_{r = 1}^{l} u_r\in[0,B_{\max}L_1+l\Delta]^{2k}\right\}} d\mu_{1,i}(u_1)\hdots d\mu_{l,i}(u_l)ds_1\hdots ds_l. \\
\end{aligned}
$$
For each coordinate, we have, up to a constant, the convolution of Dirac measures and Lebesgue measures with at least one Lebesgue measure. A convolution of this type has a bounded density with respect to the Lebesgue measure. The number of possible combinations we have for these convolutions is finite. Therefore, we take the maximum among all the bounds of these convolutions with respect to the Lebesgue measure. We obtain that there exists $\underline{C}(n(t)) >0$ such that
\begin{equation}\label{eq:proof_prop_A3_intermediate_second}
\begin{aligned}
\mathbb{E}_{(x,a)}[f\left(Z_{t}\right);T_{all}  < t < \tau_{\partial}, N_{t} \leq n(t)] &\leq \underline{C}(n(t))\sum_{l = 1}^{n(t)}\sum_{i\in\llbracket1,l\rrbracket^{2k}}\left(\max_{m_1,m_2 \in\llbracket1,l\rrbracket} \overline{C_1}(l)^{m_1}\overline{C_2}(l)^{m_2}\right)  \\
&\times\int_{v\in [-n(t)\delta,n(t)\Delta]^{2k}}\int_{s_1\in [0,t]}\hdots\int_{s_l\in [t - \sum_{r = 1}^{l-1} s_i]}\\
&\times f\left(x + v,t - \sum_{r = 1}^{l} s_r\right) 1_{\left\{x +v\in[0,B_{\max}L_1+l\Delta]^{2k}\right\}} dv ds. 
\end{aligned}
\end{equation}
The value of $\underline{C}(n(t))$ increases with the value of $n(t)$, as the cardinal of the set of all the possible cases increases when $n(t)$ increases. We now first denote 
$$
\overline{C}(n(t)) = \underline{C}(n(t))(n(t))^{2k}\left(\max_{l\in\llbracket1,n(t)\rrbracket}\max_{m_1,m_2 \in\llbracket1,l\rrbracket} \overline{C_1}(l)^{m_1}\overline{C_2}(l)^{m_2}\right),
$$
that increases when $n(t)$ increases. Then, we bound from above the term before the integrals in~\eqref{eq:proof_prop_A3_intermediate_second} by~$\sum_{l = 1}^{n(t)}\overline{C}(n(t))$. Thereafter, we do the changes of variables~$w = x+v$ and $s' = t-\sum_{r = 1}^{l} s_r$ (for the variable~$s_l$), and extend every domain of integration of the integral with respect to the variables~$(s_i)_{i\in\llbracket1,l-1\rrbracket}$ to $[0,t]$. Finally, we compute the integrals with respect to the measures $(ds_i)_{i\in\llbracket1,l-1\rrbracket}$. It comes, denoting the constant~$C(t) = \overline{C}(n(t))\left(\sum_{l = 1}^{n(t)}t^{l-1}\right)$,
$$
\begin{aligned}
\mathbb{E}_{(x,a)}\big[f\left(Z_{t}\right);T_{all}  < t < \tau_{\partial}, N_{t} \leq n(t)\big] \leq C(t)\int_{[0,B_{\max}L_1+n(t)\Delta]^{2k}} \int_{[0, t]} f\left(w,s'\right) dwds'.
\end{aligned}
$$
From the above, the proposition is proved because $C(t)$ increases when~$t$ increases~(classical result combined with Lemma~\ref{lemma:inequality_generation_mixing}).
\qed

\subsection{Proof of the auxiliary statements}\label{subsect:proof_auxiliaries_complicatedpart}
\subsubsection{Proof of Lemma \ref{lemma:preliminaries_prop_to_prove_(A3)}}\label{subsubsect:proof_auxiliary_complicatedpart_first}
Let $t >0$, $(x,a)\in E = D_{L_1}$ the initial condition of our process and let $f$ a non-negative measurable function. In the left-hand side term of \eqref{eq:to_prove_A3_intermediate_second}, first use the representation of~$Z_t$ presented in \eqref{eq:expression_particle}. Then, use the fact that 
$$
t < \tau_{\partial} \Longleftrightarrow \forall p\in\llbracket1,N_t\rrbracket: \left(I_{p},J_p\right) \neq \left(\partial,\partial\right).
$$
Finally, develop all the values that $N_t$ can take (necessarily, it holds $N_t \geq  1$ on the event~$\{T_{all} \leq t\}$). It comes
$$
\begin{aligned}
\mathbb{E}_{(x,a)}\Big[f\left(Z_{t}\right);\,T_{all} \leq t < \tau_{\partial}, N_{t} \leq n(t)\Big] &= \sum_{l = 1}^{n(t)} \mathbb{E}_{(x,a)}\Big[f\left(X_{l}, t-\mathdutchcal{T}_{l}\right)\,;\,T_{all} \leq t,\,N_{t} = l,\, \\
&\forall p\in\llbracket1,l\rrbracket: \left(I_{p},J_p\right) \neq \left(\partial,\partial\right)\Big]. \\
\end{aligned}
$$
On the event $\{T_{all} \leq t \}$, we introduce for every $j\in\llbracket1,2k\rrbracket$ the random variable $\mathdutchcal{P}_{t,j}$ that represents the last event where a jump has occured in the coordinate $j$ during the first~$N_{t}$ jumps. Formally, we define it as
$$
\mathdutchcal{P}_{t,j} = \max\left\{p\in\llbracket1,N_{t}\rrbracket,\, j\in I_p\cup J_p\right\}.
$$
By decomposing all the possible values that $(\mathdutchcal{P}_{t,j})_{j\in\llbracket1,2k\rrbracket}$ can take, we have
\begin{equation}\label{eq:to_prove_A3_intermediate_first}
\begin{aligned}
\mathbb{E}_{(x,a)}\big[f\left(Z_{t}\right);\,T_{all} \leq t < \tau_{\partial}, N_{t} \leq n(t)\big] &= \sum_{l = 1}^{n(t)}\sum_{(i_1,\hdots,i_{2k})\in\llbracket1,l\rrbracket^{2k}} \mathbb{E}_{(x,a)}\Big[f\left(X_{l}, t-\mathdutchcal{T}_{l}\right);\,T_{all} \leq t,\,\\
&\,N_{t} = l,\,\forall p\in\llbracket1,l\rrbracket: \left(I_{p},J_p\right) \neq \left(\partial,\partial\right),\\
&\,\forall j\in\llbracket1,2k\rrbracket\,:\,\mathdutchcal{P}_{t,j} = i_j\,\Big]. 
\end{aligned}
\end{equation}
We need to prove the following inclusion for all $l\in\llbracket1,n(t)\rrbracket$, $(i_1,i_2,\hdots,i_{2k})\in\llbracket1,l\rrbracket^{2k}$:
$$
\begin{aligned}
&\left\{T_{all} \leq t,\,N_{t} = l,\,\forall p\in\llbracket1,l\rrbracket: \left(I_{p},J_p\right) \neq \left(\partial,\partial\right),\,\forall j\in\llbracket1,2k\rrbracket\,:\,\mathdutchcal{P}_{t,j} = i_j\right\} \\
&\,\subset\left\{X_{l}\in [0,B_{\max}L_1+l\Delta]^{2k},\,\,\mathdutchcal{T}_l \leq t,\,\forall p\in\llbracket1,l\rrbracket:\left(I_{p},J_p\right) \neq \left(\partial,\partial\right),\right.\\
&\left.\hspace{6.95mm}\forall j\in\llbracket1,2k\rrbracket:\,j \in (I_{i_j}\cup J_{i_j})\,\right\}.
\end{aligned}
$$
Indeed, the latter allows us to easily bound from above the second term in \eqref{eq:to_prove_A3_intermediate_first} to obtain~\eqref{eq:to_prove_A3_intermediate_second}. 

Let us say why the inclusion is true. First, as the initial condition for telomere length is $x\in[0,B_{\max}L_1]^{2k}$ and as the maximum lengthening value is $\Delta$, we have \hbox{$X_{l}\in[0,B_{\max}L_1+l\Delta]^{2k}$}. Second, as $N_t = l$, we easily have that $\mathdutchcal{T}_l \leq t$. Finally, as~$\mathdutchcal{P}_{t,j}$ corresponds to  the last jump in the $j-$th coordinate, we necessarily have a jump in the $j-th$ coordinate on the event $\{\mathdutchcal{P}_{t,j} = i_j\}$ (or formally, $\mathdutchcal{P}_{t,j} = i_j \Longrightarrow j \in (I_{i_j}\cup J_{i_j})$). From these three points, it comes that the inclusion is true, and the lemma is proved.
\qed 
\subsubsection{Proof of Lemma \ref{lemma:inequality_control_mass}}\label{subsubsect:proof_auxiliary_complicatedpart_second}
Let $m\leq l\in\mathbb{N}^*$, $i=(i_1,\,i_2,\hdots,\,i_{2k})\in\llbracket1,2k\rrbracket^{l}$ such that $\mathcal{C}_{m,i} = \{j\in\llbracket1,2k\rrbracket,\,i_j = m\}$ is not empty. Throughout the proof, $f$ is a non-negative measurable function. We consider~$(\tilde{Z}_t)_{t\geq 0}$ a process with the same distribution as $(Z_t)_{t\geq 0}$, and independent of it. We also introduce the random variables $\tilde{X}_1$, $\tilde{\mathdutchcal{T}}_1$ and~$(\tilde{I}_1,\tilde{J}_1)$, that are the equivalents of $X_1$, $\mathdutchcal{T}_1$ and $(I_1,J_1)$ for the process $(\tilde{Z}_t)_{t\geq 0}$. We finally define for all~$(x,w)\in\mathbb{R}_+^{2k}\times\mathbb{R}_+$
$$
\begin{aligned}
&h_{\text{cond}}(x,w) = \mathbb{E}_{\left(x,1_{m = 1}a\right)}\left[f\big(\tilde{X}_{1}+ v,t-w-\tilde{\mathdutchcal{T}}_1\big);\tilde{X}_{1} + v\in[0,B_{\max}L_1+l\Delta]^{2k},\right.\\
&\left. w+\tilde{\mathdutchcal{T}}_1 \leq t,\,\left(I_{1},J_1\right) \neq \left(\partial,\partial\right),\,\mathcal{C}_{m,i} \subset (I_1\cup J_1)\right].
\end{aligned}
$$
Thanks to the Markov property we have
\begin{equation}\label{eq:markov_bound_lebesgue}
\begin{aligned}
&\mathbb{E}_{(x,a)}\left[f\left(X_m  + v,t - \mathdutchcal{T}_{m}\right);\,X_m + v\in[0,B_{\max}L_1+l\Delta]^{2k}, \mathdutchcal{T}_{m} \leq t,\,\left(I_{m},J_m\right) \neq \left(\partial,\partial\right),\right.\\
&\,\left.\forall j\in\llbracket1,2k\rrbracket\text{ s.t. }i_j = m \,:\,j \in I_m\cup J_m\,|\,\mathcal{H}_{m-1}\right] = h_{\text{cond}}\left(X_{m-1},\mathdutchcal{T}_{m-1}\right).
\end{aligned}
\end{equation}
Thus, our goal is to obtain an upper bound for $h_{\text{cond}}$. By summing the different events that can occur, we first have
$$
\begin{aligned}
h_{\text{cond}}(x,w) &= \sum_{\substack{(I,J)\in\mathcal{Q}_{k},\\\mathcal{C}_{m,i} \subset \left(I\cup J\right)}} \mathbb{E}_{\left(x,1_{m = 1}a\right)}\left[f\big(\tilde{X}_{1}+ v,t-w-\tilde{\mathdutchcal{T}}_1\big);\right.\\
&\left.\tilde{X}_{1} + v\in[0,B_{\max}L_1+l\Delta]^{2k},\,w+\tilde{\mathdutchcal{T}}_1 \leq t,\,\big(\tilde{I}_{1},\tilde{J}_1\big)= (I,J)\right].
\end{aligned}
$$
Now, we develop the expectation by slightly reajusting Lemma~\ref{lemm:generalized_duhamel_n=1} (we do not have a term $\overline{\mathdutchcal{H}}_0$, as we do not have a condition for $\tilde{\mathdutchcal{T}}_2 - \tilde{\mathdutchcal{T}}_1$). Then, in view of~\eqref{eq:density_events} and~\eqref{eq:useful_notation_third}, we use the second statement of Lemma~\ref{lemma:inequalities_psi} to bound from above $\mathdutchcal{G}_a$ by $\mathdutchcal{H}_a$. It comes
$$
\begin{aligned}
h_{\text{cond}}(x,w) &\leq \sum_{\substack{(I,J)\in\mathcal{Q}_{k},\\ \mathcal{C}_{m,i} \subset \left(I\cup J\right)}} \int_{u\in\mathbb{R}^{2k}}\int_{s\in[0,t-w]}f(x+u+v,t-w-s)\\
&\times1_{\{x+u+v\in[0,B_{\max}L_1+l\Delta]^{2k}\}}\mathdutchcal{V}(x+u)d\pi^{I,J}_x(u)\mathdutchcal{H}_{1_{m = 1}a}(x,s)ds. 
\end{aligned}
$$
In view of \eqref{eq:density_events}, \eqref{eq:birth_rate_assumption}, the inequalities 
$$
\frac{\frac{\partial \psi}{\partial a}(x,1_{m = 1}a+s)}{\psi(x,1_{m = 1}a+s)} \geq 0 \hspace{2mm}\text{ and }\hspace{2mm}a \leq a_0L_1,
$$
the fact that 
$$
\int_0^t \frac{\frac{\partial \psi}{\partial a}(x,1_{m = 1}a+s)}{\psi(x,1_{m = 1}a+s)}ds = \ln\left(\frac{\psi(x,1_{m = 1}a+t)}{\psi(x,1_{m = 1}a)}\right),
$$
and the third statement of Lemma \ref{lemma:inequalities_psi}, we have
$$
\begin{aligned}
\mathdutchcal{H}_{1_{m = 1}a}(x,s) &\leq \left[\lambda_{\psi} + \tilde{b}(1+(a_0L_1+s)^{d_b})\right]\left[\overline{\psi}\times\left(1+s^{d_{\psi}}\right)\right]\exp\left(-\lambda_{\psi}s\right). \\
\end{aligned}
$$
Thus, there exists a constant $C >0$ such that $\mathdutchcal{H}_{1_{m = 1}a}(x,s) \leq C$ and 
\begin{equation}\label{eq:intermediate_bound_lebesgue}
\begin{aligned}
h_{\text{cond}}(x,w) &\leq C\sum_{\substack{(I,J)\in\mathcal{Q}_{k},\\ \mathcal{C}_{m,i} \subset \left(I\cup J\right)}} \int_{u\in\mathbb{R}^{2k}}\int_{s\in[0,t-w]}f(x+u+v,t-w-s)\\
&\times1_{\{x+u+v\in[0,B_{\max}L_1+l\Delta]^{2k}\}}\mathdutchcal{V}(x+u)d\pi^{I,J}_x(u)ds. 
\end{aligned}
\end{equation}
First, we bound from above the term $\mathdutchcal{V}(x+u)$ in \eqref{eq:intermediate_bound_lebesgue} using the third statement of \hyperlink{paragraph:long_time_behaviour_S2.2}{$(S_{2.2})$}. Second, we use the fact that by the definition of $\pi_x^{I,J}$ given in~\eqref{eq:measure_by_event} and Assumption \hyperlink{paragraph:long_time_behaviour_S1.3}{$(S_{1.3})$}, we have in each coordinate of the measure $\pi_x^{I,J}$ either the Dirac measure or, up to a constant, the Lebesgue measure. Finally, we use the fact that for all $j\in \mathcal{C}_{m,i}$, $(I,J)\in\mathcal{Q}_{k}$ such that $j \in I\cup J$, we are sure that we have a jump in the coordinate $j$, so we are sure that the measure $\pi_x^{I,J}$ in the coordinate $j$ is, up to a constant, the Lebesgue measure. Hence, in view of the definition of $\mu_{m,i}^1$ given in~\eqref{eq:dft_measure_bound_A3F}, there exists $C' > 0$ such that
$$
\begin{aligned}
h_{\text{cond}}(x,w)  &\leq C' \sum_{\substack{(I,J)\in\mathcal{Q}_{k},\\ \mathcal{C}_{m,i} \subset \left(I\cup J\right)}}\int_{u\in\mathbb{R}^{2k}}\int_{s\in[0,t-w]}f(x+u+v,t-w-s)\\
&\times1_{\{x+u+v\in[0,B_{\max}L_1+l\Delta]^{2k}\}}d\mu_{m,i}^1(u) ds.
\end{aligned} 
$$
Plugging the above equation in \eqref{eq:markov_bound_lebesgue} ends the proof.
\qed

% add below the content of your .bbl file produced by bibtex.
\providecommand{\bysame}{\leavevmode\hbox to3em{\hrulefill}\thinspace}
\providecommand{\MR}{\relax\ifhmode\unskip\space\fi MR }
\providecommand{\MRhref}[2]{%
\href{http://www.ams.org/mathscinet-getitem?mr=#1}{#2}
}
\providecommand{\href}[2]{#2}

\paragraph{Acknowledgments.} This work was partially funded by the Fondation Mathématique Jacques Hadamard, and the European Union ERC-2024-COG MUSEUM-101170884. Views and opinions expressed are however those of the author(s) only and do not necessarily reflect those of the European Union or the European Research Council Executive Agency (ERCEA). Neither the European Union nor the granting authority can be held responsible for them. We warmly thank Marie Doumic for discussions and proofreadings linked to this work. We also thank Aurélien Velleret for discussions and clarifications about Assumption \hyperlink{te:assumptions_velleret_A3}{$(A_3)_F$}, and Denis Villemonais for discussions that allowed us to relax conditions for the existence of a Lyapunov function. We finally thank the anonymous reviewers for their numerous and insightful comments which helped us to improve the quality of the manuscript.

\end{document}